\documentclass[reqno]{amsart}
\usepackage{amsmath}
\usepackage[colorlinks,linkcolor=blue,anchorcolor=green,citecolor=blue]{hyperref}
\usepackage{hyperref}

\usepackage[usenames, dvipsnames]{xcolor}
\usepackage{soul} \usepackage{dsfont}
\usepackage{amsfonts}
\usepackage{amssymb}
\usepackage{amstext}
\usepackage{amsbsy}
\usepackage{amsopn}
\usepackage{amsthm}
\usepackage{enumerate}
\usepackage{amsxtra}
\usepackage{upref}
\usepackage{graphicx}
\usepackage{epstopdf}
\usepackage{mathrsfs}
\usepackage{mathtools}
\usepackage[showseconds=false,showzone=false]{datetime2}

\DeclareFontFamily{OML}{rsfs}{\skewchar\font'177}
\DeclareFontShape{OML}{rsfs}{m}{n}{ <5> <6> rsfs5 <7> <8> <9> rsfs7
  <10> <10.95> <12> <14.4> <17.28> <20.74> <24.88> rsfs10 }{}
\DeclareMathAlphabet{\mathfs}{OML}{rsfs}{m}{n}

\newcommand{\x}{\ensuremath{\underline{x}}}

\newtheorem{maintheorem}{Theorem}

\newtheorem{theorem}{Theorem}[section]
\newtheorem{addendum}[theorem]{Addendum}

\newtheorem{proposition}[theorem]{Proposition}
\newtheorem{lemma}[theorem]{Lemma}
\newtheorem{corollary}[theorem]{Corollary}

\newtheorem{definition}[theorem]{Definition}
\newtheorem{claim}[theorem]{Claim}
\newtheorem*{claim*}{Claim}
\numberwithin{equation}{section}

\theoremstyle{definition}
\newtheorem{remark}[theorem]{Remark}

\newtheorem{example}[theorem]{Example}

\renewcommand{\mod}{\mbox{$\,\mathrm{mod}\,$}}

\renewcommand{\epsilon}{\varepsilon}
\def\wt{\widetilde}

\def\text#1{\textrm{#1}}

\def\emptyset{\varnothing}

\def\vf{\varphi}

\def\L{\Lambda}

\def\cW{\mathcal W}
\def\cD{\mathcal D}

\def\x{\times}
\def \R{\mathbb R}
\def \N{{\mathbb N}}
\def \Z{\mathbb Z}
\def \T{\mathbb T}

\def\ov{\overline}
\def\un{\underline}

\def\C{\mathbb C}

\def\wh{\widehat}
\def\({\biggl(}
\def\){\biggr)}

\def\<{\bold\langle}
\def\>{\bold\rangle}

\def\wh{\widehat}

\DeclareMathOperator\const{const}
\DeclareMathOperator\dist{\widehat d}

\DeclareMathOperator\Lip{Lip}

\DeclareMathOperator\NUH{NUH}

\DeclareMathOperator\id{Id}
\DeclareMathOperator\Var{Var}

\DeclareMathOperator\cO{\mathcal O}

\newcommand\hpi{\widehat \pi}
\newcommand\hmu{\widehat \mu}
\newcommand\hnu{\widehat \nu}
\newcommand\hbeta{\widehat \beta}

\newcommand\cM{{\mathcal M}}

\def\hmu{\widehat{\mu}}
\def\hphi{\widehat{\phi}}

\def\cO{\mathcal O}

\def\hpi{\widehat{\pi}}
\def\hvf{\widehat{\vf}}

\newcounter{jbStepCounter}

\newcommand{\occult}[1]{}

\newcommand\ignore[1]{}

\newcommand\diam{{\operatorname{diam}}}

\newcommand\eps{\epsilon}

\newcommand\E{\mathbb{E}}

\newcommand\NN{{\mathbb N}}

\newcommand\Prob{{\mathbb P}}
\newcommand\Proberg{{\mathbb P}_{\operatorname{erg}}}
\newcommand\RR{{\mathbb R}}
\newcommand\supp{\operatorname{supp}}
\renewcommand\top{{\operatorname{top}}}
\newcommand\TT{{\mathbb T}}

\newcommand\ZZ{{\mathbb Z}}

\newcommand\fs{\mathfrak{sh}}

\newcommand\tpi{\widetilde\pi}

\newcommand\Bor{\text{\sc top}}
\newcommand\hTOP{h_\Bor}
\newcommand{\supnorm}{{\sup}}

\def\hpsi{\widehat{\psi}}

\newcommand\wto{\rightharpoonup}
\newcommand\wsc{\rightharpoonup}

\newcommand\stL{\widetilde{\mathscr L}}
\newcommand\stH{\widetilde{\mathscr H}}
\newcommand\stK{\widetilde{\mathscr K}}
\newcommand\stV{\widetilde{\mathscr V}}

\newcommand\tf{\widetilde{f}}

\newcommand\tmu{\widetilde{\mu}}
\newcommand\tnu{\widetilde{\nu}}

\newcommand\tU{\widetilde{U}}

\newcommand\tK{\widetilde{K}}

\newcommand\tx{\widetilde{x}}

\newcommand\pci{{+,i}}
\newcommand\mci{{-,i}}

\newcommand\pc[1]{{+,#1}}

\def\hpsi{\widehat{\psi}}

\begin{document}
\author{J. Buzzi, S.    Crovisier, O. Sarig} 
\title[Strong positive recurrence and exponential mixing]{Strong positive recurrence and exponential mixing for diffeomorphisms} 
\begin{abstract}
We introduce the \emph{strong positive recurrence} (SPR)  property  for diffeomorphisms on closed manifolds with arbitrary dimension, and show that it has many consequences and holds in many cases. SPR diffeomorphisms  can be coded by countable state Markov shifts whose transition matrices act with a spectral gap on a large Banach space, and this implies exponential decay of correlations, almost sure invariance principle, large deviations, among other properties of the  ergodic measures of maximal entropy. Any $C^\infty$ smooth surface diffeomorphism with positive entropy is SPR, and there are many other examples  with lesser regularity, or in higher dimension. 
\end{abstract}


\maketitle

\tableofcontents

\section{The SPR Property and Main Results}

\subsection{Introduction}
A fundamental question in dynamics  is to explain the  random behavior exhibited by  deterministic dynamical systems with high complexity.
This is best understood for uniformly hyperbolic systems. For these systems, one can use the spectral gap of an associated transfer  operator  to prove a variety of stochastic properties, ranging from exponential decay of correlations to almost sure invariance principles and large deviations \cite{Ruelle-TDF-book,Parry-Pollicott-Asterisque,Guivarch-Hardy,Gouezel-ASIP,Kifer}.
However, the uniform hyperbolicity condition is very restrictive \cite{Franks69,Newhouse70,Manning74}, and
a central  challenge  is to extend the theory  to a larger class of systems.
 This has given rise to various notions of ``semi-uniform'' hyperbolicity, which fall between Anosov's uniform hyperbolicity \cite{Anosov-Geodesic-Flows} and Pesin's non-uniform hyperbolicity \cite{Pesin-Izvestia-1976,Barreira-Pesin-Non-Uniform-Hyperbolicity-Book}. A principal  example is having a L.-S. Young tower with exponential tail for the return times to the base \cite{Young-Towers-Annals,young-tower-recurrence}. 

\smallskip
This work introduces a new property of this type, which we call, in analogy to a property in symbolic dynamics, ``\emph{strong positive recurrence" (SPR).}
We show that SPR is  common, and powerful. For example, we will prove  that {\em all}  $C^\infty$ surface diffeomorphisms with positive topological entropy are SPR. 
Using the general theory of SPR diffeomorphisms developed in this paper, we show:

\begin{theorem}\label{t.main}
Let $f$ be a topologically mixing $C^\infty$  diffeomorphism with positive topological entropy, on a closed surface $M$.
Let $\mu$ be the (unique) invariant measure which  maximizes the entropy. Then
for every $\beta>0$, there  are  $0<\theta<1$ and $C>1$ such that for all $\beta$-H\"older continuous functions $\varphi,\psi:M\to\R$,
 $$
 \left|\int \varphi\cdot (\psi\circ f^n)\, d\mu - \int \varphi\, d\mu \int \psi\, d\mu
    \right| \leq C\|\vf\|_\beta\|\psi\|_\beta \theta^n\ \ \ (\forall n\geq 0),\text{ where }
 $$
\begin{equation}\label{e.Holder-Norm-on-M}
\|\varphi\|_\beta:=\sup|\varphi|+\sup\bigg\{\frac{|\varphi(x)-\varphi(y)|}{d(x,y)^\beta}: x,y\in M,\ x\neq y\bigg\}.\hspace{1.1cm}
\end{equation}
\end{theorem}

\noindent
The existence of the measure of maximal entropy is due to Newhouse \cite{Newhouse-Entropy}; Uniqueness and mixing is proved in \cite{BCS-1}. 
The proof of exponential mixing is split to several parts: First we define the SPR property (\S\ref{s.def-SPR-diffeo}); Then we characterize it in terms of continuity properties of Lyapunov exponents  and deduce (using our earlier paper \cite{BCS-2})  that all $C^\infty$ surface diffeomorphisms with positive entropy are SPR (\S\ref{ss.SPR-continuity}); Finally we show that the SPR property implies the existence of a symbolic model with  a spectral gap (\S\ref{section-Sigma}). This implies  exponential mixing (\S\ref{s.decay-of-corr-diffeo-proof}). 
This argument  can be extended to the non topologically-mixing case, and yields many other stochastic properties. It also applies  to a large  class of $C^r$-diffeomorphisms and to  some examples  in  higher dimension. See Theorems \ref{t.DOC-diffeos} and \ref{t.tight-Cr},  \S~\ref{s.example} and \S~\ref{ss-Prop-SPR-diffeos}. 

Theorem \ref{t.main} discusses measures of maximal entropy. It seems likely that the theorem is false, without additional assumptions, for  SRB measures. See \cite{Martens-Liverani}.

\subsubsection*{Standing Assumptions and Notation}
Throughout this paper, $f:M\to M$ is a diffeomorphism of a
smooth Riemannian manifold $M$, which is {\em closed}, i.e. compact and without boundary. The distance between points $x,y\in M$ is denoted by  $d(x,y)$.
Henceforth, all measures are understood to be invariant Borel probability measures.

We denote the  differential of $f$ by $Df:TM\to TM$. Unless specified otherwise, we will always assume that $f$ is of class $C^{1+}$, i.e. that $Df$  is expressed in coordinates by  H\"older continuous functions.
The {\em asymptotic dilation} is
\begin{equation}\label{e.dilation}
\lambda_{\max}(f):=\lim_{n\to +\infty} \max_{x\in M} \tfrac 1 n \log \max(\|Df^n(x)\|,\|Df^{-n}(x)\|)
\end{equation}
where $\|\cdot\|$ is the operator norm defined by the Riemannian structure.

We will also need to consider more general maps $T$.
Given a measurable map  $T$ on a measurable space $(\Omega,\mathfs F)$, we let
$\Prob(T)$ (resp. $\Proberg(T)$) denote the set of $T$-invariant (resp. $T$-invariant and ergodic) probability measures on $(\Omega,\mathfs F)$.
Given an invariant measurable set $X$,
$\Prob(T|_X):=\{\mu\in\Prob(T):\mu(X)=1\}$.

The metric entropy of $\mu\in\mathbb P(T)$ is denoted by $h(T,\mu)$.
The {\em top entropy} of $T$ is \begin{equation}\label{e.Top-Entropy}
h_{\Bor}(T):=\sup\{h(T,\mu):\mu\in\Prob(T)\}. \end{equation}
By the variational principle, if $\Omega$ is a compact metric space and  $T$ is continuous,  then $h_{\Bor}(T)$ is equal to the topological entropy of $T$, $h_{\top}(T)$. 
See \cite{Walters-Book}.

A measure $\mu\in \mathbb P(T)$ such that $h(T,\mu)=h_{\Bor}(T)$ is called a {\em measure of maximal entropy (MME)}.

\subsection{Strong Positive Recurrence}
\label{s.def-SPR-diffeo}
In this subsection, we suppose $f$ to be a $C^{1}$ diffeomorphism on a closed smooth manifold $M$.
Recall the following classical notion from Pesin theory \cite{Barreira-Pesin-Non-Uniform-Hyperbolicity-Book}:
\begin{definition}\label{d.pesin}
Fix $\chi,\varepsilon>0$. A \emph{$(\chi,\varepsilon)$-Pesin block} is a non-empty set
$\Lambda\subset M$ for which there are direct sum decompositions $T_x M=E^s(x)\oplus E^u(x)$ for all $x\in \bigcup_{n\in\Z} f^n(\Lambda)$,  and a uniform number $K>0$ such that
for any $n\in \ZZ$, $k\geq 0$, and $y\in\Lambda$,
\begin{equation}\label{e.def-pesin}
\max\bigl(\|Df^k|_{E^s(f^n(y))}\|,\; \|Df^{-k}|_{E^u(f^n(y))}\|\bigr)\leq K \exp(-\chi k+\varepsilon |n|).
\end{equation}
\end{definition}
\noindent
We will be mainly interested in cases when  $\varepsilon$ is much smaller than $\chi$.
The subspaces $E^s(x)$ and $E^u(x)$ are then automatically continuous and  invariant on $\Lambda$ \cite{Barreira-Pesin-Non-Uniform-Hyperbolicity-Book}, but their dimensions may depend on $x$.
Pesin blocks may be chosen compact, but in general, they are not invariant.  

The SPR property, which we now introduce, requires that all measures with sufficiently large entropy give a definite positive mass to some fixed Pesin block:

\begin{definition}\label{def-SPR-diffeo}
A diffeomorphism $f$ of a closed manifold is \emph{strongly positively recurrent (SPR)}, if  there exists $\chi>0$ such that  for each $\varepsilon>0$,
 there are a Borel $(\chi,\varepsilon)$-Pesin block $\Lambda$ and numbers
$h_0< h_\top(f)$, $\tau>0$ as follows:
\begin{equation}\label{e.SPR}
\text{For every ergodic measure $\nu$,}\quad
    h(f,\nu)>h_0 \implies \nu(\Lambda)>\tau.
\end{equation}
\end{definition}

We will eventually see that, for $C^{1+}$ diffeomorphisms, the SPR property is equivalent to many other dynamical properties, of a more transparent nature:
 \begin{itemize}
  \item[--] {\em Robustness for the MME:} Every  measure with nearly maximal entropy is close  to the MME 
  in the $1$-Wasserstein metric. In dimension two, the Lyapunov exponents are close as well.
\item[--] {\em Effective Robustness of the MME:} The above, with explicit estimates.
  \item[--] {\em Structure of the MME}:  The  MME distribution of the first entrance time into some Pesin block has exponentially decaying  tail.    \item[--] {\em  Entropy tightness} (see  \S \ref{s.defSPR}): One can find  $h_0$, $\tau$ and $\Lambda$ as in \eqref{e.SPR} with $\tau$  arbitrarily close to one.
 \end{itemize}
All these characterizations of the SPR property will be collected in Thm~\ref{thm-characterizations}.

The name ``strong positive recurrence" originates in the theory of countable state Markov shifts.  The connection between the SPR properties for diffeomorphisms and for Markov shifts  is that they are  both equivalent to the existence of   an ``entropy gap at infinity," see  Cor~\ref{c.SPR=Entropy-Gap-Diffeos}, Thm~\ref{t.SPR-as-Entropy-Gap-Shifts}(3).  

\indent As we shall see in Thm \ref{thm-characterizations}(III), every SPR diffeomorphism can be coded by an SPR Markov shift, and vice versa.
 This is instrumental in deriving many of the properties of SPR diffeomorphisms announced in  \S\ref{ss-Prop-SPR-diffeos} (see Part~\ref{part-symbolic-diffeo}).

We will now present some examples of SPR diffeomorphisms.

Any Anosov diffeomorphism is SPR. In this case, \eqref{e.SPR} holds with   arbitrary $0<\tau<1$, because there are $\chi>0$ such that the entire manifold is a $(\chi,\eps)$-Pesin block (for each $\eps>0$).

Another trivial example is a diffeomorphism $f$ such that $\Proberg(f)$ consists of  finitely many ergodic invariant measures, each carried by a hyperbolic periodic orbit. In this case we can take $\Lambda$ to be the union of these orbits. Notice that this example has zero topological entropy. In Prop. \ref{p.SPR-in-Zero-Entropy} we will see that every SPR diffeomorphism with zero topological entropy is like that.

\medskip

There are many more examples. One of the main results of this paper is:
\begin{maintheorem}\label{t.SPR-C-infinity}
On closed surfaces, all $C^\infty$ diffeomorphisms with positive topological entropy are SPR.
\end{maintheorem}

\noindent
This will be proved as Thm~\ref{t.tight-C-infinity} as an application of a more general result, Thm~\ref{thm-CE-SPR}, valid in any dimension and finite smoothness.
\medskip

\noindent
In particular, unlike transitive Anosov surface diffeomorphisms, which can only exist on tori~\cite{Franks69, Newhouse70, Manning74}, SPR topologically mixing  $C^\infty$ diffeomorphisms with positive topologically entropy exist on all surfaces~\cite{Katok79}.

\medskip
\indent On the one hand, neither the $C^\infty$ assumption nor the dimension assumption can be removed in Theorem \ref{t.SPR-C-infinity}:
\begin{enumerate}[(1)]
\item Some {\em non-$C^\infty$} surface diffeomorphisms with positive topological entropy  are not SPR: Buzzi constructed  $C^r$ surface  diffeomorphisms with arbitrarily high $r$ and $h_{\top}(f)$, but without any measures of maximal entropy \cite{BuzziNoMax}. These diffeomorphisms cannot be SPR, by Thm~\ref{t.MME} below.

\medskip
\item Some $C^\infty$ diffeomorphisms with positive topological entropy on manifolds of {\em dimension three} are not SPR,  e.g. $f=T_A\times R_\alpha,$ where $T_A:\T^2\to\T^2$ is a hyperbolic toral automorphism and $R_\alpha:S^1\to S^1$ is the rotation by angle $\alpha$.
This diffeomorphism is not SPR, because all its orbits have one zero Lyapunov exponent, and therefore $f$ has no Pesin blocks at all.

\end{enumerate}
\indent On the other hand,  in~\S\ref{sec-surf-SPR}, we state and prove  an extension of Thm~\ref{t.SPR-C-infinity} to $C^r$ surface diffeomorphisms  with large but finite $r$, which  provides a {non-empty} $C^r$-open set of {surface} diffeomorphisms with  the SPR property.

In addition, we also have applications to higher dimensions.
We already mentioned the Anosov case. 
In \S\ref{s.example} we provide other examples, some partially hyperbolic, some with a dominated splitting (separating or not the Lyapunov exponents with different signs).
For this reason, even though our original motivation was the study of  $C^\infty$ surface diffeomorphisms, we decided to  develop the theory of strong positive recurrence in arbitrary dimension and only assuming  $C^{1+}$ smoothness.

\subsection{SPR, Continuity of Exponents, and the Proof of Theorem \ref{t.SPR-C-infinity}}\label{ss.SPR-continuity}

We will give a sufficient condition for the SPR property in terms of Lyapunov exponents.
In this section, 
we state this result in the setting where it is simplest and most powerful:
we assume $M$ to be a closed surface.
Let $\mu$ be an $f$-invariant Borel probability measure.  By the Oseledets theorem, a.e.  $x\in M$ has two Lyapunov exponents $\lambda^-(x)\leq \lambda^+(x)$ (see \S\ref{s.NUH} and \S\ref{s.SPR-exponents}). The Lyapunov exponents of  $\mu$ are $\lambda^\pm(\mu):=\int \lambda^\pm(x)d\mu(x)$. If $\mu$ is  ergodic, $\lambda^{\pm}(x)=\lambda^{\pm}(\mu)$ $\mu$-a.e.

\begin{maintheorem}\label{t.SPR-CE-SPR}
Let $f$ be a $C^{1+}$ diffeomorphism with positive topological entropy on a closed surface.
Then $f$ is SPR if and only if there is $\chi>0$ such that, for every  sequence of ergodic  measures  $\mu_k$ such that
 $h(f,\mu_k)\to h_{\top}(f)$ and
 $\mu_k\to \mu$
\mbox{weak-$*$}: \begin{enumerate}[ (C1)]
\item $\lambda^-(x)<-\chi<0<\chi<\lambda^+(x)$  $\mu$-almost everywhere; and
\item $\lambda^+(\mu_k)\to \lambda^+(\mu)$ and $\lambda^-(\mu_k)\to \lambda^-(\mu)$.
\end{enumerate}
Moreover in this setting the measure $\mu$ is an hyperbolic ergodic MME.
\end{maintheorem}

\noindent
The direct implication follows from Cor~\ref{c.Kadyrov2} and the converse one from Thm~\ref{thm-CE-SPR}, both valid in any dimension.

\medskip

\noindent 

In \S \ref{sec-surf-SPR}, we use our work on surface diffeomorphisms in \cite{BCS-2} to show  that Conditions (C1) and (C2) hold for {\em all} $C^\infty$ surface diffeomorphisms with positive topological entropy. Thm~\ref{t.SPR-C-infinity} follows. 

In fact we can do more. Burguet has extended \cite{BCS-2} to the $C^r$-case \cite{Burguet-Cr}. Using this work, we can also show that conditions (C1) and (C2) hold for all $C^r$ surface diffeomorphisms such that $h_{\top}(f)>\lambda_{\max}(f)/r$. It follows that all such diffeomorphisms are SPR. See Thm~\ref{t.tight-Cr} below.

\medbreak
We now return to the case of closed manifolds of arbitrary dimension.

\newcommand\mHC{\operatorname{X}}

\subsection{SPR Borel Homoclinic Classes}
Many diffeomorphisms have several measures of maximal entropy. In such cases, it is useful to  decompose the dynamics  into invariant pieces with some irreducibility properties, implying transitivity and uniqueness of MME. A typical example is Smale's spectral decomposition  of the non-wandering set of an axiom A diffeomorphism into ``basic sets" \cite{Smale}.
 In \cite{BCS-1}, we gave a similar decomposition for {\em general} $C^{1+}$ diffeomorphisms based on a homoclinic relation between measures. 
 
 A variant of this decomposition is described in detail \S\ref{s.measurable-homoclinic-classes}.  We call the elements of this decomposition {\em Borel homoclinic classes}.\footnote{
The connection between our Borel homoclinic classes and Newhouse's homoclinic classes as defined in \cite{Newhouse-Homoclinic} is explained in Remark \ref{r.topological-class}. }
They form  a countable (possibly finite) family of pairwise disjoint invariant Borel sets $X_i$, whose union  carries all ergodic hyperbolic invariant measures. Each $X_i$ carries at most one \emph{local MME}, i.e., $\mu\in\Prob(f|_{X_i})$ such that $h(f,\mu)=h_{\Bor}(f|_{X_i}):=\sup\{h(f,\nu):\nu\in\mathbb P(f), \nu(X_i)=1\}$.

The {\em period} $p$ of a Borel homoclinic class $X_i$ is defined to be the greatest common divisor of the periods of periodic orbits inside $X_i$.
The unique MME of  $f|_{X_i}$ (if it exists) can be shown to be   isomorphic to the product of a Bernoulli scheme and a cyclic permutation of $p$ points. In particular, it is mixing iff $p=1$.

\medbreak

One can localize the SPR property of the diffeomorphism to any $X_i$ (or more generally to
arbitrary invariant Borel sets $X\subset M$):
\begin{definition}\label{def-SPR-diffeo-local}
A diffeomorphism $f$ is {\em SPR} on an invariant Borel subset $X$
if there exists $\chi>0$ such that  for each $\varepsilon>0$, there are
a $(\chi,\varepsilon)$-Pesin block $\Lambda$ and numbers
$h_0< h_\Bor(f|_X)$ and  $\tau>0$ as follows:
\begin{equation}\label{e.SPR-local}
\text{For any ergodic measure $\nu$ on $X$,}\quad
        h(f,\nu)>h_0 \implies \nu(\Lambda)>\tau.
\end{equation}
In this case we also say that {\em $X$ is SPR for $f$}.
\end{definition}
The following result relates the SPR property of $f$ to the SPR properties of its Borel homoclinic classes:
\begin{proposition}\label{p.decomposition}
Let $f$ be a $C^{1+}$ diffeomorphism of a closed manifold.
Then $f$ is SPR if and only if there exists $h_0<h_\top(f)$ with the following properties:
\begin{enumerate}[(a)]
\item All ergodic measures $\mu$ with $h(f,\mu)>h_0$ are hyperbolic.
\item The set of Borel homoclinic classes $X$ with entropy $h_\Bor(f|_X)>h_0$ is finite.
\item Each Borel homoclinic class $X$ such that $h_\Bor(f|_X)>h_0$ is SPR.
\end{enumerate}
\end{proposition}

\noindent
This will be proved in \S\ref{sss-BHC-SPR}.

\noindent
We will use Prop.~\ref{p.decomposition} to deduce the properties of  SPR diffeomorphisms from  results on   SPR Borel homoclinic classes.
Among other things, this approach allows to extend results such as Thm~\ref{t.main} to the non-mixing case (Thm~\ref{t.DOC-diffeos}).

\begin{example}\label{e.BHC-for-axiom-A}
If $f$ is an axiom A diffeomorphism, then
the Borel homoclinic classes coincide with the (maximal) basic sets and their number is finite.
Each of the $X_i$ is SPR, because it can be covered by a single Pesin block. In the special case when $f$ is a transitive Anosov diffeomorphism, there is only one $X_i$.
\end{example}
\begin{example}\label{e.BHC-for-top-transitive}
If $\dim(M)=2$, $f$ is $C^\infty$, and $h_{\top}(f)>0$, then \cite{BCS-1} builds a spectral decomposition that shows that:
\begin{enumerate}
\item For every $0<h_0<h_{\top}(f)$, there are finitely many Borel homoclinic classes which carry all ergodic invariant measures with entropy bigger than $h_0$.
\item In the topologically transitive case,  all ergodic measures with positive entropy are carried by a single Borel homoclinic class (obviously unique).
\item In the topologically mixing case, the period of this class equals one.
\end{enumerate}
Thm~\ref{t.tight-C-infinity} below says that all the Borel homoclinic classes of a $C^\infty$ diffeomorphism of a closed surface are SPR.
\end{example}

\subsection{Properties of SPR Diffeomorphisms}\label{ss-Prop-SPR-diffeos}
We shall see that the MMEs of SPR diffeomorphisms share many of the properties of the MMEs of axiom A diffeomorphisms.
We begin with existence and uniqueness of MMEs.
As a comparison, recall that some $C^r$ diffeomorphisms do not have such measures~\cite{Misiurewicz,BuzziNoMax}, but
$C^\infty$ diffeomorphisms in any dimension always admit MME~\cite{Newhouse-Entropy}.
\begin{maintheorem}[\bf \emph{Existence and Finite Multiplicity of MME}]\label{t.MME}
Let $f$ be a $C^{1+}$ diffeomorphism of a closed manifold.
Then any SPR Borel homoclinic class $X$ of $f$
carries exactly one local MME, i.e., an invariant measure $\mu$ such that
$$
h(f,\mu)=h_{\Bor}(f|_X):=\sup\{h(f,\nu):\nu\in\mathbb P(f), \nu(X)=1\}.
$$
If $f$ is SPR and $h_{\top}(f)>0$, then the set of ergodic MME is non-empty and finite.
\end{maintheorem}
\noindent
This theorem is an immediate consequence of Prop.~\ref{p.decomposition} and Thm~\ref{t.MME-exists} on each Borel homoclinic class. The latter gives additional information on the MMEs. 
Thm~\ref{t.MME} implies Bowen's theorem on the existence and uniqueness of the MME on a basic set $X$ of an axiom A diffeomorphism \cite{Bowen-LNM}, see Example \ref{e.BHC-for-axiom-A}. 
Similarly, Theorems \ref{t.SPR-C-infinity}~and~\ref{t.MME} give a new proof of our results in \cite{BCS-1} on the finiteness of the number of ergodic MME  for  $C^\infty$ surface diffeomorphisms with positive topological entropy, and the uniqueness or mixing of the MME in the topologically transitive or topologically mixing case, see Example \ref{e.BHC-for-top-transitive}(2).

\smallskip
Thm~\ref{t.MME} says that the restriction of a $C^{1+}$ diffeomorphism to a SPR Borel homoclinic class $X$ is {\em intrinsically ergodic}: It has exactly one local MME $\mu$. Equivalently, 
 for every $\nu\in\mathbb P(f|_X)$, if $h(f,\nu)=h_{\Bor}(f|_X)$, then $\nu=\mu$.

If $h(f,\nu)\approx h_{\Bor}(f|_X)$, then $\nu\approx \mu$, with explicit error bounds:

\begin{maintheorem}[\bf \emph{Effective Intrinsic Ergodicity}]\label{t.effective}
Let $f$ be a $C^{1+}$ diffeomorphism of a closed manifold $M$. 
Let $X$ be an SPR Borel homoclinic class and let be
$\mu$ the MME of $f|_X$. Given any $\beta>0$,
there is $C>0$
such that for any invariant measure $\nu$ (non necessarily ergodic) on $X$:
\begin{enumerate}
\item  If $\varphi\colon M\to \RR$ is $\beta$-H\"older continuous and $\|\vf\|_\beta$ is defined by \eqref{e.Holder-Norm-on-M}, then
\begin{equation}\label{e.effective1}
{\textstyle \big|\int \varphi d\mu-\int \varphi d\nu\big|} \leq C \|\varphi\|_\beta \sqrt{h(f,\mu)-h(f,\nu)}.
\end{equation}
\item If $\nu$ has no atoms, there is an injective Borel map $T:M\to M$ such that $$T_\ast\nu=\mu
\text{ and }
\int c(x,Tx) d\mu \leq C\sqrt{h(f,\mu)-h(f,\nu)},$$
where
$c(x,y):=d(x,y)+ \wh{d}(E^s(x),E^s(y)) + \wh{d}(E^u(x),E^u(y))$,
and $\wh{d}$ is the metric on the Grassmannian bundle (see Appendix~\ref{app-grassmannian}).\footnote{All invariant measures on the same Borel homoclinic class are hyperbolic, homoclinically related, and have the same stable and unstable dimensions. See \S \ref{s.measurable-homoclinic-classes}.}
\footnote{We thank  Y. Pesin for the suggestion to study the distance of the Oseledets distributions of measures with high entropy to those of the MME. His suggestion motivated this statement.}
\medskip

\item $\qquad\qquad {|\Lambda^{+}(\mu)-\Lambda^{+}(\nu)|  + |\Lambda^{-}(\mu)-\Lambda^{-}(\nu)| \leq C \sqrt{h(f,\mu)-h(f,\nu)},}$
\medskip

\noindent
where $\Lambda^\pm(\nu):=\int\Lambda^{\pm}(x)d\nu$, and $\Lambda^+(x)$ and $\Lambda^-(x)$ are the sums of positive or negative Lyapunov exponents, with multiplicity (see \S\ref{ss.def-of-Lambda-plus}).
\end{enumerate}
\end{maintheorem}

\noindent
This will be proved in \S\ref{ss.Kadyrov-for-Diffeos}.

\begin{remark}
Let $W_1(\cdot,\cdot)$ denote the Wasserstein $1$-distance of measures. The Kantorovich-Rubinstein Theorem and  equation \eqref{e.effective1} in case $\beta=1$ say that  $$W_1(\mu,\nu)\leq C\sqrt{h(f,\mu)-h(f,\nu)}\text{ for all $\nu\in\mathbb P(f|_X)$.}$$
\end{remark}
\begin{remark}
Item (3) above yields the following rigidity phenomenon. Assume the hyperbolicity condition (C1) of Thm~\ref{t.SPR-CE-SPR} in the case of surfaces  (or in higher dimension, condition  (EH) of Thm~\ref{thm-CE-SPR}).
If $\Lambda^+(\nu_n)\to \Lambda^+(\mu)$ whenever $\nu_n\in\Proberg(f|_X)$, $\nu_n\to\mu$, and $h(f,\nu_n)\to h_{\Bor}(f|_X)$, then this convergence holds with the speed $\mathcal O(\sqrt{h_\top(f)-h(f,\nu_n)})$. See also Thm~\ref{thm-characterizations}.
\end{remark}

We note two important special cases of Theorem \ref{t.effective}. 
Firstly, if $f$ is a transitive Anosov diffeomorphism, then there is just one Borel homoclinic class, and the theorem holds with $X=M$. In  this case, the above is due to Kadyrov \cite{Kadyrov-Effective-Uniqueness}: Item (1) appears there explicitly, and Items (2) and (3) are  consequences of (1) and the  H\"older continuity of $E^u(x)$ and $E^{s}(x)$ on $M$.

Secondly, if $f$ is a topologically transitive $C^\infty$ surface diffeomorphism with positive topological entropy, then there is just one Borel homoclinic class with positive entropy (Example \ref{e.BHC-for-top-transitive}). Items (1)--(3) hold with $X=M$ and Item (3) reads as:
\begin{equation}\label{e.effective3}
{|\lambda^{+}(\mu)-\lambda^{+}(\nu)|  + |\lambda^{-}(\mu)-\lambda^{-}(\nu)| \leq C \sqrt{h(f,\mu)-h(f,\nu)}}\text{ for all }\nu\in\mathbb P(f).
\end{equation}

\medskip
Next we discuss the statistical properties of the ergodic  MMEs of SPR diffeomorphisms.
The {\em geometric potentials}  are the functions 
\begin{equation}\label{e.geom-pot-first}
J^s(x):=-\log |\det (Df|_{E^s(x)})|
 ,
  J^u(x):=-\log|\det (Df|_{E^u(x)})|,
  \end{equation}  or $-\infty$ if $E^t(x)$ is not defined ($t=u,s$). 
 In \S\ref{ss.exp-dec-cor}-\S\ref{ss.expansion-bounds} we show:

\begin{maintheorem}\label{t.stoch-properties}
Let $f$ be a $C^{1+}$ diffeomorphism of a closed manifold. Suppose 
$X$ is an SPR Borel homoclinic class with period $p$ such that $h_\Bor(f|_X)>0$,  and let $\mu$ be the unique local MME on $X$.
Then, the following holds for all  H\"older continuous observables on $M$
and for the geometric potentials:
 \begin{enumerate}
  \item \textbf{{Exponential Decay of Correlations}} in the mixing case, or exponential decay of correlations for the ergodic components of $\mu$ with respect to $f^p$  in the non-mixing case (Thm~\ref{t.DOC-diffeos});
  \item \textbf{Large Deviations Property} for Birkhoff sums (Thm~\ref{t.LDP-diffeos});\footnote{If $\dim M=2$,  the $n^\text{th}$ Birkhoff sums of $\varphi:=-\log\|Df|_{E^u}\|$ equal $-\log\|Df^n|_{E^u}\|$. We thank F. Ledrappier for the suggestion to analyze the large deviations of these functions. See \S\ref{ss.expansion-bounds}.}
   \item  \textbf{Almost Sure Invariance Principle} for Birkhoff sums,  and its consequences: central limit theorem,  law of the iterated logarithm, arcsine law, law of records (Thm~\ref{t.ASIP-diffeo} and Corollaries \ref{c.LIL}--\ref{c.Records});
  \item \textbf{Asymptotic Variance} of Birkhoff sums: existence,  Green-Kubo identity, linear response formula, and periodic orbit conditions for the non-vanishing of the asymptotic variance (Thm~\ref{t.asymp-var-diffeos}).
 \end{enumerate}
\end{maintheorem}

\noindent
See  \S\ref{s.statement-of-conseq} for the full statements, and \S\ref{s.Conseq-of-SPR-Proofs-Diffeos} for the proofs.

As explained above, if $f$ is a transitive Anosov diffeomorphism in any dimension, or a transitive $C^\infty$ surface diffeomorphism with positive entropy, there is just one Borel homoclinic class with positive entropy, and the theorem holds with $X=M$.
In the Anosov case, all parts of the theorem are well-known:  (1) is  in \cite{Bowen-LNM},\cite{Ruelle-TDF-book}, (2) is in \cite{Kifer}, (3) is in \cite{Denker-Philipp}, and (4) is in \cite{Guivarch-Hardy},\cite{Livsic}.
For general $C^\infty$ surface diffeomorphisms, these results are all new.
Note that Thm~\ref{t.main} follows from item (1) above together with Thm~\ref{t.SPR-C-infinity} and Example~\ref{e.BHC-for-top-transitive}.

\medskip
Finally, we characterize the  SPR property in terms of the tail of the first entrance time into  Pesin blocks. We comment on the relation between this and L.-S. Young's  tower technique \cite{Young-Towers-Annals} in the next section.

\begin{maintheorem}[\bf \emph{Exponential Tails}]\label{t.tail}
Consider a $C^{1+}$ diffeomorphism $f$ of a closed manifold, and  a Borel homoclinic class $X$ with
an ergodic measure $\mu$ maximizing the entropy of $f|_{X}$.
 $X$ is SPR if and only if there exist $\chi,\delta>0$ s.t.:
\begin{enumerate}
 \item For each $\varepsilon>0$, there are a $(\chi,\varepsilon)$-Pesin block $\Lambda$ and $\theta\in (0,1)$, $C>0$, s.t.
$$\mu\{x\colon \; \forall 0\leq k\leq n,\;  f^k(x)\notin\Lambda\}\leq C\theta^n, \quad \forall n\geq 0;$$
\item For every $\mu\in\Proberg(f|_X)$ with $h(f,\mu)>\hTOP(f|_X)-\delta$, the Lyapunov exponents  are outside 
$[-\chi,\chi]$.
\end{enumerate} 
\end{maintheorem}

\noindent
The direct implication will be proved as Thm~\ref{t.Tail-Estimates-diffeos}(1); The other direction is the implication \eqref{i.Tail}$\implies$\eqref{i.SPR} in 
Thm~\ref{thm-characterizations}.

\begin{remark}
Condition (2)  is automatic in dimension two, because of the  Ruelle inequality. In this case 
Theorem \ref{t.tail} says that the SPR property, which is a priori a property of  measures with {\em large} entropy, can be checked by looking only at the measures with {\em maximal} entropy.
\end{remark}

 Suppose  a point $x$ belongs to some $(\chi,\eps)$-Pesin block. The  {\em optimal $(\chi,\eps)$-Pesin constant } $K_{\chi,\eps}(x)$ of $x$ is  the smallest constant $K$ in \eqref{e.def-pesin}. The following corollary to Thm~\ref{t.tail} was motivated by the work of Alves, Luzzatto \& Pinheiro \cite{Alvez-Luzzatto-Pinheiro}.   
\begin{corollary}[{\bf{\emph{Optimal Pesin Constants}}}]\label{c.tail}
Under the assumptions of Thm \ref{t.tail},  there exists $\chi>0$,  such that for all $\eps>0$ there are  $\eta,C>0$ which satisfy
$$
\mu\{x: K_{\chi,\eps}(x)\geq t\}\leq C t^{-\eta}\text{ for all $t>0$.}
$$
\end{corollary}

\noindent
The  Corollary follows from Thm~\ref{t.Tail-Estimates-diffeos}(2) and a simple computation.

\subsection{Equilibrium States}\label{ss.equilibrium}
The thermodynamic formalism of Ruelle, Sinai, and Bowen  \cite{Ruelle-TDF-book,Sinai-Gibbs,Bowen-LNM} studies  a useful generalization of the measure of maximal entropy, called an {\em equilibrium measure}.  We recall the definitions.

Let $X\subset M$ be an invariant Borel set, and let $\phi:X\to \R\cup\{-\infty\}$ be a Borel function such that $\sup\phi<\infty$. We call $\phi$ a {\em potential} (see \S\ref{s.TDF-for-diffeos} for the explanation of the nomenclature).   The {\em pressure of $\phi$ relative a measure $\mu\in\mathbb P(f|_X)$} is  
$$
P(f,\mu,\phi):=h(f,\mu)+\int\phi d\mu.
$$
The {\em top pressure of $\phi$ on $X$} is 
$$
P_{\Bor}(f|_X,\phi):=\sup\{P(f,\mu,\phi): \mu\in \Prob(f|_X)\}.
$$
For compact invariant sets $X$ and continuous functions $\phi$, $P_{\Bor}(f|_X,\phi)$ equals the topological pressure of $\phi|_X$, because of the Walters variational principle \cite{Walters-Book}. 

An invariant measure on $X$ such that $
P(f,\mu,\phi)=P_{\Bor}(f|_X,\phi)$ is called an {\em equilibrium measure of $\phi$ on $X$}, or if  $X=M$, just an {\em equilibrium measure of $\phi$}. The equilibrium measures of the  potential $\phi\equiv 0$  are exactly the MME.

These classical ideas suggest the following generalization of the SPR property.

\begin{definition}\label{d.SPR-potential}
A diffeomorphism $f$ is
\emph{SPR} for a potential $\phi$ on an invariant Borel set $X$
if there exists $\chi>0$ such that for each $\varepsilon>0$, there are a Borel $(\chi,\varepsilon)$-Pesin block $\Lambda$
and numbers $P_0<P_{\Bor}(f|_X,\phi)$, $\tau>0$ as follows:
\begin{equation}\label{e.SPR-potential}
\text{For every ergodic measure $\nu$ on $X$,}\quad
    P(f,\nu,\phi)>P_0 \implies \nu(\Lambda)>\tau.
\end{equation}
When we do not specify the invariant set $X$, it is understood to be $M$. When we do not specify the potential, it is understood to be $\phi\equiv 0$.
\end{definition}

\begin{remark} Note that $f$ is SPR on $X$ in the sense of Definitions \ref{def-SPR-diffeo} and \ref{def-SPR-diffeo-local} iff it is SPR for the zero potential in the sense of Definition \ref{d.SPR-potential}.
\end{remark}

\indent
The results of the previous section
 extend in a straightforward way to equilibrium measures of SPR potentials on Borel homoclinic classes which are either H\"older or geometric (see \eqref{e.geom-pot-first}) (or more generally {\em quasi-H\"older}, in the sense of  \S\ref{s.setup-for-consequences}).
An example of an SPR geometrical potential is discussed in \S\ref{ss.geometrical}.

\begin{maintheorem}\label{t.equilibrium}
Let $f$ be a $C^{1+}$ diffeomorphism of a compact manifold. Let  $X$ be a Borel homoclinic class with period $p$, and let $\phi\colon X\to \mathbb{R}$ be a quasi-H\"older potential such that $f$ is SPR for $\phi$ on $X$. Then the following properties hold:
\begin{enumerate}
\item \textbf{Existence and Uniqueness:} There exists one and only one equilibrium measure $\mu_\phi$ on $X$ for $\phi$.
\smallskip
\item \textbf{Properties}: $\mu_\phi$  is  isomorphic to the product of a Bernoulli scheme with a cyclic permutation of $p$ points. Quasi-H\"older observables satisfy the large deviation property, the almost sure invariance principle, the properties of the asymptotic variance listed in Thm \ref{t.asymp-var-diffeos},
and, if $p=1$,  exponential decay of correlations. 

\smallskip
\item \textbf{Stability}: There is $\eps_0>0$ s.t. for every quasi-H\"older   $\psi:X\to\R$ such that  $\sup|\psi|<\eps_0$, $f$ is  SPR for $\phi+\psi$ on $X$, and  $t\mapsto P_{\Bor}(f|_X,\phi+t\psi)$ is real-analytic  on (-1,1).
\end{enumerate}
\end{maintheorem}

\noindent
This will be proved in \S\ref{ss-prop-equilibrium} and \S\ref{s.TDF-for-diffeos}.
\medskip

\begin{remark}
The SPR property is in general not robust under perturbations of the diffeomorphism.
There are SPR diffeomorphisms which can be $C^\infty$ perturbed to 
non-SPR diffeomorphisms:
for instance, consider a surface diffeomorphism $f$ whose only ergodic measures
are Dirac masses at hyperbolic fixed points but which is the limit of diffeomorphisms having zero topological entropy
and some non-hyperbolic ergodic measure (see Figure~\ref{f.fragile}). To get an example with positive entropy, consider the product of $f$ with an Axiom A diffeomorphism.
\end{remark}

\begin{figure}[h]
\begin{center}
\includegraphics[width=7cm,angle=0]{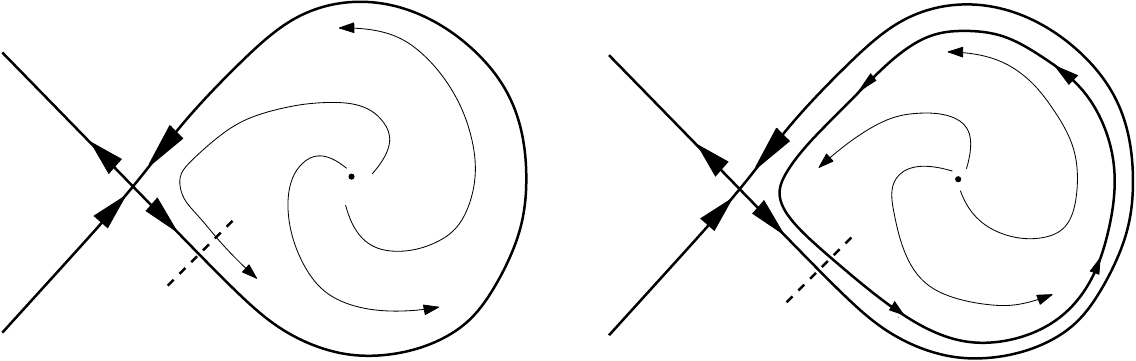}
\end{center}
\caption{An example of a SPR diffeomorphism (left) which can be perturbed to a non-SPR diffeomorphism (right) with zero entropy and a non-hyperbolic ergodic measure supported on an invariant circle. One may take the time-one map of a flow whose Poincar\'e map to a transverse section is $\infty$-tangent to the identity.} \label{f.fragile}
\end{figure}

\subsection{Historical Context}
The initial aim of this work was to understand the symbolic dynamical description of $C^\infty$ surface diffeomorphisms built in \cite{Sarig-JAMS}.

These symbolic models are \emph{countable state Markov shifts} (also called here \emph{Markov shifts}). Newhouse's theorem on the existence of MME for $C^\infty$ maps \cite{Newhouse-Entropy} and Gurevich's characterization of  Markov shifts with MME imply that these Markov shifts must satisfy a combinatorial condition called  \emph{positive recurrence} (see \eqref{e.PR}).

We wanted to know  if  they also satisfied a stronger combinatorial property, called {\em strong positive recurrence} (Def.~\ref{d.SPR-for-Shifts}).\footnote{The name ``strong positive recurrence" is from \cite{Sarig-CMP-2001}, but equivalent  conditions appeared before  in different forms the works of  Vere-Jones \cite{Vere-Jones-Geometric-Ergodicity},  Gurevich \cite{Gurevich-Stable}, and Gurevich-Zargaryan \cite{Gurevich-Zargaryan}. Other equivalent conditions   can be found in \cite{Gurevich-Savchenko},\cite{Ruette},\cite{Sarig-CMP-2001},\cite{BoyleBuzziGomez2014}.}
We were interested in this property because  of the paper \cite{Cyr-Sarig}, which characterized the SPR property for Markov shifts in terms of the existence of a Banach space where Ruelle's operator acts with spectral gap -- a condition which has long been known to have many  consequences, including the symbolic analogues of Thms~\ref{t.MME}--\ref{t.tail} above \cite{Ruelle-TDF-book},\cite{Parry-Pollicott-Asterisque},\cite{Guivarch-Hardy},\cite{Sarig-CMP-2001},\cite{Cyr-Sarig} (see Appendix~\ref{appendix-Markov-Shifts}).

We designed the SPR property for {\em diffeomorphisms} in Def.~\ref{def-SPR-diffeo}  to (a) guarantee the existence  of symbolic codings of Borel homoclinic classes by SPR {Markov shifts} (Thm~\ref{t.SPR-coding}), and (b) to hold for a large class of natural examples, e.g. all $C^\infty$ surface diffeomorphisms $f$ s.t. $h_{\top}(f)>0$ (Thm~\ref{t.SPR-C-infinity}).

We note that a large class of H\'enon maps (those with strongly regular parameters) has been shown in \cite{Berger-Henon} to have SRB measures coded by  SPR Markov shifts, yielding exponential mixing.
Other mathematicians have studied SPR-like conditions in other contexts. 
To relate our work to theirs, we note  a well-known characterization of the SPR condition for  Markov shifts in terms of  an {\em entropy gap at infinity} \cite[Sec. 6]{Buzzi-PQFT}: no sequence of measures which escapes to infinity can have entropies converging to the supremum of all entropies (see \S\ref{ss-SPR-Markov-shifts}). Entropy gaps at infinity  play a central role in the study of geodesic flows on non-compact negatively  curved manifolds and related dynamical systems \cite{ELPV},\cite{Riquelme-Velozo},\cite{Iommi-Velozo-Cusp},\cite{ST},\cite{GST}. For example, \cite[Thm~5.1]{ELPV} considers the geodesic flow on the modular surface, and gives an upper bound on the measure of the cusp in terms of the closeness of the entropy to its maximum.

In our setup, the ambient manifold is compact, but the hierarchy of Pesin blocks with larger and larger constants can be understood as a lack of compactness
and this opens the way for sequences of hyperbolic measures to lose mass in the limit to the
non-hyperbolic part of the phase space. An interpretation of the SPR condition for diffeomorphisms in terms of a suitably defined entropy gap at infinity is not only possible, but useful and illuminating,  see \S\ref{s.defSPR}.

\indent We note that the non-uniform specification properties introduced by Climenhaga and Thompson \cite{Climenhaga-Thompson} can at least sometimes be related to the existence of an SPR symbolic dynamics \cite{Climenhaga-Tower}.

L.-S. Young has introduced a powerful method for studying the rate of  decay of correlations for maps with respect to a reference measure such as the volume, based on representing the system as a factor of a ``Young tower" \cite{Young-Towers-Annals}. If the distribution of the first entrance time into the base of the tower decays exponentially fast, then one can deduce properties similar to those we prove in the SPR setting: exponential decay of correlations \cite{young-tower-recurrence}, large deviations \cite{Young-Rey-Bellet}, and the almost sure invariance principle \cite{Melbourne-Nicol-ASIP}.

Our work implies that SPR Borel homoclinic classes possess Young towers with exponential tail (for the MME). Indeed, we prove that  they
admit countable state Markov partitions such that the distribution of the  first entrance time to a partition set has exponential decay for the MME (Thms~\ref{c.Tail-CMS} and \ref{t.SPR-coding}), and every irreducible  Markov shift can be recast as a Young tower over a partition set, by inducing. 

But since our tower comes from a Markov partition, it enjoys additional properties, which are not always true for towers, such as the existence of entropy preserving lifts and projections relating the measures on the tower to the measures of the original dynamical systems (Thm~\ref{t.SPR-coding}, $(\Sigma$\ref{i.Sigma3}),$(\Sigma$\ref{i.Sigma4})).  These are crucial to our work.

\subsection{Guide to the Paper}
The paper consists of four parts, and appendices.

In Part~\ref{part-SPR-Diffeos}, 
we study the SPR property for diffeomorphisms, and  develop methods for checking it. The main result is that every Borel homoclinic class of every  $C^\infty$ surface diffeomorphism with positive entropy is SPR. 
The key idea is
to reduce SPR to  some condition on Lyapunov exponents, yielding Thm~\ref{thm-CE-SPR}. In this multidimensional generalization of Thm~\ref{t.SPR-CE-SPR} the largest exponent is replaced by the \emph{sum of the positive exponents}. Our earlier work \cite{BCS-2} establishes the required continuity for $C^\infty$ diffeomorphisms on surfaces, proving Thm~\ref{t.SPR-C-infinity} in \S\ref{sec-surf-SPR}.
We also discuss an extension to the $C^r$ case (Thm~\ref{t.tight-Cr}).
In \S\ref{s.defSPR} we introduce tightness of sequences of measures with respect to bornologies, and the precise definition of an entropy gap at infinity  that will allow us in Part~\ref{part-symbolic-diffeo}  to relate  the SPR property of a  diffeomorphisms to the SPR property of its symbolic model.  Finally Section~\ref{s.example} gives some examples of SPR diffeomorphisms in higher dimension.

In Part~\ref{part-SPR-Markov}, we recall some background on Markov shifts and then review the SPR property in this setting.  By \cite{Cyr-Sarig}, SPR  implies a spectral gap for an associated transfer operator, and this implies the analogues of Thms~\ref{t.MME}, \ref{t.effective} and \ref{t.stoch-properties} for Markov shifts. All this was essentially known before, but since the literature does not contain the results  in the generality we require, we give the  proofs in Appendix~\ref{appendix-Markov-Shifts}.

In  Part~\ref{part-symbolic-diffeo} we study the symbolic dynamics of a diffeomorphism $f$ on a Borel homoclinic class $X$. 
First, we introduce the ``bornological property" for an abstract symbolic coding, and show how it relates  the SPR property of a diffeomorphism to the SPR property of the coding (see Defs. \ref{d.Hyperbolic-Coding}, \ref{d-entropy-full}, \ref{d.bornological},  and eq.~\eqref{e.chi-entropy-hyperbolic}). Then 
 we review the specific  codings constructed in \cite{Sarig-JAMS,Ben-Ovadia-Codings} (and on Borel homoclinic classes in \cite{BCS-1}), and prove  that they have the bornological property.  This leads to Thm~\ref{t.SPR-coding}: If $f|_X$ is SPR then $f|_X$ admits an SPR coding; conversely the existence of a ``good" SPR coding  implies that $f|_X$ is  SPR. At the end of Part III, we explain why this analysis extends without effort to the SPR  property with respect to a potential
(Def.~\ref{d.SPR-potential}).

In Part~\ref{p.properties-of-SPR-diffeos}, we apply the SPR coding of Part~\ref{part-symbolic-diffeo} and the properties of SPR Markov shifts listed in Part~\ref{part-SPR-Markov}, to deduce our main results on the MMEs of  SPR diffeomorphisms (Thms~\ref{t.MME}--\ref{t.tail}), and some extensions to equilibrium measures associated to non-constant potentials (Thm~\ref{t.equilibrium}). A nice surprise is that some of the discontinuous objects of Pesin theory (such as the Oseledets distributions) are coded by H\"older continuous objects on the symbolic model. This enables us to obtain quantitative results on the robustness of Oseledets decomposition of the MME,  and the non-uniform hyperbolicity of its typical orbits.
These are discussed in \S\ref{ss.expansion-bounds}.

In \S\ref{s.converse-statements}, we show that many of the dynamical  consequences of the SPR property are actually equivalent to the SPR property. See Thm~\ref{thm-characterizations}.
\smallbreak

\indent Part V consists of appendices. These present well-known or folklore results in the setup needed for their application in this paper.

\subsection{Acknowledgements} S.C. was partially supported by  by the ERC project 692925 – NUHGD. O.S. was supported by ISF grants 
1149/18 and 264/22, and by a grant from the Minerva Stiftung. Part of this work was carried during a sabbatical visit to the  Universit\'e Paris-Saclay and the IH\'ES, and O.S. thanks these institutions for their hospitality.

\part{SPR Diffeomorphisms}
\label{part-SPR-Diffeos}

\section{Preliminaries on Non-Uniform Hyperbolicity}\label{s.NUH}
\noindent

Suppose $f:M\to M$ is a $C^{1}$ diffeomorphism of a closed smooth manifold $M$ with dimension $d$. By
the Oseledets theorem, there is an $f$-invariant set $M'$ with full measure for all invariant measures, so that for all $x\in M'$:
$$
T_xM=\bigoplus_{\lambda\in\sigma(x)} E^\lambda_x,
$$
where  $\sigma(x)\subset \mathbb{R}$ is a finite set, and $E^\lambda_x$ are non-trivial linear subspaces satisfying
$$E^\lambda_x=\{v\in T_xM\setminus 0: \lim_{n\to\pm\infty}\tfrac1n\log\|D_xf^n.v\|=\lambda\}\cup \{0\}.$$
Moreover, $\sigma(f(x))=\sigma(x)$ and  $Df(E^\lambda_x)=E^{\lambda}_{f(x)}$.

The elements of $\sigma(x)$ are called the {\em Lyapunov exponents of $x$}.
The dimension of $E^\lambda_x $ is called the {\em multiplicity of $\lambda$}. In this paper, we always order the Lyapunov exponents  in a decreasing order, repeated according to their multiplicity:
$$\lambda^1(x)\ge\dots\ge\lambda^d(x).$$

Fix $\chi>0$. A measure $\mu$ is called  \emph{$\chi$-hyperbolic},  if none of its Lyapunov exponents
belongs to $[-\chi,\chi]$.

\subsection{Pesin Blocks and the Non-Uniformly Hyperbolic Set}\label{s.PesinSets}
We examine the definition of  $(\chi,\varepsilon)$-Pesin blocks (Def.~\ref{d.pesin}) more carefully.

\begin{lemma}\label{l.splitting}
If $\Lambda$ is a $(\chi,\varepsilon)$-Pesin block with $\varepsilon<\chi$, then $\ov{\Lambda}$ is also a $(\chi,\varepsilon)$-Pesin block.
The decomposition $T_x M=E^s\oplus E^u$ in Def.~\ref{d.pesin} is uniquely defined on $\ov{\Lambda}$,  varies continuously with $x$ on $\ov{\Lambda}$, and  $E^s(x)=\{v\in T_xM: \underset{n\to +\infty}\limsup\|Df^n.v\|<+\infty\}$, and
$E^u(x)=\{v\in T_xM: \underset{n\to +\infty}\limsup\|Df^{-n}.v\|<+\infty\}$.
\end{lemma}
\begin{proof}
Suppose  $x$ is a point with decompositions  $T_{f^n(x)}M=E^u(f^n(x))\oplus E^s(f^n(x))$ which satisfy \eqref{e.def-pesin} for all $n$,
for some constant $K>0$.
Every unit vector $v\in T_xM$  decomposes into  $v=v^s+v^u$ for $v^s\in E^s(x)$, $v^u\in E^u(x)$, and for $n\geq 1$ we have
\begin{equation}\label{e.cont/exp}
\|Df^n.v^s\|\leq K\exp(-\chi n)\|v^s\|,\quad \|Df^n.v^u\|\geq K^{-1}\exp(\chi n-\varepsilon n)\|v^u\|.
\end{equation}
Since $\chi-\varepsilon>0$, if $v\notin E^s(x)$, then  $\|Df^n.v\|\sim\|Df^n.v^u\|\to\infty$, as $n\to\infty$.
Consequently, $v$ cannot belong to the stable space of another decomposition into stable and unstable spaces. So the decomposition $T_{x} M=E^u(x)\oplus E^s(x)$ is unique, and given by the statement of the lemma. Similarly, the decompositions $T_{f^n(x)} M=E^u(f^n(x))\oplus E^s(f^n(x))$ are unique.

\medskip
Let $x_i$ be a sequence of points with decompositions $T_{f^n(x_i)}M=E^u(f^n(x_i))\oplus E^s(f^n(x_i))$ as in \eqref{e.def-pesin}, for all $n$, and with the same constant $K$, for all $i$. We claim that if $x_i\to x$, then $x$ also has a decomposition $T_{f^n(x)}M=E^u(f^n(x))\oplus E^s(f^n(x))$ for all $n$, as in  \eqref{e.def-pesin}, and   with the same constant $K$.

Since condition \eqref{e.def-pesin} is closed,
it is sufficient to show that $E^u(f^n(x_i))$ and $E^s(f^n(x_i))$ converge as $i\to\infty$. A compactness argument and a diagonal argument shows that if this were not the case, then there would be two subsequences $i_k, j_k\to\infty$ so that
$$
E^t(f^n(x_{i_k}))\to E^t_1(f^n(x))\ , \ E^t(f^n(x_{j_k}))\to E^t_2(f^n(x))\ \ (t=u,s),
$$
where $E_1^t(f^n(x))\neq E_2^t(f^n(x))$ for some $n$ and some $t\in\{u,s\}$. However, it is not difficult to see using  \eqref{e.cont/exp} that
$$
T_{f^n(x)} M=E^u_i(f^n(x))\oplus E^s_i(f^n(x))\ \ \ (n\in\Z)
$$
are two direct sum decompositions which satisfy \eqref{e.def-pesin}, in contradiction to the uniqueness of the decomposition.
It follows that $E^u(x), E^s(x)$  is continuous on $\Lambda$  and extends continuously to $\ov{\Lambda}$. The lemma follows.
\end{proof}

Pesin's stable manifold theorem~\cite{Pesin-Izvestia-1976} can be stated as follows:

\begin{theorem}[Pesin]\label{t.pesin}
Given a closed Riemannian manifold $M$, and $r>1$, there exists a positive continuous function $\varepsilon_{M,r}$ with the following property.

For any $C^r$ diffeomorphism $f$ of $M$, any $\chi,\varepsilon>0$ with
$\varepsilon<\varepsilon_{M,r}(\chi,\varepsilon,\|f\|_{C^r})$, and any $(\chi,\varepsilon)$-Pesin block
$\Lambda$, each $x\in \Lambda$ admits a $C^r$ embedded disc $W^s_{loc}(x)$ satisfying:
$$\forall y\in W^s_{loc}(x),\quad \limsup_{n\to\infty}\tfrac 1 n \log d(f^n(x),f^n(y))<0.$$
Moreover the embedded disc $W^s_{loc}(x)$ varies continuously with $x\in \Lambda$ in the $C^r$ topology.
An analogous statement holds for the local unstable manifolds $W^u_{loc}(x)$.
\end{theorem}
The  {\em global stable and unstable manifolds} of $x\in\Lambda$ are defined by
$$
\displaystyle W^s(x):=\bigcup_{n\geq 0}f^{-n}[W^s_{loc}(f^n(x))]\ ,\
W^u(x):=\bigcup_{n\geq 0}f^{n}[W^s_{loc}(f^{-n}(x))].
$$ 
The choice of  $\eps$ in Theorem \ref{t.pesin} does not matter, see \cite[\S8.2]{Barreira-Pesin-Non-Uniform-Hyperbolicity-Book}.

In \cite{Pesin-Izvestia-1976}, Pesin builds $(\chi,\eps)$-Pesin blocks with positive mass for any $\chi$-hyperbolic measure. In general, these blocks are not invariant,  and they may have zero measure with respect to some
$\chi'$-hyperbolic measures with $\chi'<\chi$. This suggests the following construction:

\begin{definition}\label{def-NUH}
The \emph{non-uniformly hyperbolic set} is
\begin{equation}\label{e.NUH}
\NUH(f):=\bigcup_{\chi>0}\bigcap _{\varepsilon>0}\{x: x \text{ is contained in a $(\chi,\varepsilon)$-Pesin block}\}.
\end{equation}
The \emph{recurrent non-uniformly hyperbolic set} is
\begin{equation}\label{e.NUH2}
\NUH'(f):=\bigcup_{\chi>0}\bigcap _{\varepsilon>0}
\;\;\bigcup_{\substack{(\chi,\varepsilon)\text{-Pesin}\\ \text{block } \Lambda}}
\{x: x \text{ is a limit of periodic points in $\Lambda$}\}.
\end{equation}
\end{definition}
\noindent
Note that $\NUH'(f)\subset\NUH(f)$.
These are Borel sets, because there is a countable family of Pesin blocks $\{\Lambda_i\}_{i\in\N}$ such that every $(\eps,\chi)$-Pesin block (for any $\eps,\chi$) is a subset of $\Lambda_i$ for some $i$.
They are $f$-invariant, because the image of a $(\chi,\eps)$-Pesin block (with constant $K$) is contained in a $(\chi,\eps)$-Pesin block (with constant $Ke^{\eps}$). By Theorem~\ref{t.pesin},
each point $x\in \NUH(f)$ admits a stable and an unstable manifold.

It will be useful later to consider a fixed value of $\chi>0$ and to introduce
$$\NUH_\chi(f):=\bigcap _{\varepsilon>0}\{x: x \text{ is contained in a $(\chi,\varepsilon)$-Pesin block}\}$$
and the corresponding $\NUH'_\chi(f)$.

\begin{theorem}[Pesin~\cite{Pesin-Izvestia-1976}, Katok~\cite{KatokIHES}]\label{t.katok}
Let $f$ be a $C^{1+}$ diffeomorphism on  a closed manifold. Every hyperbolic  invariant measure of $f$  satisfies
$\mu(\NUH'(f))=1$.
\end{theorem}

\subsection{Borel Homoclinic Classes}\label{s.measurable-homoclinic-classes}\label{ss.class}
Let $f$ be a $C^{1+}$ smooth diffeomorphism of a closed manifold $M$. The set of transverse intersection points of two sub-manifolds $U,V$ in $M$ will be denoted by $U\pitchfork V$.

In \cite[\S2.4]{BCS-1}, we introduced a partial generalization of Smale's  spectral decomposition. The elements of this decomposition were
determined up to sets of  measure zero, simultaneously for all hyperbolic invariant measures.
We will now specify a a canonical choice for these sets.

\subsubsection{The Homoclinic Relation}
First, we extend the homoclinic relation on hyperbolic periodic orbits introduced by Newhouse \cite{Newhouse-Homoclinic}. See also~\cite{Rodriguez-Hertz-Squared-Tahzibi-Ures-CMP}.

\begin{definition}
Two points $x,y\in \NUH'(f)$ are \emph{homoclinically related} ($x\sim y$) if
$W^s(x)$ has a transverse intersection point with an iterate of $W^u(y)$,
and $W^u(x)$ has a transverse intersection point with an iterate of $W^s(y)$.
\end{definition}
\begin{proposition}\label{p.hr-is-equiv-rel}
The homoclinic relation $\sim$ is an equivalence relation on $\NUH'(f)$.
\end{proposition}
\begin{proof}
Symmetry and reflexivity are clear;  we will show transitivity.
Suppose  $x,y,z\in \NUH'(f)$, and $x\sim y$ and $y\sim z$.
By definition, there exists  periodic points $p,q,r$ which are arbitrarily close to $x,y,z$ in a common Pesin block.
By continuity of the stable and unstable manifolds, we can choose $p,q,r$ so  that $p\sim q$, $q\sim r$,
$x\sim p$ and $z\sim r$.

The inclination lemma (see~\cite{Katok-Hasselblatt-Book}) implies the transitivity of the homoclinic relation
between hyperbolic periodic orbits, hence $p\sim r$. Replacing $r$ by an iterate if necessary, one can thus find an unstable disc $\cD^u\subset W^u(p)$ and a stable disc $\cD^s\subset W^s(r)$ which intersect transversally.
By the inclination lemma, since $x\sim p$ there also exists an iterate $f^n(x)$ such that
$W^u(f^n(x))$ contains a disc $C^1$-close to $\cD^u$. Similarly,
there exists an iterate $f^{-m}(z)$ such that $W^s(f^{-m}(z))$ contains a disc $C^1$-close to $\cD^s$.
Consequently $W^u(f^n(x))\pitchfork W^s(f^{-m}(z))\neq \emptyset$.
With the same argument, there exist iterates of $W^s(x)$ and $W^u(y)$ which intersect transversally.
This proves $x\sim z$ as required.
\end{proof}

\begin{definition}\emph{Borel homoclinic classes} are  equivalence classes of $\sim$ in $\NUH'(f)$. A Borel homoclinic class equal to a  hyperbolic periodic orbit is called {\em trivial}.
\end{definition}
\begin{proposition}\label{p.homoclinic}
\begin{enumerate}
\item The set of Borel homoclinic classes is a finite or countable partition  of $\NUH'(f)$ into invariant Borel sets.
\item Each class contains a dense set of hyperbolic periodic points.
\item Any hyperbolic ergodic measure $\mu$ is carried by a Borel homoclinic class $X$.
\item Any measure carried by a Borel homoclinic class is hyperbolic.
\end{enumerate}
\end{proposition}
\begin{proof}
Every $x\in\NUH'(f)$ is homoclinically related to some hyperbolic periodic orbit (see the proof of Prop \ref{p.hr-is-equiv-rel}).
Since there are  at most countably many hyperbolic periodic orbits $x_i$,  there are at most countably many Borel homoclinic classes, each equal to the equivalence class of some $x_i$.
The equivalence class of $x_i$ is clearly invariant.
It is Borel, because for every $0<\eps_k<\lambda_k$ there is a countable family of Pesin blocks $\Lambda_{k,i}$ such that every $(\eps_k,\lambda_k)$-Pesin block is contained in some $\Lambda_{k,i}$, and  the sets $\{y\in\Lambda_{k,i}: y\sim x_i\}$ are  open in the relative topology. This proves (1).
Items 2 and 4 follow from \eqref{e.NUH2}. Item 3 follows from   Thm~\ref{t.katok}.
\end{proof}
Using Prop.~\ref{p.homoclinic}(3), we define the \emph{Borel homoclinic class of an ergodic hyperbolic measure $\mu$} to be the (unique)
Borel homoclinic class which carries $\mu$. If $\mu$ sits on a hyperbolic periodic orbit $\mathcal O$, we also speak of the {\em Borel homoclinic class of $\mathcal O$.}

\begin{remark}\label{r.topological-class}
Newhouse defined the  {homoclinic class} of $\mathcal O$ to be the closure of  the union
of all periodic orbits that are homoclinically related to $\mathcal O$ \cite{Newhouse-Homoclinic}.
Following \cite{BCS-1}, we call these sets
 \emph{topological homoclinic classes}. They are not the same as Borel homoclinic classes. The topological homoclinic class of $\mathcal O$  is equal to the {\em closure} $\overline X$ of the Borel homoclinic class $X$ of $\cO$.
See \cite{Newhouse-Homoclinic,bonatti-crovisier-connecting} for their  properties.

Suppose $\dim M=2$ and $f\in C^r$ with $r>1$.
Any point in $\ov{X}\setminus X$ belongs to either $\NUH'(f)^c$, 
or some set
$\overline X\cap \overline{X'}$, where $X'$ is a Borel homoclinic class different from $X$.
By \cite[Prop.~6.8]{BCS-1},  $h_{\Bor}(f|_{\ov{X}\setminus X})\leq \lambda_{\max}(f)/r$, with $\lambda_{\max}(f)$ as in \eqref{e.dilation}.

 In particular, for $C^\infty$ surface diffeomorphisms, the topological and Borel homoclinic classes of $\mathcal O$ are equal up to a set a set of measure zero, for  any ergodic invariant measure with positive entropy.

\begin{remark}
Other decompositions into pairwise disjoint sets related to Newhouse's construction were previously used in \cite{Rodriguez-Hertz-Squared-Tahzibi-Ures-CMP} (``ergodic homoclinic classes") and in \cite[Prop. 2.15]{BCS-1} (the sets denoted there by $H_{\cO}$).
\end{remark}

\end{remark}

\begin{remark}
In \cite[Def.~2.10]{BCS-1}, we discussed an equivalence relation of ergodic hyperbolic measures which, in the language of this paper,  is equivalent  to saying that $\mu_1\sim \mu_2$
when $x\sim y$ for $(\mu_1\times\mu_2)$-a.e. $(x,y)$.
We called the associated equivalence classes  \emph{measured homoclinic classes}.

Measured homoclinic classes are sets of {\em measures}; Borel homoclinic classes are sets of {\em points}. They are related by the following (obvious) statement:  the measured homoclinic class of the measure $\mu$ is the set of ergodic hyperbolic measures carried by the Borel homoclinic class of $\mu$.
\end{remark}

\subsubsection{The Period of a Borel Homoclinic Class}
We extend the results of~\cite{Abdenur-Crovisier} (compare with \cite[Prop. 2.17]{BCS-1}).

\begin{definition}
The \emph{period} of a Borel homoclinic class $X$ is the greatest common divisor of all the periods of periodic orbits  in the class.
It is denoted by $\text{period}(X)$.
\end{definition}

\begin{proposition}
If a Borel homoclinic  class $X$ has period $p$, then there is a Borel set $A$ such that $f^i(A)\cap A=\emptyset$ for $i=0,\ldots,p-1$, $f^p(A)=A$, and
$$ X=A\cup f(A)\cup\cdots\cup f^{p-1}(A).$$
Two points $x,y\in X$ belong to a same $f^i(A)$ iff there is a hyperbolic periodic point $q$ such that
$W^u(x)\pitchfork W^s(q)$ and $W^u(q)\pitchfork W^s(y)$ are both non-empty.
The set  $A$ is a Borel homoclinic class of $f^p$.
\end{proposition}

\begin{proof}
Let $X_0$ be the set of periodic points in $X$.
Proposition 1 in \cite{Abdenur-Crovisier} and its proof show that $X_0=A_0\cup f(A_0)\cup\cdots\cup f^{p-1}(A_0)$, with $A_0$  as follows:
\begin{enumerate}
\item $A_0,f(A_0),\ldots,f^{p-1}(A_0)$ are pairwise disjoint and  $f^p(A_0)=A_0$;
\item Two periodic $q,r\in X_0$ are in the same $f^i(A_0)$ iff $W^u(q)\pitchfork W^s(r)\neq \emptyset$;
\end{enumerate}
We show that the proposition holds with the set
$$A:=\{x\in X: W^u(x)\pitchfork W^s(q)\neq\emptyset\text{ for some }q\in A_0\}.$$

Obviously $f^p(A)=A$.
We claim that $A$ is disjoint from $f^k(A)$ for $1\leq k<p$.
Suppose by way of contradiction that there exists some $x\in A\cap f^k(A)$ for $1\leq k<p$.  By the definition of $\NUH'(f)$,  $x=\lim q_n$ where $q_n$ are periodic orbits in the same Pesin block. Since $q_n\to x$, $W^u_{loc}(q_n)\to W^u_{loc}(x)$ in the $C^1$ topology.
Since $x\in A$, $W^u(x)\pitchfork W^s(q)\neq \emptyset$  for some $q\in A_0$. So $W^u(q_n)\pitchfork W^s(q)\neq \emptyset$  for all $n$ sufficiently large, with the conclusion that $q_n$ must eventually all lie in $A_0$. Similarly, the assumption that  $x\in f^k(A)$ leads to the conclusion that $q_n$  all eventually lie in $f^k(A_0)$. So $f^k(A_0)\cap A_0\neq \emptyset$,  a contradiction.

Fix $x,y\in X$, and suppose $x,y\in f^i(A)$ for the same $i$.
Write $x=\lim q_n$ and $y=\lim r_n$, where $q_n$ and $r_n$ are periodic points all lying in the same Pesin block. As we saw above, $q_n$ and $r_n$  must all eventually belong to $f^i(A_0)$. Therefore, for all sufficiently large $n$, $W^u(q_n)\pitchfork W^s(r_n)\neq \emptyset$. In addition, since $q_n\to x, r_n\to y$,  for all sufficiently large $n$, $W^u(x)\pitchfork W^s(q_n)\neq \emptyset$, and $W^u(r_n)\pitchfork W^s(y)\neq \emptyset$. By the inclination lemma, $W^u(q_n)$ must accumulate on $W^u(r_n)$, and therefore $W^u(q_n)\pitchfork W^s(y)\neq \emptyset$. So there is a periodic orbit $q=q_n$ such that $W^u(x)\pitchfork W^s(q)$ and $W^s(y)\pitchfork W^u(q)$ are non-empty.

Conversely, suppose there is a hyperbolic periodic point $q$ such that
$W^u(x)\pitchfork W^s(q)\neq \emptyset$  and $W^u(q)\pitchfork W^s(y)\neq \emptyset$.
Then for all $n$ large enough,  $W^u(q_n)\pitchfork W^s(q)\neq \emptyset$ and $W^s(r_n)\pitchfork W^u(q)\neq \emptyset$. Necessarily $q_n,q,r_n$ belong to the same set $f^i(A_0)$.
One concludes that $x$ and $y$ belong to the same set $f^i(A)$.

Using the inclination lemma, it is possible to see that $A$  coincides with the set of $x\in X$ such that
$W^s(x)\pitchfork W^u(q)\neq \emptyset$ for some $q\in A_0$. Since all points in $A_0$ are $f^p$--homoclinically related,  $A$ must be contained in a  homoclinic class of $f^p$.
From the considerations above, any two points $x,y\in X$ such that $W^u(x)\pitchfork W^s(y)\neq \emptyset$
belong to a same set $f^i(A)$.
As a consequence the disjoint sets $A, f(A),\dots,f^{p-1}(A)$ are contained in different homoclinic classes of $f^p$.
This implies that $A$ is equal to a Borel homoclinic class of $f^p$.
\end{proof}

\subsubsection{MME on  Borel Homoclinic Classes}
The next result was stated and proved  in~\cite[Cor 3.3]{BCS-1} for surface diffeomorphisms, but the proof works  in any dimension:
\begin{proposition}\label{p.mme-on-bhc}
Each Borel homoclinic class $X$ carries at most one measure $\mu$ such that $h(f,\mu)=h_\Bor(f|_{X})$. This measure,
if it exists, is isomorphic to the product of a  Bernoulli scheme with a cyclic permutation of
$\text{period}(X)$ points. In particular, $\mu$ is mixing exactly when the period of the class is $1$.
\end{proposition}

\begin{proposition}\label{p.SPRfinite}
Given any Pesin block $\Lambda$, there exist at most finitely many Borel homoclinic classes which intersect $\Lambda$.
\end{proposition}
\begin{proof}
Suppose $\Lambda$ is  a Pesin block $\Lambda$, and $X_1,X_2,\ldots$ are  Borel homoclinic classes  which intersect $\Lambda$. Fix some  $x_i\in X_i\cap\Lambda$. Without loss of generality, $(x_i)_{i\geq 1}$ converges (otherwise pass to a subsequence).
By Theorem~\ref{t.pesin}, the local stable and unstable manifolds of the points in $\Lambda$ vary continuously in the
$C^1$-topology. 
As a consequence, for any $i,j$ large enough, the stable and unstable manifolds of  $x_i$ and $x_j$
have transverse intersections. This implies that for all $i$ large enough, the points $x_i$ are homoclinically related, hence the classes $X_i$ are the same.
\end{proof}

If $\dim M=2$, then we have the following global finiteness result, which follows from Remark 6.2 and Thm~6.12 in 
 \cite{BCS-1},  and Remark  \ref{r.topological-class} above:
\begin{proposition}\label{p.SPRfinite2}
Fix $1<r\leq\infty$ and let $f$ be a $C^r$ diffeomorphism on a closed surface. For any $\delta>0$
there are only finitely many Borel homoclinic classes $X$ s.t.
$$
  h_\Bor(f|_{X})>(\lambda_{\max}(f)/r)+\delta
  \qquad (\lambda_{\max}(f)/\infty=0 \text{ by convention}).
$$
\end{proposition}

\subsubsection{Borel Homoclinic Classes and the SPR Property}\label{sss-BHC-SPR}
We now prove Prop.~\ref{p.decomposition}.
We start by making the following simple general observation. Let $A, B, A_1,\dots,A_\ell$ be 
invariant measurable sets, then
 \begin{enumerate}
  \item If $A\subset B$, $B$ is SPR, and $h_\Bor(f|_A)=h_\Bor(f|_B)$, then $A$ is also SPR.
  \item $[\forall i$,\;  $A_i$ is SPR and $h_\Bor(f|_B)<h_\Bor(f|_{A_i})]$
  $\Rightarrow$  $A_1\cup \dots \cup A_\ell\cup B$ is SPR.
 \end{enumerate}

First assume that the diffeomorphism $f$ is SPR, i.e., that eq.~\eqref{e.SPR} is satisfied for some Pesin block $\Lambda$ and numbers $h_0<h_\Bor(f|_X)$ and $\tau>0$. Each ergodic measure $\mu$ with entropy larger than $h_0$ satisfies $\mu(\Lambda)>\tau$. So $\mu$ has a set of  positive measure of points with  non-vanishing Lyapunov exponents. Such measures are  hyperbolic, and we obtained Item (a) of Prop.~\ref{p.decomposition}.

Every homoclinic class with top entropy larger than $h_0$ carries an ergodic measure $\mu$ with entropy larger than $h_0$. Necessarily,  $\mu(\Lambda)>\tau$.
Prop.~\ref{p.SPRfinite} implies that there exist only finitely many classes carrying such measures, hence the number of Borel homoclinic classes
with entropy larger than $h_0$ is finite. This is Item (b).

In particular, one can increase $h_0$ so that all the homoclinic classes with top entropy larger than $h_0$ have top entropy
exactly equal to $h_\top(f)$.
Hence Item~(c) follows from  Observation (1) above.

Conversely, suppose $f$  satisfies properties  (a), (b), (c) from the statement of Proposition \ref{p.decomposition}. By (b), there is a constant  $h$ such that the  classes $X$ with $h_{\Bor}(f|_X)>h$ form  a finite family  $X_1,\dots,X_\ell$.
Increasing $h$ (while keeping $h<h_\top(f)$),  one can assume that $h_{\Bor}(f|_{X_i})=h_\top(f)$ for all $i$.
Each of these classes is SPR (by Item~(c)).
By Item~(a) and Proposition~\ref{p.homoclinic}(3), the top entropy of $Y:=M\setminus(X_1\cup\dots\cup X_\ell)$ is at most $h$. Since $M=Y\cup X_1\cup\dots\cup X_\ell$, Observation~(2) implies that $f$ is SPR.
\hfill$\Box$

\subsubsection{SPR Diffeomorphisms with Zero Topological Entropy}\
\begin{proposition}\label{p.SPR-in-Zero-Entropy}
A $C^{1+}$ diffeomorphism $f$  with zero topological entropy on a closed manifold  is  SPR iff  the following two properties hold:
\begin{enumerate}
\item[(a)] $f$ has at most  finitely many ergodic measures, and
\item[(b)] each ergodic measure is supported on a hyperbolic periodic orbit.
\end{enumerate}
\end{proposition}
\begin{proof}Suppose $f$ is  SPR and $h_{\top}(f)=0$. The constant $h$ in
Prop~\ref{p.decomposition} must be negative. Consequently, every ergodic invariant measure is hyperbolic, and the number of Borel homoclinic classes is  finite.

In addition, every Borel homoclinic class is equal to a hyperbolic periodic orbit. Otherwise the class would have contained a  transverse homoclinic intersection for some power of $f$, and $h_{\top}(f)$ would have been positive \cite{Smale,Newhouse-Homoclinic}.

It follows that the ergodic invariant measures are all carried by a finite collection of homoclinic classes, each equal to a hyperbolic periodic orbit.
This proves $(\Rightarrow)$.

Conversely, if  $f$  satisfies (a) and (b),  then $h_{\top}(f)=0$ by the variational principle, and $f$ is SPR because Definition \ref{def-SPR-diffeo} holds with $h_0$ negative and $\Lambda$ equal to the union of the (finite) collection of  hyperbolic periodic orbits of $f$.
\end{proof}

\subsection{Pliss Lemma for Measures}
Fix an invertible map  $T:\Omega\to\Omega$, and a function $\varphi:\Omega\to \R$. For every $j\geq 0$, the $j$-th {\em Birkhoff sum} of $\vf$ is
\begin{equation}\label{e.Birkhoff-Sums}
\varphi_0:=0\ ,\ \varphi_j:=\varphi+\varphi\circ T+\cdots+\varphi\circ T^{j-1}.
\end{equation}
A  {\em (forward) Pliss point} (with constant $\beta\in\R$) is a point $x$ such that
$\varphi_j(x)\geq \beta j$ for each $j\geq 0$. The next lemma estimates the size of the set of Pliss points:

\begin{lemma}\label{l.pliss}
Let $T$ be an invertible measure preserving map on a probability space $(\Omega,\mathfs F,\nu)$. Suppose $\varphi\in L^\infty(\nu)$, $A\in\R,\kappa\in [0,1]$, and
$
\nu\{x\in\Omega: \vf(x)> A\}\leq \kappa.
$
Then for every $\beta<A$,
\begin{equation}\label{e.pliss}
\nu\{x:\forall j\geq 0,\ \ \vf_j(x)\geq \beta j\}\geq \frac{\int\vf d\nu-\beta-\kappa\|\vf-A\|_\infty}{A-\beta}.
\end{equation}
\end{lemma}
\begin{remark}
In the special case $A:=\mathrm{ess}\,\sup\varphi$ and  $\kappa=0$, the lemma says that $$ \nu\{x:\forall j\geq 0,\ \ \vf_j(x)\geq \beta j\}\geq \frac{\int\vf d\nu-\beta}{A-\beta}.$$
If $\nu$ is ergodic, this implies that for $\nu$-almost every $y$, the density of $k$ such that
$x=T^k(y)$ is a Pliss point with constant $\beta$
is bounded from below by
$\frac{\int\vf d\nu-\beta}{A-\beta}$.
This is the classical Pliss lemma, see for instance
\cite[Chapter IV, Lemma 11.8]{Mane}.
Lemma~\ref{l.pliss} relaxes the assumptions and allows $\varphi$ to be larger than $A$ on a set with small measure.
Another generalized version of Pliss in a similar spirit but with different assumptions can be found in~\cite[Lem~A]{andersson-vasquez}.
\end{remark}
\begin{proof}[Proof of Lemma~\ref{l.pliss}] Without loss of generality $\beta=0$, otherwise we 
subtract $\beta$ from $\varphi$ and $A$. 
Let $P:=\{x: \forall j\geq 0\ \vf_j(x)\geq 0\}$. By the Kakutani-Yosida maximal inequality~\cite{Kakutani-Yosida}, 
 $
 \displaystyle  \int_{[\inf\limits_{j\geq 1}\frac{1}{j}\varphi_j<0]}\varphi\, d\nu\leq 0.
 $
Hence, 
\begin{align*}
 \int_\Omega \varphi d\nu = \int_P \varphi d\nu 
 + \int_{[\inf\limits_{j\geq 1}\frac{1}{j}\varphi_j< 0] }\varphi d\nu+\int_{[\inf\limits_{j\geq 1}\frac{1}{j}\varphi_j\geq  0]\cap[\varphi<0] }\varphi d\nu
 \leq \int_P \varphi d\nu.
\end{align*}
It now follows that 
\begin{align*}
\int_\Omega \varphi d\nu&\leq \int_P \varphi d\nu=
\int_{P\cap[\vf\leq A]}(\vf-A)d\nu+\int_{P\cap[\vf>A]}(\vf-A)d\nu+A\nu(P)\\
&\leq \int_{P\cap[\vf>A]}(\vf-A)d\nu+A\nu(P)\\
&\leq \|\vf-A\|_\infty \nu[\vf>A]+A\nu(P)\leq \kappa\|\vf-A\|_\infty+A\nu(P).
\end{align*}
Thus, $\nu(P)\geq \frac{1}{A}(\int\varphi\, d\nu-\kappa\|\vf-A\|_\infty)$, which is~\eqref{e.pliss} when $\beta=0$.
\end{proof}

\subsection{Tempered Envelopes}Given an invertible map  $T:\Omega\to\Omega$, a function $\Pi:\Omega\to (0,\infty)$ and $\varepsilon>0$,
we define the {\em  $\epsilon$-tempered envelope $\Pi_\eps$} by
$$
\Pi_\eps:\Omega\to(0,\infty]\ , \ \Pi_\epsilon(x):=\sup\{e^{-|n|\epsilon}\Pi(T^n x): n\in\Z\}.
$$
If $\Pi_\eps$ is finite at $x$, then $\Pi_\eps$ is finite on the orbit $\cO(x)$, and $\Pi_\eps:\cO(x)\to (0,\infty)$ is  the smallest function on $\cO(x)$ such that  $\Pi_\epsilon\geq \Pi$ and
$
e^{-\epsilon}\Pi_\epsilon\leq \Pi_\epsilon\circ T\leq e^{\epsilon}\Pi_\epsilon
$.
\begin{lemma}\label{l.Tempered-Tail}
Let $T$ be an invertible ergodic probability preserving map on a probability space  $(\Omega,\mathfs F,\nu)$. Suppose $\Pi:\Omega\to (0,\infty)$ is  measurable, and $\varepsilon>0$.
\begin{enumerate}[(1)]
\item\label{i.tempered1}  If $\log\tfrac{\Pi\circ T}{\Pi}\in L^1(\nu)$, then $\Pi_\epsilon(x)$ is finite at almost every point $x$.
\item\label{i.tempered2}  If $\exp(-c_0)\leq \frac{\Pi(Tx)}{\Pi(x)}\leq \exp(c_0)$ for some constant $c_0>0$, then for any $t>0$,
$$
\nu\{x: \Pi_\epsilon(x)>t\}\: \leq \:  4 \max(\tfrac{c_0}{\epsilon},1)\:\nu\{x: \Pi(x)>t\}.
$$
\end{enumerate}
\end{lemma}
\begin{proof}
\eqref{i.tempered1} is due to Pesin; we recall the proof.
If $\log\tfrac{\Pi\circ T}{\Pi}\in L^1(\nu)$ then by Birkhoff's ergodic theorem, $\tfrac{1}{n}\log \Pi(T^n x)\to 0$
as $n\to\pm \infty$ for almost every point $x$.
So $e^{-|n|\epsilon}\Pi(T^n x)\to 0$, and $\Pi_\epsilon(x)$ is finite.

To prove~\eqref{i.tempered2} it is sufficient to consider  ${c_0}>\epsilon$, because if ${c_0}\leq \epsilon$, then $\Pi_\epsilon\equiv\Pi$, and \eqref{i.tempered2} is obvious.
Observe that $\Pi_\epsilon(x):=\max(\Pi^s_\epsilon(x), \Pi^u_\epsilon(x))$, where
$$
\Pi^s_\epsilon(x):=\sup_{n\geq 0}e^{-n\epsilon}\Pi(T^n x)\ , \
\Pi^u_\epsilon(x):=\sup_{n\geq 0}e^{-n\epsilon}\Pi(T^{-n} x).
$$

We claim that $\nu\{x: \Pi_\epsilon^s(x)>t\}\: \leq \:  \tfrac{2c_0}{\epsilon}\:\nu\{x: \Pi(x)>t\}$. 
The following short proof was shown to us by Yuntao Zang \cite{Zang-P-C}:
 $$\begin{aligned}
  \nu(\{x:\Pi^s_\eps(x)>t\} 
  &\le^{(1)} \sum_{n\ge0} \nu(\{x:\Pi(T^nx)>t\cdot e^{\eps n},\; \Pi(T^{n+1}x)\le t\cdot e^{\eps (n+1)}\})\\
  &\le^{(2)} \sum_{n\ge0} \nu(\{x:\Pi(x)>t\cdot e^{\eps n},\; \Pi(Tx)\le t\cdot e^{\eps (n+1)}\})\\
  &\le^{(3)} \sum_{n\ge0} \nu(\{x:\Pi(x)>t\cdot e^{\eps n},\; \Pi(x)\le t\cdot e^{\eps (n+1)+c_0}\})\\
  &\le^{(4)} \int_{[\Pi>t]}\hspace{-0.2cm} \#\left\{n\in\NN: \tfrac1\eps\log\frac{\Pi(x)}{t}-1-\tfrac{c_0}\eps \le n < \tfrac1\eps\log\frac{\Pi(x)}{t}\right\}\, d\nu\\
  &\le \left(1+\tfrac{c_0}\eps\right) \cdot \nu(\{x:\Pi(x)>t\})\le \tfrac{2c_0}{\eps}\cdot \nu(\{x:\Pi(x)>t\}).
 \end{aligned}$$
 $\leq^{(1)}$ is because  $\tfrac{1}{n}\log \Pi(T^n x)\to 0$ a.e.; $\leq ^{(2)}$ uses  $T$-invariance;   $\leq ^{(3)}$ is because  $\Pi(x)\le e^{c_0}\Pi(Tx)$; and $\leq^{(4)}$ follows by exchanging the sum and the integral.

By symmetry, we also have $\nu\{x: \Pi_\epsilon^u(x)>t\}\: \leq \:  \tfrac{2c_0}{\epsilon}\:\nu\{x: \Pi(x)>t\}$. Since $\Pi_\eps=\max(\Pi_\eps^s,\Pi_\eps^u)$, a union bound gives Item \eqref{i.tempered2}.
\end{proof}

\subsection{Pesin Blocks with Measure Bounded From Below}

\begin{proposition}\label{p.pesin}
Let $f$ be a $C^1$ diffeomorphism of a closed manifold $M$.
Given $n_0\geq 1$ and $\chi>0$, let $P_{n_0}:=P_{n_0}(\chi)$ be the set
of Pliss points
defined by
\begin{equation}\label{e.PlissSet}
P_{n_0}:=
\bigg\{x
\colon\begin{array}{l}
 \text{there is a splitting } T_x M=E(x)\oplus F(x)\text{ s.t. }  \forall j\geq 0,\\
 \|Df^{j n_0}|_{E(x)}\|\leq e^{-\chi j n_0}
\text{and } \|Df^{-j n_0}|_{F(x)}\|\leq e^{-\chi j n_0}
\end{array}\bigg\}.
\end{equation}
Then, for every $\eps>0$ there exists a $(\chi,\varepsilon)$-Pesin block $\Lambda_{n_0}=\Lambda_0(\chi,\eps)$ such that
\begin{equation}\label{e.pesin}
\nu(M\setminus\Lambda_{n_0})\; \leq\;  4\max(\tfrac{c_0}\varepsilon , 1)\;  \nu(M\setminus P_{n_0}) \text{ for all invariant  measures $\nu$,}
\end{equation}
where $c_0=\chi+\max(\log\|Df\|_\supnorm,\; \log\|Df^{-1}\|_\supnorm).$
\end{proposition}

\begin{proof}
Let $\nu$ be an invariant measure. We assume $\nu$ to be ergodic, as the general case follows by averaging over the ergodic decomposition. We also assume $\nu(P_{n_0})>0$ (otherwise there is nothing to prove). These two properties imply that
 the set $M'$ of Oseledets-regular points $x$ with no zero Lyapunov exponent has full $\nu$-measure. Thus, at each point $x\in M'$, there is a splitting $T_x M=E^-_x\oplus E^+_x$ such that for every nonzero $v\in E^-_x$, $\lim\limits_{n\to\infty}\tfrac{1}{n}\log\|Df^{ n}.v\|$ and $\lim\limits_{n\to-\infty}\tfrac{1}{n}\log\|Df^{-n}.v\|$  exist, and are negative; Similarly for every nonzero $v\in E^+_x$ the  limits exist and are positive.
Thus,  for $x\in P_{n_0}\cap M'$, the decomposition $T_x M=E(x)\oplus F(x)$ must coincide with $T_xM=E^-_x\oplus E^+_x$.
We define for  $x\in M'$:
$$
\Pi(x):=\max(\Pi^s(x),\Pi^u(x)),\text{ where}
\begin{cases}
\Pi^s(x):=\sup_{k\geq 0} (\|Df^k|_{E^-(x)}\|\exp(\chi k)) & \\
\Pi^u(x):=\sup_{k\geq 0} (\|Df^{-k}|_{E^+(x)}\|\exp(\chi k)).&
\end{cases}
$$
Let $C:=\max_{0\leq k\leq n_0}\max(\|Df^k\|,\|Df^{-k}\|)$, then
$$P_{n_0} \cap M'\subset \{x\colon \Pi(x)\leq C\}.  $$

Let $\Pi_\eps$ denote the $\eps$-tempered envelope of $\Pi$. For every $n\in\ZZ$ and $k\geq 0$,
$$\|Df^k|_{E^-(f^n(x))}\|\leq \Pi(f^n(x)) \exp(-\chi k) \leq \Pi_\varepsilon(x)\exp(-\chi k+\varepsilon |n|),$$
$$\|Df^{-k}|_{E^+(f^n(x))}\|\leq \Pi(f^n(x)) \exp(-\chi k) \leq \Pi_\varepsilon(x)\exp(-\chi k+\varepsilon |n|).$$
Therefore, the following set is a $(\chi,\varepsilon)$-Pesin block:
$$\Lambda_{n_0}:=\{x\colon \Pi_\varepsilon(x)\leq C\}.$$

Note that 
 $
   \|(Df|_{E^-(x)})^{-1}\|^{-1} \le 
       \frac{\|Df^{k+1}|_{E^-(x)}\|}{\|Df^k|_{E^-(f(x))}\|} 
     \le \|Df|_{E^-(x)}\|.
 $
Consequently,
$$\|(Df|_{E^-(x)})^{-1}\|^{-1}e^{\chi}\cdot \Pi^s(f(x))
          \le \Pi^s(x) \le \max\bigl(1,\|Df|_{E^-(x)}\|\cdot e^{\chi}\cdot \Pi^s(f(x))\bigr).
$$
Therefore,
 $
   \Pi^s(f(x)) \le  \|(Df|_{E^-(x)})^{-1}\| \cdot e^{-\chi} \cdot \Pi^s(x) 
 $. In addition,  if $\Pi^s(x)>1$,
 $
   \Pi^s(x)\le \|Df|_{E^-(x)}\| \cdot e^{\chi} \cdot \Pi^s(f(x)).
 $
Otherwise $\Pi^s(x)=1\le \Pi^s(f(x))$. 
Similar inequalities hold for $\Pi^u$. In conclusion,
\begin{equation}\label{e.bound-tempered}
\exp(-c_0)\: \Pi(x)\leq \Pi(f(x))\leq \exp(c_0)\: \Pi(x),
\end{equation}
where  $c_0:=\chi+\max(\log\|Df\|_\supnorm,\; \log\|Df^{-1}\|_\supnorm)>0$.

$\Pi(x)$ is well-defined and finite  $\nu$-almost everywhere. 
By Lemma~\ref{l.Tempered-Tail}
and~\eqref{e.bound-tempered}, the envelope $\Pi_\varepsilon(x)$ is also finite $\nu$-a.e., and 
\begin{align*}
\nu(M\setminus \Lambda_{n_0}) &=
\nu(\{x:\Pi_\eps(x)>C\})\leq
4\max(\tfrac{c_0}{\varepsilon},1)\cdot\mu\{x\colon \Pi (x)> C\}\\
&\leq 4\max(\tfrac{c_0}{\varepsilon},1)\cdot\mu(M\setminus P_{n_0}).
\tag*{\qedhere}\end{align*}
\end{proof}

\section{Deducing SPR  from Properties of Lyapunov Exponents}\label{s.SPR-exponents}

In this section we give a sufficient condition for the SPR property of $C^1$ diffeomorphisms in  dimension $d\geq 2$ (Theorem \ref{thm-CE-SPR}).\footnote{
For $C^{1+}$ diffeomorphisms, this condition is  also   necessary. See Theorem~\ref{thm-characterizations}.}  The condition is in terms of the   Lyapunov exponents of   measures with nearly maximal entropy, and  consists of two properties:  {\em entropy hyperbolicity} (EH) and {\em entropy continuity} (EC), see \S \ref{s.thm-CE-SPR}.

This sufficient condition  is the key for the proof  that  {\em every} $C^\infty$ surface diffeomorphism with positive entropy is SPR (Theorem~\ref{t.SPR-C-infinity}). 
Indeed, it implies much more: For $C^\infty$ surface diffeomorphisms, all homoclinic classes (Borel or topological) are SPR (Thm~\ref{t.tight-C-infinity});  and many  $C^r$ surface diffeomorphisms (or homoclinic classes of $C^r$ diffeomorphisms) are SPR (Thm~\ref{t.tight-Cr}). 

To verify properties (EH) and (EC) in these cases, we use \cite{BCS-2} in the $C^\infty$ scenario, and  \cite{Burguet-Cr} in the $C^r$ scenario.

\subsection{Notation}\label{ss.def-of-Lambda-plus}
For the remainder of \S\ref{s.SPR-exponents}, $f$ is a $C^1$ diffeomorphism of a $d$-dimen\-sional closed manifold $M$ with $d\ge2$. There are some simplifications when $d=2$. We will point these out as we go.

Recall from  \S\ref{s.NUH} that  $\lambda^1(x)\ge\dots\ge\lambda^d(x)$ is the ordered list of Lyapunov exponents, repeated according to multiplicity. Let $\sigma(x):=\{\lambda_1(x),\ldots,\lambda_d(x)\}$, and let
$T_x M=\oplus_{\lambda\in \sigma(x)} E^\lambda_x$ be the Oseledets decomposition.
The {\em unstable space} of $x$ is 
 $E_x^+:=\underset{\lambda\in\sigma(x)}{\oplus_{\lambda>0}}E^\lambda_x$. 
 The {\em unstable dimension of $x$} is
 $i(x):=\dim E_x^+$. 
In the two-dimensional case, $\dim(E_x^+)=1$ a.e. for all ergodic measures with positive entropy. In higher dimensions, the  unstable dimension
 could have other values. We have:
$$i(x):=\max\{i:\lambda^i(x)>0\}\text{ or zero, if the set is empty}.$$
The {\em sums of the positive (resp. negative) exponents} are defined  by
$$
 \Lambda^+(x):=\sum_{i=1}^d \max(\lambda^i(x),0)\ , \ \Lambda^-(x):=\sum_{j=1}^d \min(\lambda^j(x),0).
$$
Then $\Lambda^{\pm}(x)=\Lambda^{\pm,i(x)}(x)$, where
$$\Lambda^{\pci}(\nu):=\lambda^1(\nu)+\dots+\lambda^i(\nu), \quad \Lambda^\mci(\nu):=\lambda^{i+1}(\nu)+\dots+\lambda^d(\nu).$$
If $\nu$ is ergodic, then $\lambda^i(x), \Lambda^{\pm}(x), \Lambda^{\pm,i}(x)$ and $i(x)$ are constant a.e., and we denote their constant values by $\lambda^i(\nu), \Lambda^{\pm}(\nu), \Lambda^{\pm,i}(\nu)$ and $i(\nu)$. For non-ergodic measures,
$$\lambda^i(\nu):=\int \lambda^i(x)\,d\nu(x),\quad  \Lambda^\pm(\nu):=\int \Lambda^\pm(x)\,d\nu(x), \quad \Lambda^{\pm,i}(\nu):=\int \Lambda^{\pm,i}(x)\, d\nu.$$

We write $\mu_n\wsc\mu$, when $\mu_n$ converge weak-$\ast$ to $\mu$ on $M$, i.e. when $\mu_n(\vf)\to\mu(\vf)$ for all continuous $\vf:M\to\R$. We emphasize that even  in cases when we consider $\mu_n\in\mathbb P(f|_X)$, the test functions are continuous functions on $M$, and not just on $X$.

\subsection{A Sufficient Condition for SPR in Terms of Lyapunov Exponents}\label{s.thm-CE-SPR}
Let $X\subset M$ be an  invariant Borel set ($X$ may be  non-compact). Consider the following two conditions on $f|_X$:
\begin{enumerate}
\item[(EH)] \emph{\bf Entropy Hyperbolicity on $\boldsymbol X$}:
There is $\chi>0$ as follows. 
Suppose $\mu_n\in \Proberg(f|_X)$ and $\mu_n\wsc\mu$ on $M$. If   $h(f,\mu_n)\to h_{\Bor}(f|_X)$, then  there exists a constant $i:=i(\mu)$ such that 
$ \lambda^i(x) >\chi>-\chi>\lambda^{i+1}(x)\text{ $\mu$-a.e.}$

\smallskip
\item[(EC)] \emph{\bf Entropy Continuity of $\boldsymbol{\Lambda^+}$ on $\boldsymbol{X}$:}
Suppose   $\mu_n\in \Proberg(f|_X)$ and 
$\mu_n\wsc\mu$ on $M$. If $h(f,\mu_n)\to h_{\Bor}(f|_X)$, then
$
    \lim_{n\to\infty} \Lambda^{+}(\mu_n)=\Lambda^{+}(\mu).
$

\end{enumerate}

\begin{theorem}\label{thm-CE-SPR}
Let $f$ be a $C^1$ diffeomorphism of a closed manifold, and let $X\subset M$ be a nonempty invariant Borel set. If $f$ is entropy hyperbolic and $\Lambda^+$ is entropy continuous on $X$, then $f$ is  SPR on $X$.
\end{theorem}

\begin{remark}\label{r.EH-in-Dim-2}
Property (EH) is always true for $C^\infty$ surface diffeomorphisms and invariant Borel sets $X$ such that $\hTOP(f|_X)=h_\top(f|_{\ov{X}})>0$ (e.g., $X=M$). The reason is that in this case, by Newhouse's upper-semicontinuity theorem \cite{Newhouse-Entropy}, a.e. ergodic component of $\mu$ must be an MME. Hence, by Ruelle's entropy inequality, (EH) holds with $i:=1$ and $\chi:=h_\top(f(|_{\ov{X}})/2$.
\end{remark}

\begin{remark}\label{r.Lambda-Minus}
Property (EC), the entropy continuity of $\Lambda^+$ on $X$, is equivalent to the entropy continuity of $\Lambda^-$ on $X$. This is because, by the Oseledets theorem, 
$\displaystyle\Lambda^-(\mu_n)=\int \log|\det(Df)|d\mu_n-\Lambda^+(\mu_n),$  and $\log|\det(Df)|$  is continuous on $M$.\end{remark}

\begin{remark}
If $d>2$, then  the ergodic measures  $\mu_n$ in (EC) may have different unstable dimensions. However,  
Lemma \ref{lem-AB-index} below shows that if  (EH) holds, then   $i(\mu_n)$ is eventually constant, and  equal to $i(\mu)$. In particular, for all $n$ large enough, $\Lambda^+(x)=\sum_{j=1}^{i(\mu)}\lambda^j(x)$ $\mu_n$-a.e., and not just $\mu$-a.e. (In dimension two, this follows directly from Ruelle's inequality.)
\end{remark}

\begin{remark}\label{rem-CE-beyondentropy}
The careful reader will notice that the proof of Theorem~\ref{thm-CE-SPR} below does not depend on any specific property of entropy. Indeed, Thm \ref{thm-CE-SPR} still holds if one generalizes the notions of SPR, (EH), and (EC) by replacing the entropy by some arbitrary function $P:\Proberg(f|_X)\to\RR$ and $\hTOP(f|_X)$ by the corresponding supremum $P_X:=\sup_{\mu\in\Proberg(f|_X)} P(\mu)$. More precisely, the three notions may be generalized as follows.

\medbreak\noindent
\emph{Generalized SPR on $X$:} 
there is $\chi>0$ such that for each $\eps>0$, there are a $(\chi,\eps)$-Pesin block $\Lambda$ and numbers $P_0<P_X$ and $\tau>0$ satisfying:
For any $\nu\in\Proberg(f|_X)$,
   $P(\nu)>P_0\implies \nu(\Lambda)>\tau.
 $

\smallbreak\noindent
\emph{Generalized (EH) on $X$:} there is $\chi>0$ as follows. 
Suppose $\mu_n\in \Proberg(f|_X)$ and $\mu_n\wsc\mu$ on $M$. If   $P(\mu_n)\to P_X$, then  there exists a constant $i:=i(\mu)$ such that 
$ \lambda^i(x) >\chi>-\chi>\lambda^{i+1}(x)\text{ $\mu$-a.e.}$

\smallbreak\noindent
\emph{Generalized (EC) on $X$:} Suppose $\mu_n\in \Proberg(f|_X)$ and 
$\mu_n\wsc\mu$ on $M$. If $P(\mu_n)\to P_X$, then
$
    \lim_{n\to\infty} \Lambda^{+}(\mu_n)=\Lambda^{+}(\mu).
$
\smallbreak

The reader may check that the proof of Theorem~\ref{thm-CE-SPR} generalizes almost verbatim, just replacing the Kolmogorov-Sinai entropy by the function $P$ and the top entropy of $X$ by $P_X$. It shows that if $f|_X$ satisfies the generalized (EH) and (EC), then it also satisfies generalized SPR.
\end{remark}

The  proof of the  theorem is given below. First we will use (EH)  to see that every $\mu_n\in\mathbb P(f|_X)$ with sufficiently high entropy must be hyperbolic. Every ergodic hyperbolic measure $\mu$ has a canonical lift to  the Grassmannian bundle, given by $\wt{\mu}^+:=\int_M \delta_{(x,E^+_x)}d\mu(x)$. We will use  (EC) to prove that $\wt{\mu}_n^+\wsc\wt{\mu}^+$ for many sequences of measures $\mu_n$ such that $\mu_n\wsc\mu$ and $h(f,\mu_n)\to h_{\Bor}(f|_X)$ (Lemma \ref{lemma-hyplim0}). This gives  some control of the hyperbolicity properties of  measures with high entropy, and  opens the way to using the Pliss Lemma for measures (Lemma~\ref{l.pliss}) and  Proposition~\ref{p.pesin}, to  produce  a Pesin block which has large measure for all measures with high entropy. SPR follows.

\subsection{Preparations for the Proof}\label{sec-prelim}  
\subsubsection{Continuity Lemmas}
The following semicontinuity is well-known:

\begin{proposition}\label{prop-uscLambda}
For any $C^1$ map $f$ of a closed manifold $M$,
if $\mu_n\in\Prob(f)$ and  $\mu_n\wsc\mu$, then $\limsup\limits_{n\to\infty} \Lambda^{\pci}(\mu_n)\leq \Lambda^{\pci}(\mu)$ and 
$\liminf\limits_{n\to\infty} \Lambda^{\mci}(\mu_n)\geq \Lambda^{\mci}(\mu)$, 
for all $i$.
\end{proposition}
\begin{proof}
It is sufficient to prove the result for $\Lambda^{\pci}$, see Remark \ref{r.Lambda-Minus}.

If $i=1$, then $\Lambda^{+,1}(\nu)=\int \lim\tfrac{1}{n}\log\|Df^n\|d\nu$. 
By  the sub-additive ergodic theorem, $\Lambda^{+,1}(\nu)=\inf_{n\geq 1}\tfrac{1}{n}\int \log\|Df^n\|d\nu$, an upper semi-continuous function of $\nu$.

If $d=2$ we are done. If $d>2$ and $i>1$, then we use  Raghunathan's identity to express $\Lambda^\pci(\mu)$ as the top Lyapunov exponent of an exterior power of the derivative cocycle $Df$ (see e.g.~\cite[Prop.~2.2]{bochi-viana}). Then we proceed as in the case $i=1$.
\end{proof}

The following studies the unstable dimensions of $\mu_n\wsc\mu$ in the setting of (EH):
\begin{lemma}\label{lem-AB-index}
Suppose $\mu_n\in\Proberg(f|_X)$,   $\mu_n\wsc\mu$, and  there is $0<i<d$ such that $\lambda^{i}(x)>0>\lambda^{i+1}(x)$ $\mu$-a.e. Then:
\begin{enumerate}
\item If $\Lambda^+(\mu_n)\to\Lambda^+(\mu)$, then $\mu_n$ is hyperbolic and   $i(\mu_n)=i$ for all $n$ large.
\item $\Lambda^+(\mu_n)\to\Lambda^+(\mu)$ iff $\sum_{k=1}^{i}\lambda^k(\mu_n)\to \sum_{k=1}^{i}\lambda^k(\mu_n)$.
\end{enumerate}
\end{lemma}
\noindent
(The lemma is obvious in the two-dimensional case.)

\begin{proof}
Let $\mu_n\wsc\mu$ as above.
To prove (1), it is sufficient to consider the special case when all the $\mu_n$ have the same unstable dimension $j$. Thus we assume $\Lambda^+(\mu_n)=\Lambda^{+,j}(\mu_n)\to\Lambda^{+,i}(\mu)$ and deduce $j=i$.  

Assume by contradiction that $j\ne i$.
Suppose first that $j=i+k$ with $k$ positive. By the upper semicontinuity of $\Lambda^\pc{i+k}(\nu)$ (Prop.~\ref{prop-uscLambda}) and the hyperbolicity of $\mu$,
 $$\begin{aligned}
   \limsup_{n\to\infty} \Lambda^+(\mu_n)&=\limsup_{n\to\infty} \Lambda^\pc{i+k}(\mu_n) \le \Lambda^\pc{i+k}(\mu)\\
   &= \Lambda^\pci(\mu)+\lambda^{i+1}(\mu)+\dots+\lambda^{i+k}(\mu)
   <\Lambda^\pci(\mu)=\Lambda^+(\mu),
 \end{aligned}$$
which contradicts the assumption in (1).
Next suppose $j=i-k$ with $k\geq 1$. By Prop.~\ref{prop-uscLambda} $\liminf_{n\to\infty}\Lambda^{-,j}(\mu_n)\geq \Lambda^{-,j}(\mu)$, and proceeding as before we obtain $\liminf \Lambda^-(\mu_n)>\Lambda^-(\mu)$. But this contradicts Remark \ref{r.Lambda-Minus}.
Thus $j=i$. 
\medskip

We now check the hyperbolicity of $\mu_n$ for large $n$.
In particular, $\Lambda^\pci(\mu_n)\to \Lambda^\pci(\mu)$.
Also, $\limsup_{n\to\infty} \Lambda^\pc{i+1}(\mu_n) \le \Lambda^\pc{i+1}(\mu)$ which is strictly smaller than $\Lambda^\pc{i}(\mu)$, by (EH). Thus $\Lambda^\pc{i+1}(\mu_n)<\Lambda^\pc{i}(\mu_n)$ for $n$ large, whence $\lambda^{i+1}(\mu_n)$ is eventually negative.  Thus  $\mu_n$ is hyperbolic for all large $n$, and  Item (1) follows.

Let us deduce item (2). The direct implication $\Rightarrow$ in item (2) is an immediate consequence of item (1).
For the converse, we assume $\Lambda^\pci(\mu_n)\to\Lambda^\pci(\mu)=\Lambda^+(\mu)$ and that, without loss of generality, the unstable dimension of $\mu_n$ is some integer $j$ for all large $n$. As above, the semicontinuity of $\Lambda^{+,j}$, together with the hyperbolicity of $\mu$, gives that if $j\ne i$,
 $$
   \limsup_n \Lambda^{+,j}(\mu_n)\le \Lambda^{+,j}(\mu) < \Lambda^+(\mu)=\Lambda^\pci(\mu)=\lim_n \Lambda^\pci(\mu_n).
 $$
This implies that 
for large $n$, $\Lambda^i(\mu_n)>\Lambda^j(\mu_n)$, a contradiction to $j$ being the unstable dimension of the $\mu_n$'s. Thus $j=i$, whence $\Lambda^+(\mu_n)=\Lambda^{+,i}(\mu_n)$. Direction~$\Leftarrow$ in Item (2) easily follows.
\end{proof}

\subsubsection{Grassmannian Extensions}\label{ss.grassmannian-extensions}

Suppose $1\leq i\leq d$. 
The  {\em $i$-th Grassmannian bundle} of $M$ is the bundle $\mathfrak G(i,M)$  with base space $M$, and fibres  $$\mathbb G(i,T_x M):=\{E\subset T_x M:E\text{ is a linear $i$-dimensional subspace}\}.$$ There is a natural smooth Riemannian structure on  $\mathfrak G(i,M)$, see Appendix~\ref{app-grassmannian}.
We denote the  points of $\mathfrak G(i,M)$  by $(x,E)$ ($x\in M\, , \, E\in\mathbb G(i,T_x M)$). There is a natural projection $\tpi:\mathfrak G(i,M)\to M$, given by $\tpi(x,E)=x$.

Any diffeomorphism $f$ on $M$ induces a diffeomorphism  $\tf$ on $\mathfrak G(i,M)$ acting by
$$       \tf(x,E)=(f(x),D_xf.E). $$

Let $\mu$ be a (possibly non-ergodic) $f$-invariant measure such that
the unstable dimension $i(x)$ is equal to a positive constant $i$ $\mu$-almost everywhere. Then
 $$
   \tmu^+ := \int_M \delta_{(x,E^{+}_x)}\, d\mu(x)
 $$
is an $\tf$-invariant measure on $\mathfrak G(i,M)$, called the {\em unstable lift} of $\mu$.  If $\mu$ is $f$-ergodic, then  $\tmu^+$ if $\wt{f}$-ergodic.

Recall that the \emph{the determinant} of a linear operator $T$ mapping one $i$-dimensional inner product space $E_1$ onto another $i$-dimensional inner product space $E_2$ is defined, up to sign, to be  the determinant of the matrix representing the linear operator $S_2 T S_1^{-1}:\R^i\to\R^i$, where $S_i:\R^i\to E_i$ are  
isometries. 

The {\em $i$-dimensional log-Jacobian of $f$} is the function $\vf:\mathfrak G(i,M)\to\R$ given by
$$\varphi(x,E):=\log |\det(Df_x|_E)|.$$
(The notation suppresses $i=\dim E$, because it is obvious from the context.) If $f$ is  a $C^1$ diffeomorphism, then $\vf$ is globally defined, and continuous.

\begin{proposition}\label{prop-erg-lifts}
Let $\nu$ be an $f$-invariant measure such that
the unstable dimension $i(x)$ is equal  $\nu$-almost everywhere to a constant $i$. Then: \begin{equation}\label{eq-max-dil}
   \Lambda^+(\nu)=\int_{\mathfrak G(i,M)} \vf\, d\tnu^+.
  \end{equation}
All  other $\tnu\in \mathbb P(\tf)$ such that $\tpi_*\tnu=\nu$ satisfy
$\tnu(\varphi)<\Lambda^+(\nu)$.
\end{proposition}

\noindent
This is well-known, but we could not find a reference.
We give the proof in  \S\ref{app-dilavg}.

\begin{remark}
If $i=1$, then $\Lambda^+(\nu)=\int_M\log|\det(Df|_{E^+})|d\nu$, and \eqref{eq-max-dil} is trivial. But even in this simple case, \eqref{eq-max-dil} has the advantage that $\vf$ is globally defined and continuous, whereas $\log|\det(Df|_{E^+})|$ is not.  For more information on the defect $\Lambda^+(\nu)-\tnu(\vf)$ for arbitrary lifts in this case, see \cite[Lemma 3.3]{BCS-2}.
\end{remark}

Next we use  (EH) to give a condition guaranteeing the convergence of the unstable lifts of a convergent sequence of measures $\mu_n$ the  unstable lift of $\lim\mu_n$:
\begin{lemma}\label{lemma-hyplim0}
Suppose $\mu_n\in\Proberg(f|_X)$,   $\mu_n\wsc\mu$, and  $\mu$ is hyperbolic, with a.s. constant unstable dimension $i$.
Then $\Lambda^+(\mu_n)\to\Lambda^+(\mu)$ iff
$\tmu_n^{+}\wto\tmu^{+}$ on $\mathfrak G(i,M)$.
\end{lemma}

\begin{proof}
Let $\mu_n\wsc\mu$ as above.
By assumption, there is a constant $i$ such that $i(x)=i$  $\mu$-a.e.
Suppose  $\tmu_n^{+}\wto\tmu^{+}$, with all unstable lifts living on $\mathfrak G(i,M)$. Since $\vf$ is continuous on $\mathfrak G(i,M)$, $\tmu_n^{+}(\varphi) \to \tmu^{+}(\varphi)$, and by
 eq.~\eqref{eq-max-dil}
$
    \Lambda^+(\mu_n)\to\Lambda^+(\mu).
 $

Conversely, suppose $\Lambda^+(\mu_n)\to\Lambda^+(\mu)$. 
By Lemma~\ref{lem-AB-index}, for all large $n$,  $\mu_n$ is hyperbolic with unstable dimension $i$, hence $\tmu_n$ lives on $\mathfrak G(i,M)$.
This set is compact and metrizable, therefore  $\Prob(\tf)$ is weak-$*$ compact. So it suffices to show that for any subsequence $\tmu_{n_k}^{+}$ such that $\tmu_{n_k}^{+}\wto\tnu$ on $\mathfrak G(i,M)$,  $\tnu=\tmu^+$.
 We have:
$$
   \tmu^+(\varphi) =  \Lambda^{+}(\mu) = \lim_{n\to\infty} \Lambda^{+}(\mu_n)= \lim_{k\to\infty} \tmu_{n_k}^+(\varphi)= \tnu(\varphi).
$$
Moreover $\tpi_*\tnu=\mu$. By Proposition~\ref{prop-erg-lifts}, $\tnu=\tmu^+$ as claimed.
\end{proof}

\subsection{Proof of Theorem \ref{thm-CE-SPR}}
Recall that $f$ is a $C^1$ diffeomorphism on $M$, and $X$ is some invariant measurable subset $X$. Our goal is to deduce strong positive recurrence from Conditions (EH) and (EC).

\subsubsection{Property {\rm (*)}}
\label{ss-proof-thm-CE-SPR}
We say that a $C^1$ diffeomorphism $f$  satisfies {\em Property {\rm (*)}}
on an invariant Borel set $X$ if there exists $\chi_0>0$ s.t. for each $\tau<1$ there are $h_0<\hTOP(f|_X)$ and $n_0\ge1$ so that:
\begin{equation}\label{e-exp-unif}
\text{For any ergodic measure $\nu$ on $X$,}\quad h(f,\nu)>h_0 \implies
\nu(P_{n_0}(\chi_0))>\tau.
\end{equation}

\begin{proposition}\label{prop-crit-SPR-ET}
Let $f$ be a $C^1$ diffeomorphism on a closed manifold $M$, and suppose $X$ is an  invariant measurable subset $X$. If $f$ satisfies {Property} { \rm (*)} on $X$, then $f$ is SPR  on $X$.
\smallbreak

\end{proposition}

\begin{proof}
Choose $\chi_0>0$ as in Property~(*) above and $c_0$ as in Prop.~\ref{p.pesin}. Fix     $\eps,\delta>0$  arbitrarily small. Fix $\tau<1$ so close to $1$ that
 $
    1-\tau< \tfrac\delta{4} \min(\eps/c_0,1).
 $
For this $\tau$, choose $n_0,h_0$ as in
Property (*). Proposition~\ref{p.pesin} provides us with  a $(\chi_0,\eps)$-Pesin block $\Lambda_{n_0}=\Lambda_{n_0}(\chi_0,\eps)$ such that, for any $\nu\in\Proberg(f|_X)$ with $h(f,\nu)>h_0$,
 $$
\nu(M\setminus \Lambda_{n_0})\leq 4\max(\tfrac{c_0}\varepsilon , 1) \nu(M\setminus P_{n_0}(\chi_0))<\delta.
 $$
So $\nu(\Lambda_{n_0})>1-\delta>0$. Since $\eps$ was arbitrary,  $f$ is SPR on $X$.
\end{proof}

\subsubsection{{\rm (EH)} and {\rm (EC)} Imply Property {\rm (*)}}
We suppose that $f|_X$ satisfies  Conditions (EC) and  (EH) with constant  $\chi>0$, and we prove Property (*).

Define the following sets of measures on $\mathfrak G(i,M)$:
\begin{align*}
&\stL^i:=\big\{\wt{\mu}: \exists \mu_n\in \Proberg(f|_X) \text{ s.t. } 
\wt{\mu}_n^+\rightharpoonup\wt{\mu}\text{ on }\mathfrak G(i,M),\text{ and }\\
&\hspace{ 6cm} i(\mu_n)=i,\
h(f,\mu_n)\to h_{\Bor}(f|_X)
\big\},\\
&\stH^i:=\left\{\wt{\mu}^+:\begin{array}{l}
\mu:= \tpi_*\tmu^+ \text{ has the following properties:  $i(x)=i$ $\mu$-a.e.; and }\\
\text{for $\mu$-a.e. $x$, $\lambda^1(x),\ldots,\lambda^d(x)\not\in [-\chi,\chi]$}
\end{array} \right\}.
\end{align*}

Since Property (*) is only concerned with ergodic measures of nearly maximal entropy, it is sufficient to consider measures with  unstable dimension $i$ such that 
\begin{equation}\label{eq-htop-idim}
   \sup\{ h(f,\mu): \mu\in\Proberg(f|_X)\text{ and }i(\mu)=i\}=\hTOP(f|_X).
 \end{equation}

\begin{lemma}\label{lem-stL}
$\stL^i$ is a weak-$*$ compact subset of $\stH^i$.
\end{lemma}
\begin{proof}
Compactness is clear. To see that $\stL^i\subset \stH^i$, suppose $\wt{\mu}=\lim \wt{\mu}_n^+$ with $\mu_n$ as in the definition of $\stL^i$. Then $\mu:=\tpi_\ast\wt{\mu}=\lim\mu_n$. By entropy continuity (EC), $\Lambda^+(\mu_n)\to \Lambda^+(\mu)$. By Lemma \ref{lemma-hyplim0}, $\tmu_n^+\to\tmu^+$. So $\tmu=\tmu^+$.
By entropy hyperbolicity (EH), the Lyapunov exponents of $\mu$-a.e. $x$ all lie outside $[-\chi,\chi]$, and by  Lemma~\ref{lem-AB-index}, $i(x)=\lim i(\mu_n)=i$ $\mu$-a.e.  So $\tmu=\tmu^+\in\stH^i$.
\end{proof}

For every measurable set $\tU\subset\mathfrak G(i,M)$,  $0<\gamma<1$, and $\tf$-invariant measure $\tmu$ on $\mathfrak G(i,M)$ such that $\tmu(\tU)>1-\gamma^2$,
Pliss Lemma~\ref{l.pliss} gives (with $A=1$, $\kappa=0$, $\beta=1-\gamma$ and $\varphi$ the characteristic function of $\tU$):
 \begin{equation}\label{eq-Pliss-gamma}
        \tmu\left(\left\{\tx:\forall n\ge 0\; |\{0\leq k<n:\tf^{-k}\tx\in\tU\}|\ge(1-\gamma)n\right\}\right)> 1-\gamma.
 \end{equation}
We define the number $L:=\sup_x\log\|Df_x^{-1}\|>0$ and the open sets
 $$
    \tU_{N} := \left\{(x,E)\in\mathfrak G(i,M) : \|Df_x^{-N}|_E\| < e^{-\chi N} \right\} \qquad (N\ge1),
 $$
and we prove the following estimate:

\begin{lemma}\label{lem-comp-exp-time}
Given $\stK\subset\stH^i$ compact and $\delta\in (0,1/2)$,  there are an open set
$\stV\supset\stK$ and an integer $N_0\geq 1$ such that, for all $\tmu\in\stV$, there is
$1\leq N\leq N_0$ with
 $
 \tmu(\tU_{N}) > 1-\delta^2 .
 $
\end{lemma}

\begin{proof}
Let $\tmu\in\stK$. By the definition of $\stH^i$, we have $\tmu=\tmu^+$, where $\mu$ is $\chi$-hyperbolic. By the definition of unstable lifts,  $(x,E)=(x,E^+_x)$ $\tmu$-a.e.. Therefore, by $\chi$-hyperbolicity,  there is an  $N_1=N_1(\tmu)\ge 1$
such that $\tmu(\tU_{N_1})>1-\delta^2$.

By the definition of the weak-$*$ topology, $\tmu$ has an open neighborhood $\stV(\tmu)$ s.t.
 $$
  \forall \tnu\in\stV(\tmu),\quad \tnu(\tU_{N_1})>1-\delta^2.
 $$
By the compactness of $\stK$, there are $\tmu_1,\dots,\tmu_{T}\in\stK$ such that $\stK\subset\bigcup_{\ell=1}^{T} \stV(\tmu_\ell)$. The lemma holds with $\stV:=\bigcup_{i=1}^{T} \stV(\tmu_i)$, $N_0:=\max(N_1(\tmu_1),$ $\dots,N_1(\tmu_{T}))$:
If $\tnu\in\stV(\tmu_\ell)$, take $N=N_1(\tmu_\ell)$.
\end{proof}

\begin{proof}[Proof of Theorem~\ref{thm-CE-SPR}]
Recall that condition (EH) holds with $\chi>0$, so it must be the case that $0<\chi\leq \sup_x\log\|Df_x^{-1}\|=L$.
We will prove   Property (*) for any $\chi_0\in (0,\chi)$, and the SPR property will then follow from Prop~\ref{prop-crit-SPR-ET}. 

Fix some index $1<i<d$ which satisfies \eqref{eq-htop-idim}, and some  $\tau\in (0,1)$. To begin with, we fix $\eps>0$ so that $(1-\eps)\chi=\chi_0$ and apply Lemma~\ref{lem-comp-exp-time} with $\stK:=\stL^i$ and
 $$
  \delta:=\min(\eps\cdot\chi/(4L+4),(1-\tau)/4),
 $$
obtaining $N_0,\stV$. By the definition of $\stL^i$ as a limit set, there is $h_0=h_0(i)\!<\!\hTOP(f|_X)$ such that $\stV$ contains the unstable lift $\tmu^+$ of all $\mu\in\Proberg(f|_X)$ with $i(\mu)\!=\!i$ and  $h(f,\mu)\!>\!h_0$.

Fix such a measure $\mu$. Lemma~\ref{lem-comp-exp-time} gives $1\le N\le N_0$ with $\tmu^+(\tU_{N})>1-\delta^2$.
By \eqref{eq-Pliss-gamma} (a consequence  of the Pliss Lemma),  $\tmu^+(\tK)>1-\delta \geq 1-(1-\tau)/4$, where
 $$
    \tK:=\{\tx\in\mathfrak G(i,M):\forall n\ge0\; |\{0\le k<n:\tf^{-k}(\tx)\in\tU_{N}\}|\ge
    (1-\delta)      n\}.
 $$

 Next, we bound $\|Df^{-n}_x|_E\|$ for all $\tx:=(x,E)\in\tK$ and $n\ge 0$ (not necessarily large). To do so, we divide the negative orbit segment $\tx,\tf^{-1}\tx,\dots,\tf^{-n+1}\tx$ using  the visits to $\tU_{N}$ by setting $b_0:=0$ and, inductively, for $i\ge1$:
 $$
    a_i:=\inf\{k\ge b_{i-1}:\tf^{-k}\tx\in\tU_{N}\},\; b_i:=a_i+N.
 $$
 
We set $I:=\sup\{i\ge1:b_i\le n\}$ and $V:=\bigcup_{i=1}^I [a_i,b_i)$ and note that $V$ contains all visits to $\tU_{N}$ up to time $n-N$. Therefore, since $N\leq N_0$ and $\chi\leq L$, 
 $$
   |[0,n)\setminus V|\le \frac{\eps\cdot\chi}{4L+4} n+N_0
  \text{ and }|V|\geq n\biggl(1-\frac{\eps}{4}\biggr)-N_0.
 $$
Thus, $\|Df^{-n}_x|_E\| \leq \prod_{i=1}^I e^{-\chi(b_i-a_i)} \times \prod_{k\in [0,n)\setminus V} e^L
       =e^{-\chi|V|} \times e^{L|[0,n)\setminus V|}$
\vspace{-0.05cm}       
\begin{align*}
        \hspace{1cm} & \le e^{-(1-\eps/4)n\chi + \chi N_0} \times
      e^{(\chi \eps /4)n+L\cdot N_0} \ \ \ (\text{since } |[0,n)\setminus V|\le n(\eps\cdot \chi/4L)+N_0) \\
      & \le e^{-(1-\eps/2)n\chi} \times e^{(\chi+L)N_0}\\
      &=e^{-(1-\eps)\chi n} \times e^{-\frac{\eps }{2}\chi n+(\chi+L)N_0}=e^{-\chi_0 n} \times e^{-\frac{\eps }{2}\chi n+(\chi+L)N_0} \ \  (\text{since } (1-\eps)\chi=\chi_0).    
 \end{align*}
Now we fix  an integer $n_0=n_0(i)\ge1$ so large that $\frac\eps2\chi n_0 \ge (\chi+L)N_0$. Taking $n=j n_0$, we obtain that $\|Df_x^{-jn_0}|_E\|\leq e^{-\chi_0 j n_0}$ 
 for all $j\ge 0$. Thus,
  $$
    \tmu\left(\{\tx\in\mathfrak G(i,M):\forall j\ge0\; \|Df^{-n_0j}|_{E^+_x}\|\le e^{-\chi_0 n_0 j}\}\right) \ge \tmu(\wt{K})> 1-(1-\tau)/2.
 $$
 
Since
$\tmu^+$ is the unstable lift of some $\mu\in\Proberg(f|_X)$ with $h(f,\mu)>h_0$ and unstable dimension $i$,  the above yields, for any such $\mu$,
 $$
    \mu\left(\{x\in M:\forall j\ge0\; \|Df^{-n_0j}|_{E^+_x}\|\le e^{-\chi_0 n_0 j}\}\right) > 1-(1-\tau)/2.
 $$
Similarly, 
 $
    \mu\left(\{x\in M:\forall j\ge0\; \|Df^{n_0j}|_{E^-_x}\|\le e^{-\chi_0 n_0 j} \}\right) > 1-(1-\tau)/2,
 $
and therefore $\mu(P_{n_0})>\tau$.
 We have just proved \eqref{e-exp-unif},  for every ergodic measure with entropy bigger than $h_0$, and a {\em fixed}  unstable dimension $i$ satisfying \eqref{eq-htop-idim}. 
 
 If there are several such $i$, then we let $n_0(i)$ and $h_0(i)$ be the numbers obtained above, and we 
make the following definitions:  $n_0$ is  the least common multiple of $n_0(i)$, and
$h_0$ is the maximum  of $h_0(i)$ and 
$
h_0':=\sup\{h(f,\nu): \nu\in\Proberg(f|_X),\text{ $i(\nu)$ does not satisfy \eqref{eq-htop-idim}}\}.
$
Then $h_0<h_{\Bor}(f|_X)$, and 
   $P_{n_0}\supset \bigcup_i P_{n_0(i)}$. 
  So  Property (*) holds with $\tau, n_0$ and $h_0$.
\end{proof}

\begin{remark}\label{r.entropy-tightness}
Notice that  we proved more than what was required, since we showed that the number $\tau$ in  \eqref{e.SPR-local} can be taken arbitrarily close to 1: In the proof of Prop.~\ref{prop-crit-SPR-ET},  $\tau:=1-\delta$, with $\delta$ arbitrarily small. The stronger property we proved is called {\em entropy tightness}, and will be investigated further in \S\ref{s.defSPR} and in \S \ref{s.converse-statements}.
\end{remark}

\subsection{Sufficiently Smooth Surface Diffeomorphisms are SPR (Thm~\ref{t.SPR-C-infinity})}\label{sec-surf-SPR}
We now apply   Thm~\ref{thm-CE-SPR} to surface diffeomorphisms, first in the $C^\infty$ setting, then in the more general  $C^r$ setting. In the $C^\infty$ setting, we obtain the following generalization of  Theorem \ref{t.SPR-C-infinity} in the introduction.

\begin{theorem}\label{t.tight-C-infinity}
Let $f$ be a $C^\infty$  diffeomorphism of a closed surface. 
 \begin{enumerate}
   \item If $h_\top(f)>0$, then $f$  is SPR;
   \item Every Borel  homoclinic class of $f$ is SPR;
   \item Every topological homoclinic class of $f$ is SPR;
   \item Every invariant Borel set $X$ such that $h_{\top}(f|_{\ov{X}})=h_{\Bor}(f|_X)>0$ is SPR.
 \end{enumerate}
\end{theorem}

Indeed, item (1) is Theorem \ref{t.SPR-C-infinity}. Notice that (1) only implies the SPR property for homoclinic classes with  \emph{large entropy} (Prop~\ref{p.decomposition}). By contrast, Items (2) and (3) apply to \emph{all} homoclinic classes.

\medskip
\indent The proof of Thm~\ref{t.tight-C-infinity} relies on the following corollary of  \cite[Thm C]{BCS-2}:

\begin{theorem}[\cite{BCS-2}]\label{t.surf-SPR}
Let $f$ be a $C^\infty$ diffeomorphism of a closed surface, and let $X$ be a Borel invariant set such that $h_\top(f|_{\ov{X}})=\hTOP(f|_X)>0$.
For any sequence of ergodic measures $\mu_n\in\Proberg(f|_X)$ which  converges weak-$*$ on $M$ to some possibly non-ergodic $\mu\in\Prob(f)$, if $h(f,\mu_n)\to\hTOP(f|_X)$, then \begin{itemize}
\item[$\circ$] $\mu$ is an MME of $f|_{\ov{X}}$,
\item[$\circ$] $\lambda^1(f,\mu_n)\to\lambda^1(f,\mu).$
\end{itemize}
\end{theorem}

\begin{proof}
Suppose $\mu_n\in\Proberg(f|_X)$ converge weak-$*$ on $M$ to some $\mu\in\Prob(f)$, and  $h(f,\mu_n)\to \hTOP(f|_X)>0$. By Newhouse's upper semi-continuity theorem \cite{Newhouse-Entropy},  $h(f,\mu)\ge \hTOP(f|_X)$. By Theorem C of \cite{BCS-2}, the limit measure $\mu$ can be written as $\mu=(1-\beta)\nu_0+\beta\nu_1$ with $0<\beta\le 1$, $\nu_0,\nu_1\in\Prob(f)$ such that  $h(f,\nu_1)>0$, and
  $$
     \lim_n h(f,\mu_n) \le \beta \cdot h(f,\nu_1) \text{ with }
     \beta= \frac{\lim_n \lambda^1(f,\mu_n)}{\lambda^1(f,\nu_1)}.
 $$
Note that $\nu_1(\ov{X})=1$ since $\nu_1\ll\mu$. Hence $h(f,\nu_1)\le h_\top(f|_{\ov X})=\hTOP(f|_{X})$. The above implies $h(f,\nu_1)=\hTOP(f|_X)$ and $\beta=1$ so $\mu=\nu_1$. 
\end{proof}

First, we prove Item (4), and then we will deduce the rest of the theorem.

\begin{lemma}\label{lem-surf-cpt-SPR}
Let $f$ be a $C^\infty$ diffeomorphism of a closed surface.
Any invariant Borel set $X\subset M$ such that $h_\top(f|_{\ov{X}})=\hTOP(f|_X)>0$ has the SPR property.
\end{lemma}

\begin{proof}
By Remark \ref{r.EH-in-Dim-2},  $f$ satisfies   (EH) on $X$. Moreover, the unstable dimensions in (EH) and (EC) are all equal to one, so  $\Lambda^+=\Lambda^{+,1}=\lambda^1$.
Now Thm~\ref{t.surf-SPR} implies the entropy continuity of $\Lambda^+$ on $X$.
By Theorem~\ref{thm-CE-SPR}, $f|_X$ is SPR. 
\end{proof}

\begin{proof}[Proof of Theorem~\ref{t.tight-C-infinity}] 
Lemma~\ref{lem-surf-cpt-SPR} is Item (4), and Item (1) is the special case  $X=M$ in Item (4). Item (3)  follows from Item (4) in the positive entropy case. So does Item (2), because by  Remark \ref{r.topological-class}, $\hTOP(f|_{\ov{X}\setminus X})=0$ so $\hTOP(f|_X)=h_\top(f|_{\ov{X}})$. 
To settle the zero entropy case, recall  that  a zero entropy topological homoclinic class is just a hyperbolic periodic orbit -- see  \cite{Smale,Newhouse-Homoclinic}.
\end{proof}

\begin{remark}\label{r.proof-SPR-surf}
 The proof of Item (2)  involves \cite{BCS-1}, through Remark~\ref{r.topological-class}.
 However, the proof of Items (1), (3), and (4) 
 relies on \cite{BCS-2}, but not on \cite{BCS-1}.
\end{remark}

Recall the definition of $\lambda_{\max}(f)$ from \eqref{e.dilation}.
Thm~\ref{t.tight-C-infinity} can be generalized  to  $C^r$ diffeomorphisms, with finite $r$, as follows. 
\begin{theorem}\label{t.tight-Cr}
Fix $r>1$, and  a $C^r$ diffeomorphism $f$ of a closed surface. Then:
 \begin{enumerate}
   \item If $h_\top(f)>\lambda_{\max}(f)/r$, then $f$  is SPR.  
   \item Every Borel  homoclinic class $X$ s.t. $\hTOP(f|_X)>\lambda_{\max}(f)/r$ is SPR.
   \item Every topological homoclinic class $\ov{X}$ s.t. $h_\top(f|_{\ov{X}})>\lambda_{\max}(f)/r$
   is SPR.
    \item Every invariant Borel set $X$ s.t. $h_{\top}(f|_{\ov{X}})=h_{\Bor}(f|_X)>\lambda_{\max}(f)/r$ is SPR.
 \end{enumerate}
\end{theorem}
\noindent
The proof of this result relies on the following generalization of  Thm~\ref{t.surf-SPR} to the $C^r$ case, due to D. Burguet~\cite[Cor~1]{Burguet-Cr}:

\begin{theorem}[Burguet]\label{t.surf-SPR-Cr}
Let $f$ be a $C^r$ diffeomorphism of a closed surface, and let $X$ be a Borel invariant set such that $h_\top(f|_{\ov{X}})=\hTOP(f|_X)>\lambda_{\max}(f)/r$.
For any sequence of ergodic measures $\mu_n\in\Proberg(f|_X)$ which converges weak-$*$ on $M$ to some possibly non-ergodic $\mu\in\Prob(f)$, if $h(f,\mu_n)\to\hTOP(f|_X)$,
then:
\begin{itemize}
\item[$\circ$] $\mu$ is an MME of $f|_{\ov{X}}$,
\item[$\circ$] $\lambda^1(f,\mu_n)\to\lambda^1(f,\mu).$
\end{itemize}
\end{theorem}
\noindent
Thm~\ref{t.tight-Cr} follows from Thm~\ref{t.surf-SPR-Cr} as before. Item (2) requires  Remark~\ref{r.topological-class},  which guarantees  that
$\hTOP(f|_{\ov{X}\setminus X})\leq \lambda_{\max}(f)/r$, and therefore $\hTOP(f|_{\ov{X}})=\hTOP(f|_{X})$.

\section{Other Formulations of the SPR Property}\label{s.defSPR}

\newcommand\Born{\mathscr{B}}
\newcommand\sC{\mathscr{C}}

This section gives two alternative characterizations of the SPR property:  {\em Entropy tightness}, and the existence of an {\em entropy gap at infinity}.
These characterizations are useful as a link between the SPR property for diffeomorphisms (Def~\ref{def-SPR-diffeo}), and the (classical) SPR property for countable state Markov shifts (Def~\ref{d.SPR-for-Shifts}).  This link will serve us in Part~\ref{part-symbolic-diffeo}, where we show  that SPR diffeomorphisms can be coded symbolically by SPR Markov shifts, and in Part~\ref{p.properties-of-SPR-diffeos} , where we use this coding to prove the properties of SPR diffeomorphisms we  announced in \S \ref{ss-Prop-SPR-diffeos}.

\subsection{Bornologies}\label{s.bornologies-section} 
Let $Z$ be a non-empty subset of a set $\Omega$. 

A {\em bornology} on $Z$ is a collection $\mathfrak B$ of subsets of $Z$ which covers $Z$, is closed under finite unions, and is {\em hereditary}: Every subset of an element of $\mathfrak B$ is also an element of $\mathfrak B$.
The elements of $\mathfrak B$ are called {\em bounded sets}, and their complements are
called {\em neighborhoods of infinity}. Points in $\Omega\setminus Z$ are said to {\em lie at infinity.}
\begin{enumerate}[$\bullet$]
\item If $Z\in\mathfrak B$, then all subsets of $Z$ are in $\mathfrak B$, and we call $\mathfrak B$ {\em trivial}.
If $\mathfrak B$ is not trivial, then  the neighborhoods of infinity form a proper filter on $Z$.

\smallskip
\item Let $\mathfs F$ be a $\sigma$-algebra on $Z$.  $\mathfrak B$ is called {\em $\mathfs F$-measurable} (or just ``measurable") if  every element of  $\mathfrak B$ is contained in an $\mathfs F$-measurable element of $\mathfrak B$.

\smallskip
\item A {\em base} for $\mathfrak B$ is a collection  $\mathfrak A\subset\mathfrak B$ such that every $B\in\mathfrak B$ is contained in some $A\in\mathfrak A$.  A bornology with a countable base is called {\em countably generated}.
\end{enumerate}

\medskip
\noindent
{\em Example 1\/:}  Suppose $\Omega$ is a second countable locally compact topological space. Let  $\mathfrak B$ be the bornology on $\Omega$ generated by all compact subsets of $\Omega$. Equivalently, $\mathfrak B$ is the set of all pre-compact subsets of $\Omega$. Then  $\mathfrak B$ is a countably generated Borel measurable bornology  on $\Omega$.  $\mathfrak B$ is trivial if and only if $\Omega$ is compact.

\medskip
\noindent
{\em Example 2\/:}  Suppose $(\Omega,\mathfs F)$ is a measurable space, $F:\Omega\to (0,\infty ]$ is a measurable function, and   $Z:=\{x\in \Omega:F(x)<\infty\}$. Let $\mathfrak B$ denote the collection of  $B\subset \Omega$ such that  $\sup\limits_B F<\infty$. Then 
 $\mathfrak B$ is a countably generated measurable bornology on $Z$.  If $F$ is bounded on $Z$, then  $\mathfrak B$ is trivial.

\subsubsection*{{\bf Escape to Infinity with Respect to a Bornology.}}\nopagebreak
\begin{enumerate}[$\bullet$]
\item We say that  a sequence of points $x_n\in \Omega$  {\em escapes to infinity relatively to $\mathfrak B$}, and write
$
x_n\to\infty\ \ (\mathfrak B),
$
 if
$$\forall B\in\mathfrak B,\,   \exists N\text{ such that }\forall n>N,\, x_n\not\in B.$$
(Any sequence of points in  $\Omega\setminus Z$ escapes to infinity relative to $\mathfrak B$.)

\smallskip
\item Suppose $\mathfrak B$ is a measurable bornology on a measurable subset $Z$ of
a measurable space  $(\Omega,\mathfs F)$. We say that a sequence of  probability measures $\mu_n\in\mathbb P(\Omega) $ {\em escapes to infinity relative to $\mathfrak B$}, and write
$\mu_n\to\infty \ \ (\mathfrak B)$,
if
$$\forall B\in\mathfrak B\cap\mathfs F,\  \limsup_{n\to\infty}\mu_n(B)=0.$$
(Any sequence of measures on  $\Omega\setminus Z$  escapes to infinity relative $\mathfrak B$.)

\end{enumerate}

\subsubsection*{{\bf {Tightness} with Respect to a Bornology}}\label{ss.tightness}
We are interested in conditions which prevent  escape to infinity within a given collection of measures.

Let $(\Omega,\mathfs F)$ be a measurable space,  and  suppose $\mathfrak B$ is a measurable bornology on some measurable set $Z\subset \Omega$. Fix some set  $\cM$  of measures on $\Omega$, and  some function $h:\cM\to[0,\infty)$.
We denote  $$h(\cM):=\{h(\mu):\mu\in\cM\},$$ and form the supremum  $\sup h(\cM)$.
The following definition is modeled on  Def~\ref{def-SPR-diffeo}:

\begin{definition}\label{def.partial-tight}
$\cM$ is \emph{{partially} $h$-tight}  if for \underline{some} $\tau\in (0,1)$, there are a number $h_0<\sup h(\cM)$ and  a measurable set $B\in\mathfrak B$  such that
 \begin{equation}\label{eq.tight}
     \forall \mu \in\cM\quad h(\mu)>h_0 \implies \mu(B)>\tau.
  \end{equation}
\end{definition}
\noindent
If $\mathfrak B$ is countably generated, then this is equivalent to saying that there is no sequence of $\mu_n\in\cM$,  such that $\mu_n\to\infty \ \ (\mathfrak B)$ and $h(\mu_n)\to \sup h(\cM)$.

The following stronger property is very useful (see \S \ref{ss-entropy-tightness}):

\begin{definition}\label{def.full-tight}
$\cM$ is \emph{{fully} $h$-tight} (or just {\em $h$-tight}) if for \underline{every} $\tau\in (0,1)$, there are a number $h_0<\sup h(\cM)$ and a measurable set $B\in\mathfrak B$  such that eq.~\eqref{eq.tight} holds.
\end{definition}

\subsubsection*{{\bf {$\boldsymbol h$-Gap} at Infinity
}}
Definition~\ref{def.partial-tight} can be made more quantitative, as follows. Let

\newcommand\hinfty{h^\infty_{\mathfrak B}}
\newcommand\Hinfty{H^\infty_{\mathfrak B}}

\begin{equation}\label{e.def-limsup}
\hinfty(\cM):=\sup\left\{t\in\R:\begin{array}{c}
\forall \tau\in (0,1), \forall B\in\mathfrak B\cap\mathfs F, \exists\mu\in\cM \\ \text{ s.t. }
 \mu(B)<\tau\text{ and }h(\mu)>t
\end{array}
\right\},
\end{equation}
with the convention that $\sup\emptyset:=-\infty$.
In particular, if the bornology is trivial, then $\hinfty(\cM):=-\infty$. We call this number 
the {\em $h$-value at infinity}. Note that:
\begin{lemma}\label{l.limsup-general}
If $\mathfrak B$ is countably generated and measurable, then
$$
\hinfty(\cM)=\limsup_{\mu\to\infty\ (\mathfrak B)} h(\mu):=\sup\left\{
\limsup_{n\to\infty} h(\mu_n): \mu_n\in\cM,\ \mu_n\to\infty\ (\mathfrak B)
\right\}.
$$
\end{lemma}

\begin{definition}
We say  $\cM$ has an \emph{$h$-gap at infinity}, if
 $$
   \hinfty(\cM) < \sup h(\cM).
 $$
\end{definition}

The following lemma is straightforward:

\begin{lemma}\label{l-partial-tight}
$\cM$ is $h$-partially tight if and only if $\cM$ has an $h$-gap at infinity.
\end{lemma}

\begin{remark}
Let us define an increasing function $(0,1)\to \{-\infty\}\cup [0, \sup h(\cM)]$:
$$
     H^\infty_{\mathfrak B,\cM}(\tau):= \sup\left\{ t\in\R:  \forall   B\in\mathfrak B\cap\mathfs F\; \exists\mu\in\cM
       \text{ s.t. } \mu(B)<\tau \text{ and }h(\mu)>t\right\}.
$$
Note that $\hinfty(\cM)=\inf_\tau H^\infty_{\mathfrak B,\cM}(\tau)$.
Then $\cM$ is partially (resp. fully) $h$-tight if and only if for some (resp. for all) $\tau\in (0,1)$ ,  $H^\infty_{\mathfrak B,\cM}(\tau)< \sup h(\cM)$.
\end{remark}

\subsection{The Pesin Bornologies}\label{ssPesinBorlonogies}
Let $f$ be a $C^1$ diffeomorphism on a closed manifold $M$, 
and let  $X\subset M$ be an invariant Borel set. We define $\cM:=\Proberg(f|_{X})$,  and let $h\colon \cM\to [0,\infty)$ denote the entropy function, $h(\mu):=h(f,\mu)$. 
Then $$\sup h(\cM)=h_{\Bor}(f|_X), \text{ see \eqref{e.Top-Entropy}}.$$ 

Fix $0<\eps<\chi$, and assume $f|_X$ has at least one $\chi$-hyperbolic measure on $X$. Then $X$ contains 
some $(\chi,\eps)$-Pesin block. Let $Z_{{\chi,\eps}}$ denote the intersection of $X$ with  the union of all $(\chi,\eps)$-Pesin blocks.

\begin{definition}
The {\em Pesin bornology of $f|_X$ (with parameters $(\chi,\eps)$)} is the bornology on $Z_{\chi,\eps}\subset X$ given by  
\begin{equation}\label{e.Pesin-Bornologies}
\mathfrak P_{\chi,\eps}=\mathfrak P_{\chi,\eps}(f|_X):=\{\Lambda\subset Z_{\chi,\eps}: \Lambda \text{ is a $(\chi,\eps)$-Pesin block}\}\cup\{\emptyset\}.\end{equation}
\end{definition}
\begin{definition} The number 
$h^\infty_{\chi,\eps}=h^\infty(f|_X;\chi,\eps):=h^\infty_{\mathfrak P_{\chi,\eps}}(\cM)=\hspace{-0.2cm}\limsup\limits_{\mu\to\infty\, (\mathfrak P_{\chi,\eps})}\!\!h(f,\mu)$ 
is called {the} \emph{entropy at infinity of $f|_X$}, {with parameters $(\chi,\eps)$}.
\end{definition}

\begin{lemma}\label{l.P-is-bornology}
Under the above assumptions, 
\begin{enumerate}[(1)]
\item $\mathfrak P_{\chi,\eps}$ is a countably generated Borel  bornology on  $Z_{\chi,\eps}$.
\item If $\mathfrak P_{\chi,\eps}$ is trivial, then {$f:\ov{Z_{\chi,\eps}}\to\ov{Z_{\chi,\eps}}$} is uniformly hyperbolic.
\item If $\chi_1\leq \chi_2$ and $\eps_1\geq \eps_2$, then $h^\infty_{\chi_1,\eps_1}\leq h^\infty_{\chi_2,\eps_2}$.
\item Partial (resp. full) $h$-tightness w.r.t. $\mathfrak P_{\chi_2,\eps_2}$ implies the same w.r.t. $\mathfrak P_{\chi_1,\eps_1}$.
\end{enumerate}
\end{lemma}
\begin{remark}
The converse of (2) may fail: $f:\ov{Z_{\chi,\eps}}\to\ov{Z_{\chi,\eps}}$ may be uniformly hyperbolic, while $\ov{Z_{\chi,\eps}}\setminus Z_{\chi,\eps}$ carries a measure with some exponent smaller than $\chi$.
\end{remark}
\begin{proof}
Clearly, $\mathfrak P_{\chi,\eps}$ covers $Z_{\chi,\eps}$, is made of subsets of $Z_{\chi,\eps}$, and is hereditary. Also, by the definition of Pesin blocks, $\mathfrak P_{\chi,\eps}$ is closed under finite unions. So $\mathfrak P_{\chi,\eps}$ is a bornology.
$\mathfrak P_{\chi,\eps}$ is measurable, because the closure of a $(\chi,\eps)$-Pesin block in $Z_{\chi,\eps}$ is also a $(\chi,\eps)$-Pesin block,  see Lemma \ref{l.splitting}. Finally, $\mathfrak P_{\chi,\eps}$ is countably generated, because it is generated by the countable collection of sets  
\begin{align*}
A_N:=\left\{
x\in Z_{\chi,\eps}:\begin{array}{l}\forall n\in\Z, \; T_{f^n(x)} M=E^s(f^n(x))\oplus E^u(f^n(x)) \text{ where}\\
\text{$E^s(f^n(x))$, $E^u(f^n(x))$ satisfy \eqref{e.def-pesin} with $K:=N$}
\end{array}
\right\},\ N\in\N.
\end{align*}

If $\mathfrak P_{\chi,\eps}$ is trivial, then
$Z_{\chi,\eps}$ is a $(\chi,\eps)$-Pesin block. By Lemma \ref{l.splitting},  $\ov{Z_{\chi,\eps}}$ is a $(\chi,\eps)$-Pesin block, and this set is $f$-invariant.  So $f$ is uniformly hyperbolic on $\overline{Z_{\chi,\eps}}$.

Suppose $\chi_1\leq \chi_2$ and $\eps_1\geq \eps_2$. Then \eqref{e.def-pesin} with $(\chi_2,\eps_2)$ implies \eqref{e.def-pesin} with $(\chi_1,\eps_1)$. Therefore, every $(\chi_2,\eps_2)$-Pesin block is a $(\chi_1,\eps_1)$-Pesin block and $\mathfrak P_2:=\mathfrak P_{\chi_2,\eps_2}$ is a subset of $\mathfrak P_1:=\mathfrak P_{\chi_1,\eps_1}$. Thus,
$
\mu_n\to\infty\ (\mathfrak P_1)\Rightarrow \mu_n\to\infty\ (\mathfrak P_2).
$

Parts (3) and (4) now easily follow from Lemmas \ref{l.limsup-general} and \ref{l-partial-tight}. 
\end{proof}

\subsection{SPR, Partial Entropy Tightness, and Entropy Gap at Infinity}
In the setup of the previous section~\S\ref{ssPesinBorlonogies}, partial (or full) $h$-tightness is called \emph{partial (or full) entropy-tightness}, and an $h$-gap at infinity is called an {\em entropy gap at infinity}. 

Definition \ref{def-SPR-diffeo-local}, Definition \ref{def.partial-tight} and Lemma \ref{l-partial-tight} immediately yield:

\begin{lemma}\label{l-SPR-as-tight}
Let $X$ be an invariant Borel set of   a $C^{1}$ diffeomorphism $f$ on a closed manifold $M$. Then 
$f$ is SPR on $X$ iff there is $\chi>0$ such that for all $0<\eps<\chi$,  $\Proberg(f|_X)$ is  partially entropy-tight relative to $\mathfrak P_{\chi,\eps}(f|_X)$.
\end{lemma}

\begin{corollary}\label{c.SPR=Entropy-Gap-Diffeos}
Let $X$ be an invariant Borel set of   a $C^{1}$ diffeomorphism $f$ on a closed manifold $M$. Then 
$f$ is SPR on $X$ iff there is $\chi>0$ such that for all $0<\eps<\chi$, $\Proberg(f|_X)$ has an entropy gap at infinity with respect to $\mathfrak P_{\chi,\eps}(f|_X)$, i.e. 
\begin{equation}\label{e.Entropy-Gap}
h^\infty(f|_X;\chi,\eps)< h_{\Bor}(f|_{X}).
\end{equation}
\end{corollary}

\begin{remark}
Notice that these characterizations use infinitely many Pesin bor\-nologies, one for each $0<\eps<\chi$. However, we will eventually see that it is sufficient to show partial tightness relatively to a  \emph{single} Pesin bornology $\mathfrak P_{\chi,\eps(M,f,\chi)}$.
This can be proved by building an SPR coding using Remark~\ref{rem-SPR-single-eps} which in turn proves the SPR property by Proposition~\ref{p.project-SPR}.
The value of $\eps=\eps(M,f,\chi)$   could in principle be calculated by a following the estimates of  \cite{Sarig-JAMS,Ben-Ovadia-Codings}, but we did not do this.
\end{remark}

\newcommand\hinftyPesin{h^\infty_{\mathfrak P_{\chi,\eps}}}

\subsection{Entropy Tightness}\label{ss-entropy-tightness}
Fix a $C^{1}$ diffeomorphism $f$ of a closed manifold $M$, and let $X\subset M$ be an invariant Borel set. 
We will now consider the implications of {\em (full)} entropy tightness, with respect to a {\em single} Pesin bornology:

\begin{definition}\label{d.entropy-tight-diffeo}
  $f$ is {\em entropy-tight} on $X$, if there exists \underline{at least one} Pesin bornology $\mathfrak P_{\chi,\eps}$ with the following property: For every $\tau\in (0,1)$,  there exist $h_\tau<h_{\Bor}(f|_{X})$ and $\Lambda_\tau\in\mathfrak P_{\chi,\eps}$ such that for every $ \nu\in\Proberg(f|_{X})$,
 $$
 h(f,\nu)>h_\tau\Rightarrow \nu(\Lambda_\tau)>\tau.
 $$
When $X=M$, we just say that $f$ is entropy-tight.
\end{definition}

Remark \ref{r.entropy-tightness} can be rephrased as:

\medskip
\noindent{\bf Theorem A'.} {\em On closed surfaces, all $C^\infty$ diffeomorphisms with positive topological entropy are  entropy-tight. Moreover any Borel $f$-invariant subset $X$ such that $h_{\top}(f|_{\overline{X}})=\hTOP(f|_X)>0$ is also entropy-tight.}
\medskip

While the SPR property involves  infinitely many Pesin bor\-nologies, with different epsilons, entropy tightness  uses just one  Pesin bornology.
However, we will show that entropy-tightness with respect to {\em some} Pesin bornology $\mathfrak P_{\chi,\eps}$ implies entropy-tightness with respect to {\em every} Pesin bornology $\mathfrak P_{\chi,\eps'}$ with  $\eps'>0$, and deduce:

\begin{proposition}\label{prop-tight-implies-SPR}
Let $X$ be an invariant Borel subset of a $C^1$ diffeomorphism $f$ on a closed manifold $M$.  If $f$ is entropy-tight on $X$, then $f$ is SPR on $X$.
\end{proposition}

\begin{proof}
For simplicity, we give the proof in the case $X=M$. The modifications needed to treat the general case are routine.
Consider the following properties:
$$
\begin{array}{l l }
(\text{SPR}_{\chi,\eps}) & (\chi,\eps)\text{-Strong Positive Recurrence: } \exists \tau>0, h_0<h_{\top}(f), \Lambda\in\mathfrak P_{\chi,\eps}\text{ s.t. }\\
& \forall\nu\in\Proberg(f),\   h(f,\nu)>h_0\Rightarrow\nu(\Lambda)>\tau;\\
(\text{ET}_{\chi,\eps}) & (\chi,\eps)\text{-Entropy-Tightness: }\forall \tau\in (0,1), \exists h_\tau<h_{\top}(f),\Lambda_\tau\in \mathfrak P_{\chi,\eps}\text{ s.t. }\\
& \forall\nu\in\Proberg(f),\   h(f,\nu)>h_\tau\Rightarrow\nu(\Lambda_\tau)>\tau.
\end{array}
$$
If $f$ is entropy-tight, then  $(\text{ET}_{\chi,\eps})$ holds for some $\chi$ and $\eps$. We will show that
$$
(\text{ET}_{\chi,\eps})\;{\Rightarrow}\; {\exists \chi'>0,\forall\eps'\ (\text{ET}_{\chi',\eps'})\;\Rightarrow\; \exists \chi'>0,\forall\eps'\ (\text{SPR}_{\chi',\eps'})}\;\Rightarrow\;\text{SPR}.
$$
The second and third implications are self-evident. We will prove the first.

The proof is based on Prop. \ref{p.pesin}.
Let $c_0:=\tfrac{\chi}2+\max(\log\|Df\|_{\sup},\log\|Df^{-1}\|_{\sup})$, and fix $\eps'>0$, $\tau\in (0,1)$. By $(\text{ET}_{\chi,\eps})$ there are $h_\tau>0$ and  $\Lambda_0\in\mathfrak P_{\chi,\eps}$, such that
$$
\forall \nu\in\Proberg(f),\ h(f,\nu)>h_\tau\Rightarrow \nu(\Lambda_0)>{1-\tfrac{1-\tau}{4\max(c_0/\eps',1)}.}
$$

By  Def.~\ref{d.pesin}, there is $n_0\geq 1$ s.t.
for every $x\in \Lambda_0$ and  $j\geq 0$,  $\|Df^{jn_0}|_{E^s(x)}\|\leq e^{-jn_0{\chi /2}}$ and $\|Df^{-j}|_{E^u(x)}\|\leq e^{- jn_0{\chi/2}}$. 

By Prop. \ref{p.pesin}, there exists $\Lambda_\tau\in\mathfrak P_{{\chi/2},\eps'}$ such that
$$
\forall \nu\in\Proberg(f),\ \nu(M\setminus \Lambda_\tau)\leq 4\max(c_0/\eps',1)\nu(M\setminus \Lambda_0).
$$
Thus, if $h(f,\nu)>h_\tau$, then $\nu(\Lambda_\tau)>1-4\max(c_0/\eps',1)(1-\nu(\Lambda_0))>\tau$. 

This is $(\text{ET}_{\chi',\eps'})$ with $\chi'=\chi/2$.
\end{proof}

\begin{remark}
We will eventually see that for $C^{1+}$ diffeomorphisms,  entropy-tightness and SPR are  equivalent (Thm~\ref{thm-characterizations}). We do not know if this is also the case for $C^1$ diffeomorphisms.
\end{remark}

\section{Examples in Higher Dimension}\label{s.example}
We exhibit in this section open sets of diffeomorphisms, in dimension $3$ and higher, that are not uniformly hyperbolic but are SPR.
Note how the SPR property is easily deduced from facts already established in previous works.

\subsection{Partially Hyperbolic SPR Diffeomorphisms with Center Foliation of  Circles}\label{ss.center-compact}
We consider the $C^1$-open class of \emph{partially hyperbolic diffeomorphisms} on a $3$-dimensional closed connected manifold $M$:
these are the diffeomorphisms which preserve a splitting $TM=E^s\oplus E^c\oplus E^u$ into three non-trivial bundles
such that $E^s$  is uniformly contracted, $E^u$ is uniformly expanded, and there exists $N\geq 1$ such that for any $x\in M$, any unit vectors
$v^s\in E^s_x$, $v^c\in E^c_x$, $v^u\in E^u_x$ satisfy
$$
 2\|Df^N.v^s\|\leq \|Df^N.v^c\| \leq \tfrac 1 2 \|Df^N.v^u\|.
$$
It is well-known that there exist two (unique) invariant foliations $\mathcal{W}^s$, $\mathcal{W}^u$ whose leaves are tangent to
$E^s$ and $E^u$ respectively.
We will focus on diffeomorphisms which satisfy the following additional properties:
\begin{enumerate}
\item There exists an invariant foliation $\mathcal{W}^c$ tangent to $E^c$, whose leaves are circles with uniformly bounded length.
\item The system is \emph{accessible}: any two points $x,y\in M$ can be joined by a path
that is tangent to $E^s\oplus E^u$.
\item There exists a hyperbolic periodic point $x\in M$ (thus the action along the center is not isometric). 
\end{enumerate}
\begin{proposition}
The set of diffeomorphisms of $\TT^3$ satisfying (1), (2), (3) is $C^1$-open, and it contains a non-empty $C^1$-open set of transitive systems which are not uniformly hyperbolic.
\end{proposition}
\begin{proof}
 Properties (1) and (2) are $C^1$-open by ~\cite[Theorem A']{Martinchich} and~\cite{didier} respectively.  Property (3) is $C^1$-open due to the implicit function theorem.

One can build a diffeomorphism satisfying (1) by considering the product of an Anosov diffeomorphism of $\TT^2$
with the identity on $S^1$. This diffeomorphism is accumulated by a $C^1$-open set of transitive non-hyperbolic diffeomorphisms satisfying (3),
see~\cite{Bonatti-Diaz}. Properties (2) can then be realized by $C^1$-perturbation
by~\cite{Dolgopyat-Wilkinson}.
\end{proof}

For these systems, the ergodic MME are non-uniformly hyperbolic, their number is finite and there exist at least one ergodic MME
with 2 negative Lyapunov exponents and one with two positive Lyapunov exponents~\cite{RodriguezHertz-RodriguezHertz-Tahzibi-Ures}.
Moreover \cite{Tahzibi-Yang} establishes a separation between positive and negative center exponents, yielding what we called entropy hyperbolicity (EH) in Thm~\ref{thm-CE-SPR}. As a consequence we obtain:

\begin{theorem}
Let $f$ be a $C^2$ partially hyperbolic diffeomorphism on a $3$-dimen\-sio\-nal closed connected manifold,
satisfying (1),(2),(3) above. Then $f$ is SPR.
\end{theorem}
\begin{proof}
We check the  conditions for the SPR property in Proposition~\ref{p.decomposition}:
\medskip

\noindent
$\bullet$ \emph{There is $\chi>0$ such that the ergodic measures with large entropy are $\chi$-hyperbolic:} This is Theorem B in~\cite{Tahzibi-Yang}.
\medskip

\noindent
$\bullet$ \emph{The number of Borel homoclinic classes with large entropy is finite.}
Indeed, the partially hyperbolic structure, with one-dimensional center, together with Pliss lemma,
imply that for any $\chi$-hyperbolic ergodic measure,
there exists a positive measure set of points whose stable and unstable manifolds
have uniform size (independent from the measure).
Consequently there exists $\ell$ such that among any set of $\ell$ measures that are $\chi$-hyperbolic,
two of them have to be homoclinically related. This gives the property.
\medskip

\indent Now we can fix $h\in (0,h_{\top}(f))$ such that any Borel homoclinic class with entropy larger than $h$ has
entropy equal to $h_\top(f)$.
\medskip

\noindent
$\bullet$ \emph{Each of these Borel classes is SPR.}
Theorem C in~\cite{Tahzibi-Yang} states the following property:
\emph{Let $(\mu_n)$ be a sequence of ergodic measures of $f$
such that $(h(f,\mu_n))$ converges towards $h_\top(f)$ and $(\mu_n)$ converges towards an invariant measure $\mu$.
Then $h(f,\mu)=h_\top(f)$ and the center Lyapunov exponent of all the ergodic components of $\mu$ have the same sign.}
By continuity of the bundles $E^s, E^c, E^u$, we also have convergence of the Lyapunov exponents
$\lambda^i(\mu_n)\to \lambda(\mu)=\int \lambda^i(\mu)d\mu$.
The Entropy Hyperbolicity (EH) and the Entropy Continuity (EC) of $\Lambda^+$ are satisfied.
Theorem~\ref{thm-CE-SPR} then implies that $f$ is SPR.
\end{proof}

\begin{remark}
The construction in~\cite{Bonatti-Diaz} gives a nonempty $C^1$-open set of diffeomorphisms with two hyperbolic fixed points $P,Q$
with stable dimensions respectively equal to $1$ and $2$ such that any hyperbolic periodic point is homoclinically related
to $P$ or $Q$ and such that its topological homoclinic class coincides with the whole manifold. As a consequence, for these systems
there exist exactly two Borel homoclinic classes and two ergodic MMEs, both with full support.
\end{remark}

\subsection{Robustly Transitive but not Partially Hyperbolic SPR Diffeomorphisms}\label{ss.dominated-notPH}
Bonatti and Viana \cite{Bonatti-Viana} have  introduced a family of diffeomorphisms derived from Anosov which are robustly transitive, volume-hyperbolic but admit no uniform invariant sub-bundle.
We show that some of them are SPR:

\begin{theorem}
For any $d\geq 4$, there exists a non-empty $C^1$-open set of SPR transitive diffeomorphisms of $\TT^d$ which admit no uniform invariant sub-bundle.
\end{theorem}

We rely on the analysis performed in \cite{Buzzi-Fisher-2013}. There,  a special class of Bonatti-Viana diffeomorphisms $f:\TT^d\to\TT^d$ is derived from an Anosov and ergodic automorphism $A:\TT^d\to\TT^d$ and the following properties are checked for some number $h<h_\top(f)$:
\begin{itemize}
 \item This class of systems contains a non-empty open set in the $C^1$ topology.
 \item They have positive entropy, equal to that of the Anosov automorphism $A$. 
More precisely, if $\pi:\TT^d\to\TT^d$ is the topological semiconjugacy making $f$ an extension of $A$, then,  for every ergodic $\mu\in\Prob(f)$ with $h(f,\mu)>h$, $h(f,\mu)=h(A,\pi_*(\mu))$.
 \item They have a continuous dominated splitting $T\TT^d=E^{cs}\oplus E^{cu}$.
  \item There are an open set $U\subset\TT^d$ and some numbers $\eps>0$ and $\kappa>0$ such that $\mu(U)<\eps$ for any $\mu\in\Proberg(f)$ with $h(f,\mu)>h$. 
Moreover, for  $x\notin U$, $\|Df|_{E^{cs}_x}\|\le e^{-\kappa}$ and $\|(Df|_{E^{cu}_x})^{-1}\|\le e^{-\kappa}$.
This number $\eps>0$ can be chosen so small that $\eps<\frac{1}{4}\frac{\kappa}{\log\max(\|Df\|,\|Df^{-1}\|)}\leq\frac14$.
\end{itemize}
It is now easy to see that for any $\mu\in\Proberg(f)$ with $h(f,\mu)>h$,  the Lyapunov exponents along $E^{cu}$ are  lower-bounded by
 $$\begin{aligned}
    &\ge (1-\mu(U))\kappa - \mu(U)\cdot\log\|Df^{-1}\|=\kappa - \mu(U) (\log\|Df^{-1}\|+\kappa)\\
    &\ge \kappa \left(1-\eps\left(1+\kappa^{-1}\log\|Df^{-1}\| \right) \right)
    > \kappa/2>0.
  \end{aligned}$$
Thus the exponents along $E^{cu}$ (resp. $E^{cs}$) are lower-bounded by a positive number (resp. upper-bounded by a negative number).
Since the splitting $T\TT^d=E^{cs}\oplus E^{cu}$ is continuous, the sum of the positive exponents, which coincides with the sum of the exponents along $E^{cu}$ is continuous with respect to the measure $\mu$ in the weak-$*$ topology.

Let us consider a converging sequence of ergodic measures $\mu_n\wsc\mu$ such that
$h(f,\mu_n)\to h_\top(f)$.
Note that $\mu$ is an MME. Indeed, using that $\pi$ is a semiconjugacy, $A$ is uniformly hyperbolic, and $\pi_*$ preserves the entropy of ergodic measures with large entropy:
 $$
   h(f,\mu)\ge h(A,\pi_*(\mu))\ge \lim_{n\to\infty}h(A,\pi_*(\mu_n)) = \lim_{n\to\infty}h(f,\mu_n) = h_\top(f).
 $$

This implies that a.e. ergodic component of $\mu$ is also an MME. 
In particular,  a.e. ergodic component of $\mu$ has unstable dimension $\dim E^{cu}$ and is $\kappa/2$-hyperbolic. It follows that the measure $\mu$ itself is $\kappa/2$-hyperbolic with constant unstable dimension $\dim E^{cu}$.
We have thus checked Properties (EH) and (EC); hence $f$ is SPR.

\begin{remark}
The construction in~\cite{Bonatti-Viana} shows that all the hyperbolic periodic points whose stable dimension equals $\dim(E^{cs})$
are homoclinically related. As a consequence these diffeomorphisms admit a unique Borel homoclinic class with maximal entropy
and a unique MME; its support is the whole manifold.
\end{remark}

\subsection{Further Examples of SPR Diffeomorphisms}
We believe the situation in Sections \ref{ss.center-compact}-\ref{ss.dominated-notPH} to be quite common: for many classes of non-uniformly hyperbolic diffeomorphisms, previous analysis, though only aimed at proving the finite number of ergodic MMEs, actually contain sufficient information to deduce the SPR property using the tools from this work.
Let us cite some interesting classes of partially hyperbolic diffeomorphisms which we expect can be handled this way:
 \begin{enumerate}
   \item Those with center foliation in circles, satisfying conditions similar to \S\ref{ss.center-compact}, but with arbitrary dimensions of the stable and unstable bundles.
  \item Those with invariant non-compact one-dimensional center leaves such as the perturbations of time-$1$ maps of Anosov flows
  and more generally discretized Anosov flows \cite{Martinchich}. The properties in this setting are similar to \S\ref{ss.center-compact}: from \cite{Buzzi-Fisher-Tahzibi,Crovisier-Poletti}
there exists a non-empty $C^1$-open set of robustly transitive $C^2$ diffeomorphisms which admit two ergodic MMEs, both fully supported, non-uniformly hyperbolic but with different unstable dimensions.
  \item Those with center Lyapunov exponents of given sign such as:
  \begin{itemize}
   \item[--] the Ma\~n\'e{'s} derived from Anosov and their generalizations \cite{Buzzi-Fisher-Sambarino-Vasquez-2012,Ures-2012,Fisher-Potrie-Sambarino-2014},
   \item[--] the dynamically defined class \cite{Mongez-Pacifico} introduced by Mongez and Pacifico.
   \end{itemize}
 \end{enumerate}

\subsection{Application to Physical Measures of Mostly Contracting Diffeomorphisms}\label{ss.geometrical}
We now give an example of SPR diffeomorphisms for some non-constant potential (Def \ref{d.SPR-potential}).
Let us consider a $C^{1+}$ diffeomorphism $f$ on a closed manifold $M$ with a (continuous) invariant dominated splitting $TM=E^{cs}\oplus E^u$ such that $E^u$ is uniformly expanded: for some $N\ge1$ and any unit  vectors $v^{cs}\in E^{cs}_x$, $v^u\in E^u_x$ at any $x\in M$,
 $$
   \|Df^N.v^{cs}\| \le \frac12 \|Df^N.v^u\|  \text{ and } \|Df^N.v^u\|\ge 2.
 $$
A \emph{u-state} is an invariant measure $\mu$ whose disintegrations along the leaves of the strong unstable foliation $\cW^u$ (tangent to $E^u$) are absolutely continuous with respect to the intrinsic volume; equivalently the partial entropy $h^u(f,\mu)$ in the direction $E^u$ coincides with
the average Jacobian $\int \log |\det(Df|_{E^u})|d\mu$.
A diffeomorphism is \emph{mostly contracting} \cite{Bonatti-Viana} if for any u-state, all the almost everywhere Lyapunov exponents of $\mu$ along $E^{cs}$ are negative.
The set of mostly contracting diffeomorphisms is $C^1$-open~\cite{JYang-u-entropy}. In this setting the u-states are
physical measures and the union of their basins has full volume in the manifold.
Mixing properties and limit theorems for these systems have been obtained in various works, for instance \cite{Dolgopyat, Castro, Melbourne-Nicol0}.

Any ergodic u-state $\mu$ is a hyperbolic measure. Ruelle's inequality implies that it is an equilibrium measure for
the potential $\phi(x)=-\log |\det(Df_{x}|_{E^u})|$ on the homoclinic class $X$ that carries $\mu$ and that $P(f,\mu,\phi)= 0$.

\begin{theorem}
Let $f$ be a $C^{1+}$ mostly contracting diffeomorphism and $\mu$ be an ergodic u-state carried by a homoclinic class $X$.
Then $X$ is SPR for the potential $\phi(x)=-\log |\det(Df_{x}|_{E^u})|$.
\end{theorem}
In particular this allows to recover some of the aforementioned results:
each physical measure $\mu$ is exponentially mixing (if its period is $p=1$), satisfies
the large deviation property and the almost sure invariance principle.
\begin{proof}
By Ruelle's inequality, $h^u(f,\nu)\leq \int -\phi d\nu$ for any invariant measure; the definition of mostly contracting diffeomorphisms then implies that u-states are the measures which realize the equality.
The map $\nu\mapsto h^u(f,\nu)$ on $\Prob(f)$ is upper semi-continuous~\cite{JYang-u-entropy} and $\phi$ is continuous (since $E^u$ is continuous),
hence the set of u-states is compact. Since the upper Lyapunov exponent along $E^{cs}$ is upper semi-continuous with respect to the measure (similar to Lem~\ref{prop-uscLambda}),
there exists $\chi>0$ such that any ergodic u-state $\mu$ satisfies $\lambda^i(\mu) >\chi>-\chi>\lambda^{i+1}(\mu)$,
where $i=\dim(E^{u})$.

Let $X$ be a homoclinic class which carries a u-state $\mu$. The arguments in the previous paragraph imply that
for any sequence of ergodic measures $\nu_n$ on $X$ satisfying $P(f,\nu_n,\phi)\to 0$,  any limit measure $\nu$ is a u-state.
To summarize what we have obtained: suppose $\nu_n\in \Proberg(f|_X)$ and $\nu_n\wsc \nu$ on $M$; if $P(f,\nu_n,\phi)\to P_{\Bor}(f|_X,\phi)=0$ then,
\begin{itemize}
\item there exists $i:=i(\mu)$ such that $\lambda^i(x) >\chi>-\chi>\lambda^{i+1}(x)$ $\nu$-a.e.,
\item $\lim_{n\to\infty} \Lambda^+(\nu_n)=\Lambda^+(\nu)$ as $\Lambda^+$ is the average of a continuous function since $i(\nu_n)=i(\nu)=\dim E^u$.
\end{itemize}

These properties are the analogous to (EH) and (EC) in \S\ref{s.thm-CE-SPR} when the entropy is replaced by the pressure associated to $\phi$. Thm~\ref{thm-CE-SPR} then extends to that setting with exactly the same proof and implies that $X$ is SPR for $\phi$. (See Remark~\ref{rem-CE-beyondentropy}.)
\end{proof}

\part{SPR Markov Shifts}
\label{part-SPR-Markov}

\section{Preliminaries on Countable State Markov shifts}\label{s.Markov-shifts}

\subsection{Directed Graphs}\label{s.graphs}
Let $\mathfs G$ denote a directed graph with finite or countable collection of vertices $\mathfs V$.
Suppose $a,b$ are vertices.  

If there is an edge from $a$ to $b$, then we write $a\to b$.
The {\em out-going degree} of a vertex $a$ is $\#\{b\in\mathfs V: a\to b\}$. The {\em in-going degree} of a vertex $a$ is $\#\{b\in\mathfs V: b\to a\}$. If every vertex has positive  out-going and in-going degree, then we say that $\mathfs G$ is {\em proper}. If every vertex has finite out-going and in-going degree, then we say that $\mathfs G$ is {\em locally finite}.

A  {\em path of length $n$ from $a$ to $b$} is an $n+1$ word $(\xi_0,\xi_1,\ldots,\xi_{n-1},\xi_n)\in \mathfs V^{n+1}$ such that
$\xi_0=a$, $\xi_n=b$, and $\xi_i\to\xi_{i+1}$ for $0\leq i\leq n-1$. If there is a path of length $n>0$ from $a$ to $b$, then we write $a\xrightarrow[]{n}b$.

 $\mathfs G$ is called {\em connected},  if for every $a,b\in\mathfs V$, there are positive integers $m,n$ such that $a\xrightarrow[]{m}b$ and $b\xrightarrow[]{n}a$.  A {\em maximal connected component} of $\mathfs G$ is a connected sub-graph of $\mathfs G$, which is not contained in a bigger connected sub-graph of $\mathfs G$.

Suppose $\mathfs G$ is connected. It can be shown that
$
p(a):=\gcd\{n:a\xrightarrow[]{n}a\}
$
is the same value for all vertices $a$. The common value is called the {\em period of $\mathfs G$}. If the period is $1$, $\mathfs G$ is said to be \emph{aperiodic}.

\subsection{Markov Shifts} Let $\mathfs G$ be a countable directed graph. Let
\begin{equation}\label{e.CMS}
\Sigma=\Sigma(\mathfs G):=\{(\cdots,x_{-1},x_0,x_1,\cdots)\in \mathfs V^\Z: x_i\to x_{i+1}\text{ for all }i\in\Z\}.
\end{equation}
The {\em regular part} of $\Sigma(\mathfs G)$ is
\begin{equation}\label{e.regular-part}
\Sigma^\#:=\{(x_i)_{i\in\Z}\in\Sigma: (x_i)_{i\geq 0} \text{ and } (x_i)_{i\leq 0}\text{  have constant subsequences}\}.
\end{equation}
The {\em left shift map} is the transformation
$$
\sigma:\Sigma(\mathfs G)\to \Sigma(\mathfs G),\ \ \sigma(\un{x}):=\un{y},\text{ where }y_i:=x_{i+1}.
$$
The  resulting dynamical system is called the {\em countable state Markov shift (CSMS)} (or just {\em Markov shift}) associated to $\mathfs G$. 

\medbreak

Let $\mathfs G'$ denote the maximal subgraph of $\mathfs G$ all of whose vertices appear as coordinates of at least one  element of $\Sigma(\mathfs G)$, then $\Sigma(\mathfs G)=\Sigma(\mathfs G')$ and $\mathfs G'$ is proper. It follows that when studying CSMS, it is sufficient to consider proper graphs.

\subsection{Cylinders and Topology}
A non-empty set of the form
$$
{_m}[a_0,\ldots,a_k]:=\{\un{x}\in\Sigma(\mathfs G): {x_{i+m}=a_{i}}\ (0\leq i\leq k)\}
$$
is called a {\em cylinder}.
In case $m=0$, we write
$
[a_0,\ldots,a_{k}]:={_0}[a_0,\ldots,a_k].
$
The cylinders
$
[a]:=\{\un{x}\in\Sigma(\mathfs G):x_0=a\}
$
are called {\em partition sets}. The words $(a_0,\ldots,a_k)$ for which $[a_0,\ldots,a_k]\neq \emptyset$ are called {\em admissible}.

The cylinders form a basis for a topology on $\Sigma(\mathfs G)$, which makes $\sigma$ continuous. 
The associated Borel $\sigma$-algebra $\mathfs B$ is the $\sigma$-algebra generated by the cylinders.

The cylinders are  clopen. Moreover, 

\begin{lemma}[\cite{Kitchens-Book}]\label{l.local-compactness-crit}
Let $\mathfs G$ be a proper countable directed graph. Then $\mathfs G$ is locally finite if and only if $\Sigma(\mathfs G)$ is locally compact, and  in this case the cylinders are compact.
\end{lemma}
\noindent

\subsection{Metric}\label{ss.metric}
The topology of $\Sigma(\mathfs G)$ is generated by the following  metric:
\begin{equation}\label{e.metric-on-sigma} d(\un{x},\un{y}):=
\exp[-\min\{|i|:x_i\neq y_i\}]\text{ when $\un{x}\neq \un{y}$, and $d(\un{x},\un{y}):=0$ otherwise.}
\end{equation}

Fix {$\beta>0$}. The $\beta$-H\"older norm of a complex valued function $\psi$ on $\Sigma(\mathfs G)$ is
\begin{equation}\label{e.Holder-Norm}
\|\psi\|_\beta:=\sup|\psi|+\sup\left\{\tfrac{|\psi(\un{x})-\psi(\un{y})|}{d(\un{x},\un{y})^\beta}:\un{x}\neq \un{y}\right\}.
\end{equation}
If this is finite, we say that $\psi$ is {\em $\beta$-H\"older continuous}. 

Occasionally (see e.g. the proof of Thm \ref{t.SGP}), we will need to modify $\psi$ by certain  unbounded coboundaries.  This forces  us to consider the following larger classes of functions: $\psi$
is \emph{locally $\beta$-H\"older continuous} (resp. \emph{weakly $\beta$-H\"older continuous}) if there is $C>0$ such that for any $\un{x},\un{y}$ with $x_0=y_0$ (resp. with $x_0=y_0$ and $x_1=y_1$),
$$|\psi(\un{x})-\psi(\un{y})|\le C \cdot d(\un{x},\un{y})^\beta.$$

\subsection{Irreducibility and  Aperiodicity}\label{s.spectral-decomp}
Markov shifts associated to connected graphs are called {\em irreducible}. The {\em period} of $\Sigma(\mathfs G)$ is the period of $\mathfs G$. Irreducible Markov shifts with period one are called {\em irreducible and aperiodic}. 

The following well-known lemmas explain the dynamical significance of these definitions. We omit the proofs, which are standard
(see, e.g., \cite[\S7.2]{Kitchens-Book}): 

\begin{lemma}\label{l.top-transitive-for-MS}
Let  $\mathfs G$ be a proper countable directed graph. Then
\begin{enumerate}[(1)]
\item $\sigma:\Sigma(\mathfs G)\to \Sigma(\mathfs G)$ is topologically transitive iff     $\Sigma(\mathfs G)$ is irreducible.
\item $\sigma:\Sigma(\mathfs G)\to \Sigma(\mathfs G)$ is topologically mixing iff 
        $\Sigma(\mathfs G)$ is irreducible and aperiodic.
\end{enumerate}
\end{lemma}

\begin{lemma}\label{l.irreducible-components}
The non-wandering set of a countable state Markov shift $\Sigma(\mathfs G)$ is the finite or countable disjoint union of the irreducible Markov shifts defined by the maximal connected components of $\mathfs G$.
\end{lemma}

\begin{lemma}[{\bf Spectral Decomposition}]\label{l.spectral-decomposition-CMS}
Any irreducible Markov shift with period  $p>1$ can be  decomposed into
$
\Sigma(\mathfs G)=\biguplus_{i=0}^{p-1}\Sigma_i
$ in such a way that each $\Sigma_i$ is closed, $\sigma(\Sigma_i)=\Sigma_{i+1\text{ mod }p}$, and where $\sigma^p:\Sigma_i\to\Sigma_i$ are topologically conjugate to a topologically mixing Markov shift.
\end{lemma}
\begin{proof}
This is well-known, but we would like to sketch the construction, because we will use it later.
Fix $a\in\mathfs V$, and let $S_i:=\{b\in\mathfs V:a\xrightarrow[]{n}b, n=i\mod p\}$, $i=0,\ldots,p-1$. These sets are disjoint, because $p$ is the period of $\mathfs G$.
It follows that for $u,v\in\mathfs V$, if $u\in S_i$ and $u\to v$,  then $v\in S_{i+1\text{ mod }p}$.

Letting
$
\Sigma_i:=[S_i]=\{\un{x}\in\Sigma: x_0\in S_i\}
$, we find that $\Sigma_i$ are closed, $\Sigma(\mathfs G)=\biguplus_{i=0}^{p-1}\Sigma_i
$, and  $\sigma(\Sigma_i)=\Sigma_{i+1\text{ mod }p}$. In addition,  $\sigma^p:\Sigma_i^+\to\Sigma_i^+$ is topologically conjugate to the left shift on $\Sigma(\mathfs H_i)$, where $\mathfs H_i$ is the graph  with set of vertices  $\{{_0}[a_0,\ldots,a_{p-1}]:a_0\in S_i\}\setminus\{\emptyset\}$,  and edges  $[\un{a}]\to[\un{b}]$ when $a_{p-1}\to b_0$. This graph is connected and aperiodic, therefore $\sigma^p:\Sigma_i\to\Sigma_i$ is topologically mixing.
\end{proof}
\noindent
We remark that  $\sigma^p:\Sigma_i\to\Sigma_i$ are all topologically conjugate to one another, and the conjugacies may be taken to be suitable powers of $\sigma$.

\subsection{The Measure of Maximal Entropy and Positive Recurrence}\label{ss.positive-recurrence}
Suppose $\mathfs G$ is a connected countable directed graph, and $a$ is a vertex.
A {\em loop of length $n$ at $a$} is an admissible word of the form $(a,\xi_1,\ldots,\xi_{n-1},a)$. If $\xi_i\neq a$ for all $0<i<n$, we call this a {\em first return loop of length $n$ at  $a$}. 
\begin{align}
\hspace{-0.23cm}\text{Let }& Z_n(a):=\text{ the number of loops of length $n$ at vertex $a$};\label{e.Z_n}\\
& Z_n^\ast(a):=\text{ the number of first return loops of length $n$ at vertex $a$};\label{e.Z_n-star}\\
& h(\mathfs G):=\limsup_{n\to\infty}\tfrac{1}{n}\log Z_n(a)\text{ (this could be infinite)}.\label{e.Gurevich-entropy}
\end{align}

Gurevich showed that in the connected case, $h(\mathfs G)$ is independent of $a$. This number is called the {\em Gurevich entropy} of $\mathfs G$, and has the following properties \cite{Gurevich-Topological-Entropy},\cite{Gurevich-Measures-Of-Maximal-Entropy}:

\begin{theorem}[Gurevich]
\label{t.Gurevich-Theorem}
Suppose $\mathfs G$ is a connected countable directed graph with period $p$ and set of vertices $\mathfs V$. Then:
\begin{enumerate}[(1)]
\item $\displaystyle h(\mathfs G)=\lim_{n\to\infty}\tfrac{1}{pn}\log Z_{pn}(a)$, for all  $a\in\mathfs V$. Moreover, if $h(\mathfs G)<\infty$, then $Z_n(a), Z_n^\ast(a)<\infty$ for all $n$ and $a$.
\item $\displaystyle h(\mathfs G)$ coincides with $h_{\Bor}(\Sigma(\mathfs G)):=\sup\left\{h(\sigma,\nu): \begin{array}{l}
    \nu\text{ is a shift invariant}\\
     \text{Borel probability measure}
     \end{array} \right\}.$
\item Suppose $h(\mathfs G)<\infty$, then a measure of maximal entropy exists iff
    \begin{equation}\label{e.PR}
    \sum e^{-nh(\mathfs G)} Z_n(a)=\infty\text{ and }\sum n e^{-nh(\mathfs G)} Z_n^\ast(a)<\infty\text{ for some $a\in\mathfs V$.}
    \end{equation}
    In this case the MME is unique, and
    \eqref{e.PR} holds for all vertices $a$.
\end{enumerate}
\end{theorem}
\noindent

Condition \eqref{e.PR} is called {\em positive recurrence}.\footnote{\cite{Gurevich-Measures-Of-Maximal-Entropy} uses  a different but equivalent positive recurrence condition. For the equivalence, see \cite{Gurevich-Savchenko}.} It is  different from the  {\em strong} positive recurrence condition, which we will discuss in the next section.

We now describe the (unique) MME of an irreducible Markov shift.

\begin{theorem}\label{t.Parry-Gurevich-ASS}
Let $\Sigma=\Sigma(\mathfs G)$ be a positive recurrent and irreducible Markov shift, with period $p$, and with finite Gurevich entropy. Let $\lambda=\exp(h_{\Bor}(\Sigma))$. Then:
\begin{enumerate}[(1)]
\item The MME $\mu$ is ergodic, fully supported, and  isomorphic to the product of a Bernoulli scheme and a cyclic permutation of $p$ points.
    
\smallskip
\item There are positive vectors $\un{\ell}=(\ell_u)_{u\in \mathscr{V}}$, $\un{r}=(r_u)_{u\in \mathscr{V}}$ such that\begin{equation}\label{e.vectors}
\un{\ell}A=\lambda\un{\ell}\ ,\ A\un{r}=\lambda\un{r}\ , \text{ and } \sum_{u\in\mathscr{V}} \ell_u r_u=1.
\end{equation}
\item The MME $\mu$ is the Markov measure with initial distribution $
p_v=\ell_v r_v$, and transition probabilities $\displaystyle p_{uv}=\tfrac{r_v}{\lambda r_u}$
(when $u\to v$ is allowed),
i.e.  for every cylinder {of $\Sigma$}, 
 $$
\mu({_m}[v_0,\ldots,v_{n-1}])=p_{v_0} p_{v_0 v_1}\cdots p_{v_{n-2}v_{n-1}}
{=\lambda^{-n} \frac{r_{v_{n-1}}}{r_{v_0}}}.
$$
\end{enumerate}
\end{theorem}
\noindent
Parry proved (2) and (3) when $|\mathfs G|<\infty$ \cite{Parry-Intrinsic-MC};
Gurevich proved (2) and (3) when $|\mathfs G|=\infty$ \cite{Gurevich-Measures-Of-Maximal-Entropy}. The uniqueness, ergodicity, and full support of the MME follow. When $|\mathfs G|<\infty$, the isomorphism in (1) was proved by Friedman \& Ornstein \cite{Ornstein-Friedman} when $p=1$,  and by 
 Adler, Shields \& Smorodinsky \cite{Adler-Shields-Smorodinsky}, when $p>1$. The extension to infinite, positive recurrent, graphs is routine. A detailed proof in the more general case of equilibrium measures can be found in \cite[Thm 3.1, Lem 4.1]{Sarig-Bernoulli-JMD}.

\section{Strongly Positively Recurrent  Markov Shifts}\label{s.SPR-shift}

\subsection{Strong Positive Recurrence} \label{ss-SPR-Markov-shifts}
Let $\mathfs G$ be a  connected countable directed graph with set of vertices $\mathfs V$, and assume that the entropy $h_{\Bor}(\Sigma)=h(\mathfs G)$ is finite.
Recall that $Z_n^\ast(a)$, respectively $Z_n(a)$, is  the number of first return loops, respectively arbitrary loops, of length $n$ at a vertex~$a$. 

\begin{definition}\label{d.SPR-for-Shifts}
An irreducible Markov shift $\Sigma(\mathfs G)$ is \emph{strongly positive recurrent (SPR)}, if, for some vertex $a$,
 \begin{equation}\label{ineq-MS-SPR}
    \limsup_{n\to\infty}\tfrac1n\log Z_n^\ast(a) <  \limsup_{n\to\infty}\tfrac1n\log Z_n(a)  
 \end{equation}
\end{definition}
\noindent
By Theorem \ref{t.Gurevich-Theorem}, the right hand side above is $h(\mathfs G)$. 

The nomenclature ``SPR"  is from \cite{Sarig-CMP-2001}, but equivalent conditions were considered before under different names. 
Gurevich, Zargaryan, and Savchenko \cite{Gurevich-Zargaryan,Gurevich-Stable,Gurevich-Savchenko} called this condition \emph{stable positive recurrence}, and Vere-Jones \cite{Vere-Jones-Geometric-Ergodicity}  used it to characterize \emph{geometric ergodicity}.
The following  useful characterization of the SPR property is due to Vere Jones  \cite{Vere-Jones-Geometric-Ergodicity} and Gurevich \& Savchenko \cite{Gurevich-Savchenko}. 

\begin{lemma}[Vere Jones]\label{l.SPR-for-Shifts} Let $R_a$ denote the radius of convergence of the generating function
$
\displaystyle{F_a(t):=\sum_{n=1}^\infty t^n Z_n^\ast(a).}
$
Then $F_a(R_a)>1$ iff \eqref{ineq-MS-SPR} holds.
\end{lemma}\begin{proof}\cite{Vere-Jones-Geometric-Ergodicity},\cite{Gurevich-Savchenko}.
Let $T_a(t):=1+\sum_{n\geq 1} t^n Z_n(a)$. Every loop of length $n$ at the vertex $a$ is either a first return loop, or a concatenation of a first return loop of length $k$ and another loop of length $n-k$. It follows that
$$
Z_n(a)=Z_n^\ast(a)+\sum_{k=1}^{n-1} Z_k^\ast(a) Z_{n-k}(a), \text{ whence }T_a(t)-1=T_a(t)F_a(t).
$$
Note that $\displaystyle{T_a(t)=1/(1-F_a(t))}$.
The radius of convergence of $T_a(t)$ is $r_a:=\exp[-\limsup_n\frac{1}{n}\log Z_n(a)]$, and for $F_a(t)$ it is $R_a:= \exp[-\limsup_n\frac{1}{n}\log Z_n^\ast(a)]$. 

If $F_a(R_a)>1$, then there exists $0<\rho_a<R_a$ such that $F_a(\rho_a)=1$. For all $z\in\mathbb C$ such that $|z|<\rho_a$, $|F_a(z)|<1$. Thus $\frac{1}{1-F_a(z)}$ is an analytic extension of $T_a$ to  $\{z\in\mathbb C:|z|<\rho_a\}$, and $T_a(\rho_a)=\infty$. So $r_a=\rho_a$. Since $\rho_a<R_a$, we get \eqref{ineq-MS-SPR}. 

Similarly, if $F_a(R_a)\leq 1$, then $T_a$ can be analytically extended to $\{z\in\mathbb C:|z|<R_a\}$, whence $r_a\geq R_a$. So \eqref{ineq-MS-SPR} fails.
 \end{proof}

Vere-Jones  \cite{Vere-Jones-Geometric-Ergodicity}  proved that either $F_a(R_a)>1$ for all vertices $a$, or $F_a(R_a)\leq 1$ for all vertices $a$.
Thus, if \eqref{ineq-MS-SPR} holds for some vertex, then it holds for all vertices.

\begin{example} Lemma \ref{l.SPR-for-Shifts} can be used to construct many examples and non-examples of SPR shifts. A {\em bouquet graph} is a  directed graph $\mathfs G$ obtained by starting with a ``base vertex" $a$, and then adding, for each $n$,  $\ell_n$  first-return loops of length $n$ at the vertex $a$ in such a way that all these loops meet only at $a$. 
In this case
$F_a(t)=\sum_{n=1}^\infty \ell_n t^n,\text{ and } R_a=1/\limsup_{n\to\infty}\sqrt[n]{\ell_n}.$

Fix $M\geq 1$,  and define $\ell_n=\lceil 2^n/n^2\rceil$ for $n\geq M$, and $\ell_n=0$ for $n< M$. Then $R_a=\frac{1}{2}$, and if $M=1$ then $F_a(R_a)=\tfrac12\lceil\tfrac21\rceil+\cdots>1$,  and $\Sigma(\mathfs G)$ is SPR. But if $M$ is very large, then 
$F_a(R_a)\leq \sum_{n\ge M}\frac{1}{n^2}+ 2^{-M+1}<1$,  and $\Sigma(\mathfs G)$ is not SPR.
\end{example}

As in the case of diffeomorphisms, the SPR property of Markov shifts can be expressed using the ``entropy at infinity." Define the \emph{cylinder bornology} to be
\begin{equation}\label{e.cylinder-bornology}
\mathfrak C:=\{B\subset \Sigma(\mathfs G): B\subset \text{ finite union of cylinders}\}.
\end{equation}
(If $\mathfs G$ is locally finite, then $\mathfrak C=\{\text{pre-compact sets}\}$,  see Lemma \ref{l.local-compactness-crit}.)
Next,  define the  {\em entropy at infinity} of $\Sigma(\mathfs G)$ (see \S\ref{s.bornologies-section} and \cite{Buzzi-PQFT, Ruette, iommi-todd-velozo-2020, Iommi-Todd-Velozo-Advances}) to be 
$$h^\infty(\Sigma(\mathfs G)):=h^\infty_{\mathfrak C}(\Proberg(\sigma)):=\displaystyle 
\sup\left\{\limsup_{n\to\infty} h(\mu_n):\mu_n\in\Proberg(\sigma),\; \mu_n\to\infty (\mathfrak C)\right\}.$$
The next theorem states, among other things, the equivalence of the  SPR property to the existence of an entropy gap at infinity, $h^\infty(\sigma)<h_{\Bor}(\Sigma)$:

\begin{theorem}\label{t.SPR-as-Entropy-Gap-Shifts}
Let $\mathfs G$ be a connected countable directed graph, with finite Gurevich entropy, and let $a$ be a vertex. The following are equivalent:
\begin{enumerate}[(1)]
\item $\Sigma(\mathfs G)$ is strongly positively recurrent.
\item $F_a(R_a)>1$.
\item Entropy gap at infinity: $\displaystyle h^\infty(\Sigma(\mathfs G))<h_{\Bor}(\Sigma(\mathfs G))$.
\item The shift is partially entropy-tight w.r.t. the cylinder bornology  (see  \S\ref{s.bornologies-section}).
\item The shift is (fully) entropy-tight with respect to the cylinder bornology.
\item There is a finite set $W\subset\mathfs V$ s.t.
$\Sigma':=\Sigma(\mathfs G)\cap W^\Z$ is irreducible, and  
 $$
   \forall a,b\in W,\  h_{\Bor}(\Sigma')\ge \limsup\limits_{n\to\infty}\frac{1}{n}\log \#\left\{(a,\xi_1,\ldots,\xi_n,b):
   \begin{array}{c}
    \xi_1,\ldots,\xi_n\not\in W \\
    {[a,\xi_1,\ldots,\xi_n,b]\neq \emptyset}
    \end{array}
    \right\}.
  $$
\item There exists an MME $\mu$, and for any sequence $(\mu_n)$ of measures in $\Proberg(\sigma)$, if $h(\mu_n)\to h_{\Bor}(\Sigma(\mathfs G))$, then $(\mu_n)$ converges weak-$*$ to $\mu$.
\end{enumerate}
\end{theorem}

(1)$\Leftrightarrow$(2) is Lemma~\ref{l.SPR-for-Shifts} due to Vere-Jones.
(2)$\Rightarrow$(3) is due to Buzzi \cite[Prop. 6.1]{Buzzi-PQFT}, see also \cite{Iommi-Todd-Velozo-Advances}. 
(3)$\Rightarrow$(2) is due to Ruette \cite{Ruette}. 
(3)$\Leftrightarrow$(4) is a special case of Lemma~\ref{l-partial-tight}.
(1)$\Rightarrow$(7) is due to Iommi, Todd and Velozo, see \cite[Def. 2.14 and Thm 8.12]{Iommi-Todd-Velozo-Advances}. 
(7)$\Rightarrow$(5)$\Rightarrow$(4) are trivial.
(2)$\Leftrightarrow$(6) is due to Gurevich and Zargaryan \cite{Gurevich-Zargaryan}, see \cite[Thm~3.8]{Gurevich-Savchenko}. 
So (1)--(7) are all equivalent.

Theorem \ref{t.SPR-as-Entropy-Gap-Shifts} suggests a convenient extension of the definition of the  SPR property to the case when $\mathfs G$ is not connected (see \cite[Def.~1.13]{Buzzi-PQFT}):

\begin{definition}\label{d.SPR-for-irreducible-shifts}
We say that a (possibly non-irreducible) Markov shift $\Sigma$ with finite Gurevich entropy has the \emph{SPR property} if $\displaystyle h^\infty(\Sigma)<h_\Bor(\Sigma)$.
\end{definition}

\begin{remark}\label{r.reducible-SPR-shifts}The reader can check using Thm \ref{t.SPR-as-Entropy-Gap-Shifts} that a Markov shift is SPR iff:
(1)  there is $h_0<h_\Bor(\Sigma)$ such that all but finitely many irreducible components $\Sigma_i'$ satisfy $h_{\Bor}(\Sigma_i')\leq h_0$, and (2) the irreducible components $\Sigma_i'$ such that  $h_{\Bor}(\Sigma_i')=h_{\Bor}(\Sigma)$ are all SPR in the sense of Definition \ref{d.SPR-for-Shifts}. \end{remark}
The equivalence of SPR to entropy tightness  ((1)$\Leftrightarrow$(5) above, due to~\cite{Iommi-Todd-Velozo-Advances}) immediately extends to  reducible shifts:

\begin{corollary}[{after} Iommi, Todd \& Velozo]\label{coro-tight-SPR}
Suppose
$\Sigma$ is a Markov shift with finite entropy $h_\Bor(\Sigma)$.
Then $\Sigma$ is SPR iff for every $\delta>0$, there is $h_\delta<h_{\Bor}(\Sigma)$ and a finite union $F$ of cylinders such that
$$
\forall\mu\in\Proberg(\Sigma), \ h(\sigma,\mu)>h_\delta\Rightarrow \mu(F)>1-\delta.
$$
\end{corollary}

\subsection{SPR and Positive Recurrence}\label{s.SPR-vs-PR}
Thm \ref{t.Gurevich-Theorem} characterizes the irreducible Mar\-kov shifts with MME in terms of  the positive recurrence condition \eqref{e.PR}. Let us compare this condition to strong positive recurrence.

\medskip
\noindent
{\em SPR implies Positive Recurrence:\/} We saw in Lemma \ref{l.SPR-for-Shifts} and its proof that  SPR implies that  $F_a(R_a)>1$,  $T_a(r_a)=T_a(\rho_a)=\infty$ and $F_a'(r_a)<\infty$. The equality says that $\sum e^{-n h(\mathfs G)}Z_n(a)=\infty$, and the inequality says that $\sum n e^{-n h(\mathfs G)}Z_n^\ast(a)<\infty$. This is eq.~\eqref{e.PR}, i.e., positive recurrence.

\medskip
\noindent
{\em SPR is Strictly Stronger than Positive Recurrence:\/} The following example is due to Ruette \cite{Ruette}. Define $
\ell_n:=2^{n-\sqrt{n}}$ when $n\geq 1$ is a square, and $\ell_n:=0$ when $n\geq 1$ is not a square.
Take the bouquet graph with base vertex $a$, and $\ell_n$ first return loops of length $n$ at $a$.
Then $F_a(t)$ has radius of convergence $R_a=\frac{1}{2}$, and $F_a(R_a)=1$. So the SPR property fails.
On the other hand,
$
T_a(t):=1+\sum_{n=1}^\infty t^n Z_n(a)=\frac{1}{1-F_a(t)}
$ has radius of convergence $1$, $T_a(\frac{1}{2})=\infty$, and $F_a'(\frac{1}{2})=\sum k^2 2^{-k+1}<\infty$. This shows positive recurrence (with Gurevich entropy $\ln 2$).

\medskip
A good way to understand the difference between SPR and positive recurrence is to think of the recurrence properties of the MME $\mu$. Let $\lambda:=\exp[h(\mathfs G)]$, and let
$$
\vf_a(\un{x}):= \mathds{1}_{[a]}(\un{x})\min\{n\geq 1: x_n=a\}.
$$

\medskip
\noindent
{\em Positive Recurrence Says that  $\E_\mu[\vf_a]<\infty$}. Indeed by Thm \ref{t.Parry-Gurevich-ASS} we see that for every admissible loop $(a,\xi_1,\ldots,\xi_{n-1},a)$,
$$
\mu({_a}[a,\xi_1,\ldots,\xi_{n-1},a])=p_a p_{a\xi_1}\cdots p_{\xi_{n-1}a}=c(a)\lambda^{-n}
$$
where $c(a)$ is a positive constant which only depends on $a$. It follows that
$
\mu[\vf_a=n]=c(a)\lambda^{-n}Z_n^\ast(a)
$, whence by positive recurrence,
$$
\E_\mu [\vf_a]=\sum_{n=1}^\infty n\mu[\vf_a=n]=c(a)\sum_{n=1}^\infty n \lambda^{-n}Z_n^\ast(a)<\infty.
$$

\noindent
{\em SPR Says that there is $t>0$ such that $\E[e^{t\vf_a}]<\infty$}. Indeed, in the SPR case,   $\mu[\vf_a=n]=c(a)\lambda^{-n}Z_n^\ast(a)\to 0$ exponentially fast, by  Thm \ref{t.SPR-as-Entropy-Gap-Shifts}(2).

\begin{corollary}[Vere-Jones]\label{c.Tail-CMS}
Let $\Sigma$ be an SPR irreducible Markov shift with finite entropy $h_\Bor(\Sigma)$ and a MME $\mu$.
For every $a\in \mathscr{V}$ there are $0<\theta<1$, $C>1$ s.t.
$$\mu\big\{\un{x}\in \Sigma: \inf\{k\geq 1: x_k=a\} >n\big\}<C\theta^n.$$
\end{corollary}
\noindent
{\em Proof.\/}
Let $\tau_a^+(\un{x}):=\inf\{n\geq 1:x_n=a\}$ and $\tau_a^-(\un{x}):=\sup\{n\leq {0}
: x_n=a\}$  when defined, and  $\tau_a^\pm(\un{x}):=\pm\infty$ otherwise.
 These functions are finite a.e., because
 $\mu$ is ergodic and fully supported, see Thm~\ref{t.Parry-Gurevich-ASS}. Note that
 $\vf_a(\un{x}):=\mathds{1}_{[a]}(\un{x})\tau_a^+(\un{x})$.

As mentioned above, in the  SPR case, there are constants $C_0>0$ and $0<\theta<1$ such that  $\mu[\vf_a=n]\leq C_0\theta^n$.
Since $\tau_a^-$ is finite a.e.,
 \begin{align*}
 &\mu[\tau_a^+>n]=\mu\left([\tau_a^+>n]\cap [\tau_a^->-\infty]\right)=\sum_{\ell={0}
 }^\infty\sum_{r=n+1}^\infty\mu[\tau_a^-=-\ell, \tau_a^+=r]\\
 &=\sum_{\ell={0}
 }^\infty\sum_{r=n+1}^\infty(\mu\circ\sigma^{\ell})[\vf_a=\ell+r]\leq \sum_{\ell={0}
 }^\infty\sum_{r=n+1}^\infty C_0\theta^{\ell+r}\leq\mathrm{const.}\theta^n.\hspace{2.4cm}\qed
 \end{align*}
\begin{remark}
It is shown in \cite{Vere-Jones-Geometric-Ergodicity} that $\theta$ depends on $a$.
\end{remark}

\subsection{Spectral Gap and Properties of SPR Markov Shifts}
{Since SPR Markov shifts are positive recurrent, the Gurevich Theorem~\ref{t.Gurevich-Theorem}(3)
implies that every irreducible SPR Markov shift admits a unique MME.
In the reducible case:}

\smallskip
\noindent
\begin{corollary}[{after} Gurevich]\label{t.Existence-of-MME-CMS}
{Every SPR   Markov shift admits a positive and finite number of ergodic measures of maximal entropy.}
\end{corollary}
\begin{proof}
{Each ergodic invariant measures is carried by an irreducible component.
By Remark \ref{r.reducible-SPR-shifts},  an SPR Markov shift has a positive and finite number of transitive components with full entropy. The result now follows from  Thm \ref{l.irreducible-components}.}
\end{proof}
Obviously, the ergodic properties of one MME depend only on  the irreducible component which carries it. 
From this point on, suppose $\Sigma(\mathfs G)$ is an {\em irreducible} SPR Markov shift with finite Gurevich entropy and  period $p$. Let $\mu$ be its MME.
In Appendix \ref{appendix-Markov-Shifts} we state  and prove the following results:
 \begin{enumerate}
  \item \textbf{{Exponential Decay of Correlations}} for H\"older continuous functions in case $p=1$, and  exponential decay of correlations for the ergodic components of $\mu$ with respect to $\sigma^p$  in case $p>1$ (Thm~\ref{t.DOC-for-CMS}).
  \item \textbf{Large Deviations Property} for Birkhoff sums of H\"older continuous observables (Thm~\ref{t.LDP-CMS}).
   \item  \textbf{Almost Sure Invariance Principle} for Birkhoff sums of H\"older continuous observables (Thm~\ref{t.ASIP-Sigma}).
  \item \textbf{Asymptotic Variance} of Birkhoff sums of H\"older continuous observables: existence,  Green-Kubo identity, linear response formula, and periodic orbit conditions for the non-vanishing of the asymptotic variance (Thm~\ref{t.zero-variance-CMS}).
   \item \textbf{Equilibrium Measures for Small Potentials}: there exists $\eps_0>0$ such that for every H\"older continuous {potential} $\phi:\Sigma(\mathfs G)\to\R$, if $\sup|\phi|<\eps_0$, then $\phi$ is SPR and has a unique equilibrium measure which satisfies properties (1)--(4) above,
            (see Thm~\ref{t.no-phase-transition}, \S\ref{s.TDF-for-diffeos}, and \cite{Cyr-Sarig}).
 \end{enumerate}

{The proofs of these results use the one-sided shift $\Sigma^+(\mathfs G)$ associated
to the graph $\mathfs G$ and its transfer operator.  They}
rely in a crucial way on the following property of SPR Markov shifts, proved in \cite{Cyr-Sarig}, and stated precisely in Appendix \ref{appendix-Markov-Shifts}:

\medskip
\noindent
 {\em There is  a ``nice" Banach space $\mathcal L$ of one-sided functions $\Sigma^+(\mathfs G)\to\R$, where the transfer operator of the MME acts ``with spectral gap."}

\medskip
\noindent
In fact some form of this  property is {\em equivalent} to the SPR property \cite{Cyr-Sarig}.

The derivation of properties (1)--(5) from the spectral gap was done before in different setups \cite{Guivarch-Hardy},\cite{Ruelle-TDF-book},\cite{Parry-Pollicott-Asterisque},\cite{Kifer},\cite{Gouezel-ASIP},\cite{Cyr-Sarig}. For our setup, see Appendix \ref{appendix-Markov-Shifts}.

\subsection{Strongly Positively Recurrent Markov Shifts with a Potential}
Throu\-ghout this section, let  $\Sigma:=\Sigma(\mathfs G)$ be an irreducible Markov shift, and let $\phi:\Sigma\to\R$ be a H\"older continuous function (in particular, $\sup|\phi|<\infty$).\footnote{Although we will not use this in this paper, we mention that the results of this section hold under the weaker assumption that $\phi$ is weakly H\"older and $\sup\phi<\infty$ \cite{Sarig-ETDS-99,Buzzi-Sarig,Sarig-CMP-2001,Daon}. }

A shift invariant probability measure $\mu$ is called an {\em equilibrium measure} for $\phi$ if the  {\em pressure of $\mu$}, $P(\sigma,\mu,\phi):=h(f,\mu)+\int \phi d\mu$, equals the  
 {\em top pressure of $\phi$},  $$P_{\Bor}(\Sigma,\phi):=\sup\{P(\sigma,\mu,\phi)\; :\; \mu\in \Prob(\sigma)\}.$$
Equilibrium measures of constant potentials are simply MME.

The theory for MME in  \S\ref{ss.positive-recurrence} and \S\ref{s.SPR-shift} has been extended to equilibrium measures of  H\"older continuous potentials
$\phi\colon \Sigma\to \RR$ in \cite{Gurevich-Stable,Sarig-CMP-2001}. Here is a  summary.
Suppose $P_{\Bor}(\Sigma,\phi)\!<\!\infty$ (since $\sup\phi<\infty$, this happens iff $h_{\Bor}(\Sigma)<\infty$).  Then:
\begin{enumerate}
\item If the equilibrium measure exists, then it is unique \cite{Buzzi-Sarig};
\item Call the unique equilibrium measure $\mu_\phi$, then  $(\Sigma,\mu_\phi)$ is measure theoretically isomorphic to the product of a Bernoulli scheme and a cyclic permutation of $p$ points, where $p$ is the period of $\Sigma$  \cite{Sarig-Bernoulli-JMD,Daon}.
\end{enumerate}

Given $a\in\mathfs V$, let  
$\vf_a(\underline x)=\mathds{1}_{[a]}(\underline x)\inf\{n\geq 1 : \sigma^n(\underline x)\in [a]\}$, $\phi_n:=\sum\limits_{k=0}^{n-1}\phi\circ\sigma^k$, and 
$$\displaystyle Z_n(\phi,a):=\sum_{\sigma^n(\underline x)=\underline x} e^{\phi_n(\underline x)}\mathds{1}_{[a]}(\underline x)\;\;
\text{ and }\;\;
Z_n^*(\phi,a):=\sum_{\sigma^n(\underline x)=\underline x} e^{\phi_n(\underline x)}\mathds{1}_{[\varphi_a=n]}(\underline x).$$
It is shown in \cite{Sarig-ETDS-99} that 
\begin{equation}\label{e.gurevich-pressure}
P_{\Bor}(\Sigma,\phi)=\limsup_{n\to +\infty} \tfrac 1 n \log Z_n(\phi,a).
\end{equation}
\begin{definition}\cite{Sarig-CMP-2001}\label{d.SPR-potential2}
We say that $\phi$ is a {\em strongly positive recurrent (SPR) potential} on  $\Sigma$ (or that {\em $\Sigma$ is SPR for $\phi$}),  if for some vertex $a$:
\begin{equation}\label{e.SPR-potential2}
\underset{n\to +\infty}\limsup \tfrac 1 n \log Z^*_n(\phi,a)<\underset{n\to +\infty}\limsup \tfrac 1 n \log Z_n(\phi,a).
\end{equation}
\end{definition}

\begin{remark}\label{r.SPR-potential}
In \cite{Sarig-CMP-2001} one defines a quantity $\Delta_a(\phi)$ called {\em the  discriminant},  and one defines the SPR property for $\phi$
by the condition $\Delta_a(\phi)>0$. The ``discriminant theorem"~\cite[Thm 2]{Sarig-CMP-2001} shows that $\Delta_a[\phi]>0$ iff \eqref{e.SPR-potential2} holds. Corollary 1 in  \cite{Sarig-CMP-2001} says that if \eqref{e.SPR-potential2} holds with {\em some} vertex $a$, then it holds with {\em all} vertices $a$.
\end{remark}

\begin{proposition}\cite{Ruhr-Sarig}\label{p.SPR-potential} 
Let  $\Sigma$ be an irreducible Markov shift with  finite Gurevich entropy, and suppose  $\phi$ is a H\"older continuous function.  Then $\phi$ is SPR on $\Sigma$ iff there exist $P_0<P_{\Bor}(\Sigma,\phi)$, $\tau>0$ and a finite union of cylinders $B$ such that:
\begin{equation}\label{e.SPR-potential-irred}
\text{For every ergodic measure $\nu$ on $\Sigma$,}\quad
    P(\sigma,\nu,\phi)>P_0 \implies \nu(B)>\tau.
\end{equation}
\end{proposition}

\begin{proof}
Suppose $\phi$ is SPR on $\Sigma$. By \cite{Sarig-CMP-2001} and the assumptions that $h_{\Bor}(\Sigma)<\infty$ and $\phi$ is  H\"older continuous, $\phi$ has an equilibrium measure $\mu_\phi$. 
By 
 \cite[Thm 7.1]{Ruhr-Sarig} and the discussion in \S9 there, there is a constant $C$ so that for every partition set $[a]$, for every shift-invariant measure $\nu$, 
 $
 |\nu[a]-\mu_\phi[a]|\leq C\sqrt{P_{\Bor}(\Sigma,\phi)-P(\sigma,\nu,\phi)}.
 $

In particular, for every $a$ such that $\mu_\phi[a]>0$ and for every $P_0$ sufficiently close to $P_{\Bor}(\Sigma,\phi)$, if $P(\sigma,\nu,\phi)>P_0$, then $\nu[a]>\frac{1}{2}\mu_{\phi}[a]$. {One easily deduces} \eqref{e.SPR-potential-irred}.

The other direction requires the following observation: By Thm~\ref{t.Gurevich-Theorem}(1) and the assumption that  $h_{\Bor}(\Sigma)<\infty$, we have that  $Z_n(b)<\infty$ for every vertex $b$ and every $n$. This is the ``$\mathcal F$-property" of \cite{iommi-todd-velozo-2020}.

Suppose \eqref{e.SPR-potential-irred} holds,  and assume by way of contradiction that $\phi$ is not SPR. 
By the $\mathcal F$-property and Lemma 8.1 in \cite{Ruhr-Sarig}, there exists a sequence of ergodic shift-invariant  probability measures $\nu_n$ such that $P(\sigma,\nu_n,\phi)\to P_{\Bor}(\Sigma,\phi)$, but $\nu_n[a]\to 0$ for all vertices $a$. Clearly, this contradicts \eqref{e.SPR-potential-irred}. So $\phi$ must be SPR.
\end{proof}
We can now extend Def~\ref{d.SPR-potential2} to reducible Markov shifts:
\begin{definition}\label{d.SPR-potential3}
We say that a Markov shift $\Sigma$ with finite Gurevich entropy is \emph{SPR for a H\"older continuous $\phi$}, if there exist a finite union of cylinders $B$
and numbers $P_0<P_{\Bor}(\Sigma,\phi)$, $\tau>0$ s.t.
\begin{equation}\label{e.SPR-potential3}
\text{For every ergodic measure $\nu$ on $\Sigma$,}\quad
    P(\sigma,\nu,\phi)>P_0 \implies \nu(B)>\tau.
\end{equation}
\end{definition}
\noindent
Again, $\phi$ is SPR on  $\Sigma$ iff  for some $P_0<P_{\Bor}(\Sigma(\mathfs G),\phi)$,  $\phi$ is SPR on all  irreducible components with pressure $>P_0$,  and the number  
of such components is finite.

\part{Symbolic Dynamics of SPR Diffeomorphisms}\label{part-symbolic-diffeo}

\section{Hyperbolic, Entropy-full, and Bornological  Codings}\label{s.hyperbolic-codings}

Suppose $T$ is a map on a topological space $Y$.
A {\em Markovian coding}  (or just {\em ``coding"}) of $T$ is a map $\pi:\Sigma\to Y$, where $\Sigma$ is a  countable state Markov shift,  $\pi$ is continuous, and  {$\pi\circ \sigma=T\circ\pi.$} We do not require injectivity or surjectivity. Codings with weak forms of these  properties are discussed below (Def \ref{d.Hyperbolic-Coding}~or~\ref{d-entropy-full}).

Many dynamical systems have useful Markovian codings: hyperbolic toral auto\-mor\-phisms  \cite{Adler-Weiss-Similarity-Toral-Automorphisms}, Anosov diffeomorphisms \cite{Sinai-Construction-of-MP,Sinai-MP-U-diffeomorphisms}, Axiom A diffeomorphisms
 \cite{Bowen-MP-Axiom-A}, piecewise monotonic interval maps \cite{Takahashi, Hofbauer-Tower,Buzzi-SIM,Buzzi-PQFT,Lima-1D},  and non-uniformly hyperbolic maps \cite{young-tower-recurrence,Sarig-JAMS,Ben-Ovadia-Codings,Araujo-Lima-Poletti}. 

We will only consider $C^{1+}$ diffeomorphisms $f$  on closed manifolds $M$. 
Our ultimate aim is to show that if $f$ is SPR, then $f$ admits ``good" Markovian codings $(\Sigma,\pi)$ with SPR Markov shifts $\Sigma$ (Thm~\ref{t.SPR-coding}). 
This section prepares the ground for this, by identifying  several desirable properties of abstract Markovian codings, which allow to relate the SPR property of $f$ to the SPR property of $\Sigma$.

\subsection{Hyperbolic Codings}
The lack of injectivity of the coding map may create many difficulties, including the following  crucial one:
some measures on $\pi(\Sigma)$ may fail to have lifts to $\Sigma$, or may have lifts with different entropy.

The ``hyperbolic codings" we will now introduce will not suffer from this problem (Lemma \ref{l.Hyp-Mark-Coding-Prop} below). In addition, they will have the following nice feature: every point in the coded set $\pi(\Sigma)$ has well defined ``stable" and ``unstable" directions.

Before giving the definition, we remind the reader that $\Sigma$ is endowed with the metric \eqref{e.metric-on-sigma}, and  the {\em regular part} of a Markov shift $\Sigma$ is the set$$
\Sigma^\#:=\{(x_i)_{i\in\Z}\in\Sigma: (x_i)_{i\geq 0}, (x_i)_{i\leq 0}\text{ both contain constant sub-sequences}\}.
$$

\begin{definition}\label{d.Hyperbolic-Coding}
Let $X$ be an invariant Borel subset. 
A {\em hyperbolic coding in}  $X$ is a pair $(\Sigma,\pi)$ where $\Sigma$ is a locally compact countable state Markov shift, and $\pi:\Sigma\to M$
is a H\"older continuous map satisfying:
\begin{enumerate}[(a)]
\item\label{i.def1}  $\pi\circ\sigma=f\circ\pi$.
\item\label{i.def2} $\pi:\Sigma^\#\to M$ is finite-to-one, i.e. $(\pi|_{\Sigma^\#})^{-1}(x)$ is finite for each $x\in M$.\footnote{We do not require the cardinality of these sets to be uniformly bounded.}
\item\label{i.def5} $(\pi_*\hmu)(X)=1$ for any $\sigma$-invariant probability measure $\hmu$  on $\Sigma$.
\item\label{i.def3} For some $\chi_0>0$, for every $\un x\in \Sigma$, there is  a splitting $T_{\pi(\un x)}M=E^s(\un x)\oplus E^u(\un x)$ s.t. $\underset{n\to +\infty} \limsup \tfrac 1 n \log \|Df^n|_{E^s(\un x)}\|<-\chi_0$,   $\underset{n\to +\infty} \limsup \tfrac 1 n \log \|Df^{-n}|_{E^u(\un x)}\|<-\chi_0$, and such that $\underline x\mapsto E^{s}(\un x)$ and $\underline x\mapsto E^{u}(\un x)$ are H\"older continuous.
\end{enumerate}
\end{definition}

\medbreak

\begin{lemma}[Basic Properties of Hyperbolic Codings]\label{l.Hyp-Mark-Coding-Prop}
Let $(\Sigma,\pi)$ be a hyperbolic coding in an $f$-invariant measurable set $X$.
Then:
\begin{enumerate}[(1)]
\item\label{i.basic0} $\pi(\Sigma^\#)$
is Borel measurable.
\item\label{i.basic1} Every ergodic measure $\hmu$ on $\Sigma$ satisfies $h(f,\pi_*\hmu)=h(\sigma,\hmu)$.

\item\label{i.basic2} Every $f$-invariant ergodic measure $\mu$ with $\mu(\pi(\Sigma^\#))=1$ admits an  ergodic lift to $\Sigma$,
i.e., some $\sigma$-invariant ergodic measure $\hmu$ satisfying $\pi_*\hmu=\mu$.

\item\label{i.basic2bis} $h_{\Bor}(\Sigma)\leq h_{\Bor}(f|_X)$.

\item\label{i.basic3} If $\phi:M\to\R$ is H\"older continuous, then $\hphi:=\phi\circ\pi$ is H\"older continuous on $\Sigma$.

\item\label{i.basic4} If $\Sigma$ is irreducible, then all  $\sigma$-invariant ergodic measures  on $\Sigma$ project to  homoclinically related $f$-invariant measures on $M$.
\end{enumerate}
\end{lemma}

\begin{proof}
Part \eqref{i.basic0} is a consequence of the following general result \cite[Thm 18.10]{Kechris}: {\em The image of a Borel set by a countable-to-one Borel map is Borel measurable.}

Part \eqref{i.basic1} is a consequence of the Abramov-Rokhlin entropy  formula \cite{Abramov-Rokhlin}, and the fact that by Def.~\ref{d.Hyperbolic-Coding}\eqref{i.def2}, $\sigma:\Sigma^\#\to\Sigma^\#$ is measure theoretically isomorphic to a skew-product over $f:M\to M$, with finite fibers.

Part \eqref{i.basic2} follows from Def.~\ref{d.Hyperbolic-Coding}\eqref{i.def2} as in  the proof of  \cite[Prop. 13.2]{Sarig-JAMS}. Briefly, $\hmu$ can be taken to be  a typical ergodic component of the measure
 $$\int_{\Sigma^\#}\frac{1}{|(\pi|_{\Sigma^\#})^{-1}(x)|}\sum_{y\in (\pi|_{\Sigma^\#})^{-1}(x)}\delta_y\, d\mu(x).
 $$
Here the finiteness-to-one of $\pi:\Sigma^\#\to M$ is essential.

Part \eqref{i.basic2bis} follows from Part \eqref{i.basic1}.
Part \eqref{i.basic3} is a consequence of the H\"older continuity of $\pi:\Sigma\to M$.
Part \eqref{i.basic4} is the content of Proposition 3.6 in~\cite{BCS-1}.
\end{proof}

\subsection{Entropy-Full Codings}
A hyperbolic coding need not carry the  lift of every invariant  measure. ``Entropy-full" codings lift   all measures with large entropy:
\begin{definition}\label{d-entropy-full}
A hyperbolic coding $(\Sigma,\pi)$ in $X$ is \emph{entropy-full} 
if there is $h_0<\hTOP(f|_X)$ such that every ergodic invariant probability measure $\mu$ on $X$ with entropy bigger than $h_0$ is carried by $\pi(\Sigma^\#)$.
\end{definition}

\begin{lemma}\label{l-entropy-equal}
If $(\Sigma,\pi)$ is hyperbolic and entropy-full, then $\hTOP(f|_X)=\hTOP(\Sigma)$.
\end{lemma}
\begin{proof}
Take $\mu_n\in\Proberg(f|_X)$ such that $h(f,\mu_n)\to h_{\Bor}(f|_X)$. By entropy-fullness, for all $n$ sufficiently large, $\mu_n(\Sigma^\#)=1$, and by Lemma \ref{l.Hyp-Mark-Coding-Prop}\eqref{i.basic2}, $\mu_n=\pi_\ast\hmu_n$ for some $\hmu_n\in\Proberg(\sigma)$. So $h(f,\mu_n)\leq h(\sigma,\hmu_n)\leq h_{\Bor}(\Sigma)$. Passing to the limit $n\to\infty$, we obtain $h_{\Bor}(f|_X)\leq h_{\Bor}(\Sigma)$. By Lemma \ref{l.Hyp-Mark-Coding-Prop}\eqref{i.basic2bis},  $h_{\Bor}(f|_X)=h_{\Bor}(\Sigma)$.
\end{proof}

\subsection{Bornological Codings}
The next property relates the Pesin bornologies of $f|_X$ to the cylinder bornology of $\sigma:\Sigma\to\Sigma$. We will use it to relate the SPR property of $f|_X$ to the SPR property of $\Sigma$.

\begin{definition}\label{d.bornological}

A hyperbolic coding $(\Sigma,\pi)$ in $X$ is {\em $\chi$-bornological in $X$}  if:
\begin{enumerate}[(a)]
\item\label{i.born1} For some $\varepsilon>0$, for any  $(\chi,\varepsilon)$-Pesin block $\Lambda$,
there exists a finite collection of cylinders  $A_i$ in  $\Sigma$ such that
$\hmu(\pi^{-1}(\Lambda)\setminus \bigcup_i A_i)=0$ for every $\hmu\in\mathbb P(\sigma)$.

\smallskip
\item\label{i.born2} For some $0<\tilde \varepsilon<\tilde \chi$, for  any cylinder $A$ in $\Sigma$,
there is a $(\tilde \chi,\tilde \varepsilon)$-Pesin block $\Lambda$  such that $\mu(\pi(A\cap \Sigma^\#)\setminus
\Lambda)=0$ for every $\chi$-hyperbolic  $\mu\in\mathbb P(f|_X)$.
\end{enumerate}
We say that the coding is {\em bornological}, if it is $\chi$-bornological for some $\chi>0$.
\end{definition}
\noindent

The {\em smaller} the $\chi$, the stronger  the $\chi$-bornological property, since there are  more Pesin blocks and more hyperbolic measures to consider.

\begin{proposition}\label{p.lift-SPR}
If $f$ is $\chi$-SPR on an invariant Borel set $X$ (Def.~\ref{def-SPR-diffeo-local}),
then every hyperbolic, $\chi$-bornological  coding $(\Sigma,\pi)$ in $X$ s.t. $\hTOP(\Sigma)=\hTOP(f|_X)$  is SPR.
\end{proposition}
\begin{remark}\label{rem-SPR-single-eps}This result  only requires the condition in  Def.~\ref{d.bornological}\eqref{i.born1}. Also, we do not need the full force of the SPR property for $f$, only the partial entropy tightness  with respect to $\mathfrak P_{\chi,\eps}$ for a {\em single} (but small enough) $\eps$.
\end{remark}
\begin{proof}

Take $\eps>0$ as in Def.~\ref{d.bornological}\eqref{i.born1}.
By the $\chi$-SPR property, there exist a $(\chi,\eps)$-Pesin block $\Lambda$ and $h_0<h_{\Bor}(f|_X), \tau>0$ such that
\begin{equation}\label{e.three-twenty}
\text{for every ergodic measure $\nu$ on $X$, $h(f,\nu)>h_0\Longrightarrow \nu(\Lambda)>\tau$}.
\end{equation}

 By Def.~\ref{d.bornological}\eqref{i.born1}, there is a finite union of cylinders $A\subset \Sigma$ such that
every $\sigma$-invariant measure gives zero mass to
$\pi^{-1}(\Lambda)\setminus A$.
By assumption,
$$
h_0<h_\Bor(f|_X)=h_{\Bor}(\Sigma).
$$
Let $\hnu$ be an ergodic measure  on $\Sigma$ such that $h(\sigma,\hnu)>h_0$. The projection $\nu:=\pi_\ast\hnu$ is an $f$-ergodic measure on $X$, and by Lemma \ref{l.Hyp-Mark-Coding-Prop}\eqref{i.basic1}, $h(f,\nu)=h(\sigma,\hnu)>h_0$. By~\eqref{e.three-twenty}, $\nu(\Lambda)>\tau$, whence by the choice of $A$,
$
\hnu(A)\geq \hnu(\pi^{-1}(\Lambda))=\nu(\Lambda)>\tau.
$

We started with the assumptions that $\hnu\in\Proberg(\sigma)$ and $h(\sigma,\hnu)>h_0$, and ended with the conclusion that $\hnu(A)>\tau$. Thus the entropy at infinity of $\Sigma$ does not exceed $h_0$, and we have an entropy gap at infinity. By Def.~\ref{d.SPR-for-irreducible-shifts}, $\Sigma$ is SPR.
\end{proof}

\begin{proposition}\label{p.project-SPR}
{Given $\chi>0$, let $(\Sigma,\pi)$ be an entropy-full, $\chi$-bornological, hyperbolic coding in $X$, and assume that $f|_X$ satisfies the following property:}
\begin{equation}\label{e.chi-entropy-hyperbolic}
\exists h_0<h_{\Bor}(f|_X),\quad (\forall \nu\in\Proberg(f|_X),\ h(f,\nu)>h_0\Longrightarrow\text{ $\nu$ is $\chi$-hyperbolic}).
\end{equation}
If $\Sigma$ is SPR, then $f$ is  entropy-tight, whence SPR,  on $X$.
\end{proposition}
\begin{remark}\label{r.SPR-implies-hyperbolicity}
(a) The proof only requires condition  Def.~\ref{d.bornological}\eqref{i.born2}.\\
(b) In dimension two, \eqref{e.chi-entropy-hyperbolic}  holds whenever 
$\chi<h_{\Bor}(f|_X)$,  by Ruelle's  inequality.\\ (c) If $f|_X$ is $\chi_0$-SPR for some $\chi_0>0$, then \eqref{e.chi-entropy-hyperbolic}
holds with any number $\chi\in (0,\chi_0)$.
\end{remark}
\begin{proof}
Def.~\ref{d.bornological}\eqref{i.born2}
provides  constants $\tilde\chi>\tilde\eps>0$ such that for any finite union of cylinders $A\subset\Sigma$ there is a $(\tilde\chi,\tilde\eps)$-Pesin block $\Lambda$ such that $
\mu(\pi(A\cap\Sigma^\#)\setminus\Lambda)=0$  for all $\chi$-hyperbolic invariant measures $\mu$ for $f|_X$.

Fix $\chi_0>0$ slightly smaller than $\tilde\chi$. We will show that for every $\delta>0$, there exist $h_\delta<h_\Bor(f|_X)$ and a $(\chi_0,\tilde\eps)$-Pesin block $\Lambda_\delta$ such that
\begin{equation}\label{e.9-point-2}
\forall\mu\in\Proberg(f|_X)\ \left(
h(f,\mu)>h_\delta\Rightarrow \mu(\Lambda_\delta)>1-\delta
\right).
\end{equation}
Eq~\eqref{e.9-point-2} says that   $f$ is entropy-tight, whence SPR, on $X$, see Def.~\ref{d.entropy-tight-diffeo} and Prop.~\ref{prop-tight-implies-SPR}.

\smallskip
\noindent
{\it Choice of ${h_\delta}$ and  ${A_\delta}$.}
By Lemma~\ref{l-entropy-equal}, $h_\Bor(\Sigma)=h_\Bor(f|_X)$.
Choose $h_0<h_{\Bor}(\Sigma)$ as in Def.~\ref{d-entropy-full} and Eq~\eqref{e.chi-entropy-hyperbolic}.
Fix $\delta>0$. Let $c_0:=\chi_0+\max(\log\|Df\|_{\sup}{,}\log\|Df^{-1}\|_{\sup})$, and choose $0<\tau_0<1$ so close to one, that
\begin{equation}\label{e.tau-zero}
4\max\left(\tfrac{c_0}{\tilde{\eps}},1\right)(1-\tau_0)<\delta.
\end{equation}

By assumption, $\Sigma$ is SPR. For  Markov shifts, the SPR property is equivalent to entropy tightness (Cor~\ref{coro-tight-SPR}). Therefore there exist $h_0<h_\delta<h_{\Bor}(\Sigma)$ and a finite union of cylinders $A_\delta$ such that
$$
\forall\hmu\in\Proberg(\sigma)\
\left(
h(f,\hmu)>h_\delta\Rightarrow \hmu(A_\delta)>1-\tau_0
\right).
$$

\smallskip
\noindent
{\it Choice of ${\Lambda_\delta}$.}
By the choice of $\tilde\chi$ and  $\tilde\eps$, there is a $(\tilde\chi,\tilde\eps)$-Pesin block $\tilde\Lambda_\delta$ such that
\begin{equation}\label{e.Lambda-delta-prime}
\mu(\pi(A_\delta\cap\Sigma^\#)\setminus\tilde\Lambda_\delta)=0\text{ for all $\chi$-hyperbolic $\mu\in\mathbb P(f|_X)$.}
\end{equation}
Since $\chi_0<\tilde\chi$, there exists an $n_0$ such that $\tilde\Lambda_\delta\subset P_{n_0}$, the Pliss set as in eq.~\eqref{e.PlissSet}:
$$
P_{n_0}:=\left\{x\in M:\begin{array}{l}
T_x M=E^u(x)\oplus E^s(x)\text{ and }\forall j\geq 0,\\
\|Df^{j n_0}|_{E^s(x)}\|\leq e^{-j n_0 \chi_0}\ , \  \|Df^{-j n_0}|_{E^u(x)}\|\leq e^{-j n_0 \chi_0}
\end{array} \right\}.
$$
By Prop. \ref{p.pesin}, there exists a $(\chi_0,\tilde\eps)$-Pesin block $\Lambda_\delta$ such that
$$
\forall\nu\in\Proberg(f),\ \nu(M\setminus\Lambda_\delta)\leq 4\max\left(\tfrac{c_0}{\tilde\eps},1\right)\nu(M\setminus P_{n_0}).
$$

\smallskip
\noindent
{\it Proof of \eqref{e.9-point-2}.} Suppose   $\mu\in \Proberg(f|_X)$ and $h(f,\mu)>h_\delta$.
By the choice of $\Lambda_\delta$,
\begin{align*}
&\mu(M\setminus\Lambda_\delta)\leq 4\max\left(\tfrac{c_0}{\tilde\eps},1\right)\mu(M\setminus P_{n_0})
\leq 4\max\left(\tfrac{c_0}{\tilde\eps},1\right)\mu(M\setminus\tilde\Lambda_\delta).
\end{align*}
Since $h_\delta>h_0$ and $h_0$ satisfies Eq~\eqref{e.chi-entropy-hyperbolic}, $\mu$ is $\chi$-hyperbolic. By \eqref{e.Lambda-delta-prime}, $$
\mu(M\setminus\Lambda_\delta)\leq
4\max\left(\tfrac{c_0}{\tilde\eps},1\right)\mu(M\setminus\pi(A_\delta\cap\Sigma^\#)).
$$

Since $h_\delta>h_0$ and $h_0$ satisfies Def.~\ref{d-entropy-full}, $\mu(\pi(\Sigma^\#))=1$.  By Lemma \ref{l.Hyp-Mark-Coding-Prop}\eqref{i.basic2}, there is a measure $\hmu\in\Proberg(\sigma)$ such that  $\mu=\pi_\ast\hmu$ and $h(\sigma,\hmu)=h(f,\mu)$. By the Poincar\'e recurrence theorem, $\hmu(\Sigma^\#)=1$.  By the identity $\mu=\hmu\circ\pi^{-1}$,
$$\mu(M\setminus\Lambda_\delta)\leq  4\max\left(\tfrac{c_0}{\tilde\eps},1\right)\hmu({\Sigma\setminus (A_\delta\cap\Sigma^\#)})=4\max\left(\tfrac{c_0}{\tilde\eps},1\right)\hmu({\Sigma}\setminus A_\delta).
$$
By the choice of $A_\delta$, $\hmu({\Sigma}\setminus A_\delta)<\tau_0$; by the choice of $\tau_0$, $\mu(M\setminus\Lambda_\delta)<\delta$, completing the proof of \eqref{e.9-point-2}. So $f|_X$ is entropy-tight.
By Prop.~\ref{prop-tight-implies-SPR}, $f|_X$ is {SPR}.
\end{proof}

\section{Construction of Hyperbolic Codings}\label{ss.construction-hyperbolic}
The papers \cite{Sarig-JAMS,Ben-Ovadia-Codings} show that every  $C^{1+}$ diffeomorphism {with at least one hyperbolic invariant measure}  has a hyperbolic Markovian coding $(\Sigma,\pi)$,  and that {for surface diffeomorphisms}, $(\Sigma,\pi)$ is  entropy-full. The paper \cite{BCS-1} shows that if we restrict to Borel homoclinic classes, then $\Sigma$  can be chosen to be irreducible. 

In this paper, we will  prove that if $f$ is SPR, then $\Sigma$ is SPR (and in the two-dimensional case and some other cases, vice-verse). 

The proof relies on the particular features of the constructions in 
\cite{Sarig-JAMS,Ben-Ovadia-Codings,BCS-1}. For the convenience of the reader,   we survey these constructions 
in this section.  
We try to keep as close as possible to the notational scheme of \cite{Ben-Ovadia-Codings}. We let  $d:=\dim M$ and assume that $d\geq 2$.
We fix  $\beta>0$ and assume that $f\in C^{1+\beta}$. 
Finally, we assume that $f$ admits some hyperbolic invariant measure, and  fix a
Borel homoclinic class $X$ and a number $\chi>0$ so small that $X$ carries a $\chi$-hyperbolic invariant measure (i.e. a measure such that all Lyapunov exponents are  outside $[-\chi,\chi]$ a.e.).

We fix some $\bar \chi\in (0,\chi)$, and another parameter $\bar{\eps}$, arbitrarily small, such that  
\begin{equation}\label{e.epsilon-bar}
{0<\bar \varepsilon<\min(\bar \chi, \chi-\bar \chi)\leq \bar \chi <\chi.}
\end{equation}
We will need to assume that $\bar{\eps}$ is ``small enough" so that the constructions in \cite{Sarig-JAMS,Ben-Ovadia-Codings,BCS-1} apply. How small depends only on $f,M,\beta,\chi,\bar{\chi}$, see \cite[\S1.6]{Sarig-JAMS}. The particular condition will not be important to us.

\subsection{Pesin Charts} Pesin introduced an atlas of charts for the ``hyperbolic part" of $M$, which bring $f$ to the form of a hyperbolic linear map plus a small perturbation. We review a variant of his construction, from \cite{Ben-Ovadia-Codings}.
\subsubsection{The Oseledets-Pesin Reduction ${C_{\bar \chi}(x)^{-1}}$}
\cite[Def. 2.10]{Ben-Ovadia-Codings} introduces a specific Borel set $\NUH^*_{\bar \chi}\subset M$, which only depends on $f,M$ and $\bar\chi$. Among its many properties, it also satisfies the following:
\begin{enumerate}[$\circ$]
\item $\NUH^*_{\bar \chi}$ has full measure for every $\bar{\chi}$-hyperbolic ergodic invariant measure.
\item Every point in $\NUH^*_{\bar \chi}\subset M$ is Lyapunov regular (i.e., satisfies the conclusion of the Oseledets Theorem for the derivative cocycle).
\item Every point in $\NUH^*_{\bar \chi}\subset M$ has some positive Lyapunov exponents, some negative Lyapunov exponents,  and
no Lyapunov exponents in $[-\bar{\chi},\bar{\chi}]$.
\end{enumerate}
($\NUH^*_{\bar \chi}$ has some other properties, which are needed  in \cite{Ben-Ovadia-Codings}, but which will not be  used explicitly below, and therefore we omit them.)

\begin{lemma}
There are a constant $\kappa=\kappa(\bar \chi,f)$ and  a collection of linear operators
$
C_{\bar \chi}(x):\R^d\to T_x M
$    $(x\in \NUH^*_{\bar \chi})$ with the following properties:
\begin{enumerate}[(1)]
\item $\|C_\chi(x)\|\leq 1$.
\item $x\mapsto C_{\bar \chi}(x)$ is a measurable map.
\item $C_{\bar \chi}(x)$ brings the derivative cocycle into a hyperbolic block-diagonal form: There are
square matrices $D_s(x), D_u(x)$  of dimensions $\dim E^s(x),\dim E^u(x)$
s.t.
\begin{equation}\label{e.OP-reduction-1}
C_{\bar \chi}^{-1}(f(x)) Df_x C_{\bar \chi}(x)=\left(\begin{array}{cc}
D_s(x) & 0\\
0 & D_u(x)
\end{array}\right),
\end{equation}
and for all unit vectors $\xi^s\in E^s(x)$, $\xi^u\in E^u(x)$,
\begin{equation}\label{e.OP-reduction-2}
\kappa^{-1}\leq \|D_s(x)\xi^s\|\leq e^{-{\bar \chi}}\ ,\ e^{{\bar \chi}}\leq \|D_u(x)\xi^u\|\leq \kappa.
\end{equation}
\end{enumerate}
\end{lemma}
\noindent
$C_\chi(x)^{-1}:T_x M\to \R^d$ is called the {\em Oseledets-Pesin Reduction} (at $x$).

\begin{proof}[Sketch of the Proof]
There are many constructions \cite{Pesin-Izvestia-1976,Katok-Hasselblatt-Book,Barreira-Pesin-Non-Uniform-Hyperbolicity-Book}, the following is from \cite[\S 2.1.2]{Ben-Ovadia-Codings}.
Suppose $x\in \NUH^*_{\bar \chi}$, and let $\<\cdot,\cdot\>_x$ denote the Riemannian metric on $T_x M$.
Let $\pi_{t}:T_x M\to E^{t}(x)$ $(t=s,u)$ denote the projections associated with the splitting $T_x M=E^u\oplus E^s$, and define
\begin{align*}
&\<\xi,\eta\>'_x:=\<\pi_s \xi,\pi_s \eta\>_{x,s}'+\<\pi_u \xi,\pi_u \eta\>_{x,u}',\text{ where }\\
&\<\xi,\eta\>'_{x,s}:=2\sum_{m=0}^\infty \<Df_x^m \xi,Df_x^m \eta\> e^{2m{\bar \chi}},\text{ for }\xi,\eta\in E^s(x),\\
&\<\xi,\eta\>'_{x,u}:=2\sum_{m=0}^\infty \<Df_x^{-m} \xi,Df_x^{-m} \eta\> e^{2m{\bar \chi}},\text{ for }\xi,\eta\in E^u(x).
\end{align*}

Let    $C_{\bar \chi}(x): \R^d\to T_x M$ be  a linear map so that
$
C_{\bar \chi}(x)[\R^{s(x)}\x\{0\}]=E^s(x)$,   $C_{\bar \chi}(x)[\{0\}\x \R^{u(x)}]=E^u(x)$, and  which carries the Euclidean metric  to $\<\cdot,\cdot\>_x'$. Then
$$
\<\xi,\eta\>_x'=\<C_{\bar \chi}(x)^{-1}\xi,C_{\bar \chi}(x)^{-1}\eta\>_{\R^2}\text{ for all }\xi,\eta\in T_x M.
$$
There are many choices for $C_{\bar \chi}(x)$ (differing by orthogonal self-maps of $E^u(x)$ and $E^s(x)$). We fix one choice, which depends  measurably on $x$.
Standard calculations show that $C_\chi(x)$ satisfies the conclusion of the lemma, see \cite[\S 2.1.2]{Ben-Ovadia-Codings}.
\end{proof}

\subsubsection{The Norm of  ${C_{\bar \chi}(x)^{-1}}$}
We collect several inequalities which relate the largeness of $\|C_{\bar \chi}(x)^{-1}\|$ to the weakness of the hyperbolicity bounds at $x$.

Let $\|\cdot\|_x$ denote the  Riemannian norm on $T_x M$ and define, for $x\in\NUH_{\bar\chi}^\ast$,
\begin{align}
S_{\bar \chi}(x)&:=\max_{\xi\in E^s(x), \|\xi\|_x=1}\sqrt{2\sum_{m=0}^\infty \|Df_x^m \xi\|^2_{f^m(x)} e^{2m{\bar \chi}}}=\|C_{\bar \chi}(x)^{-1}|_{E^s(x)}\| \label{e.S}\\
U_{\bar \chi}(x)&:=\max_{\eta\in E^u(x), \|\eta\|_x=1}\sqrt{2\sum_{m=0}^\infty \|Df_x^{-m} \eta\|^2_{f^{-m}(x)} e^{2m{\bar \chi}}}=\|C_{\bar \chi}(x)^{-1}|_{E^u(x)}\|\label{e.U}
\\
\alpha(x)&:=\measuredangle(E^u(x),E^s(x)),\text{always taken to be inside $(-\tfrac{\pi}{2},\tfrac{\pi}{2}]$.}\notag\\
\varrho_{\bar\chi}
(x)&:=\max\{S_{\bar\chi}(x),U_{\bar\chi}(x), {1}/{|\sin\alpha(x)|}\}.\notag
\end{align}
Note that $\varrho_{\bar\chi}(x)$ is large when the  hyperbolicity bounds at $x$ are weak.

\begin{lemma}\label{l.C-and-Xi} If $x$ belongs to $\NUH^*_{\bar \chi}$, then
\begin{equation}\label{e.C-and-Xi}
\varrho_{\bar\chi}(x)\leq \|C_{\bar\chi}(x)^{-1}\|\leq \sqrt{2}\,\varrho_{\bar\chi}(x)^2.
\end{equation}
\end{lemma}
\begin{proof}
Decompose a non-zero $v\in T_x M$ into  $v=\xi+\eta$, with  $\xi\in E^s(x), \eta\in E^u(x)$.
\begin{align*}
&\frac{\|C_{\bar \chi}(x)^{-1}v\|^2_{\R^d}}{\|v\|^2}=\frac{\|C_{\bar \chi}(x)^{-1}\xi\|^2_{\R^d}+\|C_{\bar \chi}(x)^{-1}\eta\|^2_{\R^d}}{\|\xi+\eta\|^2} \leq \frac{S_{\bar \chi}(x)^2\|\xi\|^2+U_{\bar \chi}(x)^2\|\eta\|^2}{\|\xi+\eta\|^2}\\
&\leq \max\{S_{\bar \chi}(x)^2,U_{\bar \chi}(x)^2\}\frac{\|\xi\|^2+\|\eta^2\|}{\|\xi+\eta\|^2}\\
&\equiv \max\{S_{\bar \chi}(x)^2,U_{\bar \chi}(x)^2\}\bigg/\frac{\|\xi\|^2+\|\eta^2\|+2\|\xi\|\|\eta\|\cos\measuredangle(\xi,\eta)}{\|\xi\|^2+\|\eta\|^2}\\
&\leq \frac{\max\{S_{\bar \chi}(x)^2,U_{\bar \chi}(x)^2\}}{1-\cos\measuredangle(\xi,\eta)}
\leq \frac{\max\{S_{\bar \chi}(x)^2,U_{\bar \chi}(x)^2\}}{1-\cos\alpha(x)} 
\leq \frac{2\max\{S_{\bar \chi}(x)^2,U_{\bar \chi}(x)^2\}}{\sin^2\alpha(x)},
\end{align*}
since $\sin^2\alpha\leq 2(1-\cos(\alpha))$ on $[-\pi/2,\pi/2]$.

Since $|ab|\leq \max(|a|,|b|)^2$, we obtain 
\begin{align*}
&\frac{\|C_{\bar \chi}(x)^{-1}v\|^2_{\R^d}}{\|v\|^2}\leq 2\left(\max\{S_{\bar \chi}(x)^2,U_{\bar \chi}(x)^2,\frac{1}{\sin^2\alpha(x)}\}\right)^2.
\end{align*}
Taking square root, we obtain the right half of \eqref{e.C-and-Xi}.

Next, $\|C_{\bar \chi}(x)^{-1}\|\geq \max\{S_{\bar \chi}(x), U_{\bar \chi}(x)\}$.
To see that $\|C_{\bar \chi}(x)^{-1}\|\geq 1/|\sin\alpha(x)|$, pick unit vectors $\xi\in E^s(x), \eta\in E^u(x)$ such that $\measuredangle(\xi,\eta)=\alpha(x)$, and take $v:=\xi-\eta$.
\begin{align*}
&\frac{\|C_{\bar \chi}(x)^{-1}v\|^2_{\R^d}}{\|v\|^2}=\frac{\|C_{\bar \chi}(x)^{-1}\xi\|^2_{\R^d}+\|C_{\bar \chi}(x)^{-1}\eta\|^2_{\R^d}}{\|\xi-\eta\|^2}\\
&\geq \frac{2}{\|\xi-\eta\|^2},\quad \text{(take $m=0$ in \eqref{e.S} and \eqref{e.U}})\\
&=\frac{2}{2-2\cos\alpha}\geq \frac{1}{\sin^2\alpha},\quad \text{since $\sin^2\alpha\geq (1-\cos(\alpha))$ on $[{-\pi/2},\pi/2]$.}
\qedhere
\end{align*}
\end{proof}
In \S\ref{ss.construction-bornological}, we will need  the following bound for $\|C_{\bar{\chi}}(x)^{-1}\|$ on $(\chi,\eps)$-Pesin blocks:
\begin{lemma}\label{l.K-C}
There exist $A(f,\bar \chi),B(f, \bar \chi,{\bar \eps})\in (0,1)$,
such that the following holds,  for every $\varepsilon\in (0,\bar\eps]$ and $K\geq 1$. Suppose  $x\in\NUH^*_{\bar \chi}$  satisfies the Pesin bounds~\eqref{e.def-pesin}
with the numbers $\chi,\varepsilon,K$. Then 
$K\geq B\|C_{\bar \chi}(x)^{-1}\|^{A}.$
\end{lemma}
\begin{proof}
Fix unit vectors  $\eta\in E^u(x)$ and $\xi\in E^s(x)$ such that $\measuredangle(\xi,\eta)=\alpha(x)$. Then $$
\|\eta-\xi\|=2|\sin\tfrac{\alpha}{2}|=\frac{|\sin\alpha|}{\cos(\alpha/2)}\leq \sqrt{2}|\sin\alpha|\ \ \text{ (recall that }|\alpha|\leq \frac{\pi}{2}).
$$
By~\eqref{e.def-pesin},  $\|Df^m_x\xi\|\leq K e^{-m\chi}\|\xi\|$ and
$\|\eta\|= \|Df_{f^m(x)}^{-m}Df_{x}^m\eta\|\leq K e^{{\eps} m}e^{-m\chi}\|Df_{x}^m\eta\|$ for  $m\geq 0$. So
$
\|Df_{x}^m\eta\|\geq K^{-1}e^{m({\chi}-{\eps})}\|\eta\|
$, whence
\begin{equation*}
 \|\eta-\xi\|\geq \frac{\|Df^m_x \eta-Df^m_x\xi\|}{\|Df^m_x\|}\geq \frac{K^{-1}e^{m({\chi}-{\eps})}-K e^{-m{\chi}}}{\Lip(f)^m}.
\end{equation*}
Let $m:=\lceil \log (4K)/{\bar \chi}\rceil$. By \eqref{e.epsilon-bar},
 $K^2e^{-m(2\chi-\varepsilon)}\leq K^2e^{-2m{\bar \chi}}\leq \tfrac 1 2$, whence
$K e^{-m{\chi}}\leq  \tfrac{1}{2}K^{-1}e^{m({\chi}-{\eps})}$, and therefore
\begin{equation}\label{e.m-xi}
\|\eta-\xi\|\geq \frac{K^{-1}e^{m({\chi}-{\eps})}}{2\Lip(f)^m}.
\end{equation}
We saw above that $|\sin\alpha|\geq (1/\sqrt{2})\|\eta-\xi\|$.
By the choice of $m$ and \eqref{e.m-xi},
\begin{align*}
|\sin\alpha|\geq \const K^{-\log\Lip(f)/{\bar \chi}},
\end{align*}
where the constant only depends on $\Lip(f)$ and ${\bar \chi}$.
Hence, for some $c_1,c_2$ which only depend on $\bar\chi$ and $\Lip(f)$,
$
K\geq c_1\left(\frac{1}{|\sin\alpha|}\right)^{c_2}.
$

Next, choose some arbitrary unit vector $\xi\in E^s(x)$.
By \eqref{e.pesin} and \eqref{e.epsilon-bar},
 $$
 \|Df_x^m\xi\|^2 e^{2m\bar\chi}\leq K^2e^{-2m\chi}e^{2m\bar\chi}\leq K^2 e^{-2m(\chi-\bar{\chi})}\leq K^2 e^{-2m\bar{\eps}}.
 $$ Summing over $m$, we deduce that
\begin{align*}
&S_{\bar \chi}(x)\leq \left(2\sum_{m=0}^\infty K^2 e^{-2m\bar \eps} \right)^{1/2}\leq \frac{K\sqrt{2}}{\sqrt{1-e^{-2{\bar \eps}}}}\leq K\sqrt{2} e^{\bar \eps}/{{\sqrt{\bar \eps}}}\leq \frac{K\sqrt{2}\,e^{\lambda_{\max}(f)}}{{\bar \eps}}.
\end{align*}
We now decrease   $c_1$ so that $c_1<\frac{{\sqrt{\bar \eps}}}{{\sqrt{2}e^{\lambda_{\max}(f)}}}$, and  obtain
$
K\geq c_1 S_{\bar \chi}(x) \geq c_1 S_{\bar \chi}(x)^{c_2}.
$
Similarly, $
K\geq c_1 U_{\bar \chi}(x)^{c_2}.
$
So $K\geq c_1\varrho_{\bar\xi}(x)^{c_2}$.
The lemma  follows from \eqref{e.C-and-Xi}.
\end{proof}

\subsubsection{Pesin Charts and the Parameter ${Q_{\bar \eps}(x)}$.}
Recall that $d=\dim M$.
In the previous section we constructed linear changes of coordinates $C_{\bar\chi}(x)$  which transform $Df_x: T_x M\to T_{f(x)}M$ into a linear hyperbolic map on $\R^d$, see \eqref{e.OP-reduction-1}. We will now construct non-linear changes of coordinates $\Psi_x$ which transform $f$ on a neighborhood of $x$ into a non-linear perturbation of a linear hyperbolic map on $\R^d$.

Let $\exp_x: T_x M\to M$ denote the exponential map defined by the Riemannian structure. The {\em Pesin chart} of {\em size} $p$ and {\em center} $x$  is the map
$$
\Psi^p_x:[-p,p]^d\to M\ , \ \Psi^p_x(v):=\exp_x[C_{\bar \chi}(x)v].
$$
We have $\|C_{\bar \chi}(x)\|\leq 1$. So
if $p$ is smaller than the injectivity radius of $M$, then $\Psi_x$ is a diffeomorphism onto its image, and  $\Psi_x$ defines a system of  coordinates on a neighborhood of $x$. Pesin showed that if  $p$ is small enough, then ``$f$ in coordinates"
$
\Psi_{f(x)}^{-1}\circ f \circ \Psi_x
$
is close to a uniformly hyperbolic linear map $\R^d\to\R^d$, see~\cite{Pesin-Izvestia-1976},\cite{Barreira-Pesin-Non-Uniform-Hyperbolicity-Book}. How small should $p$ be depends on $x$. We will work with an explicit (but non-canonical) threshold $Q_{\bar{\epsilon}}(x)$, which we proceed to define.

Let $I_{\bar \eps}:=\{e^{-\frac{1}{3}\ell{\bar \eps}}:\ell\in\N\}$. For any number $0<Q<1$, we set $$\lfloor Q\rfloor_{{\bar \eps}}:=\max\{0<\tau\leq Q:\tau=e^{-\frac{1}{3}\ell{\bar \eps}}, \ell\in\NN\}\in I_{\bar \eps}.$$
Recall that $\beta>0$ is the H\"older exponent of $Df$.
Given $x$ in $\NUH^*_{\bar \chi}$, let \begin{align}
Q_{\bar \eps}(x):=\left\lfloor\frac{{\bar \eps}^{90/\beta}}{3^{6/\beta}}
\|C_{\bar \chi}(x)^{-1}\|^{-\frac{48}{\beta}}\right\rfloor_{\bar \eps}.\label{e.Q}
\end{align}
The following can be shown as in \cite[Thm~5.3.1]{Barreira-Pesin-Non-Uniform-Hyperbolicity-Book}, \cite[Thm~ 2.13]{Ben-Ovadia-Codings}:
\begin{lemma}\label{l.f-in-coord}
Suppose $\bar{\eps}=\bar{\eps}(f,M,\beta,{\chi},\bar{\chi})$ is small enough.
For any point $x\in \NUH^*_{\bar \chi}$ and for all $0<p\leq Q_{\bar \eps}(x)$:
\begin{enumerate}[(1)]
\item $\Psi_x:[-p,p]^d\to M$ is a diffeomorphism onto its image.
\item $f_x:=\Psi_{f(x)}^{-1}\circ f\circ \Psi_x$ is well-defined and injective on $[-p,p]^d$.
\item $f_x(0)=0$, the derivative matrix of $f_x$ at zero $(Df_x)_0$ equals the block matrix on the right-hand side of \eqref{e.OP-reduction-1}, and
    $
    \|f_x-(Df_x)_0\|_{C^{1+\frac{\beta}{2}}}\leq {\bar \eps}\text{ on }[-p,p]^d
    $.
\item The symmetric statements hold for $f_x^{-1}:=\Psi_x^{-1}\circ f^{-1}\circ\Psi_{f(x)}$.
\end{enumerate}
\end{lemma}
In particular, ``$f$ in coordinates," $f_x$,  is uniformly  close in $C^{1+\frac{\beta}{2}}$ to the hyperbolic linear map $(Df_x)_0$. By \eqref{e.OP-reduction-2}, $(Df_x)_0$ is  hyperbolic {\em uniformly} in $x$.

\medskip
\noindent
{\em Henceforth we will only consider Pesin charts $\Psi_x^p$ such that $0<p\leq Q_{\bar \eps}(x)$. }

\subsubsection{Overlap Conditions}\label{ss.overlap}
If $\Psi_x$ and $\Psi_{f(x)}$ satisfy the conclusion of Lemma \ref{l.f-in-coord}, then so do all sufficiently small perturbations of $\Psi_x$ and $\Psi_{f(x)}$, albeit on a smaller neighborhood of $0$, and in a slightly weaker topology.

This statement can be made quantitative, and uniform. Building on \cite{Sarig-JAMS}, \cite{Ben-Ovadia-Codings} specifies what it means for $\Psi_y$ to {\em $\bar{\eps}$-overlap} $\Psi_x$, and then shows that  {\em if $\Psi_y^{q}$ ${\bar \eps}$-overlaps $\Psi_{f(x)}^p$, then
$\Psi_y^{-1}\circ f\circ \Psi_x$ is ``uniformly close" to the uniformly hyperbolic linear map $(Df_x)_0$ on $[-p,p]^2$, in the $C^{1+\frac{\beta}{3}}$ topology}.  We omit the precise statements, which can be found in \cite[Def.~2.18, Prop.~2.21]{Ben-Ovadia-Codings}.

We denote the ${\bar \eps}$-overlap condition by
$
\Psi_x^p\overset{{\bar \eps}}{\approx}\Psi_y^{q}.
$

\subsection{Chains and Shadowing}\label{ss.chains-and-shadowing}
The construction of hyperbolic codings uses the shadowing theory for general non-uniformly hyperbolic maps in \cite{Sarig-JAMS}. In this theory,  pseudo-orbits
are replaced by objects called ``$\bar{\eps}$-chains," and defined below.

\subsubsection{Double Charts.}
A {\em double chart} is a pair $(\Psi_x^{p^u}, \Psi_x^{p^s})$ of two concentric Pesin charts
where $x\in \NUH^*_{\bar \chi}$
and $p^s,p^u$ belong to $(0,Q_{\bar \eps}(x)]\cap I_{\bar \eps}$.
We will denote double charts  by the formal symbol $\Psi_x^{p^u,p^s}$.

\subsubsection{Chains.}\label{ss.chains}
We denote $p\wedge p':=\min(p,p')$.
Given  double charts $\Psi_x^{p^u,p^s}$ and $\Psi_y^{q^u,q^s}$, we write $\Psi_x^{p^u,p^s}\to \Psi_y^{q^u,q^s}$, if the following conditions hold:
\begin{enumerate}[\quad $\circ$]
\item
$
q^u=\min\{e^{\bar \eps} p^u,Q_{\bar \eps}(y)\}$ and $p^s=\min\{e^{\bar \eps} q^s, Q_{\bar \eps}(x)\},
$ see \eqref{e.Q},
\item $\Psi_{x}^{p^u\wedge p^s}\overset{{\bar \eps}}{\approx}\Psi_{f^{-1}(y)}^{p^u\wedge p^s}$ and $\Psi_{y}^{q^u\wedge q^s}\overset{{\bar \eps}}{\approx}\Psi_{f(x)}^{q^u\wedge q^s}$, see \S\ref{ss.overlap},
\item $\dim E^s(x)=\dim E^s(y)$ and $\dim E^u(x)=\dim E^u(y)$.
\end{enumerate}
The following consequence will be useful to us later \cite[Lem~4.4]{Sarig-JAMS}:
\begin{equation}\label{tralala}
\Psi_x^{p^u,p^s}\to\Psi_y^{q^u,q^s} \Rightarrow \frac{q^u\wedge q^s}{p^u\wedge p^s}\in [e^{-{\bar \eps}},e^{{\bar \eps}}].
\end{equation}

\begin{definition}
An {\em ${\bar \eps}$-chain} is a sequence of double charts $\un \Psi=(\Psi_{x_n}^{p^u_n, p^s_n})_{n\in\Z}$ such that  $\Psi_{x_n}^{p^u_n, p^s_n}\to \Psi_{x_{n+1}}^{p^u_{n+1}, p^s_{n+1}}$ for all $n$.
\end{definition}

We use ${\bar \eps}$-chains as substitutes to pseudo-orbits.

\subsubsection{The Shadowing Theorem}
We say that an $\bar{\eps}$-chain $(\Psi_{x_i}^{p^u_i,p^s_i})_{i\in\Z}$  {\em shadows} the orbit of $z\in M$, if for all $i\in\Z$,
$
f^i(z)\in \Psi_{x_i}\bigl([-\eta_i,\eta_i]^2\bigr),\text{ where }\eta_i:=p^u_i\wedge p^s_i.
$

\begin{lemma}\label{l.shadowing}
Suppose ${\bar \eps}={\bar \eps}(f,\beta,M,{\chi},\bar \chi)$ is small enough.
Then every $\bar{\eps}$-chain shadows the orbit of a unique point $z:=\mathfrak{sh}(\un{\Psi})$,
and  $\fs:\{\text{${\bar \eps}$-chains}\}\to M$ satisfies:
\begin{enumerate}[(i)]
\item {\em Equivariance:} $\fs\circ\sigma=f\circ\fs$ ($\sigma$ denotes the left shift).
\item  {\em H\"older Continuity:} There are $\theta\in (0,1)$ and $C>0$ such that if $\un{\Psi},\un{\Psi}'$ are two ${\bar \eps}$-chains with $\Psi_n=\Psi'_n$ for all $|n|\leq m$, then $d(\fs(\un{\Psi}),\fs(\un{\Psi'}))<C\theta^m$.
\item There exists an invariant measurable subset $\NUH^\#_{\bar \chi}\subset \NUH^*_{\bar \chi}$, with full measure for every
$\bar{\chi}$-hyperbolic invariant measure, such that every $x\in \NUH^\#_{\bar \chi}$ is shadowed by some $\bar{\eps}$-chain.
\item[(iv)] {\em Hyperbolicity:} Suppose  $\un \Psi$ is an $\bar\eps$-chain, and $z=\mathfrak{sh}(\un{\Psi})$. Then we can decompose $T_z M= E^s(\un \Psi)\oplus E^u(\un \Psi)$ s.t.
for every non zero  $v^s\in E^s(\un \Psi)$, $v^u\in E^u(\un \Psi)$, $$
\limsup_{n\to\pm\infty}\tfrac{1}{n}\log\|Df^n_z.v^s\|\leq-\tfrac{2}{3}\bar \chi \text{ and }
\limsup_{n\to\pm\infty}\tfrac{1}{n}\log\|Df^{-n}_z.v^u\|\leq -\tfrac{2}{3}\bar \chi.
$$
\item[(v)] The  functions $\un{\Psi}\mapsto E^s(\un{\Psi}), E^u(\un{\Psi})$ in (iv) are H\"older continuous with respect to the metric \eqref{e.metric-on-sigma} on the space of sequences, and the Grassmannian metric.
\end{enumerate}
\end{lemma}

\begin{proof}
The lemma follows from Thm~3.13 and its proof in \cite{Ben-Ovadia-Codings}, together with
\cite[Prop.~2.30, Prop.~3.5, Prop.~3.12, and Prop. 4.4(2)]{Ben-Ovadia-Codings}. (The 2D case is in \cite{Sarig-JAMS}.)
\end{proof}

\begin{remark}
By Lemma~\ref{l.splitting}, if  $z$ belongs to a Pesin block with parameters $\bar\chi, \bar \varepsilon$,
the decomposition in (iv) only depends on $z$ and not on the sequence $\un \Psi$ that shadows $z$,
and it coincides with the splitting of Definition~\ref{d.pesin}.
\end{remark}

\subsubsection{Discretization.}\label{ss.discretization}
The set of  $\bar{\eps}$-chains equals the set of paths on  a directed graph,
whose vertices are  double charts, and whose edges are defined in \S\ref{ss.chains}. This graph is uncountable.

To obtain a countable graph, we  replace the uncountable set of all double charts by a  countable subset  $\mathfs V$, with the following properties:
\begin{enumerate}[\quad (a)]
\item {\em Sufficiency\/:}
For the  measurable set  $\NUH^\#_{\bar \chi}$ in Lemma \ref{l.shadowing}(iii) (which has full measure for every
${\chi}$-hyperbolic ergodic measure), we have the following:

\smallskip
\noindent
Every $x\in \NUH^\#_{\bar \chi}$ is shadowed by some $\bar{\eps}$-chain  $\un{\Psi}=(\Psi_{x_n}^{p^u_n,p^s_n})_{n\in\Z}$ such that
\begin{enumerate}[(i)]
\item $\Psi_{x_n}^{p^u_n,p^s_n}\in\mathfs V$ for all $n$;
\item $(\Psi_{x_n}^{p^u_n,p^s_n})_{n\geq 0}$ and
$(\Psi_{x_n}^{p^u_n,p^s_n})_{n\leq 0}$ {both} contain constant subsequences;
\item $\forall n\in \ZZ$, $Q_{\bar \eps}(f^n(x))/Q_{\bar \eps}(x_n)\in [e^{-{\bar \eps}/3},e^{{\bar \eps}/3}]$.
\end{enumerate}

\smallskip
\item {\em Discreteness:\/} For every $t>0$ there are only finitely many $\Psi_x^{p^u,p^s}\in\mathfs V$ such that $p^u\wedge p^s>t$.
\end{enumerate}

\noindent
The construction of $\mathfs V$ uses a coarse-graining procedure, which is described in \cite[\S 2.2.3, \S 2.3.1, Thm~3.13]{Ben-Ovadia-Codings} and  \cite[\S 3.3, Prop.~4.5, Theorem 4.16]{Sarig-JAMS}.

Let $\mathfs H$ denote the graph whose set of vertices is $\mathfs V$, and whose edges are defined in  \S\ref{ss.chains}. This is a countable graph, and $\Sigma(\mathfs H)$ equals the set of $\bar{\eps}$-chains in $\mathfs V^\Z$.

The sufficiency of $\mathfs V$ implies that $\fs(\Sigma^\#(\mathfs H))$ has full measure for all ${\chi}$-hyperbolic invariant measures. The discreteness of $\mathfs V$ and \eqref{tralala} imply that $\mathfs H$ is locally finite.

\subsubsection{The Shadowing Inverse Problem} Many different ${\bar \eps}$-chains can shadow the same orbit, but all the  chains in $\Sigma^\#(\mathfs H)$ which project to a given point $x$ share some common features:
\begin{lemma}\label{l.inverse} If $(\Psi_{x_n}^{p^u_n,p^s_n})_{n\in\Z},(\Psi_{y_n}^{q^u_n,q^s_n})_{n\in\Z} \in\Sigma^\#(\mathfs H)$
 shadow the same point, then for all $n\in\Z$,
\begin{enumerate}[(i)]
\item $p^u_n/q^u_n,\; p^s_n/q^s_n\in [e^{-\sqrt[3]{{\bar \eps}}},e^{\sqrt[3]{{\bar \eps}}}]$,
\item There are linear maps $\Xi_n: T_{x_n}M\to T_{y_n} M$ with $\|\Xi_n\|, \|\Xi_n^{-1}\|\leq e^{{\bar \eps}}$ such that
$$
\frac{\|C_{\bar \chi}(x_n)^{-1}\xi\|}{\|C_{\bar \chi}(y_n)^{-1}\Xi_n\xi\|}\in [e^{-4\sqrt{{\bar \eps}}},e^{4\sqrt{{\bar \eps}}}]
\quad \text{ for all non-zero $\xi\in T_{x_n} M$.}
$$
\item $|\log[Q_{\bar \eps}(x_n)/Q_{\bar \eps}(y_n)]|\leq \tfrac{200\sqrt {\bar \eps}}{\beta}$.
\end{enumerate}
\end{lemma}
\noindent
For the proof of parts (i) and (ii), see  \cite[Prop.~4.8 and  Lemma~4.12]{Ben-Ovadia-Codings}. Part (iii) follows from (ii) and  \eqref{e.Q}. (For the simpler 2D case, see \cite[Part 2]{Sarig-JAMS}.)

\subsection{Hyperbolic Codings of Diffeomorphisms}\label{s.G-hat}
$\fs:\Sigma(\mathfs H)\to M$ is H\"older continuous, and it satisfies properties \eqref{i.def1} and \eqref{i.def3} in Def.~\ref{d.Hyperbolic-Coding}. But it  is not a hyperbolic coding, because  $\fs:\Sigma^\#(\mathfs G)\to M$ is not necessarily finite-to-one.
To get a finite-to-one coding, we project the natural Markov partition of
$\Sigma^\#(\mathfs H)$ to $M$, and refine it to a Markov partition for $f$. Specifically, let
\begin{equation}\label{e.Z}
\mathfs Z:=\{Z(v):v\in \mathfs V\}\ \ ,\ \ Z(v):=\{\fs(\un{\Psi}): \un{\Psi}\in\Sigma^\#(\mathfs G)\ ,  \Psi_0=v\}.
\end{equation}
Since $\mathfs V$ is sufficient, $\mathfs Z$ covers $\NUH_{\bar\chi}^\#$, a set of full measure for all ergodic ${\chi}$-hyperbolic invariant measures. The elements of $\mathfs Z$ may overlap, but by Lemma \ref{l.inverse}, the overlaps have the following local finiteness property \cite[\S 5.1.1]{Ben-Ovadia-Codings}:
\begin{lemma}
$\{Z'\in\mathfs Z:Z'\cap Z\neq \emptyset\}$ is finite for every $Z\in\mathfs Z$.
\end{lemma}
This crucial finiteness property makes it possible to refine $\mathfs Z$ to a {countable partition}. Moreover, a  procedure similar to the one used by Bowen in \cite{Bowen-LNM}, generates a countable refinement which is a {\em Markov partition}. For details, see \cite[\S11.1]{Sarig-JAMS} and \cite[\S6]{Ben-Ovadia-Codings}.
We denote this partition by $\mathfs R$.

Now we build  a new countable locally finite graph $\wh{\mathfs G}$ with set of vertices $\mathfs R$, and edges $R\to S$ whenever $R\cap f^{-1}(S)\neq \emptyset$.

\begin{lemma}\label{l.coding-final}
There is a map $\hpi:\Sigma(\wh{\mathfs G})\to M$ with  the following properties:
\begin{enumerate}[(i)]
\item $\hpi\circ\sigma=f\circ\hpi$, the map $\hpi$ is H\"older continuous, $\mu[\hpi(\Sigma^\#(\wh{\mathfs G}))]=1$ for all ${\chi}$-hyperbolic ergodic measures $\mu$, 
and  $\hpi:\Sigma^\#(\wh{\mathfs G})\to M$ is finite-to-one.
\item For every $R\in\mathfs R, Z\in\mathfs Z$, the sets $\{Z'\in\mathfs Z:Z'\cap R\neq\emptyset\}$ and $\{R'\in\mathfs R:R'\cap Z\neq\emptyset\}$ are finite.
Moreover either $R\subset Z$ or $R\cap Z=\emptyset$.
\item $\hpi(\Sigma^\#(\wh{\mathfs G}))=\fs(\Sigma^\#(\mathfs H))$.  Moreover, for every $\un{R}\in\Sigma^\#(\widehat{\mathfs G})$, there exists an ${\bar \eps}$-chain $\un{\Psi}\in\Sigma^\#(\mathfs H)$ such that $\wh{\pi}(\un{R})=\fs(\un{\Psi})$, and $R_i\subset Z(\Psi_i)$ for all $i$.
\item The stable and unstable spaces of $\hpi(\un{x})$ depend H\"older continuously on $\un{x}$.
\end{enumerate}
\end{lemma}
\noindent
The Markov partition $\mathfs R$ is constructed in \cite[\S 6]{Ben-Ovadia-Codings}, where one can find the proofs of (i) and (ii).  Parts (iii)
and (iv) are explained in footnote 12 and Prop. 6.1 there.

\begin{corollary}\label{c.coro-epsilon(M,f,chi,beta)}
Let $f$ be a $C^{1+\beta}$ diffeomorphism of a closed manifold $M$ which admits a $\chi$-hyperbolic invariant measure for some $\chi>0$. 
For every $\bar{\chi}\in (0,\chi)$ sufficiently small,  
for every $\bar{\eps}$ small enough,
the following holds:
\begin{enumerate}[(1)]
\item $(\Sigma(\wh{\mathfs G}),\hpi)$ {built above (with this choice of $\bar\chi,\bar \epsilon$)} is a hyperbolic coding of $f$ in $M$;
\item If $\mu$ is a $\chi$-hyperbolic $f$-invariant measure on $M$, then $\mu(\wh{\pi}(\Sigma^\#(\wh{\mathfs G})))=1$.
\end{enumerate}
\end{corollary}

\begin{proof}
We have already seen that $\mathfs H$ is locally finite.
By Lemma \ref{l.coding-final}(ii) and (iii),
every vertex in the graph $\widehat{\mathfs G}$ has finite incoming and outgoing degrees. By Lemma \ref{l.local-compactness-crit}, $\Sigma(\wh{\mathfs G})$ is locally compact. The remaining properties of $(\Sigma(\wh{\mathfs G}),\hpi)$ follow from Lemmas \ref{l.coding-final} and \ref{l.shadowing}(iv).
\end{proof}
\subsection{Irreducible Hyperbolic Codings of Borel Homoclinic Classes}
We now localize to a  Borel homoclinic class $X$, and show that $f|_X$ has a hyperbolic coding with an irreducible $\Sigma$ (see \S\ref{s.spectral-decomp}).
The key  is the following result   from \cite{BCS-1}.

\begin{lemma}\label{l.irreducible}
{Given $\chi>0$,}
let $(\Sigma(\wh{\mathfs G}),\hpi)$ be the coding  in Lemma \ref{l.coding-final}.
If $X$ is a Borel homoclinic class of $f$, then
there exists a
maximal connected component $\mathfs G\subset\wh{\mathfs G}$ such that:\begin{enumerate}[(i)]
\item  $\Sigma({\mathfs G})$ is irreducible and  $\sigma:\Sigma(\mathfs G)\to\Sigma(\mathfs G)$ is topologically transitive;
\item Every $\chi$-hyperbolic ergodic measure $\mu$ on  $X$ satisfies $\mu[\hpi(\Sigma^\#(\mathfs G))]=1$;
\item Every $\sigma$-invariant measure $\hmu$ on $\Sigma(\mathfs G)$ satisfies  $(\hpi_*\hmu)(X)=1$.
\end{enumerate}
\end{lemma}
\noindent
See Theorem 3.1 and Section 3.4 in \cite{BCS-1}.
(The proofs in \cite{BCS-1} were written in the two dimensional case, but as noted in \S 1.6 there, they work verbatim in higher dimension. A proof in the higher-dimensional case, which also applies to infinite conservative invariant measures can be found in \cite{Ben-Ovadia-coded-set}.)

\begin{corollary}\label{c.coro-transitive-full-hyp-coding}
Let $X$ be a Borel homoclinic class  of a $C^{1+\beta}$ diffeomorphism $f$ with some $\chi$-hyperbolic invariant measure, for some $\chi>0$.
For every $\bar{\chi}\in (0,\chi)$ small enough, for every $\bar{\eps}>0$ small enough, we have:
\begin{enumerate}[(1)]
\item   $\wh{\mathfs G}$ from \S\ref{s.G-hat} has a maximal connected component $\mathfs G$ such that
$(\Sigma(\mathfs G),\hpi|_{\Sigma(\mathfs G)})$ is a hyperbolic coding of $f$ in $X$. In particular, $\Sigma(\mathfs G)$ is irreducible.
\item If $\mu$ is a $\chi$-hyperbolic $f$-invariant measure on $X$, then $\mu(\wh{\pi}(\Sigma^\#(\mathfs G)))=1$.
\item If $(\Sigma(\wh{\mathfs G}),\hpi)$ is $\chi$-bornological, then  $(\Sigma(\mathfs G),\hpi|_{\Sigma(\mathfs G)})$ is $\chi$-bornological.
\end{enumerate}
\end{corollary}

\begin{proof}
Parts (1) and (2) follow directly from Lemma \ref{l.irreducible} and
 Cor~\ref{c.coro-epsilon(M,f,chi,beta)}. Part (3) is proved as follows.
Suppose $(\Sigma(\wh{\mathfs G}),\hpi)$ is $\chi$-bornological (Def.~\ref{d.bornological}).
{Note that}   $(\Sigma({\mathfs G}),\hpi|_{\Sigma({\mathfs G})})$ satisfies Item~\eqref{i.born2} of the bornological property. To see that it also satisfies Item~\eqref{i.born1},
fix $\eps$ so that for every $(\chi,\eps)$-Pesin block $\Lambda$, there is  a finite family of cylinders $A_i$ in $\Sigma(\wh{\mathfs G})$ so that $\hmu(\hpi^{-1}(\Lambda)\setminus \bigcup_i A_i)=0$ for all $\hmu\in\mathbb P(\Sigma(\wh{\mathfs G}))$. Given $i$, let
$B_i:=A_i\cap\Sigma(\mathfs G)$. This  is either empty, or it is a cylinder in  $\Sigma({\mathfs G})$.  Clearly,  $\hmu(\hpi^{-1}(\Lambda)\setminus \bigcup_i B_i)=0$ for all $\hmu\in\mathbb P(\Sigma({\mathfs G}))$.
\end{proof}

\section{Existence of SPR Codings}
\label{ss.construction-bornological}

In this section we state and prove our main results  on the symbolic codings of SPR diffeomorphisms {(Theorem~\ref{t.SPR-coding})}: {\em An SPR diffeomorphism can be coded by an SPR Markov shift, and an SPR Borel homoclinic class can be coded by an irreducible Markov shift.} Conversely, in dimension two and some other cases,   any diffeomorphism with a ``good`` SPR coding must itself be SPR.

The papers \cite{Sarig-JAMS,Ben-Ovadia-Codings,BCS-1} provide hyperbolic codings with the required irreducibility properties. We will show that these codings are bornological, and invoke Prop.~\ref{p.lift-SPR} and Prop.~\ref{p.project-SPR}. 
Unlike the material in the previous section, this is new.

\subsection{Proof of the Bornological Property}
Suppose $f$ is a $C^{1+\beta}$ diffeomorphism,  with a Borel homoclinic class $X$ such that $h_{\Bor}(f|_X)>0$.
{We fix parameters $\chi>\bar \chi>0$}
and we let $\bar{\eps}>0$ be some number as in \eqref{e.epsilon-bar}, which is so small  that all the results of \S\ref{ss.construction-hyperbolic} are valid.
We also need the number $A(f, \bar \chi)$ from Lemma~\ref{l.K-C}, and the hyperbolic coding $(\Sigma(\wh{\mathfs G}),\hpi)$ from Cor \ref{c.coro-epsilon(M,f,chi,beta)}.
Fix
\begin{equation}\label{def.gamma}
\gamma:=\tfrac{\beta}{48}A(f, \bar \chi)\text{ and }
\varepsilon \in {(0, \gamma \bar\eps)}.
\end{equation}

\begin{definition}\label{d.Optimal-Pesin-Constant}
The
\emph{optimal $(\chi,\eps)$-Pesin bound}  at $x$  is  the infimum $K_\ast(\chi,\eps;x)$ of the numbers $K>0$ for which there are splittings $T_{f^n(x)} M=E^u(f^n(x))\oplus E^s(f^n(x))$ $(n\in\Z)$ so that
$$
\forall n\in\Z\, \forall k\geq 0\, \max(\|Df^k|_{E^s(f^n(x))}\|, \|Df^{-k}|_{E^u(f^n(x))}\|)\leq K e^{-\chi k+\varepsilon |n|}.
$$
{By convention, if there is no such splitting, then $K_*(\chi,\eps; x):=+\infty$.}
\end{definition}
\noindent
Henceforth we work with the parameters $\chi$ and $\eps$ fixed above, and we let $K_*(x):=K_*(\chi,\eps; x)$. The plan is to check
Def.~\ref{d.bornological}(a) with the parameters $\chi$ and $\eps$.

By  Lemma~\ref{l.splitting},  {there is at most one} splitting  $T_{f^n(x)}M=E^s(f^n(x))\oplus E^u(f^n(x))$ {as above}. {It follows that if the infimum $K_\ast(x)$ is finite, then it is in fact a minimum}. It is also clear that  the level sets $\{x:K_\ast(x)\leq K\}$ are all $(\chi,\eps)$-Pesin blocks, and that every  $(\chi,\eps)$-Pesin block is a subset of a set like that.
Finally, by the optimality of $K_*(x)$, this function is {\em $\eps$-tempered}:
\begin{equation}\label{e.temperability}
e^{-\eps}K_\ast(x)\leq K_\ast(f(x))\leq e^{\eps}K_\ast(x).
\end{equation}

The following bound for $K_\ast(x)$ uses the shadowing theory of \S\ref{ss.chains-and-shadowing}:
\begin{lemma}\label{l.bound-Pesin-bound}
There exists ${c:=}c(f,\beta,\bar\chi,\bar \eps)\in (0,1)$ such that
for any $x$ in $\NUH^\#_{\bar \chi}$ and $\un \Psi=(\Psi_{x_i}^{p^u_i,p^s_i})_{i\in\Z}$ in $\Sigma^\#(\mathfs H)$
such that  $x=\fs(\un\Psi)$,
$
K_\ast(x)\geq c\cdot(p^u_0\wedge p^s_0)^{-\gamma}.
$
\end{lemma}
\begin{proof}
By Lemma~\ref{l.K-C},
$K_*(x)\geq B\|C_{\bar\chi}(x)^{-1}\|^{A}$.

Recall the definition of $Q_{\bar{\eps}}(x)$ in
 \eqref{e.Q}.
By \S\ref{ss.discretization} (a)(iii) and Lemma~\ref{l.inverse}(iii),
$Q_{\bar \eps}(x)\geq e^{-\big({\bar \eps}/3+\tfrac{200\sqrt {\bar \eps}}{\beta}\big)} Q_{\bar \eps}(x_0).$
Looking at the definition of $Q_{\bar{\eps}}(x)$ again,  and recalling the choice of $\gamma$ in \eqref{def.gamma}, we conclude that for some $c(f,\beta,\bar\chi,\bar \varepsilon)\in (0,1)$,
\begin{equation}\label{e.initial-step}
K_*(x)\geq c\cdot Q_{\bar \eps}(x_0)^{-\gamma}.
\end{equation}

\begin{claim*}
There exist $i\leq 0\leq j$ such that $p^u_{i}=Q_{{\bar \eps}}(x_{i})$ and $p^s_j=Q_{\bar \eps}(x_j)$.
\end{claim*}
\smallskip
\noindent
{\em Proof of the Claim.\/}
Let $\Psi_i:=\Psi_{x_i}^{p^u_i,p^s_i}$. Since $\un \Psi\in \Sigma^\#(\mathfs H)$,  there exists $i_k\downarrow -\infty$ such that $\Psi_{i_k}$ is constant, say equal to $\Psi_z^{r^u,r^s}$ (see \eqref{e.regular-part}).

Assume by way of contradiction that $p_i^u\neq Q_{\bar \eps}(x_i)$ for all $i\leq 0$. Then $p_i^u<Q_{\bar \eps}(x_i)$ for all $i\leq 0$, see \S\ref{ss.chains}.
 Since $\un{\Psi}$ is a path on $\mathfs H$, $\un{\Psi}$ is an $\bar{\eps}$-chain, whence
$$
p^u_i=\min\{e^{{\bar \eps}}p^u_{i-1},Q_{\bar \eps}(x_i)\}=e^{\bar \eps} p^u_{i-1}\text{ for all }i\leq 0,
$$
see \S\ref{ss.chains}.
 But this implies that $p^u_0=e^{{\bar \eps}|i_k|}p^u_{i_k}=e^{{\bar \eps}|i_k|}r^u\xrightarrow[k\to\infty]{}+\infty$, which is absurd, since $p^u_0\leq Q_{\bar \eps}(x_0)<1$. The {existence  of $j$ is proved in a similar manner.} 

\medskip
Let $k:=\max\{i\leq 0: p_i^u=Q_{\bar \eps}(x_i)\}$. Applying \eqref{e.initial-step} to $f^k(x)$,  we obtain
$$
K_\ast(f^{k}(x))\geq c\cdot Q_{\bar \eps}(x_{k})^{-\gamma}
=
c \cdot (p^u_{k})^{-\gamma}.
$$
By the  maximality of $k$ and the definition of $\bar{\eps}$-chains in \S\ref{ss.chains},
$$
\forall k+1\leq n\leq 0,\
p^u_n= \min\{e^{{\bar \eps}}p^u_{n-1},Q_{\bar \eps}(x_n)\}= e^{{\bar \eps}}p^u_{n-1}.$$
So $p^u_0=e^{-k{\bar \eps}}p^u_{k}$. Next, by~\eqref{e.temperability}, and since $\varepsilon<\gamma\bar \varepsilon$,
 \begin{align*}
K_\ast(x)\geq e^{-|k|\eps}K_{\ast}(f^{k}(x))\geq e^{-|k|\eps} c\cdot (p^u_{k})^{-\gamma}
=e^{|k|(\bar \eps\gamma-\eps)}c\cdot (p^u_0)^{-\gamma}\geq c\cdot (p^u_0)^{-\gamma}.\end{align*}
Similarly, one can use $k:=\min\{j\geq 0: p^s_j=Q_{\bar \eps}(x_j)\}$ to show that
$K_\ast(x)\geq c (p^s_0)^{-\gamma}$. Together, this gives $K_\ast(x)\geq c(p^u_0\wedge p^s_0)^{-\gamma}$.
\end{proof}

\begin{lemma}\label{l.cylinder-finite}
Let $(\Sigma(\wh{\mathfs G}),\hpi)$ be the coding in Cor~\ref{c.coro-epsilon(M,f,chi,beta)}.
For every $t\!>\!0$,
the set
$ \{\un{R}\!\in\!\Sigma^\#(\widehat{\mathfs G})\cap \hpi^{-1}(\NUH^\#_{\bar \chi})\!:\! K_\ast(\hpi(\un{R}))\!\leq\! t\}$
is contained in a finite  union of  cylinders.
\end{lemma}
\begin{proof}
Let $\un{R}\in\Sigma^\#(\widehat{\mathfs G})\cap \hpi^{-1}(\NUH_{\bar{\chi}}^\#)$, $x:=\hpi(\un{R})$, and suppose $K_\ast(x)\leq t$.

By Lemma~\ref{l.coding-final}(iii), there exists an $\bar\eps$-chain $\un{\Psi}=(\Psi_{x_i}^{p^u_i,p^s_i})\in \Sigma^\#(\mathfs H)$ such that $R_i\subset Z(\Psi_i)$ for all $i$ and $x=\fs(\un{\Psi})$.
By Lemma~\ref{l.bound-Pesin-bound},
    $
c\cdot (p^u_0\wedge p^s_0)^{-\gamma}\leq K_\ast(x)\leq t$.
Recall the definition of $Z(\Psi_0)$ from \eqref{e.Z}. Then
\begin{enumerate}[\quad$\circ$]
\item  $Z(\Psi_0)\in\mathfs Z_t$, where
$\mathfs Z_t:=\{Z(\Psi_{x}^{p^u,p^s})\in\mathfs Z: p^u\wedge p^s\geq (t/c)^{-\frac{1}{\gamma}}\}$, and
\item $R_0\in\mathfs R_t$, where
    $
    \mathfs R_t:=\bigcup_{Z\in\mathfs Z_t}\{R\in\mathfs R:R\subset Z\}.
    $
\end{enumerate}
$\mathfs Z_t$ is finite by the discreteness property of $\mathfs V$, see \S\ref{ss.discretization}(b). $\mathfs R_t$ is finite by Lemma~\ref{l.coding-final}(ii). So
    $\{\un{R}\in\Sigma^\#(\widehat{\mathfs G})\cap\hpi^{-1}(\NUH^\#_{\bar \chi}):K_\ast(\hpi(\un{R}))\leq t\}$ is contained in
    $\{\un{y}\in\Sigma^\#(\widehat{\mathfs G}):\; y_0\in \mathfs R_t\}$, whence in a finite union of cylinders.
\end{proof}

\begin{proposition}\label{c.bornologicalA}
The codings in Cor \ref{c.coro-epsilon(M,f,chi,beta)} and \ref{c.coro-transitive-full-hyp-coding} are   $\chi$-bornological.
\end{proposition}
\begin{proof}
First we show that $(\Sigma(\wh{\mathfs G}),\hpi)$ is $\chi$-bornological.

Let us consider some $(\chi,\varepsilon)$-Pesin block $\Lambda$.
By definition, there exists $t>0$ such that any $x\in \Lambda$ satisfies $K_*(x)<t$.
By Lemma~\ref{l.cylinder-finite} there is a  finite union of cylinders $A$
such that $\Sigma^\#(\widehat{\mathfs G})\cap \hpi^{-1}(\Lambda\cap \NUH^\#_{\bar \chi})\subset A$.

Let $\hmu$ be some ergodic $\sigma$-invariant measure on $\Sigma(\widehat{\mathfs G})$
and set $\mu=\hpi_*\hmu$. If $\mu(\Lambda)=0$, then $\hmu(\hpi^{-1}(\Lambda))=0$.
Otherwise, by ergodicity, all the Lyapunov exponents of $\mu$ are outside $(-\chi,\chi)$, whence by \eqref{e.epsilon-bar}, outside $[-\bar{\chi},\bar{\chi}]$. In this case $\mu$
is $\bar{\chi}$-hyperbolic, and  it follows that $\mu$ gives full measure to the set
$\NUH^\#_{\bar \chi}$ from Lemma \ref{l.shadowing}(iii).
Therefore $\hmu$ gives full measure to $\hpi^{-1}(\NUH_{\bar{\chi}}^\#)\cap\Sigma^\#(\wh{\mathfs G})$, and as a result,
\begin{equation*}
\hmu(\hpi^{-1}(\Lambda)\setminus A)=\hmu\big(\bigl[\Sigma^\#(\widehat{\mathfs G})\cap \hpi^{-1}(\Lambda\cap \NUH^\#_{\bar \chi})\bigr]\setminus A\big)=0.
\end{equation*}
This proves that $\hmu(\hpi^{-1}(\Lambda)\setminus A)=0$ for all ergodic $\sigma$-invariant measures on $\Sigma(\wh{\mathfs G})$. By the ergodic decomposition, this is also the case for non-ergodic measures, and we verified the first half of the bornological property, Def.~\ref{d.bornological}(a).

\medskip
Next we prove  the second half of the bornological property, Def.~\ref{d.bornological}(b), with
\begin{equation}\label{e.born-tild-const}
\widetilde \chi:=\bar \chi \text{ and } \widetilde \varepsilon:=\tfrac{\bar \varepsilon\beta}{48}.
\end{equation}

Given $R\in \mathfs R$, let $A_R:=\{\un R\in \Sigma^\#(\widehat{\mathfs G}):R_0=R\}$. We must show that for some $(\wt{\chi},\wt{\eps})$-Pesin block $\Lambda$, for all $\chi$-hyperbolic invariant measures $\mu$,
\begin{equation}\label{e.rishi}
\mu[\pi(A_R)\setminus\Lambda]=0.
\end{equation}

By Lemma~\ref{l.coding-final}(ii) the set $\widetilde  {\mathfs Z}_R:=\{Z(\Psi_{y}^{q^u,q^s})\in\mathfs Z: R\subset Z (\Psi_{y}^{q^u,q^s})\}$ is finite.
By \S\ref{ss.discretization}(b), the following set is finite:
$$\mathfs Z_R:=\{Z(\Psi_{x}^{p^u,p^s})\in\mathfs Z:
\exists  Z(\Psi_{y}^{q^u,q^s})\in\widetilde  {\mathfs Z}_R \text{ s.t. } p^u>e^{-\sqrt[3]{{\bar \eps}}} q^u
\text{ and } p^s>e^{-\sqrt[3]{{\bar \eps}}} q^s \}.$$
In particular there exists $c_0>0$ such that $p^u\wedge p^s>c_0$ for all $Z(\Psi_{x}^{p^u,p^s})\in \mathfs Z_R$.

Fix some $x\in\NUH^{\#}_{{\bar{\chi}}}\cap \wh{\pi}(A_{R})$, and write $x= \hpi(\un R)$ with $\un R\in A_R$.
By \S\ref{ss.discretization}(a), there exists some  $\un \Psi=(\Psi_{x_i}^{p^u_i,p^s_i})_{i\in\Z}$ in $\Sigma^\#( \mathfs H)$
satisfying
$$
\fs(\un \Psi)=x\text{ and }Q_{\bar \varepsilon}(f^n(x))\geq Q_{\bar \varepsilon}(x_n) e^{-\bar\varepsilon/3}\text{
for each }n\in \ZZ.
$$
By Lemma~\ref{l.coding-final}(iii), since $x\in\NUH_{\bar{\chi}}^\#$, there exists $\un \Psi'=(\Psi_{y_i}^{q^u_i,q^s_i})_{i\in\Z}$ in $\Sigma^\#( \mathfs H)$
with $\fs(\un \Psi')=\hpi(\un{R})=x$ and $R=R_0\subset Z(\Psi_{y_0}^{q^u_0,q^s_0})$. In particular $Z(\Psi_{y_0}^{q^u_0,q^s_0})\in \widetilde {\mathfs Z}_R$.
$\un{\Psi}$ and $\un{\Psi}'$ both shadow $x$.
By Lemma~\ref{l.inverse}(i),  $Z(\Psi_{x_0}^{p^u_0,p^s_0} )\in {\mathfs Z}_R$.
Then, for all $n\in\Z$:
\begin{align*}
Q_{\bar \varepsilon}(f^n(x))&\geq Q_{\bar \varepsilon}(x_n) e^{-\bar\varepsilon/3}\quad \text{ {by \S\ref{ss.discretization}(a)(iii),}}\\
&\geq (p_n^u\wedge p^s_n) e^{-\bar\varepsilon/3}  \quad \text{by the definition of double charts},\\
&\geq (p_0^u\wedge p_0^s) e^{-\bar\varepsilon |n| -\bar\varepsilon/3}  \quad \text{by~\eqref{tralala} and the definition of $\bar{\eps}$-chains},\\
&\geq c_0\cdot e^{-\bar\varepsilon |n| -\bar\varepsilon/3}  \quad \text{since $ Z(\Psi_{x_0}^{p^u_0,p^s_0})\in {\mathfs Z}_R$.}\end{align*}
 By \eqref{e.Q}, there is a constant $\wt{K}=\wt{K}(\beta,\bar\varepsilon,c_0)>0$ such that
$
\|C_{\bar{\chi}}(f^n(x))^{-1}\|\leq \wt{K} e^{\bar{\eps}\beta|n|/48}
$.
It now follows from
\eqref{e.S} that for all $n\in\Z$ and all $m\geq 0$,
\begin{align*}
\|Df^m|_{E^s(f^n(x))}\|&\leq \|C_{\bar \chi}(f^n(x))^{-1}\| \exp(-\bar \chi m)\leq \wt{K} \exp(-\widetilde \chi m +\widetilde \varepsilon |n|)
\end{align*}
(recall that $\wt\chi=\bar\chi$ and $\wt{\eps}=\bar{\eps}\beta/48$). A similar estimate holds for $\|Df^{-m}|_{E^u(f^n(x))}\|$.

This shows that  $\hpi(A_R)\cap \NUH^\#_{\bar \chi}$ is contained in the $(\widetilde \chi,\widetilde \varepsilon)$-Pesin block $\Lambda_R$ of all points which satisfy \eqref{e.def-pesin}, but with $(\wt{K},\wt{\chi},\wt{\eps})$ replacing $(K,\chi,\eps)$.

Now suppose $\mu$ is a $\chi$-hyperbolic invariant measure $\mu$. Such measures are $\bar{\chi}$-hyperbolic, and therefore they are  carried by $\NUH_{\bar{\chi}}^\#$. Necessarily,
$$
\mu[\hpi(A_R)\setminus\Lambda_R]=\mu[\hpi(A_R)\cap \NUH_{\bar{\chi}}^\#\setminus\Lambda_R]=\mu(\emptyset)=0,
$$
and we obtained \eqref{e.rishi}.
This completes the proof of that  $(\Sigma(\wh{\mathfs G}),\hpi)$ is  $\chi$-bornological.

 By Part {3} of Cor~\ref{c.coro-transitive-full-hyp-coding},  $(\Sigma({\mathfs G}),\hpi|_{\Sigma({\mathfs G})})$ is also $\chi$-bornological.
\end{proof}

\subsection{SPR Codings for SPR Diffeomorphisms}\label{section-Sigma}

Collecting the results of the previous sections, we obtain our main coding result:
\begin{theorem}\label{t.SPR-coding}
Let $f$ be a $C^{1+}$ diffeomorphism on a closed manifold $M$,
and $X$ be $M$ or a Borel homoclinic class. For every {$\chi>0$}
there is a map $\pi:\Sigma\to M$ s.t.:

\begin{enumerate}[($\mathit\Sigma$1)]

\item \textbf{The Markov Shift.}\label{i.Sigma1}  $\Sigma$ is a locally compact countable state Markov shift.  If  $X$ is a Borel homoclinic class, we can choose  $\Sigma$ to be  irreducible.

\smallskip
\item \textbf{The Coding Map.} \label{i.Sigma2} $\pi:\Sigma\to M$ is H\"older continuous, $\pi\circ\sigma=f\circ\pi$, and $\pi:\Sigma^\#\to M$ is finite-to-one. (See \eqref{e.regular-part} for the definition of  $\Sigma^\#$.)

\smallskip
\item \textbf{Projection of Measures.} \label{i.Sigma3} Every ergodic $\sigma$-invariant measure $\hmu$ on $\Sigma$ projects to an ergodic $f$-invariant measure $\mu:=\pi_\ast\hmu:=\hmu\circ\pi^{-1}$ such that 
$\mu(X)=1$ and $h(f,\mu)=h(\sigma,\hmu)$.

\smallskip
\item \textbf{Lifting Measures.} \label{i.Sigma4}
{For every} ergodic $\chi$-hyperbolic $f$-invariant measure $\mu$ on $X$
{there exists} some  ergodic $\sigma$-invariant measure $\hmu$ on $\Sigma$
such that $\pi_*\hmu=\mu$.

\smallskip
\item \textbf{Hyperbolicity.} \label{i.Sigma5} {For some $\chi_0\!>\!0$, at any $\un x\in \Sigma$
there is a splitting $T_{\pi(\un x)}M=E^s(\un x)\oplus E^u(\un x)$
such that $\underline x\mapsto E^{s}(\un x)$ and $\underline x\mapsto E^{u}(\un x)$ are H\"older continuous and $\underset{n\to +\infty} \limsup \tfrac 1 n \log \|Df^n|_{E^s(\un x)}\|<-\chi_0$,   $\underset{n\to +\infty} \limsup \tfrac 1 n \log \|Df^{-n}|_{E^u(\un x)}\|<-\chi_0$.} 

\smallskip
\item \label{i.Sigma6} \textbf{Bornological Property.} $(\Sigma,\pi)$ is a $\chi$-bornological hyperbolic coding in $X$.

\smallskip
\item \textbf{SPR.}\label{i.Sigma1bis} If $f|_X$ {is $\chi'$-SPR} with $\chi'>\chi$, then $\Sigma$ is SPR, $(\Sigma,\pi)$ is entropy-full, and  $h_\Bor(\Sigma)=h_{\Bor}(f|_X)$.

\smallskip
\item \textbf{Additional Properties.} \label{i.Sigma7}  $(\Sigma,\pi)$ satisfies all the properties described in \cite[\S3]{BCS-1}, and in particular the ``locally finite Bowen property" stated there.\footnote{We do not need these additional properties in this paper.}

\end{enumerate}
\textbf{Converse Statement:} Suppose \eqref{e.chi-entropy-hyperbolic} holds with  $\chi$ (as is always the case when $\dim M=2$ and $\chi<h_{\Bor}(f|_X)$). Then 
the {following} holds:
If $f|_X$ has an SPR Markov coding satisfying ($\mathit{\Sigma}$\ref{i.Sigma1})--($\mathit\Sigma$\ref{i.Sigma6}), then $f|_X$ is entropy-tight, whence SPR.
\end{theorem}

\begin{proof}
Without loss of generality, $X$ carries some hyperbolic invariant measure: If $X$ is a Borel homoclinic class, this is automatic; If $X=M$ and there are no hyperbolic invariant measures, then the theorem holds vacuously with the empty Markov shift.
Next, we observe that decreasing $\chi$ makes the (direct) statement of the theorem stronger; Therefore, we may also assume without loss of generality that $\chi$ is so small that $X$ carries a $\chi$-hyperbolic invariant measure.

{The Markov coding provided by Cor~\ref{c.coro-epsilon(M,f,chi,beta)} or \ref{c.coro-transitive-full-hyp-coding} satisfies Properties  (${\Sigma}$\ref{i.Sigma1})--($\Sigma$\ref{i.Sigma5}) and Prop.~\ref{c.bornologicalA} gives $(\Sigma\ref{i.Sigma6})$.}
Since the constructions in \S\ref{ss.construction-hyperbolic} are compatible with those of \cite{BCS-1}, we also have $(\Sigma\ref{i.Sigma7})$. It remains to check $(\Sigma\ref{i.Sigma1bis})$: Suppose $f|_X$ {is $\chi'$-SPR with $\chi'>\chi$.}
By Remark~\ref{r.SPR-implies-hyperbolicity},  \eqref{e.chi-entropy-hyperbolic} holds 
and $(\Sigma\ref{i.Sigma4})$ implies that $(\Sigma,\pi)$ is entropy-full.  By Lemma~\ref{l-entropy-equal},  $h_{\Bor}(\Sigma)=h_{\Bor}(f|_X)$.  Next, $f|_X$ must be  $\chi$-SPR, because $f|_X$ is $\chi'$-SPR and $\chi'>\chi$, see Lemma \ref{l.P-is-bornology}(4). By Prop.~\ref{p.lift-SPR} and the $\chi$-bornological property, $\Sigma$ is SPR. So we also have $(\Sigma\ref{i.Sigma1bis})$.

Conversely, suppose \eqref{e.chi-entropy-hyperbolic}, and that $f$ has an SPR Markov coding in $X$ satisfying  (${\Sigma}$\ref{i.Sigma1})--($\Sigma$\ref{i.Sigma6}).
In particular it is hyperbolic and $\chi$-bornological. By \eqref{e.chi-entropy-hyperbolic} and
(${\Sigma}$\ref{i.Sigma4}), {the coding is entropy-full}. Prop.~\ref{p.project-SPR} then implies that $f|_X$ is entropy-tight.
\end{proof}

\begin{remark}\label{r.eps-tilde-arbit-small}
Looking at \eqref{e.born-tild-const}, and recalling the choices of $\bar \chi,\bar\eps$,
we deduce the following. For any $\wt{\chi}\in (0,\chi)$ and all $\wt{\eps}>0$ small enough, there are codings $(\Sigma,\pi)$ as  in  Thm~\ref{t.SPR-coding}  which satisfy the bornological property (Def.~\ref{d.bornological}\eqref{i.born2}) with $\wt{\chi},\wt{\eps}$.
\end{remark}

We end with an extension of Item $(\Sigma\ref{i.Sigma1bis})$  to potentials of equilibrium measures.
Suppose $\phi$ is a real-valued H\"older continuous function on $M$ (or more generally one of the quasi-H\"older functions  considered  in \S\ref{s.setup-for-consequences}). For a Borel homoclinic class $X$:
\begin{itemize}
\item $f$ is {\em $\chi$-SPR for $\phi$ on $X$}, if it satisfies Def. \ref{d.SPR-potential} with the parameter $\chi$; 
\item 
a hyperbolic coding $(\Sigma,\pi)$ in $X$ is {\em pressure-full for $\phi$}, if there is $p_0<P_{\Bor}(f|_X,\phi)$ such that for every $\mu\in\Proberg(f|_X)$, if   $h_{\mu}(\sigma)+\int\phi\, d\mu>p_0$, then  $\mu[\pi(\Sigma^{\#})]=1$, and $\mu$ has a lift with the same entropy and pressure to $\Sigma$
\end{itemize}

\begin{addendum}\label{a.SPR-coding}
The coding in Thm~\ref{t.SPR-coding} satisfies
\begin{enumerate}[($\mathit\Sigma$7')]
\item \textbf{SPR for Potentials.}\label{i.Sigma1bisbis} If $f|_X$ is $\chi'$-SPR for $\phi$ with $\chi'>\chi$, then $\Sigma$ is SPR for $\phi\circ \pi$,   $P_\Bor(\Sigma,\phi\circ \pi)=P_{\Bor}(f|_X,\phi)$, and  $(\Sigma,\pi)$ is pressure-full for $\phi$.
\end{enumerate}
\end{addendum}
\noindent
The proof is identical to the proof of $(\Sigma\ref{i.Sigma1bis})$, and we omit it.

\part{Properties of SPR Diffeomorphisms}\label{p.properties-of-SPR-diffeos}

{

\section{Consequences of the SPR Property} \label{s.statement-of-conseq}
This section is devoted to the precise  statements of the ergodic properties announced in Thms \ref{t.MME}--\ref{t.equilibrium} and the direct part of Thm~\ref{t.tail}. The proofs will be given in the next section.

Let $f$ be a $C^{1+}$ diffeomorphism of a closed smooth $d$-dimensional  manifold $M$, with $d\geq 2$. Let  $X$ be an SPR Borel homoclinic class. We will see below that $f|_X$ has a unique measure of maximal entropy. We denote this measure by $\mu$.

\subsection{Quasi-H\"older functions}\label{s.setup-for-consequences}
Given  $\psi:X\to\C$,  let
$$
\psi_n=\psi+\psi\circ f+\cdots + \psi\circ f^{n-1}.
$$
We are mostly interested in H\"older continuous $\psi$  and  in {\em geometric potentials}
\begin{equation}\label{e.geometric}
J^s(x)=-\log |\det(Df|_{E^s(x)})|\ \
\text{ and } J^u(x)=-\log |\det(Df|_{E^u(x)})|,
\end{equation}
which are not always continuous (or even globally defined).
The geometric potentials are  needed  in \S\ref{ss.expansion-bounds} and \S\ref{ss.Kadyrov-for-Diffeos}.

To treat $J^s(x), J^u(x)$ and H\"older continuous functions on the same footing, we introduce the  class of {\em quasi-H\"older functions}, which contains them both.
Informally, $\psi(x)$ is {\em quasi-H\"older}, if it can be recast as a H\"older continuous function of $(x,E^u(x),E^s(x))$.
For example, $J^u(x)=\Psi(x,E^u(x),E^s(x))$, where $\Psi(x,E_1,E_2)=-\log|\det(Df|_{E_1})|$, and  $E_i\subset T_x M$ are linear subspaces. We now make this precise.

Let $\mathfrak G(k,M)$ denote the $k$-th Grassmannian bundle of $M$, with its canonical Riemannian structure (see Appendix~\ref{app-grassmannian}).
Let
\begin{equation}\label{e.frak-G}
\mathfrak G:= \bigsqcup_{k=1}^{d-1}\mathfrak G(k,M)\oplus \mathfrak G(d-k,M)\ \ \text{(disjoint union)}.
\end{equation}

If $x\in X$, then $(x;E^u(x),E^s(x))\in \mathfrak G$. Indeed, on  $X\subset\NUH(f)$ (see \eqref{e.NUH}),  $E^u(x),E^s(x)$ are well-defined and unique by Lemma \ref{l.splitting}. \begin{definition}A {\em $\mathfrak G$-lift} of a function $\psi:X\to\C$ is a function $\Psi:{\mathfrak G}\to\C$  such that for all $x\in X$,
$
\Psi(x;E^s(x),E^u(x))=\psi(x)
$.
\end{definition}
\noindent

Every connected component of $\mathfrak G$ has a canonical metric, induced by the Riemannian structure. We extend this metric to  $\mathfrak G$ by declaring that the  distance   between points  in different connected components equals the diameter of $M$.
Let $d_{\mathfrak G}(\cdot,\cdot)$ denote the restriction of this distance function to $\mathfrak G$.

\begin{definition}
$\psi:X\to{\C}$ is {\em quasi-H\"older on $X$} (with exponent $\beta$), if $\psi$ admits an $\mathfrak G$-lift $\Psi$ which is H\"older continuous with exponent $\beta$, with respect to $d_\mathfrak{G}$. In this case, the {\em $\beta$-quasi-H\"older norm} of $\psi$ is defined to be
$$
{\|\psi\|_\beta':=\inf\left\{\sup_{\mathfrak G}|\Psi|+\sup_{\xi\neq \eta}\frac{|\Psi(\xi)-\Psi(\eta)|}{d_{\mathfrak G}(\xi,\eta)^\beta}:\Psi:\mathfrak G\to\C\text{ is a lift of }\psi\right\}}.
$$
\end{definition}
\noindent
(Actually we will only use the lift  to $\Gamma:=\{(x;E^u(x),E^s(x)):x\in X\}\subset \mathfrak G$. 
One could wonder whether requiring a lift to all of $\mathfrak G$ is wasteful. It is not: Any H\"older function on $\Gamma$ can be extended to $\mathfrak G$ without  decreasing the H\"older exponent.)

\begin{example}\label{e.Holder-is-Quasi-Holder}
If $\psi:M\to\C$ is H\"older with exponent $\beta$, then $\psi|_{X}$ is quasi-H\"older with exponent $\beta$, and  $\|\psi\|_{\beta}'\leq\|\psi\|_\beta$: use the lift $\Psi(x;E_1,E_2)=\psi(x_1)$.
\end{example}

\begin{example}\label{e.Geometric-Is-Quasi-Holder}
If $f\in C^{1+\beta}$ with $\beta>0$,  then $x\mapsto J^s(x)$ and $x\mapsto J^u(x)$ (introduced in~\eqref{e.geometric}) are
quasi-H\"older on $X$ with exponent $\beta$. \end{example}

\begin{lemma}\label{l.QH-functions-are-Bounded}
Quasi-H\"older functions on $X$ are measurable and  bounded. \end{lemma}
\begin{proof}{
By Lemma \ref{l.splitting}, $x\mapsto (x; E^u(x),E^s(x))$ is continuous on Pesin blocks. Therefore any quasi-H\"older function $\psi:X\to\mathbb C$ is continuous on Pesin blocks. Since $X$ is covered by countably many Pesin blocks, $\psi$ is measurable on $X$. Next,
$\mathfrak G$ has finite diameter, and every  H\"older continuous function on $\mathfrak G$ is bounded.}
\end{proof}

We proceed to state our results on the properties of SPR diffeomorphisms.

\subsection{Existence and Structure of  Measures of Maximal Entropy}
The following implies Thm~\ref{t.MME}.

\begin{theorem}\label{t.MME-exists}
Let $f$ be a $C^{1+}$ diffeomorphism of a closed manifold, and let $X$ be an SPR Borel homoclinic class with period $p$. Then:
\begin{enumerate}[(1)]
\item There exists a measure $\mu$ on $X$ such that $h(f,\mu)=\hTOP(f|_X)$.
\item This measure is unique (on $X$), ergodic, {hyperbolic},  and its support equals $\overline X$. If $p=1$, then $\mu$ is Bernoulli. If $p>1$, then $\mu$ is
 isomorphic to the product of a Bernoulli scheme with a cyclic permutation of $p$ atoms.
\item {Suppose $f$ is SPR, then $f$ has a finite positive number of ergodic measures of maximal entropy, each carried by a distinct SPR Borel homoclinic class. }
\end{enumerate}
\end{theorem}

\begin{definition} The integer $p$ in (2)  is called the \emph{period of $\mu$}. \end{definition}

\begin{remark}\label{r.Thm-13.1}
Part (1) is proved in \S\ref{s.Proof-existence-MME-diffeo}. Part
(2) was obtained in~\cite{BCS-1}, and (3) follows from  (2) and Prop. \ref{p.decomposition}. By \cite[Sec.~1.6]{BCS-1}, the period $p$ of $\mu$ coincides with {\em the period of the  Borel homoclinic class of $\mu$}, defined to be the greatest common divisor of the period of all hyperbolic periodic orbits homoclinically related to $\mu$.

\end{remark}

\begin{remark}\label{r.top-mixing-implies-Bernoulli}{
Suppose $\dim M=2$, $f\in C^\infty$ and $h_{\top}(f)>0$.
\begin{enumerate}[\quad$\circ$]
\item If $f$ is a  topologically transitive, then there exists exactly one Borel homoclinic class $X$ carrying measures with positive entropy \cite{BCS-1}. The MME of $f|_X$ is the unique MME of $f$.

\item  If $f$ is topologically mixing, then  $\mu$ is Bernoulli and  $p=1$  \cite[Thm 6.12(5)]{BCS-1}.
\end{enumerate}
These results use two-dimensionality and $C^\infty$-smoothness in an essential way.}
\end{remark}
\subsection{Exponential Decay of Correlations}
\label{ss.exp-dec-cor}
We now study the ergodic properties of $(X,f,\mu)$,
starting with the exponential mixing announced in Thm~\ref{t.main}.

The {\em expectation} of an $L^1$ function $\psi$ will be denoted by $
\E_\mu(\psi):=\int \psi d\mu.
$
The {\em variance} of an $L^2$ function $\psi$ is the number $\mathrm{Var}_\mu(\psi):=\|\psi-\E_\mu(\psi)\|_2^2$.
The
 {\em covariance} of two real-valued  $\vf,\psi\in L^2(\mu)$ is
$$
\mathrm{Cov}_\mu(\vf,\psi):=\E_\mu\bigl[(\vf-\E_\mu(\vf))(\psi-\E_\mu(\psi))\bigr]=\int \vf \psi d\mu-\int\vf d\mu \int \psi d\mu.
$$
The {\em correlation coefficient} of non-constant real-valued $\vf,\psi\in L^2(\mu)$ is $${\rho_\mu}(\vf,\psi):=\frac{\mathrm{Cov}_\mu(\vf,\psi)}
{\sqrt{\mathrm{Var}_\mu(\vf)\mathrm{Var}_\mu(\psi)}}\in [-1,1].
$$

We are interested in the correlation between the results of ``a measurement at time $n$" $\vf\circ f^n$ and ``a measurement at time $m$" $\psi\circ f^m$, for $|n-m|\gg 1$. By the $f$-invariance of $\mu$, for every $\vf,\psi\in L^2\setminus\{\text{constants}\}$, there is a $c>0$ such that
$$
\rho_\mu(\vf\circ f^n,\psi\circ f^m)=c\cdot \mathrm{Cov}_\mu(\vf,\psi\circ f^{m-n})
=c\cdot \mathrm{Cov}_\mu(\psi,\vf\circ f^{n-m})\ \ \ (\forall n,m\in\mathbb Z).$$ So it is sufficient to  estimate $\mathrm{Cov}_\mu(\vf,\psi\circ f^{k})$ for $k\in\N$.

Let $p$ denote the period of $\mu$. If $p=1$, then $\mu$ is  mixing, and  $ \rho_\mu(\vf,\psi\circ f^n)\to 0$. 
If $p>1$, then $\mu$ is not mixing,  and we will need to work with the ergodic components of $\mu$ for the iterate $f^p$:

\begin{theorem}\label{t.DOC-diffeos}
Let $f$ be a $C^{1+}$ diffeomorphism of a closed manifold, let $X$ be an SPR Borel homoclinic class with period $p$, and let $\mu$ be the MME of $f|_X$. For every $\beta>0$ there are $0<\theta<1$ and $C>1$ such that for all $\vf,\psi$ which are
$\beta$-quasi-H\"older on $X$,   and for every ergodic component $\mu'$ of $(X,\mu,f^p)$,
\begin{equation}\label{e.exp-doc}
\left|\int \vf \cdot (\psi\circ f^{np}) \;d\mu'-\int \vf \;d\mu'\int \psi \;d\mu'\right|\leq C\|\phi\|_\beta'\|\psi\|_\beta'\theta^{np} \ \ \ (\forall n\geq 0).
\end{equation}
If $(X,f,\mu)$ is mixing, then the theorem holds with $p=1$ and $\mu'=\mu$.
\end{theorem}

\begin{remark}
By Example \ref{e.Holder-is-Quasi-Holder}, if $\vf,\psi$ are $\beta$-H\"older continuous on $M$, then $\|\cdot\|_\beta'$ can be replaced by the usual $\beta$-H\"older norm $\|\cdot\|_\beta$.
\end{remark}

\begin{remark}\label{r.doc-for-top-mixing-surface-diffeos}
Suppose $\dim M=2$,  $f\in C^\infty$, $h_{\top}(f)>0$, {and $f$ is topologically mixing}. By Remark \ref{r.top-mixing-implies-Bernoulli}, $f$ has a unique MME $\mu$, the Borel homoclinic class of $\mu$ is SPR, and the period equals one. So \eqref{e.exp-doc} holds with $\mu'=\mu$ and $p=1$.
\end{remark}

Thm~\ref{t.main} follows from Thm~\ref{t.DOC-diffeos}, and the previous  remarks.

For the proof of Thm~\ref{t.DOC-diffeos}, see \S\ref{s.decay-of-corr-diffeo-proof}.  {The theorem is known for Anosov diffeomorphisms, Axiom A attractors, and ``Smale systems," see Bowen \cite{Bowen-LNM} and Ruelle \cite{Ruelle-TDF-book}. 
In the general non-uniformly hyperbolic setting, it is the first result of its type for measures of maximal entropy.

\subsection{Asymptotic Variance}
Suppose $\psi$ is a real-valued function in $L^2(\mu)$. Recall that $\psi_n=\psi+\psi\circ f+\cdots+\psi\circ f^{n-1}$. The  \emph{asymptotic variance} of $\psi$ is
\begin{equation}\label{eq-def-avar}
\sigma_\psi^2:=\lim\limits_{n\to\infty}\tfrac{1}{n}\mathrm{Var}_\mu(\psi_n)
\end{equation}
(whenever the limit exists).
The following theorem studies this limit.
We have defined in \S\ref{ss.equilibrium}
the top pressure $P_\Bor(f|_X,\psi)$ of $f$ on $X$ with respect to $\psi$.
\begin{theorem}\label{t.asymp-var-diffeos}
Let $f$ be a $C^{1+}$ diffeomorphism of a closed manifold, let $X$ be an SPR~Borel homoclinic class, let $\mu$
be the measure of maximal entropy of $f|_X$, and suppose $\psi:X\to\R$ is $\beta$-quasi-H\"older. Then the limit \eqref{eq-def-avar} exists, and:
\begin{enumerate}[(1)]
\item {\bf Green-Kubo Formula:} Let $p$ be the period of $\mu$, then
    $$
    \sigma_\psi^2=\frac{1}{p}\left[\mathrm{Var}_\mu(\psi_p)
    +2\sum_{n=1}^\infty \mathrm{Cov}_\mu(\psi_{p},\psi_{p}\circ f^{np})\right].
    $$
\item {\bf Linear Response Formula:}  $\sigma_\psi^2=\frac{d^2}{dt^2}\big|_{t=0}P_\Bor(f|_X,t\psi)$.
\item  {\bf Asymptotic Laplace Transform:} Suppose $\int\psi d\mu=0$, then
    $$
    \E_\mu(e^{z\psi_n/\sqrt{n}})\xrightarrow[n\to+\infty]{}e^{\frac{1}{2}\sigma_\psi^2 z^2}\ \ \ \ (\forall z\in\mathbb C).
    $$
\item {\bf Upper Bound\/:} For some $M_\beta>0$ which depends on $\beta,f$
but not on  $\psi$,
\begin{equation}\label{e.sigma-bound-diffeo}
\sigma_\psi\leq M_\beta\|\psi\|_\beta'.
\end{equation}

\item {\bf Conditions for Zero Asymptotic Variance:} The following are equivalent:
    \begin{enumerate}[(a)]
    \item $\sigma_\psi^2=0$.
    \item $\psi-\int\psi d\mu=u-u\circ f$ $\mu$-a.e. for some Borel function $u:M\to\R$.
    \item For all  hyperbolic periodic points $x$, if $x$ is  homoclinically related to $\mu$, then $\psi_n(x)=n\int\psi d\mu$,
    where $n$ is the period of $x$.
        \item $\int \psi d\nu=\int\psi d\mu$ for all  $f$-invariant measures $\nu$  on $X$.
    \end{enumerate}
\end{enumerate}
\end{theorem}
\noindent
Thm~\ref{t.asymp-var-diffeos} is proved in \S\ref{s.asymp-var-diffeos}.
{The special cases when $f$ is Anosov, axiom A, or a Smale system (see Example \ref{e.BHC-for-axiom-A}) are due to Bowen \cite{Bowen-LNM}, Ruelle \cite{Ruelle-TDF-book}, Livsic \cite{Livsic}, and Guivarc'h \& Hardy \cite{Guivarch-Hardy}. The equivalence of (b) and (c) in (4) is close to results of Katok \& Mendoza \cite{Katok-Hasselblatt-Book} and Pollicott \cite{Pollicott-Livsic}.}

\subsection{The Central Limit Theorem and Convergence of Moments} The following result is an immediate consequence of Theorem \ref{t.asymp-var-diffeos}(3):
\begin{corollary}
\label{c.CLT-diffeos}
Let $f$ be a $C^{1+}$ diffeomorphism of a closed manifold, let $X$ be an SPR Borel homoclinic class, and let $\mu$
be the MME of $f|_X$. Suppose $\psi:X\to\R$ is  quasi-H\"older, and $\int\psi d\mu=0$. Then:
\begin{enumerate}[(1)]
\item {\bf Central Limit Theorem (CLT).\/} For every $a<b$,
$$
\lim_{n\to+\infty}\mu\left\{x:\frac{\psi_n(x)}{\sqrt{n}}\in (a,b)\right\}=\begin{cases}
\frac{1}{\sqrt{2\pi\sigma_\psi^2}}\int_a^b e^{-t^2/2\sigma_\psi^2}dt & \text{if } \sigma_\psi\neq 0,\\
\mathds{1}_{(a,b)}(0) & \text{if } \sigma_\psi=0 \text{ and } ab\neq 0.
\end{cases}
$$
\item {\bf Convergence of Moments.} For every $k\in\N$, $n^{-k/2}\E_\mu(\psi_n^{k})\xrightarrow[n\to+\infty]{} m_{k} \sigma_\psi^{k} $, where $m_{k}$ is the $k$-th moment of the standard Gaussian distribution.
\end{enumerate}
\end{corollary}
\begin{proof}
The proof of (1) follows the classical approach: Theorem \ref{t.asymp-var-diffeos} gives the limit $\E_\mu(e^{z\psi_n/\sqrt{n}})\to e^{\frac{1}{2}\sigma_\psi^2 z^2}$ for every $z\in i\mathbb R$,  and the CLT follows from  L\'evy's continuity theorem.

The CLT says that $\psi_n/\sqrt{n}$ converges in distributions to $\sigma_\psi \mathcal{N}$ where $\mathcal{N}$ is a standard Gaussian random variable. To see that  $\E_\mu((\psi_n/\sqrt{n})^k)\to \E(\sigma_\psi^k N^k)$, it is sufficient to check that $X_n:=(\psi_n/\sqrt{n})^k$ is {\em uniformly integrable}, i.e. that for every $\eps>0$ there exists $K$ so that $\E_\mu\bigl(|X_n|\mathds{1}_{[|X_n|>K]}\bigr)\leq\eps$ for all $n$.

This is indeed the case: $|\frac{x^{2k}}{(2k)!}|\leq \frac{e^x+e^{-x}}{2}$, therefore for every $K>1$,
$$
\E_\mu\bigl(|X_n|\mathds{1}_{[|X_n|>K]}\bigr)\leq \tfrac{1}{K^2} \E_\mu\bigl(X_n^2 \bigr)\leq  \frac{(2k)!}{2K^2}\biggl(\E_\mu[e^{\psi_n/\sqrt{n}}]+\E_\mu[e^{-\psi_n/\sqrt{n}}]\biggr).
$$
This is $O(1/K^2)$, because  $\E_\mu[e^{\pm \psi_n/\sqrt{n}}]\to e^{\frac{1}{2}\sigma_\psi^2 }$, by Thm~\ref{t.asymp-var-diffeos}(3).
\end{proof}
In the special case of Anosov diffeomorphisms and Smale systems (see Example \ref{e.BHC-for-axiom-A}), Cor~\ref{c.CLT-diffeos} is due to
 Bowen \cite{Bowen-LNM}, Ruelle \cite{Ruelle-TDF-book}, and Guivarc'h \& Hardy \cite{Guivarch-Hardy}.

\subsection{Large Deviations}\label{s.LDP-diffeos}
Given a bounded measurable function $\psi:X\to\R$, let
$$
\Lambda_\psi(t):=\limsup_{n\to+\infty}\tfrac{1}{n}\log \E_\mu[ e^{t\psi_n}],
\quad I_\psi(s):=\sup_{t\in\R}\{st-\L_\psi(t)\}.$$
$\Lambda_\psi$ is called the
{\em asymptotic log-moment generating function},  and $I_\psi$ is called the  {\em rate function}, see \cite{Dembo-Zeitouni}.

\begin{theorem}\label{t.LDP-diffeos}
Let $f$ be a $C^{1+}$ diffeomorphism of a closed manifold, let $X$ be an SPR Borel homoclinic class,  and let $\mu$ be the MME of $f|_X$.
For every $\beta>0$ there is $c>0$ with the following property.
Suppose $\psi:X\to\R$ is
$\beta$-quasi-H\"older with norm $\|\psi\|_\beta'=1$,  $\int\psi d\mu=0$,
and $\sigma_\psi^2\neq 0$. Then: \begin{enumerate}[(1)]
\item {For every closed set $F\subset\R$,
$
\displaystyle\limsup_{n\to\infty}\tfrac{1}{n}\log\mu\{x: \tfrac{1}{n}\psi_n(x)\in F\}\leq -\inf_F I_\psi.$}
\item {For every open set $G\subset\R$, $
\displaystyle\liminf_{n\to\infty}\tfrac{1}{n}\log\mu\{x: \tfrac{1}{n}\psi_n(x)\in G\}\geq -\inf_{G\cap (-c\sigma_\psi^4,c\sigma_\psi^4)} I_\psi.$}
\item
$\displaystyle\lim_{n\to\infty}\tfrac{1}{n}\log\mu\{x\,:\, \psi_n(x)\geq na\}\, =\, -I_\psi(a)$ for all $0<a<c\sigma_\psi^4$.
\item $I_\psi(a)=\tfrac {a^2}{2\sigma^{2}_\psi}(1+o(1))$ as $a\to 0^+$.
\end{enumerate}
\end{theorem}
\noindent
The proof, and additional information on  $\Lambda_\psi(t)$ and $I_\psi(s)$, are  in \S\ref{s.Proof-of-LDP-diffeos}.  {The special case of Anosov or Axiom A diffeomorphisms (see Example \ref{e.BHC-for-axiom-A}) is due to  Kifer \cite{Kifer}.}

\begin{remark}\label{r.LDP}
{Had the infimum in (2) been over  $G$ and not just  $G\cap (-c\sigma_\psi^4,c\sigma_\psi^4)$, then  (1) and (2) would have constituted the   ``large deviations principle," see \cite{Dembo-Zeitouni}. In \S\ref{s.Proof-of-LDP-diffeos} we will see that
 $
 \displaystyle\inf_{G\cap (-c\sigma_\psi^4,c\sigma_\psi^4)} I_\psi=\inf_G I_\psi
 $
 whenever $G$ intersects $[-c\sigma_\psi^4,c\sigma_\psi^4]$.}
\end{remark}

\subsection{Approximation by Brownian Motion}\label{s.ASIP-Statement}
The main result of this section is an almost sure   invariance principle (ASIP) in the spirit of Strassen \cite{Strassen}. 

We need  some classical definitions from Probability Theory \cite{Billingsley}.
A probability space $(\Omega,\mathfs F,\mu)$ is called {\em standard} if $\Omega$ is a complete separable metric space, $\mu$ is a Borel probability measure, and $\mathfs F$ is the Borel $\sigma$-algebra (which can always be completed with respect to $\mu$, though it is not always convenient for us to do so).
A (real-valued) {\em stochastic process} is a measurable parameterized family of real-valued functions
$(X_t)_{t\in T}$ on the same standard  probability space. In this paper, $T=\N, [0,T_0]\text{  or }[0,\infty)$ with the usual Borel structure, and
``measurability" means that $(t,\omega)\mapsto X_t(\omega)$ is Borel measurable.

Two real-valued stochastic processes $(X_t^{(i)})_{t\in T}$,  $i\in \{1,2\}$, on (possibly different) probability spaces $(\Omega_i, m_i)$ are said to be  {\em equal in distribution}, if for every $k\geq 1$, every $t_1<\cdots<t_k$ in $T$, and every finite family  of Borel sets $E_1,\ldots,E_k\subset \R$,
$$
m_i\{\omega\in \Omega_i: X_{t_j}^{(i)}(\omega)\in E_j\ \ \forall 1\leq j\leq k\}\text{ are equal for $i\in \{1,2\}$}.
$$
In this case,  the same events happen almost surely for $(X_t^{(1)})_{t\in T}$ and for $(X_t^{(2)})_{t\in T}$.

Suppose $T=[0,T_0]$ or $[0,\infty)$ and $(\Omega,\mathfs F,m)$ is a standard probability space. A parameterized family of functions $B_t:\Omega\to\R$ $(t\in T)$ is called  a {\em standard Brownian motion} on $(\Omega,\mathfs F,m)$, if
\begin{enumerate}[(1)]
\item $(t,\omega)\mapsto B_t(\omega)$ is measurable;
\item $B_0\equiv 0$, and $B_t-B_s$ has Gaussian distribution with mean zero and variance $|t-s|$, i.e.
    $m\{\omega:B_t(\omega)-B_s(\omega)<\tau\}=(2\pi|t-s|)^{-1/2}\int_{-\infty}^\tau e^{-x^2/2|t-s|}dx$;
\item For each $0<t_1<\cdots<t_{n+1}$, $B_{t_{i+1}}-B_{t_i}$ are independent random variables, i.e., $m(\bigcap_i\{\omega: B_{t_{i+1}}(\omega)-
    B_{t_{i}}(\omega)\in E_i \})=\prod_i m\{\omega: B_{t_{i+1}}(\omega)-
    B_{t_{i}}(\omega)\in E_i\}$ for every Borel sets $E_i\subset\R$.
\end{enumerate}

\begin{definition}[{\bf ASIP}]\label{def-ASIP}
We say that a stochastic process $(S_n)_{n\ge1}$ satisfies the \emph{almost sure invariance principle (ASIP)} with parameter $\sigma\ge0$ and rate $o(n^\gamma)$ for $0<\gamma<\tfrac{1}{2}$, if there exist two stochastic processes  $(\wt{S}_n)_{n\geq 1}$ and $(\wt{B}_t)_{t\geq 0}$ defined on a common standard probability space
such that
\begin{enumerate}[(1)]
\item the stochastic processes $(\wt{S}_n)_{n\geq 1}$ and  $(S_n)_{n\geq 1}$ are equal in distribution;
\item $({\wt B}_t)_{t\geq 0}$ is a standard Brownian motion;
\item $|\wt{S}_n-\sigma {\wt B}_n|=o(n^{\gamma})$ a.e. as $n\to\infty$. 
\end{enumerate}
\end{definition}

\begin{remark}
The value of $\sigma$  is uniquely defined. The ASIP implies that    $S_n/\sqrt{n}$ converges to a normal distribution, and $\sigma$ is the standard deviation of the limit.
\end{remark}

\begin{theorem}\label{t.ASIP-diffeo}
Let $f$ be a $C^{1+}$ diffeomorphism of a closed manifold, let $X$ be an SPR Borel homoclinic class, and let
$\mu$ be the MME of $f|_X$. Suppose $\psi:X\to \R$ is  quasi-H\"older with $\int\psi d\mu=0$.
Then $S_n(x):=\psi(x)+\psi(f(x))+\cdots+\psi(f^{n-1}(x))$
satisfies the ASIP 
with $\sigma=\sigma_\psi$. 
 \end{theorem}
\noindent
The proof is in \S\ref{s.ASIP-diffeos}.  {The Anosov case  is due to Denker \& Philipp \cite{Denker-Philipp}.}

\begin{remark}\label{r-dynamical-ASIP}
Lemma~\ref{l.Borelomania} below implies  the following strengthening of the ASIP, which we call the  ``\emph{dynamical ASIP}":
There is a Borel probability measure $\nu$ on $X\times[0,1]$
which projects to $\mu$ by $(x,t)\mapsto x$ such that
$\wt{S}_n$, $\wt{B}_t$ are defined on $(X\times[0,1],\nu)$ and
$\wt{S}_n(x,\xi)=\psi(x)+\dots+\psi(f^{n-1}(x)).$ Hence
\begin{equation}\label{e.ASIP-dynamical}
     \psi(x)+\psi(f(x))+\dots+\psi(f^{n-1}(x)) = \sigma \wt{B}_n(x,\xi)
          +o(n^{\gamma})\text{, as $n\to\infty$, $\nu$-a.e.}
\end{equation}
\end{remark}

\begin{remark}
It is likely that Thm~\ref{t.ASIP-diffeo} can be improved. Firstly, the error bound $o(n^{\frac{1}{4}+\eps})$ is probably not optimal \cite{Komlos-Major-Tusnady}. Secondly, the factor $[0,1]$ in the previous remark
is  an artifact of Lemma \ref{l.Borelomania}, and it would be nice to get rid of it. But is not clear that  these improvements would lead to new applications.
\end{remark}

\subsection{Functional Central Limit Theorem}\label{s.FCLT}
The ASIP discusses the a.s. behavior of the graph of  $n\mapsto \psi_n(x)$ for  {\em fixed}  $x$. The functional CLT describes the  behavior of the {\em ensemble} of graphs of $n\mapsto \psi_n(x)$, with   $x$ ranging over $(X,\mu)$.

{We first recall some basic definitions from probability theory \cite{Billingsley}. Suppose $(\Omega_n,\mathfs F_n,\mu_n)$ $(n\geq 0)$  are probability spaces, and let $(Y,d)$ be a complete separable metric space, equipped with its Borel $\sigma$-algebra $\mathfs B(Y)$. Let $Z_n:\Omega_n\to Y$ be measurable functions. We say that $Z_n$ {\em converges in distributions to $Z_0$}, and write
$$
Z_n\xrightarrow[n\to\infty]{dist} Z_0,
$$
if $\int g(Z_n)d\mu_n\to \int g(Z_0)d\mu_0$ for every bounded continuous function $g:Y\to\R$. }

Measurable functions $Z:\Omega\to Y$ are called  {\em $Y$-valued random variables} (on $\Omega$). We will be interested in the case  $Y=\mathcal C([0,1]):=\{\omega:[0,1]\to\R:\omega\text{ is continuous}\}$, equipped with the metric $d(\omega_1,\omega_2):=\max|\omega_1-\omega_2|$ and the Borel $\sigma$-algebra.

\begin{example}[{\bf Brownian Motion}]\label{e.Wiener}
Wiener constructed a measure $\mu_W$ on the space $\mathcal C([0,1])$, called {\em Wiener's measure}, so that the
 one-parameter family of measurable functions $B_t\colon \mathcal{C}([0,1])\to \R$ $(t\in [0,1])$
defined by $B_t(\omega):=\omega(t)$ is a standard Brownian motion on $(\mathcal{C}([0,1]),\mu_{W})$.
It is useful to introduce
$$
\overline{B}:\mathcal{C}([0,1])\to \mathcal{C}([0,1])\ ,\ \overline{B}(\omega):=\omega.
$$
This $\mathcal{C}([0,1])$-valued random variable models a random Brownian {\em path} during [0,1].
\end{example}

\begin{example}[{\bf Renormalized Linear Interpolations}]\label{e.linear-interp}
 Let $X$ be a Borel homoclinic class of a diffeomorphism $f$. Assume $f|_X$ is SPR and let $\mu$ be its MME.
For every $x\in X$ and $n\geq 1$, we define
$
\overline{\psi}_n(x)\in \mathcal{C}([0,1])
$ to be the piecewise linear function on $[0,1]$, which is affine on the intervals $[\frac{k-1}{n},\frac{k}{n}]$ $(1\leq k\leq n)$, and with  the following values at the endpoints:
$$
\overline{\psi}_n(x)(0):=0 \text{ and }\overline{\psi}_n(x)(\tfrac{k}{n})=\frac{\psi_k(x)}{\sqrt{n}}.
$$
We see it as a $\mathcal{C}([0,1])$-valued random variable on the probability space $(X, \mu)$.
\end{example}

\begin{corollary}[{\bf Functional CLT}]\label{c.FCLT-Diffeos}
In the setting of  Thm~\ref{t.ASIP-diffeo}, $$
\overline{\psi}_n\xrightarrow[n\to+\infty]{dist}\sigma_\psi \overline{B}.
$$
\end{corollary}
\noindent
{This is a direct consequence of the ASIP, see \cite{Billingsley} or {Thm~\ref{thm-FCLT-from-ASIP}}.

\subsection{Law of the Iterated Logarithm (LIL)}\label{s.LIL-diffeos} The LIL gives the optimal almost sure bounds for the growth of $\psi_n(x)$, as $n\to\infty$.
\begin{corollary}\label{c.LIL} In the setting of  Thm~\ref{t.ASIP-diffeo}, if $\sigma_\psi\ne0$,
for $\mu$-a.e. $x$ and every $0< c< 1$ we have {\em the law of the iterated logarithm:}
\begin{equation} \limsup_{n\to\infty}\frac{\psi_n(x)}{\sigma_\psi\sqrt{2n\log\log n}}=1, \
\liminf_{n\to\infty}\frac{\psi_n(x)}{\sigma_\psi \sqrt{2n\log\log n}}=-1.\label{e.LIL1}
\end{equation} We also have as well as {\em Strassen's identity:}
\begin{equation} \limsup_{N\to\infty} \tfrac{1}{N}\#\left\{1\leq n\leq N: \psi_n(x)>c \sigma_\psi\sqrt{2n\log\log n}\right\}=1-e^{-4(c^{-2}-1)}.\label{e.LIL2}
\end{equation} \end{corollary}
\noindent
This is a consequence of the
ASIP, see \cite{Billingsley} or Thm~\ref{thm-LIL-from-ASIP}.

\medbreak

Strassen noticed an amusing consequence of  \eqref{e.LIL2}: The upper density of $n$ such that  $\psi_n(x)>\tfrac 1 2\sigma_\psi\sqrt{2n\log\log n}$ is a.s. bigger than  $0.99999$, but smaller than $1$.

\subsection{Arcsine Law}
\begin{corollary}\label{c.Arcsine}
Let
$
d_n(x):=\frac{1}{n}\#\{1\leq k\leq n: \psi_k(x)>0\}
$.
Under the assumptions of Thm~\ref{t.ASIP-diffeo}, if $\sigma_\psi\ne0$, then for every $s\in [0,1]$,
\begin{equation}\label{e.arcsine}
\lim_{n\to\infty}\mu\{x\in X: d_n(x)\leq s\}=\tfrac{2}{\pi}\arcsin(\sqrt{s}).
\end{equation}
\end{corollary}
\noindent
This is a consequence of the} ASIP, see \cite{Billingsley} or Thm~\ref{thm-Arcsine-from-ASIP}.

To understand the arcsine law, consider the density  of the arcsine distribution,  $\frac{d}{ds}(\frac{2}{\pi}\arcsin\sqrt{s})=\pi^{-1}[s(1-s)]^{-1/2}$. This function has a global  minimum at $s=\frac{1}{2}$ and increases to infinity as $s\downarrow 0$ or $s\uparrow\infty$. 
 Hence, \eqref{e.arcsine} says that the most likely scenario is that most of $\psi_k(x)$ have the same sign, and the least likely scenario is that  roughly half are positive and roughly half are negative.

\subsection{Law of Records}\label{s.records}
\begin{corollary}\label{c.Records}
Under assumptions of Thm~\ref{t.ASIP-diffeo},
if $\sigma_\psi\ne0$,
for every $s>0$,
\begin{equation}\label{e-Records}
\mu\left\{x\in X: \tfrac{1}{\sqrt{n}}\max_{1\leq k\leq n}\psi_k(x)\geq s\right\}\xrightarrow[n\to\infty]{}\sqrt{\tfrac{2}{\pi\sigma^2_\psi}}
\int_s^\infty e^{-t^2/2\sigma^2_\psi}dt.
\end{equation}
\end{corollary}
\noindent
This is a consequence of the Functional CLT, 
see \cite{Billingsley} or Thm~\ref{thm-Records-from-ASIP}}.

\subsection{Sharp Expansion Bounds}\label{ss.expansion-bounds}
The ASIP has dynamical consequences. Recall the definition of  $\Lambda^+(\mu)$ from \S \ref{ss.def-of-Lambda-plus}.
\begin{corollary}\label{t.expansion-bounds}
Let $X,f,\mu$ be as in \S\ref{s.setup-for-consequences}.
If $X$ contains two hyperbolic periodic orbits with different sums of positive Lyapunov exponents (counted with multiplicity), then there is $\sigma>0$ such that for $\mu$-a.e. $x\in X$ the following properties hold.
\begin{enumerate}[(1)]
\item {\bf Law of Iterated Logarithm\/}: For all $c>1$, for all $n$ large enough,
$$
e^{n\Lambda^+(\mu)-c\sigma\sqrt{2n\log\log n}}
\leq |\det (Df^n_x|_{E^u(x)})|\leq e^{n\Lambda^+(\mu)+c\sigma\sqrt{2n\log\log n}}.
$$
\item
{\bf These Bounds are Sharp:} For all $0<c<1$,
\begin{align*}
&\limsup_{N\to\infty}\tfrac{1}{N}\#\{1\leq n\leq N:
 |\det (Df^n_x|_{E^u(x)})|> e^{n\Lambda^+(\mu)+c\sigma\sqrt{2n\log\log n}}\}\\
&=\limsup_{N\to\infty}\tfrac{1}{N}\#\{1\leq n\leq N:
|\det (Df^n_x|_{E^u(x)})|< e^{n\Lambda^+(\mu)-c\sigma\sqrt{2n\log\log n}}\}\\
&=1-e^{-4(c^{-2}-1)}.
\end{align*}

\smallskip
\item {\bf Arcsine Law:}
Let $
d_N(x):=\frac{1}{N}\#\{1\leq n\leq N: |\det (Df^n_x|_{E^u(x)})|>e^{n\Lambda^+(\mu)}\},
$
then
$\displaystyle
\lim_{N\to\infty}\mu\{x: d_N(x)\leq s\}=\frac{2}{\pi}\arcsin(\sqrt{s})\text{ for all $s{\in[0,1]}$.}
$

\smallskip
\item {\bf Large Deviations:} As $\varepsilon\to 0$,
$$
\lim_{n\to\infty}\tfrac1n\log \mu\left\{x: |\det(Df^n|_{E^u(x)})|\geq e^{n(\Lambda^+(\mu)+\eps)}\right\}= -\tfrac{\eps^2}{2\sigma^2}(1+o(1)),
$$

$$
\lim_{n\to\infty}\tfrac1n\log \mu\left\{x: |\det(Df^n|_{E^u(x)})|\leq e^{n(\Lambda^+(\mu)-\eps)}\right\}= -\tfrac{\eps^2}{2\sigma^2}(1+o(1)).
$$
\end{enumerate}
\end{corollary}
\begin{proof}
By the chain rule, $|\det(Df^n_x|_{E^u(x)})|=\exp\left(\psi_n(x)+n \Lambda^+(\mu)\right)$, where
$$
\psi(x):=\log|\det(Df_x|_{E^u(x)})|-\Lambda^+(\mu).
$$
This function is quasi-H\"older on $X$, by Example \ref{e.Geometric-Is-Quasi-Holder}, and $\int\psi d\mu=0$ because by the Oseledets theorem and the ergodicity of $\mu$,
$$
\int\log|\det(Df_x|_{E^u(x)})|d\mu =\lim_{n\to
\infty}\tfrac{1}{n}\log|(\det Df^n_x|_{E^u(x)})|=\Lambda^+(\mu).
$$
By Thm \ref{t.asymp-var-diffeos}(5) and the assumptions on $X$,
 $\sigma:=\sigma_\psi\neq 0$. The corollary now follows from 
Cor~\ref{c.LIL} and \ref{c.Arcsine}, and Thm~\ref{t.LDP-diffeos}.
\end{proof}
\begin{remark}
 Similar results for $Df|_{E^s}$ follow by applying the theorem to $f^{-1}$.
 \end{remark}
 \begin{remark}
 Note that in dimension two, $|\det(Df^n|_{E^u(x)})|=\|Df^n|_{E^u(x)}\|$.
\end{remark}

\subsection{Effective Intrinsic Ergodicity}\label{ss.Kadyrov-for-Diffeos}
The following statements imply Thm~\ref{t.effective}.
Since $f|_X$ has a unique MME $\mu$, every measure with entropy $h_{\Bor}(f|_X)$ equals  $\mu$.  By the following result, measures with entropy {\em close} to $h_{\Bor}(f|_X)$ are {\em close} to $\mu$:

\begin{theorem}\label{t.Kadyrov-Ineq-diffeos}
Let $f$ be a $C^{1+}$ diffeomorphism of a closed manifold, let $X$ be an SPR Borel homoclinic class, and let
$\mu$ be the MME of $f|_X$.
For every $\beta>0$ there is a number $C>0$ as follows. For every $\beta$-quasi-H\"older observable $\psi$ on $X$ and every non-necessarily ergodic $f$-invariant probability measure $\nu$ on $X$,
$$
\left| \int\psi d\mu-\int\psi d\nu \right|\leq C\|\psi\|_\beta'\sqrt{h(f,\mu)-h(f,\nu)}.
$$
\end{theorem}
\noindent
In the special case of Anosov diffeomorphisms (see Example \ref{e.BHC-for-axiom-A}),
Thm~\ref{t.Kadyrov-Ineq-diffeos} is due to Kadyrov \cite{Kadyrov-Effective-Uniqueness}.
The next theorem is more precise, when $h(f,\nu)\to h(f,\mu)$:

\begin{theorem}\label{t.Kadyrov-Ineq-diffeos-Strong}
Under the assumptions of Thm~\ref{t.Kadyrov-Ineq-diffeos},
for every $\beta>0$, for every $\eps>0$ small enough and for every $\beta$-quasi-H\"older function $\psi\colon X\to \RR$ with  asymptotic variance $\sigma_\psi^2$ (w.r.t. $\mu$), there is $\delta>0$ as follows. For every non-necessarily ergodic $f$-invariant measure $\nu$ on $X$ such that $h(f,\nu)>h(f,\mu)-\delta$,
$$
\left|\int \psi d\mu-\int \psi d\nu\right|\leq e^{\eps}\sqrt{2\sigma_\psi^2(h(f,\mu)-h(f,\nu))}.
$$
\end{theorem}
\begin{remark}\label{r.Kadyrov-delta}The proof shows that the factor $2\sigma_\psi^2$ is optimal and that there are
$\tilde{\eps},\tilde{C}>0$ which depend only on $X$ and $\beta$ (but not on $\psi$), so that
the theorem holds with
$\delta:=\tilde{C}\min(\eps,\tilde{\eps})^6\big(\min(1,{\sigma_\psi}/{\|\psi-\mu(\psi)\|_\beta'})\big)^{14}$
for all $\eps>0$.
\end{remark}
\noindent
Theorems \ref{t.Kadyrov-Ineq-diffeos} and \ref{t.Kadyrov-Ineq-diffeos-Strong} follow from known properties of SPR Markov shifts \cite{Ruhr-Sarig}, see \S\ref{s.Kadirov-Proof-diffeos}.
They have several interesting consequences.

\begin{corollary}\label{c.Kadyrov1}
Under the assumptions of Thm~\ref{t.Kadyrov-Ineq-diffeos},
if  $\nu_n$ are {$f$-invariant measures on $X$} and  $h(f,\nu_n)\to h(f,\mu)$, then $\nu_n\to\mu$ weak-$*$ on $M$.
\end{corollary}
\noindent
The proof is immediate, but the result is not trivial, because   $X$ may be non-closed,  and the entropy map  need not be upper semi-continuous (so it is not a simple consequence of Newhouse existence theorem for $C^\infty$ maps).

Next we compare the Lyapunov exponents of $\nu_n$ and $\mu$.
Recall that $\Lambda^+(\nu)$ and $\Lambda^-(\nu)$ are the integrals over all ergodic components of $\nu$ of the sum of positive (resp. negative) Lyapunov exponents with multiplicity (see \S\ref{ss.def-of-Lambda-plus}).
\begin{corollary}\label{t.Lyap-Stab}
Let $X,f,\mu$ be as in \S\ref{s.setup-for-consequences}. Then there exists $C>0$ so that for any $f$-invariant probability measure $\nu$ on $X$
\begin{align*}
&|\Lambda^\pm(\mu)-\Lambda^\pm(\nu)|\leq C\sqrt{h(f,\mu)-h(f,\nu)}.
\end{align*}
In particular, if $\nu_n\in\Prob(f|_X)$ and $h(f,{\nu_n})\to h(f,\mu)$, then $\Lambda^\pm({\nu_n})\to\Lambda^\pm({\mu})$.
\end{corollary}
\begin{proof}
Apply Thm~\ref{t.Kadyrov-Ineq-diffeos} to the geometric potentials $J^s,J^u$ in Example \ref{e.Geometric-Is-Quasi-Holder}.\end{proof}

 \begin{remark}
In dimension two, 
the above gives a rate of convergence of the Lyapunov exponents of measures $\nu_n$ such that $h(f,\nu_n)\to h(f,\mu)$, to the Lyapunov exponents of $\mu$.
 \end{remark}

\begin{corollary}\label{c.Kadyrov2}
Let $f$ be a SPR $C^{1+}$ diffeomorphism of a closed manifold and let $\nu_n\in \Proberg(M)$ be a weak-$\ast$ convergent sequence of measures with some limit $m$. If  $h(f,\nu_n)\to h_\top(f)$, then  $m$ is a hyperbolic ergodic MME and $\Lambda^\pm({\nu_n})\to\Lambda^\pm(m)$.
\end{corollary}
\begin{proof}
By Prop~\ref{p.decomposition}, for some $h_0<h_\top(f)$ there exists only finitely many Borel homoclinic classes with entropy larger than $h_0$; moreover they are all SPR and their union carries all the measures $\nu_n$ for $n$ sufficiently large. One can thus find a Borel homoclinic class $X$ which supports all the measures from a subsequence of $(\nu_n)$. In particular $h_\Bor(X)=h_\top(f)$. Since $X$ is SPR, it supports a hyperbolic ergodic MME $\nu$. So  $h(f,\nu_n)\to \hTOP(f|_X)=h(f,\nu)$ and Cor~\ref{c.Kadyrov1} says that $\mu=\nu$.
\end{proof}

In fact the assumption $h(f,\nu_n)\to h(f,\mu)$ for non-atomic invariant measures $\nu_n$ on $X$ implies a stronger form of convergence, related to optimal transport:
We can deform $\nu_n$ into $\mu$ using a mass preserving map $T$ which ``does not alter $x$, $E^u(x)$ or $E^s(x)$ too much, on average''.
To make this precise, let $d(\cdot,\cdot)$ denote the Riemannian distance function on $M$, and let  $\dist$ denote the Riemannian distance function on the disconnected union of the Grassmannian bundles of $M$,  $\bigsqcup_{1\leq k\leq \dim M} \mathfrak G(k,M)$    (Appendix~\ref{app-grassmannian}), with the convention that the distance between points in different connected components is $\diam(M)$.
\begin{corollary}\label{t.Osceledets-Stab}
For any $X,f,\mu$ as in \S\ref{s.statement-of-conseq}, there exists $C>0$ as follows: For every non-atomic invariant probability  measure $\nu$ on $X$, there exists a Borel one-to-one map $T:M\to M$ so that $T_\ast\nu=\mu$, and
\begin{align}
&\int_M c(x,Tx)\,d\nu\leq C\sqrt{h(f,\mu)-h(f,\nu)}, \label{e.Transport}
\end{align}
where $c(x,y):=d(x,y)+\dist(E^u(x),E^u(y))+\dist(E^s(x),E^s(y))$.
\end{corollary}
\noindent
The corollary follows from Thm~\ref{t.Kadyrov-Ineq-diffeos}, and standard techniques in the theory of optimal transport, see \S\ref{s.Proof-Oseledets-Stab}.

\subsection{Exponential Tails in Pesin Theory}\label{s.Conseq-Pesin-Theory}
The study of L.-S. Young towers \cite{Young-Towers-Annals} leads naturally to the question of the rate of decay of the tail of the first entrance time to a Pesin set.   S. Luzzatto suggested that it would also be interesting to estimate the measure of the complement of Pesin sets ``with large constants" (private communication).
The next result (containing Cor~\ref{c.tail} and part of Thm~\ref{t.tail}), proved in \S\ref{s.Tail-Diffeos}, is in this direction.

\begin{theorem}\label{t.Tail-Estimates-diffeos}
Suppose $X,f,\mu$ are as in \S\ref{s.setup-for-consequences}. Then {for every $\chi>0$ small enough and for all $0<\eps<\chi$,} the following holds.
\begin{enumerate}[(1)]
\item {\em {\bf First Entrance Times to Pesin blocks:}\/} There exist a $(\chi,\eps)$-Pesin block $P$ and $0<\theta<1$ so that
    $
    \tau_P(x):=\inf\{n\geq 1: f^n(x)\in P\}
    $
    satisfies $$\mu\{x\in X: \tau_P(x)>n\}=O(\theta^n)\text{ as $n\to\infty$.}$$
\item {\em {\bf Optimal Pesin Bounds:}\/} Let $K_\ast(x)$ denote the optimal $(\chi,\eps)$-Pesin bound (see Def.~\ref{d.Optimal-Pesin-Constant}). Then there exists  $0<\theta<1$ such that
    $$
    \mu\{x\in X: \log K_\ast(x)>n\}=O(\theta^n),\text{ as }n\to\infty.
    $$
\end{enumerate}
\end{theorem}

\subsection{Equilibrium States (Thm~\ref{t.equilibrium} Items~(1) and (2))}\label{ss-prop-equilibrium}
So far we have discussed measures of maximal entropy, but many of our  results extend without much difficulty to equilibrium measures of SPR quasi-H\"older potentials. 

Specifically, the proofs of Thms \ref{t.MME-exists}, \ref{t.DOC-diffeos}, \ref{t.asymp-var-diffeos}, \ref{t.LDP-diffeos}, \ref{t.ASIP-diffeo} and Cor~\ref{c.CLT-diffeos},
given in \S\ref{s.Conseq-of-SPR-Proofs-Diffeos} below are based (a) on the coding of SPR Borel homoclinic classes by SPR Markov shifts stated in \S\ref{section-Sigma} and (b) on results for SPR irreducible Markov shifts in Appendix~\ref{appendix-Markov-Shifts}.
In the case of a Borel homoclinic class $X$ which is SPR for a quasi-H\"older potential $\phi\colon X\to \RR$,
there also exists a coding which is SPR for the (H\"older continuous) lifted potential $\phi\circ \pi$ (see Addendum~\ref{a.SPR-coding}).
We have taken care to prove the  results in Appendix~\ref{appendix-Markov-Shifts} for general H\"older continuous potential. This allows a straightforward extension of the proofs in \S\ref{s.Conseq-of-SPR-Proofs-Diffeos} and \S\ref{section-Sigma} from the case of the zero potential to the case of quasi-H\"older SPR potentials.

This gives the existence, uniqueness, and structure of equilibrium states for the potential $\phi$ on $X$,
the exponential decay of correlations (when the period is $p=1$) and the large deviation property, the almost sure invariance principle,
the properties of the asymptotic variance, i.e. Items (1) and (2) of Thm~\ref{t.equilibrium}.

\subsection{No Phase Transitions In High Temperature (Thm~\ref{t.equilibrium} Item~(3))}\label{s.TDF-for-diffeos}

The thermodynamical formalism is based on analogies between certain dynamical notions and corresponding   notions  in equilibrium statistical physics:
\begin{enumerate}[$\bullet$]
\item The average of a potential
 $\int\phi\,d\mu$ corresponds to $-\frac{\const}{T} U$ where $U$ is the ``specific energy" of a macroscopic state $\mu$ at some ``temperature" $T$.
\item The pressure $h(f,\mu)+\int\phi d\mu=-\frac{\const}{T}F$, where $F$ is the  ``specific free energy." 
\end{enumerate}
In equilibrium statistical physics, the first derivatives of the free energy are thermodynamic quantities and its second derivatives are linear response functions. In Ruelle's theory,   a ``first-order phase transition" is a situation when some directional derivative $\frac{d}{dt}|_{t=0}P_{\top}(\phi+t\psi)$ does not exist, and a ``high-order phase transition" when some higher directional derivative  $\frac{d^k}{dt^k}|_{t=0}P_{\top}(\phi+t\psi)$ does not exist.

The following theorem (which corresponds to the third item of Thm~\ref{t.equilibrium}), when applied to $\phi:=-\frac{const}{T}U$ with $T$ sufficiently large,  says that for quasi-H\"older energy functions $U$,  ``there are no phase transitions at sufficiently high temperatures".

\begin{theorem}\label{t.no-phase-transition-diffeo}
Given a $C^{1+}$ diffeomorphism $f$ of a closed manifold
which is SPR on a Borel homoclinic class $X$ for some quasi-H\"older potential $\phi\colon X\to \RR$,
there exists $\eps_0>0$ as follows. If $\psi:X\to\R$ is quasi-H\"older with $\sup|\psi|<\eps_0$, then
$f$ is SPR on $X$ for $\phi+\psi$, and
$t\mapsto P_{\Bor}(\phi+t\psi)$ is real-analytic on $(-1,1)$.
\end{theorem}
\begin{proof}
Code $X$ by an irreducible SPR shift, and apply Thm~\ref{t.no-phase-transition}.
\end{proof}

\section{Proofs of the Consequences}\label{s.Conseq-of-SPR-Proofs-Diffeos}
\subsection{Standing Assumptions I}
Throughout this section,  $f$ is a $C^{1+}$ diffeomorphism of a closed manifold $M$
and $X$ is an SPR Borel homoclinic class with period $p$. Without loss of generality we may assume that
the entropy $h_{\Bor}(f|_X)$ is positive since otherwise all our results  hold  trivially, or vacuously.

For all $0<\chi<\chi'$, every ergodic measure which charges some $(\chi',\eps)$-Pesin block must be $\chi$-hyperbolic.
Let us denote:
$$\Proberg(\chi,f|_X)=\{\nu\in \Proberg(f|_X):\nu\text{ is $\chi$-hyperbolic}\}.$$
By the  SPR property, there exist $\chi,\kappa>0$ such that
\begin{equation}\label{e.chi-kappa-def}
\left.\begin{array}{c}
h(f,\nu)>\hTOP(f|_X)-\kappa\\
\nu\in\Proberg(f|_X)
\end{array}\right\}\Longrightarrow
\nu\in \Proberg(\chi,f|_X).
\end{equation}
We can take $\chi,\kappa$ arbitrarily small, because decreasing $\chi,\kappa$  weakens \eqref{e.chi-kappa-def}.
In particular  we may choose $\chi$ so small that Thm~\ref{t.SPR-coding} yields an irreducible SPR coding $\pi:\Sigma\to M$ satisfying Properties ($\Sigma$\ref{i.Sigma1})--($\Sigma\ref{i.Sigma6}$),  \eqref{e.chi-entropy-hyperbolic} for $f|_X$ and $\hTOP(\Sigma)=\hTOP(f|_X)$.

Let us consider the spectral decomposition of $\Sigma$ (as in Lemma~\ref{l.spectral-decomposition-CMS}):
$$\Sigma=\biguplus_{j=0}^{q-1}\sigma^j(\Sigma').$$ In particular, $\sigma^q:\Sigma'\to\Sigma'$ is naturally conjugated to an {\em aperiodic} irreducible Markov shift. The period $q$ of $\Sigma$ may a priori be distinct from the period $p$ of $X$.

\begin{lemma}\label{l.psi-hat-Holder}
For any $\beta>0$,
there are numbers $C=C(\Sigma,\pi,\beta)$ and $\wh{\beta}=\wh{\beta}(\Sigma,\pi,\beta)$ such that for every $\beta$-quasi-H\"older $\psi:X\to\R$, there exists a $\wh{\beta}$-H\"older continuous function $\hpsi:\Sigma\to\R$ which satisfies $\hpsi=\psi\circ\pi$ {on $\pi^{-1}(X)$,} and
$
\|\hpsi\|_{\wh{\beta}}\leq C\|\psi\|_\beta'.
$
\end{lemma}
\begin{proof}
Since $\psi$ is $\beta$-quasi-H\"older,  there exists a $\beta$-H\"older continuous $\Psi:\mathfrak G\to\R$ such that $\|\Psi\|_\beta\leq 2\|\psi\|_\beta'$, and  $\psi=\Psi\circ\imath$ on $X$, where $\imath(x)=(x,E^u(x);x,E^s(x))$.

Let $\hpsi=\Psi\circ\imath\circ\pi$, then $\hpsi=\psi\circ\pi$ on $\pi^{-1}(X)$.
By Properties ($\Sigma$\ref{i.Sigma2}) and ($\Sigma$\ref{i.Sigma5}) (in \S\ref{section-Sigma}),  $\imath\circ\pi$ is  $\gamma$-H\"older continuous for some $\gamma$; so $\hpsi$ is $\wh{\beta}$-H\"older continuous with $\wh{\beta}:=\gamma\beta$, and
$
\|\hpsi\|_{\wh{\beta}}\leq \|\Psi\|_\beta \|\imath\circ\pi\|_{\gamma}^{\beta}\leq C\|\psi\|_\beta'$, with  $C:=2\|\imath\circ\pi\|_{\gamma}^{\beta}$.
\end{proof}
\subsection{Existence of the Measure of Maximal Entropy (Thms \ref{t.MME} and \ref{t.MME-exists})}\label{s.Proof-existence-MME-diffeo}
Since $\Sigma$ is SPR and irreducible,  $\Sigma$ is positively recurrent and has a unique MME (Thm~\ref{t.Gurevich-Theorem}).
Call this measure $\hmu$, and let  $\mu:=\pi_*\hmu$. By ($\Sigma$\ref{i.Sigma3}) in \S\ref{section-Sigma}, $\mu$ is carried by $X$, and $h(f,\mu)=h(\sigma,\hmu)=h_{\Bor}(\Sigma)=h_{\Bor}(f|_X)$. So   $\mu$ is a MME of $f|_X$.
This proves the first part of Thm~\ref{t.MME-exists}.
The other parts of the theorem now follow from known results, see Remark \ref{r.Thm-13.1}.
\hfill$\Box$

\medbreak

\subsection{Standing Assumptions II} In addition to Assumptions I, we let  $\mu$ be the MME of $f|_X$, and  $\hmu$ denote the MME of $\Sigma$.
Recall that
$p$ is the period of $X$ (and therefore of $\mu$), and that $q$ is the period of $\Sigma$.

\begin{lemma}\label{l.p|q}
$
\mu=\hmu\circ\pi^{-1}
$
and
$
p|q.
$
\end{lemma}
\begin{proof}
Suppose $\mu$ is an MME of $f|_X$. By the affinity of the entropy function, a.e. ergodic component of $\mu$ has entropy $h_{\Bor}(f|_X)=h_{\Bor}(\Sigma)=h$. By \eqref{e.chi-kappa-def}, such measures are  $\chi$-hyperbolic. By ($\Sigma$\ref{i.Sigma4}), they lift to ergodic measures with entropy $h$  on $\Sigma$. But on $\Sigma$, there is just one such measure: $\hmu$ (Thm~\ref{t.Parry-Gurevich-ASS}). Consequently,  a.e. ergodic component of $\mu$ equals $\pi_*\hmu$, whence  $\mu=\pi_*\hmu$.

By Thm~\ref{t.MME-exists}(2), there is a measurable function $g:X\to\C$ such that $g\circ f=e^{2\pi i/p}g$
and $|g(x)|=1$ for all $x$. Let $G:=g\circ\pi$; then $G$ is defined $\hmu$-a.e., see ($\Sigma$\ref{i.Sigma3}). Using the commutation relation $\pi\circ\sigma=f\circ\pi$, we obtain  $G\circ\sigma^q=e^{2\pi i q/p}G$ $\hmu$-a.e.

However, by Thm~\ref{t.Parry-Gurevich-ASS}(1) (with $q$ replacing what is denoted there $p$),  each of the ergodic components of $(\Sigma,\sigma^q,\hmu)$ is Bernoulli, whence mixing. So $G$ is a.e. constant, and $q/p$ is an integer.
\end{proof}

\subsection{Exponential Decay of Correlations (Thms~\ref{t.main} and \ref{t.DOC-diffeos})}\label{s.decay-of-corr-diffeo-proof}
First, we prove  Thm~\ref{t.DOC-diffeos};  Then, {Thm~\ref{t.main} will follow from Remark \ref{r.doc-for-top-mixing-surface-diffeos}.}

By Thm~\ref{t.MME-exists}(2), if the period $p$ of $\mu$ is $1$, then $\mu$ is mixing. Otherwise, $( X,\mu,f^p)$ has $p$ ergodic components $\mu'_i:=\mu'\circ f^{-i}$ ($0\leq i\leq p-1$), each Bernoulli with respect to $f^p$.
We will show that for every $\beta$  there are constants $C>0, 0<\theta<1$ such that for all $\beta$-quasi-H\"older $\psi,\vf:X\to\R$
and for all $n\geq 0$,
\begin{equation}\label{e.aim-i}
\mathrm{Cov}_{\mu'_i}(\vf,\psi\circ f^{pn})<C{\|\psi\|_{\beta}'\|\vf\|_{\beta}'}\theta^{pn}.
\end{equation}
It is sufficient to do this for $i=0$, i.e. for $\mu_i'=\mu'$.

\begin{lemma}
$\mu'=\pi_\ast\hmu_j$, where  $\hmu_j$ is an ergodic component of $\hmu$ for the map $\sigma^q$.
\end{lemma}
\begin{proof}The spectral decomposition $\Sigma=\biguplus_{j=0}^{q-1}\sigma^j(\Sigma')$ induces the decomposition
$
\displaystyle
\hmu=\tfrac{1}{q}\sum_{j=0}^{q-1}\hmu'\circ\sigma^{-j},\text{ where }\hmu'$
is the renormalization of the restriction of $\hmu$ to $\Sigma'$.

Clearly,   $\hmu'$ is the  unique MME of $\sigma^q|_{\Sigma'}$, and since $\sigma^q|_{\Sigma'}$ is conjugate to an {\em aperiodic} irreducible Markov shift, $\hmu'$ is $\sigma^q$-Bernoulli, see Thm~\ref{t.Parry-Gurevich-ASS}(1).
By Lemma~\ref{l.p|q},
 \begin{equation}\label{e.p-q-J}
\tfrac{1}{q}\sum_{j=0}^{q-1}(\hmu'\circ\pi^{-1})\circ f^{-j}=\hmu\circ\pi^{-1}=\mu=\tfrac{1}{p}\sum_{i=0}^{p-1}\mu'\circ f^{-i}.
\end{equation}
We claim that
 \eqref{e.p-q-J} gives two ergodic decompositions for $\mu$ with respect to $f^{pq}$. (If $q>p$, then some summands on the left appear with multiplicity.)
Indeed:
\begin{enumerate}[\quad$\circ$]
\item Since $\hmu'$ is $\sigma^q$-Bernoulli, $\hmu'$ is $\sigma^{pq}$-ergodic and invariant. {Therefore $\hmu'\circ\pi^{-1}\circ f^{-j}$ are $f^{pq}$-ergodic and invariant.}

\item
Since $\mu'$ is $f^p$-Bernoulli, $\mu'$ is $f^{pq}$--ergodic and invariant. Therefore $\mu'\circ f^{-i}$ are $f^{pq}$--ergodic and invariant.
\end{enumerate}

By the uniqueness of the ergodic decomposition, since $\mu'$ appears on the right side of \eqref{e.p-q-J}, it must also appear on its  left side. Choose $0\leq j\leq q-1$ such that
$$\mu'=\hmu'\circ\pi^{-1}\circ f^{-j}=\hmu'\circ\sigma^{-j}\circ\pi^{-1}\equiv\hmu_j'\circ\pi^{-1},\text{where $\hmu'_j:=\hmu'\circ\sigma^{-j}$.}$$
Since $\hmu'$ is $\sigma^q$-Bernoulli,    $\hmu'_j$ is $\sigma^q$-Bernoulli, whence $\sigma^q$-ergodic. By construction, $\hmu=\frac{1}{q}\sum_{\ell=0}^{q-1}\hmu'_j\circ\sigma^{-\ell}$. Necessarily,
 $\hmu_j'$ is one of the   $\sigma^q$-ergodic components of $\hmu$.
\end{proof}

\smallskip
\noindent
{\em Proof of Theorem~\ref{t.DOC-diffeos}.\/}
Let $C$ and $\hbeta$ be the constants from Lemma \ref{l.psi-hat-Holder}. Then for every $\beta$-quasi-H\"older functions $\vf$ and $\psi$ on $X$,  there are two $\hbeta$-H\"older continuous functions $\hpsi,\hvf:\Sigma\to\R$ such that
$\hpsi=\psi\circ\pi$, $\hvf=\vf\circ\pi$ {on $\pi^{-1}(X)$,} and
$
\|\wh{\varphi}\|_{\wh{\beta}}\leq C\|\vf\|_{\beta}'\text{ , }\|\wh{\psi}\|_{\wh{\beta}}\leq C\|\psi\|_{\beta}'.
$
{Since $\hmu\circ\pi^{-1}=\mu$ and $\mu$ is carried by $X$,
$
\hpsi=\psi\circ\pi\text{ and }\hvf=\vf\circ\pi\ \text{$\hmu$-a.e.}
$}
{Note that for each $n\geq 0$, $\|\hpsi\circ\sigma^{n}\|_{\hbeta}\leq \hbeta^{-n}\|\hpsi\|_{\hbeta}\leq C\hbeta^{-n}\|\psi\|_\beta'$.  }

In Appendix \ref{appendix-Markov-Shifts}, we state and prove the following result (Thm~\ref{t.DOC-for-CMS}): {\em Every  $\sigma^q$-ergodic component of an MME on an SPR irreducible Markov shift with period $q$, has exponential decay of correlations for $\hbeta$-H\"older continuous observables.}

Applying this to $\hmu_j'$, gives $C_{j}>0$, $0<\theta_{j}<1$ which depend only on $\Sigma,\pi,\beta$, s.t. $$|\mathrm{Cov}_{\hmu'_j}(\hvf,(\hpsi\circ \sigma^r)\circ \sigma^{q\ell})|\leq C_{j}\|\hvf\|_{\hbeta}\|\hpsi\circ\sigma^r\|_{\hbeta}\theta_{j}^\ell\leq
{C_{j}C^2\hbeta^{-r}\|\vf\|_{\beta}'\|\psi\|_{\beta}'}\theta_{j}^\ell,
 \ \ \text{ $\forall r,\ell\geq 0$.}$$
For every $n$, we write
$
pn=q\ell_n+r_n,\text{ where }\ell_n:=\lfloor\frac{pn}{q}\rfloor\ , \ 0\leq r_n\leq q-1.
$
Then
\begin{align*}
|\mathrm{Cov}_{\mu'}(\vf,\psi\circ f^{pn})|=|\mathrm{Cov}_{\hmu'_j}(\hvf,\hpsi\circ \sigma^{r_n}\circ \sigma^{q\ell_n})|\leq C_{j}C^2\hbeta^{-r_n}\|\vf\|_{\beta}'\|\psi\|_{\beta}'\theta_{j}^{\ell_n}.
\end{align*}
Thus,
$|\mathrm{Cov}_{\mu'}(\vf,\psi\circ f^{pn})|\leq C'\|\vf\|_{\beta}'|\psi\|_{\beta}'\theta^n$ where
{$C':=\max_{1\leq j\leq q}\{C_{j}C^2\hbeta^{-r}\theta_{j}^{-1}\}$, and $\theta:=\max_{1\leq j\leq q}\{\theta_{j}^{p/q}\}$. }
We obtained \eqref{e.aim-i}, in case $i=0$.
The case $i\neq 0$ follows by symmetry, and we proved
Thm~\ref{t.DOC-diffeos}.
\qed

\subsection{Asymptotic Variance (Thm~\ref{t.asymp-var-diffeos})}\label{s.asymp-var-diffeos}
The theorem holds for SPR Markov shifts, see \S\ref{ss.CMS.variance}. We will use symbolic dynamics to deduce it for diffeomorphisms.

By Lemmas \ref{l.psi-hat-Holder}, \ref{l.p|q}, there are $\hpsi$, $\hmu$ such that (a) $\hmu$ is the MME of $\Sigma$; (b) $\mu=\pi_*\hmu$; (c)  $\hpsi=\psi\circ\pi$  on $\pi^{-1}(X)$; (d)  $\|\hpsi\|_{\hbeta}\leq C\|\psi\|_\beta'$. Since $\mu(X)=1$, (b) and (c) imply
$
\hpsi_n:=\sum_{k=0}^{n-1}\hpsi\circ\sigma^k=\psi_n\circ\pi\text{ $\hmu$-a.e., for all $n$.}
$
Thus by  (b), $\mathrm{Var}_\mu(\psi_n)=\mathrm{Var}_{\hmu}(\hpsi_n)$.
By Thm~\ref{t.asymp-var-CMS}, $\lim\limits_{n\to\infty}\frac{1}{n}\mathrm{Var}_\mu(\psi_n)$ exists, and
\begin{equation}\label{e.sigma-psi-hat}
\sigma_\psi^2=\sigma_{\hpsi}^2:=\lim_{n\to\infty}\tfrac{1}{n}\mathrm{Var}_{\hmu}(\hpsi_n),
\end{equation}
\begin{equation}\label{e.Laplace-Transform-f}
\E_\mu[e^{z\psi_n/\sqrt{n}}]=\E_{\hmu}(e^{z\hpsi_n/\sqrt{n}})\longrightarrow e^{\frac{1}{2}\sigma_\psi^2 z^2}\   \ (z\in\mathbb C)\ \  \text{ if $\int \psi d\mu= 0$}.
\end{equation}
(The variance $\sigma^2_\psi=\sigma^2_{\wh\psi}$ does depend on $\mu$, even if omitted from the notation).

Next, by (d), $\sigma_\psi=\sigma_{\hpsi}\leq  M_{\wh{\beta}}{\|\hpsi\|_{\hbeta}}\leq CM_{\wh{\beta}}{\|\psi\|_{\beta}'}$, with $M_{\wh{\beta}}$ as in \eqref{e.sigma-bound-CMS}.
This proves the existence of the limit  \eqref{eq-def-avar}, and Parts (3) and (4) of Thm~\ref{t.asymp-var-diffeos}.

Part (1), the Green-Kubo identity, follows from the exponential decay of correlations of $f^p$ with respect to the ergodic components of  $(X,f^p,\mu)$, exactly as in the proof of Theorem \ref{t.asymp-var-CMS}.\footnote{Note that we cannot deduce the Green-Kubo formula for $\sigma_\psi$ directly from the Green-Kubo formula for $\sigma_{\hpsi}$,  because of the possible difference between the period of $\mu$ and the period of $\Sigma$.}

By \eqref{e.sigma-psi-hat} and Thm \ref{t.asymp-var-CMS}, $\sigma_\psi^2=\sigma_{\hpsi}^2= \frac{d^2}{dt^2}\big|_{t=0}P_\Bor(\Sigma,t\hpsi)$.
Part (2) follows from:
\begin{claim}\label{c.lift-pressure}
Recall $\kappa$ from \eqref{e.chi-kappa-def}.
There is $\varepsilon(\kappa)>0$ s.t.
$P_{\Bor}(f|_X,t\psi)=P_\Bor(\Sigma, t\hpsi)$ 
for all $|t|<\tfrac{\varepsilon}{\|\psi\|_\infty}$.
\end{claim}
\begin{proof}By the affinity of the entropy function, 
the supremum defining $P_\Bor(\Sigma, t\hpsi)$ and $P_{\Bor}(f|_Xt\psi)$ can be replaced by a supremum over  ergodic measures.
\begin{enumerate}[(1)]
\item[($\geq$)]
By ($\Sigma$\ref{i.Sigma3}) in \S\ref{section-Sigma},  every   $\hnu\in \Proberg({\sigma})$ projects to an ergodic measure $\nu$ on $X$, with same entropy. So  $h(\sigma,{\hnu})+t\int \hpsi d\hnu=h(f,{\nu})+t\int \psi d\nu\leq P_{\Bor}(f|_X,t\psi)$ and $P_\Bor(\Sigma,t\hpsi)\leq P_{\Bor}(f|_X,t\psi) \text{ for all $t$}.$
\item[($\leq$)]
Recall $\chi,\kappa>0$ from \eqref{e.chi-kappa-def},  and choose $\eps,\delta>0$ so small that
$
2\eps+\delta<\kappa
$.
Fix $|t|<\eps/\|\psi\|_\infty$. If $\nu\in\Proberg(f,X)$ and  $h(f,\nu)+t\int\psi d\nu>P_{\Bor}(f|_X,t\psi)-\delta$, then $h(f,\nu)>h_{\Bor}(f|_X)-\kappa$. By \eqref{e.chi-kappa-def}, $\nu\in\mathbb \Proberg(\chi,f|_X)$, and by ($\Sigma$\ref{i.Sigma4}),  $\nu$ has an ergodic lift to $\Sigma$, with same entropy. So $h(f,\nu)+t\int\psi d\nu\leq P_\Bor(\Sigma,t\hpsi)$ and
$
 P_{\Bor}(f|_X,t\psi) \leq P_\Bor(\Sigma,t\hpsi)\text{ for all $|t|<\eps/\|\psi\|_\infty$.}$ 
\end{enumerate}
So $P_{\Bor}(f|_X,t\psi) = P_\Bor(\Sigma,t\hpsi)$ for all $t\in(-\eps/\|\psi\|_\infty,\eps/\|\psi\|_\infty)$.
\end{proof}

\smallskip
It remains to prove Part (5) of the theorem. It is sufficient to consider the case  $\int\psi d\mu=0$, and  show that the following conditions are equivalent:
\begin{enumerate}[(1)]
\item $\sigma_\psi^2=0$.
\item $\psi=u-u\circ f$ $\mu$-a.e. for some Borel function $u:M\to\R$ $\mu$-a.e.
\item If $f^n(x)=x$, then  $\psi_n(x)=0$.
\item If $\nu$ is an  invariant measure on $X$, then $\int\psi d\nu=0$.
\end{enumerate}

\medskip
\noindent
$(3)\Rightarrow(2)$: By the SPR property, $\mu$ is a hyperbolic measure. If (3) holds, then (2) must follow, by the Katok-Mendoza generalization of the Livsic Theorem to non-uniformly hyperbolic diffeomorphisms \cite[Thm S.4.7]{Katok-Hasselblatt-Book}. (The Katok-Mendoza Theorem assumes that the weights of {\em all} periodic orbits is zero, but a close look at the proof shows that only hyperbolic periodic orbits homoclinic to $\mu$ are needed.)

\medskip
\noindent
$(2)\Rightarrow (1)$: Suppose $\psi=u-u\circ f$ $\mu$-a.e.  It is easy to see that $|u-u\circ f^n|/\sqrt{n}\to 0$ in probability, therefore
$\E_\mu(e^{it\psi_n/\sqrt{n}})\to 1$ for all $t$ real. By \eqref{e.Laplace-Transform-f}, $\sigma_\psi=0$.

\medskip
\noindent
$(1)\Rightarrow (4)$:
Every $\chi$-hyperbolic $\nu\in\Proberg(f|_X)$  lifts to an ergodic measure $\hnu$ on $\Sigma$. By (1) and \eqref{e.sigma-psi-hat},  $\sigma_{\hpsi}=0$, and by  the characterization of zero variance for SPR Markov shifts (Thm~\ref{t.asymp-var-CMS}), $\int \hpsi d\hnu=\int \hpsi d\hmu=0$. It follows that $\int \psi d\nu=0$ for all $\chi$-hyperbolic ergodic invariant measures on $X$. Since $\chi$ can be chosen arbitrarily small, this is the case for all $\nu\in\Proberg(f|_X)$ (all such measures are hyperbolic). By the ergodic decomposition, $\int \psi d\nu=0$ for all $\nu\in\mathbb P(f|_X)$.

\medskip
\noindent
$(4)\Rightarrow(3)$ is obvious. Thm~\ref{t.asymp-var-diffeos} is now proved.
\hfill$\Box$

\subsection{Large Deviations (Thm \ref{t.LDP-diffeos})} \label{s.Proof-of-LDP-diffeos}
Choose $M_\beta$ as in~\eqref{e.sigma-bound-diffeo} {and} $C$, $\hpsi$, $\hbeta$, $\hmu$ as in Lemmas \ref{l.psi-hat-Holder}--\ref{l.p|q}. By \eqref{e.Laplace-Transform-f}, $\E_\mu[e^{t\psi_n}]=\E_{\hmu}[e^{t\hpsi_n}]$ for all $n$ and $t\in\R$. Hence
$$
\Lambda_\psi(t)=\underset{n\to\infty}{\limsup}\tfrac{1}{n}\log\E_{\hmu}(e^{t\hpsi_n})=:\Lambda_{\hpsi}(t)\ ,\ I_\psi(s)=\sup_t\{st-\Lambda_{\hpsi}(t)\}=:I_{\hpsi}(s).
$$
Similarly, for every Borel set $E\subset\R$,
$
\mu\{x\in X:\frac{1}{n}\psi_n(x)\in E\}=(\hmu\circ\pi^{-1})\{x\in X:\frac{1}{n}\psi_n(x)\in E\}=\hmu\{\un{x}\in\Sigma:\frac{1}{n}\hpsi_n(\un{x})\in E\}
$.
 Thm~\ref{t.LDP-diffeos} now follows from the large deviations theorem for SPR Markov shifts (Thm~\ref{t.LDP-CMS}), and \eqref{e.sigma-psi-hat}. \qed

\medskip
This argument can be refined to obtain some more information.
Item (3) below justifies Remark \ref{r.LDP}.
\begin{lemma}\label{l.large-deviation-diffeo}
Under the setting of Thm~\ref{t.LDP-diffeos}, for any $\beta>0$ there are $\varepsilon,c>0$
such that for any $\beta$-quasi-H\"older function $\psi$ with $\|\psi\|_\beta'=1$, $\int \psi d\mu=0$ and  $\sigma_\psi^2\neq 0$:
\begin{enumerate}[(1)]
\item \begin{enumerate}[(a)]
\item $\Lambda_\psi$ is finite, convex, non-negative on $\R$. On $(-\varepsilon,\varepsilon)$, $\Lambda_\psi$ is $C^\infty$ and strictly convex.
\item
$I_\psi$ is  convex and   non-negative  on  $\R$. On $(-c{\sigma_\psi^4},c{\sigma_\psi^4})$, $I_\psi$ is  finite, $C^\infty$, strictly convex, and satisfies $\frac{1}{2}\sigma_\psi^{-2}\leq I_\psi''\leq 2\sigma_\psi^{-2}$.
\end{enumerate}
\item $\Lambda_\psi(0)=0\ , \ \Lambda_\psi'(0)=0\ , \ \Lambda_\psi''(0)=\sigma^2_\psi\ ,$
$I_\psi(0)=0\ , \ I_\psi'(0)=0\ , \ I_\psi''(0)={\sigma^{-2}_\psi}.
$

\item If  $G$ is open and $G\cap[-c\sigma_\psi^4,c\sigma_\psi^4]\neq \emptyset$, then
$
 \displaystyle\inf_{G\cap (-c\sigma_\psi^4,c\sigma_\psi^4)} I_\psi=\inf_G I_\psi
 $.
\item
\begin{enumerate}[(a)]
\item $\Lambda_\psi(t)=\sup\left\{h(f,\nu)+t\nu(\psi):\nu\in \Prob(f|_X)\right\}-h_{\Bor}(f|_X)$ for all $|t|<\varepsilon$.
\item $I_\psi(s)=h_{\Bor}(f|_X)-\sup\left\{h(f,\nu):\nu\in \Proberg(f|_X),\  \nu(\psi)=s\right\}$, for $|s|<c\sigma_\psi^4$.
\item
The suprema are attained by unique measures; these measures are ergodic.
\end{enumerate}
\end{enumerate}
\end{lemma}
\begin{remark}
We do not know if the supremum in (4b) also holds over non-necessarily ergodic measures.
However the proof shows that it holds over all invariant measures which admit lifts to $\Sigma$. \end{remark}
\begin{proof}
Choose constants $\varepsilon_1,c_1>0$ as in 
Lemma~\ref{t.Free-Energy-CMS}, applied to  for $\Sigma,\hbeta$ and $\phi\equiv 0$.
By Lemma~\ref{l.psi-hat-Holder},
$\|\hpsi\|_{\hbeta}\leq C \|\psi\|_\beta'=C$.
Items (1)--(2) follow immediately when $\|\psi\|_{\wh{\beta}}'=1$,  and using Remark \ref{r.Free-Energy-CMS}, when $\|\psi\|_{\wh{\beta}}'\neq 1$, and $\varepsilon\leq \varepsilon_1/C$ and $c\leq c_1/C^3$.
This implies that $I_\psi$ is strictly increasing on $(0,+\infty)$, strictly decreasing on $(-\infty,0)$, hence Item (3).

\medskip
\noindent
{\em Proof of (4a).\/}
By Lemma~\ref{t.Free-Energy-CMS}(3a), for $|t|\leq \varepsilon_1/ \|\hpsi\|_{\wh \beta}$, $\Lambda_{\hpsi}(t)=P_\Bor(\Sigma,t\hpsi)-h(\sigma,\hmu)$.
By Claim~\ref{c.lift-pressure}, there is $\eps_2(\kappa)>0$ such that for $|t|<\eps_2/\|\hpsi\|_{\wh \beta}$,
 $P_\Bor(\Sigma,t\hpsi)=P_{\Bor}(f|_X,t\psi)$.
 So for all $|t|\leq \min(\eps_1,\eps_2)/\|\hpsi\|_{\wh{\beta}}$,
$$\Lambda_\psi(t)=\Lambda_{\hpsi}(t)=P_\Bor(\Sigma,t\hpsi)-h(\sigma,{\hmu})
=P_{\Bor}(f|_X,t\psi)-h(f,\mu).
$$
Hence
 (4a) holds for  $\varepsilon<\min(\varepsilon_1,\varepsilon_2)/C$.\qed

\medskip
\noindent
{\em Proof of (4b).\/} We will construct $c_2=c_2(\Sigma,\beta)$ such that for all $|s|\leq c_2\sigma_\psi^4$,
\begin{align*}
\mathrm{LHS}:=&\sup\{h(f,\nu):\nu\in\Proberg(f|_X),\ \textstyle{\int} \psi d\nu=s\}\\
&\hspace{1cm}=\sup\{h(\sigma,\hnu):\hnu\in\mathbb \Proberg(\sigma),\textstyle{\int} \hpsi d\hnu=s\}=:\mathrm{RHS}.
\end{align*}
Once this is done, (4b) follows from Lemma~\ref{t.Free-Energy-CMS}(3b) and the identity $I_\psi=I_{\hpsi}$.

The inequality $(\geq)$ holds for all $s$, because by ($\Sigma$\ref{i.Sigma3}), if $\hnu\in\mathbb P_{erg}(\sigma)$ and $\int\hpsi d\hnu=s$, then $\nu:=\hnu\circ\pi^{-1}$ is carried by $X$, is ergodic,   $\int\psi d\nu=s$ and $h(f,\nu)=h(\sigma,\hnu)$.

The inequality $(\leq)$ is more subtle.
Let $h:=h_\Bor(f|_X)=h_\Bor(\Sigma)$.
Lemma \ref{t.Free-Energy-CMS} and Taylor's approximation give $c_3(\Sigma,\pi,\beta)$ such that  for every $|s|\leq c_3\sigma_{\psi}^4/\|\hpsi\|_{\wh{\beta}}^{3}$,
$$
\mathrm{RHS}=h-I_{\hpsi}(t)\geq h-\tfrac{1}{2}s^2\max_{|s|\leq c_3\sigma_{\psi}^4/\|\hpsi\|_{\wh{\beta}}^{3}} I_{\hpsi}''\geq h-{s^2}/\sigma_\psi^2.
$$
{We recall $\sigma_\psi\leq M_\beta\|\psi\|_\beta'$.}
Thus,  there is $c_4(\Sigma,\pi,\beta)$ such that  RHS$>h-\kappa$ for all $|s|<c_4\sigma_\psi^4$. 
By  ($\geq$),  $\mathrm{LHS}>h-\kappa$ for all such $s$. Therefore the LHS is determined by $\nu\in\Proberg(f|_X)$ with entropy bigger than $h-\kappa$. By \eqref{e.chi-kappa-def}, such $\nu$ belong to $\Proberg(\chi, f|_X)$, and by ($\Sigma$\ref{i.Sigma4}), they have ergodic lifts to $\Sigma$, with the same entropy. Hence
$\mathrm{LHS}\leq \sup\{h(\sigma,\hnu):\hnu\in\mathbb P_{erg}(\sigma)\text{ s.t. }\int \hpsi d\hnu=s\}=\mathrm{RHS}$,
proving $(\leq)$.\hfill$\Box$

\medskip
\noindent
{\em Proof of (4c).\/} The proofs of Items (4a) and (4b) show that the measures on $X$ which solve the variational problems in these two items lift to measures on $\Sigma$ which solve the variational problems there.  Conversely, solutions on $\Sigma$ project to solutions on $X$. Now Item (4c) follows from Lemma \ref{t.Free-Energy-CMS}(3).
\end{proof}

\subsection{The Almost Sure Invariance Principle (Thm \ref{t.ASIP-diffeo})}\label{s.ASIP-diffeos}
\begin{proof}[Proof of Theorem \ref{t.ASIP-diffeo}]
Choose  $\hmu$ and $\hpsi$ as in \S\ref{s.asymp-var-diffeos}, then the stochastic processes $(\psi_n)_{n\geq 1}$ on $(M,\mathfs B(M),\mu)$ and $(\hpsi_n)_{n\geq 1}$ on $(\Sigma,\mathfs B(\Sigma),\hmu)$ are equal in distribution.

By the ASIP for SPR Markov shifts (Thm~\ref{t.ASIP-Sigma}), there are  a probability space $(\Omega,\mathfs F,\nu)$
and  measurable functions $\wt{S}_n,\wt B_t:\Omega\to \R$ such that $(\wt B_t)_{t\geq 0}$ is a standard Brownian motion,
$(\wt{S}_n)_{n\geq 1}$ is equal in distributions to $(\hpsi_n)_{n\geq 1}$,  and   for all $\gamma>\tfrac{1}{4}$,
$
|\wt{S}_n(\omega)-\sigma_{\hpsi} \wt B_n(\omega)|=o(n^\gamma)\text{ a.e. in $\Omega$, as $n\to\infty$.}
$

By \eqref{e.sigma-psi-hat}, $\sigma_\psi=\sigma_{\hpsi}$.  Since $({\wt S}_n)_{n\geq 1}=(\hpsi_n)_{n\geq 1}$  in distribution,  $({\wt S}_n)_{n\geq 1}=({\psi}_n)_{n\geq 1}$ in distribution. 
This shows that the ASIP is satisfied.

The  ``dynamical ASIP" from Remark \ref{r-dynamical-ASIP} follows from the following lemma:

\begin{lemma}\label{l.Borelomania}
Let $S_n:X\to\R$ be  Borel functions on a standard probability space $(X,\mathfs B,\mu)$. Suppose $(\wt{S}_n)_{n\geq 1}$ and $(\wt{B}_t)_{t\geq 0}$ are two stochastic processes on some other standard probability space $(\Omega,\mathcal F,m)$, which satisfy parts (1)--(3) in Def \ref{def-ASIP} with constants $\sigma$ and $\gamma$. 
Then there exist a Borel probability measure $\nu$ on $X\times [0,1]$ which projects to $\mu$,  and a standard Brownian motion $(B_t)_{t\geq 0}$ on $(X\times [0,1],\nu)$ s.t.
$$
|{S}_n(x)-\sigma {B}_n(x,\xi)|=o(n^{\gamma}) \text{ $\nu$-a.e. in $X\times [0,1]$.}
$$
\end{lemma}
\begin{proof}
If $t>0$, the law of $\wt{B}_t:\Omega\to\R$ has no atoms. Therefore $(\Omega,\mathfs F,m)$ has no atoms. By  Kuratowski's isomorphism theorem, there is a Borel probability measure $m'$ on $[0,1]$, with a  measure-theoretic isomorphism $\theta:([0,1],\mathfs B,m')\to (\Omega,\mathfs F,m).$

Let $Y:=\R^\N$ and define  $\un{S}:X\to Y$ by $\un{S}(x):=(S_1(x),S_2(x),\cdots)$. Let $\Xi$ denote the measurable partition of $X$  into the sets $\{x\in X:\un{S}(x)=\un{y}\}$, $\un{y}\in Y$.
By Rokhlin's theorem, there is a family of measures $\mu_{\un{y}}$, each carried by the $\Xi$--atom $\{x\in X:\un{S}(x)=\un{y}\}$, such that
$$
\mu=\int_Y \mu_{\un{y}} d\mu\circ\un{S}^{-1}.
$$

Next, let $\un{T}:\Omega\to Y$ be the function $\un{T}(\omega)=(\wt{S}_1(\omega),\wt{S}_2(\omega),\cdots)$. Since
 $(S_n)_{n\geq 1}$ is equal to $(\wt{S}_n)_{n\geq 1}$ in distribution,
$
\mu\circ\un{S}^{-1}=m\circ\un{T}^{-1}=m'\circ \theta^{-1}\circ\un{T}^{-1}.
$
Therefore
$
\mu=\int_{[0,1]} \mu_{\un{T}(\theta(\xi))}dm'(\xi).
$
Thus, the following measure on $X\times [0,1]$ projects to $\mu$:
$$
\nu:=\int_{[0,1]}\mu_{\un{T}(\theta(\xi))}\times \delta_\xi\, dm'(\xi).
$$

Define ${S}_n,{B}_t:X\times [0,1]\to\R$ by
$
{S}_n(x,\xi):=S_n(x) \ , \ {B}_t(x,\xi):=\wt{B}_t(\theta(\xi)).
$
For $\nu$-almost every $(x,\xi)$,  $x$ is in the $\Xi$--atom  which carries  $\mu_{\un{T}(\theta(\xi))}$. Therefore, $\nu$-almost surely, $\un{S}(x)=\un{T}(\theta(\xi))$. Equivalently,
$
({S}_n(x))_{n\geq 1}=(\wt{S}_n(\theta(\xi)))_{n\geq 1}
$. So
\begin{align*}
&\bigl(({S}_n(x))_{n\geq 1},({B}_t(x,\xi))_{t\geq 0}\bigr)=
\bigl((\wt{S}_n(\theta(\xi)))_{n\geq 1},(\wt{B}_t(\theta(\xi)))_{t\geq 0}\bigr)\text{ $\nu$-almost surely}\\
&=\bigl((\wt{S}_n(\omega))_{n\geq 1},(\wt{B}_t(\omega))_{t\geq 0}\bigr)\text{ in distributions (because $m'\circ \theta^{-1}=m$)}.
\end{align*}
So $({B}_t(x,\xi))_{t\geq 0}$ is a standard Brownian motion on $(X\times [0,1],\nu)$, and
$\nu\{(x,\xi):|S_n(x)-\sigma {B}_{n}(x,\xi)|=o(n^{\gamma})\}=m\{\omega: |\wt{S}_n(\omega)-\sigma \wt{B}_{n}(\omega)|=o(n^{\gamma})\}=1
$.
\end{proof}

We finish the proof of Thm~\ref{t.ASIP-diffeo} by showing $\sigma=\sigma_\psi$.
Using~\eqref{e.ASIP-dynamical}, the bounded convergence theorem, and the characteristic function of normal laws, we get

$$
    \lim_{n\to\infty} \int e^{i\psi_n/\sqrt{n}} d\mu =
     \lim_{n\to\infty} \int e^{i\sigma B_n/\sqrt{n}} d\nu =
     \int e^{i\sigma B_1} d\nu =
    e^{\frac{-\sigma^2}{2}}.
    $$
By Thm~\ref{t.asymp-var-diffeos}(3) it is also equal to $\exp({\frac{-\sigma_\psi^2}{2}})$, hence $\sigma=\sigma_\psi$.
\end{proof}

\subsection{Effective Intrinsic Ergodicity (Thms~\ref{t.effective}, \ref{t.Kadyrov-Ineq-diffeos}, \ref{t.Kadyrov-Ineq-diffeos-Strong})}\label{s.Kadirov-Proof-diffeos}

Suppose that  $X,f,\mu$ are as in \S\ref{s.setup-for-consequences} and $(\Sigma,\pi)$ as in our standing assumptions, and choose $\kappa$ and $\chi$ as in \eqref{e.chi-kappa-def}.
Fix $\psi\colon X\to \RR$ quasi-H\"older of exponent $\beta$. Choose $\wh{\beta}:=\wh{\beta}(\Sigma,\pi,\beta)$, $C_1(\Sigma,\pi,\beta)$ as in  Lemma \ref{l.psi-hat-Holder} and $C^*,K>0$ as in Thm~\ref{t.effective-ergodicity-shift}.
Let $\hpsi$ be a H\"older continuous function such that $\hpsi=\psi\circ\pi$ on $\pi^{-1}(X)$, and
 $\|\hpsi\|_{\wh{\beta}}\leq C_1\|\psi\|_\beta'$.

\begin{proof}[Proof of Theorem~\ref{t.Kadyrov-Ineq-diffeos}]

Suppose $\nu$ is an {ergodic} $f$-invariant measure carried by $X$. If $h_{\nu}(f)> h_{\mu}(f)-\kappa$, then
$\nu$ is $\chi$-hyperbolic by \eqref{e.chi-kappa-def}, and  by ($\Sigma$\ref{i.Sigma4}) in \S\ref{section-Sigma}, $\nu$  has an ergodic lift $\hnu$  to $\Sigma$ with the same entropy.
By Thm~\ref{t.effective-ergodicity-shift}(c),
$$
|\mu(\psi)-\nu(\psi)|=|\hmu(\hpsi)-\hnu(\hpsi)|\leq C_1 K\|\psi\|_{\beta}'\sqrt{h(f,\mu)-h(f,\nu)}.
$$
If $h_{\nu}(f)\leq h_{\mu}(f)-\kappa$, then we have the trivial bound:
$$
|\mu(\psi)-\nu(\psi)|\leq 2\sup_{ X}|\psi|\leq 2\|\psi\|_{\beta}'\leq 2\kappa^{-1/2}\|\psi\|_{\beta}'\sqrt{h(f,\mu)-h(f,\nu)}.
$$
Taking $C:=\max\{C_1K,2\kappa^{-1/2}\}$, we see that for all $\nu\in \Proberg(f|_X)$,
\begin{equation}\label{e.Kad-Erg-Step}
|\mu(\psi)-\nu(\psi)|\leq C\|\psi\|_{\beta}'\sqrt{h(f,\mu)-h(f,\nu)}.
\end{equation}

For non-ergodic measures $\nu$ on $X$, we use the ergodic decomposition $\nu=\int \nu_\xi d\xi$: \begin{align*}
&|\mu(\psi)-\nu(\psi)|\leq \int |\mu(\psi)-\nu_\xi(\psi)|d\xi\leq C\|\psi\|_\beta' \int|h(f,\mu)-h(f,\nu_\xi)|^{1/2}d\xi\\
&\leq C \|\psi\|_\beta' \left(\int |h(f,\mu)-h(f,\nu_\xi)|d\xi\right)^{1/2}=C \|\psi\|_\beta' \left(\int h(f,\mu)-h(f,\nu_\xi)d\xi\right)^{1/2}  \\
&=C\|\psi\|_\beta'\sqrt{h(f,\mu)-h(f,\nu)},
\text{ because }h_\nu(f)=\int h_{\nu_{\xi}}(f) d\xi.\qedhere
\end{align*}
\end{proof}

\begin{proof}[Proof of Theorem~\ref{t.Kadyrov-Ineq-diffeos-Strong} and Remark \ref{r.Kadyrov-delta}]
If $\sigma_\psi=0$, then $\mu(\psi)=\nu(\psi)$ for all $f$-invariant  measures $\nu$ on $X$ (Thm \ref{t.asymp-var-diffeos}), and the claim is trivial.
 Thus,
without loss of generality, assume $\mu(\psi)=0$, $\|\psi\|'_\beta=1$.

Fix $\eps>0$ small and $\delta_0:= C^\#
\eps^2\sigma_\psi^6$ with $C^\#=C^\#(\Sigma,\pi,\beta):=C^*/4$ (the number given by Thm~\ref{t.effective-ergodicity-shift}) so that $\delta_0\leq C^\ast(\tfrac \eps 2)^2\sigma_{\hpsi}^6$. 
Recall that
$\sigma_{\hpsi}=\sigma_\psi$ (see \eqref{e.sigma-psi-hat}), $\|\hpsi\|_{\wh{\beta}}\leq C_1\|\psi\|_\beta'$, and  $\sigma_{\psi}\leq  M_{{\beta}}\|\psi\|_{{\beta}}$ (see \eqref{e.sigma-bound-diffeo}).

 Suppose $\nu$ is an {\em ergodic} probability measure on $X$ such that $h(f,\nu)>h(f,\mu)-\delta_0$. Since $\delta_0<\kappa$,  $\nu$ lifts to a shift invariant measure $\hnu$ with the same entropy.

Then
$
h(\hnu,\sigma)>h(\hmu,\sigma)-\delta_0
$. 
Since $\delta_0<C^*\eps^2\sigma_\psi^6$ and $0<\eps\le\eps^*$ with $\eps^*$ small enough, Thm~\ref{t.effective-ergodicity-shift}(a) gives
$$\begin{aligned}
|\mu(\psi)-\nu(\psi)|&=|\hmu(\hpsi)-\hnu(\hpsi)|
\leq e^{\eps/2}\sqrt{2\sigma_{\hpsi}^2(h(\sigma,\hmu)-h(\sigma,\hnu))}\\
&=
e^{\eps/2}\sqrt{2\sigma_{\psi}^2(h(f,\mu)-h(f,\nu))}.
\end{aligned}$$

The optimality of the constant $2\sigma_\psi^2$ can be proved by using the measures $\nu_n:=\hnu_n\circ\pi^{-1}$, where $\hnu_n$
are given by Thm~\ref{t.effective-ergodicity-shift}(b).

We now consider the non-ergodic  case. The idea is to collect the ergodic components with small entropies and  with large entropies to separate measures $\nu_0,\nu_1$ and treat them separately.

To see this, we fix small number $0<\delta\ll\delta_0$
(specified below) and suppose $h(f,\mu)-h(f,\nu)<\delta$. Decompose $\nu=\alpha\nu_0+(1-\alpha)\nu_1$, where almost every ergodic component of $\nu_0$ (resp. $\nu_1$) has entropy less than or equal to (resp. greater than) 
$h(f,\mu)- \delta_0$ 
(here $0\le\alpha\le 1$ and by convention $\nu_0=\nu_1=\nu$ if $\alpha=0$ or $1$). We note that
 $
  h(f,\nu)=\alpha h(f,\nu_0)+(1-\alpha)h(f,\nu_1) \le \alpha (h(f,\mu)-\delta_0) + (1-\alpha) h(f,\mu),
 $
hence
 $$
    0\le\alpha \le \frac{h(f,\mu)-h(f,\nu)}{\delta_0} < \frac{\delta}{\delta_0} < 1.
 $$

Let $\nu_1=\int\nu_1^\xi\, d\xi$ be the ergodic decomposition of $\nu_1$.  By construction, 
 $h(f,\mu)-h(f,\nu_1^\xi)<\delta_0$ for a.e. $\xi$. By the first part of the proof  and  the concavity of $\sqrt{t}$, 
 $$\begin{aligned}
   |\mu(\psi)-\nu_1(\psi)| &\leq \int |\mu(\psi)-\nu_1^\xi(\psi)|\, d\xi
    \leq   \sqrt{2e^\eps\sigma_\psi^2} \int \sqrt{h(f,\mu)-h(f,\nu_1^\xi)}\, d\xi\\
    &\leq  \sqrt{2e^\eps\sigma_\psi^2 \int h(f,\mu)-h(f,\nu_1^\xi)}\, d\xi
    = \sqrt{2e^\eps\sigma_\psi^2 ( h(f,\mu)-h(f,\nu_1))}\\
    &\leq \sqrt{2e^\eps\sigma_\psi^2 ( h(f,\mu)-h(f,\nu))}
   \quad \text{as $h(f,\nu_1)\ge h(f,\nu)$.}
 \end{aligned}$$
Next, the trivial bound $|\mu(\psi)-\nu_0(\psi)|\leq 2\sup|\psi|$ leads us to 
 $$\begin{aligned}
  & |\mu(\psi)-\nu(\psi)|
      \leq (1-\alpha)|\mu(\psi)-\nu_1(\psi)| + \alpha|\mu(\psi)-\nu_0(\psi)|\\
      &\leq (1-\alpha)\sqrt{2e^\eps\sigma_\psi^2 ( h(f,\mu)-h(f,\nu))}  + \alpha\cdot 2\sup|\psi|\\
      &\leq e^{\eps/2}\sqrt{2\sigma_\psi^2 ( h(f,\mu)-h(f,\nu))}  + 2\|\psi\|_\beta'\frac{h(f,\mu)-h(f,\nu)}{\delta_0}\\
      &\leq e^{\eps/2}\sqrt{2\sigma_\psi^2 ( h(f,\mu)-h(f,\nu))}\left(1 + \frac{\sqrt2}{e^{\eps/2}\sigma_\psi\delta_0} \sqrt{h(f,\mu)-h(f,\nu)}\right),\ \text{ since } \|\psi\|_\beta'=1\\
      &\leq e^{\eps/2}(1+\eps/2)\sqrt{2\sigma_\psi^2 ( h(f,\mu)-h(f,\nu))}
      \leq e^{\eps}\sqrt{2\sigma_\psi^2 ( h(f,\mu)-h(f,\nu))}
 \end{aligned}$$
provided $\delta\leq \frac{\epsilon^2}{8}\sigma_\psi^2\cdot\delta_0^2=
\frac{1}{8} (C^\#)^2\epsilon^6\sigma_\psi^{14}=
\frac{1}{128}(C^*)^2\epsilon^6\sigma_\psi^{14}$.
\end{proof}

\subsection{Oseledets Decompositions in High Entropy (Cor \ref{t.Osceledets-Stab})}\label{s.Proof-Oseledets-Stab}
Let $\mathfrak G$ be the bundle \eqref{e.frak-G}, with its distance function $d_{\mathfrak G}(\cdot,\cdot)$.
For every $x\in X$, let $\xi(x):=(x;E^u(x),E^s(x))\in \mathfrak G$.
Given a non-atomic invariant measure $\nu$ on $X$,
we will construct a Borel one-to-one map $T:M\to M$ so that $T_\ast\nu=\mu$, and
\begin{equation}\label{e.rach}
\int d_{{\mathfrak G}}(\xi(x),\xi(Tx))d\nu(x)\leq \const\cdot \sqrt{h(f,\mu)-h(f,\nu)}.
\end{equation}
We assume $h(f,\nu)<h(f,\mu)$ since otherwise $\nu=\mu$ and we can take $T=\id$.

Lift $\mu,\nu$ to
$\wt{\mu}:=\xi_*\mu$ and
$\wt{\nu}:=\xi_*\nu$.
By the Kantorovich--Rubinstein theorem, and the observation that 
 any $1$-Lipschitz map $\Psi:{\mathfrak G}\to\C$ induces
a $1$-quasi-H\"older map $\psi\colon X\to \C$ defined by $\psi(x)=\Psi(x;E^u(x),E^s(x))$
s.t. $\|\psi\|_1'\leq 1 +\diam(\mathfrak G)$,
\begin{align*}
&\inf\left\{\int d_{\mathfrak G}(\xi,\eta)dm:\begin{array}{l}
\text{$m$ is a probability measure on $\mathfrak G\times\mathfrak G$ s.t.}\\
\text{$m(A\times \mathfrak G)=\wt{\nu}(A)$, $m(\mathfrak G\times A)=\wt{\mu}(A)$ for all $A$}
\end{array}
\right\}\\
&=\sup\left\{
\left|\int\Psi d\wt{\mu}-\int\Psi d\wt{\nu}\right|: \Psi:\mathfrak G\to\R\text{ is bounded, Lipschitz,  }\mathrm{Lip}(\Psi)\leq 1
\right\}\\
&\leq \sup\left\{
\left|\int\psi d{\mu}-\int\psi d{\nu}\right|: \|\psi\|_1'\leq 1 +\diam(\mathfrak G)
\right\}\\
&\leq (1+\diam(\mathfrak G))C\sqrt{h(f,\mu)-h(f,\nu)},
\end{align*}
with $C$ as in Theorem \ref{t.Kadyrov-Ineq-diffeos}.
So  there is a coupling $m$ of $\wt{\nu}$ and $\wt{\mu}$ with ``cost" 
 $$
   \int d_{\mathfrak G}(\xi,\eta) dm<2(1+\diam(\mathfrak G))C\sqrt{h(f,\mu)-h(f,\nu)}.
 $$
Then \cite[Thm B]{Pratelli-2007},
gives a Borel map $\wt{T_0}:\mathfrak G\to\mathfrak G$ s.t. $\wt{\mu}=\wt{\nu}\circ \wt{T_0}^{-1}$ and
$$
\int_{\mathfrak G} d_{\mathfrak G}(\xi, \wt{T_0}\xi) d\wt{\nu}<3(1+\diam(\mathfrak G))C\sqrt{h(f,\nu)-h(f,\mu)}.
$$

One can replace $\wt{T_0}$ by a {\em one-to-one}  map $\wt T$ which is arbitrarily close to it, as follows. First, we  build two finite measurable partitions $\{A_i\}$, $\{B_i\}$ of $\mathfrak G$ such that
$A_i$, $B_i$ have small diameters (in particular the variations of $d_{\mathfrak G}$ on each $A_i\times B_j$ are small).
Observe that for each $j$, 
$$
\sum_i \wt{\nu}(A_i\cap \wt{T}_0^{-1} B_j)=\wt{\nu}(\wt{T}_0^{-1} B_j)
=\wt{\mu}(B_j).
$$
Now build a measurable partition $\{B_{ij}\}$ of $B_j$ so that $\wt{\mu}(B_{ij})=
\wt{\nu}(A_i\cap \wt{T}_0^{-1} B_j)$, and a measure theoretic Borel isomorphism $(A_i\cap \wt{T}_0^{-1}B_j,\wt{\nu})\to (B_{ij},\wt{\mu})$. These maps glue to a Borel measure theoretic  isomorphism $\wt{T}:(\mathfrak G,\wt{\nu})\to  (\mathfrak G,\wt{\mu})$. Clearly,
$$
|d_{\mathfrak G}(\xi,\wt{T}\xi)-d_{\mathfrak G}(\xi,\wt{T}_0\xi)|\leq \diam (B_j)\ \ \text{ for all }\xi\in A_i\cap \wt{T}_0^{-1} B_j.
$$
Integrating and summing over $i,j$, we find that 
$|\int_{\mathfrak G} d_{\mathfrak G}(\xi, \wt{T}\xi) d\wt{\nu}-\int_{\mathfrak G} d_{\mathfrak G}(\xi, \wt{T_0}\xi) d\wt{\nu}|<\sup_j [\diam(B_j)]$, which can be made arbitrarily small.

Since $\wt{\mu}$ and $\wt{\nu}$ are carried by the graph
$\{(x,E^u(x);x,E^s(x)):x\in X\}$, the map
$$
T(x):=\text{first coordinate of }\wt{T}(x;E^u(x),E^s(x))
$$
is  one-to-one, Borel,  $\mu=\nu\circ T^{-1}$ and \eqref{e.rach} holds with  $const:=3(1+\diam(\mathfrak G))C$.

Since $c(x,y):=d(x,y)+\dist(E^s(x),E^s(y))+ \dist(E^u(x),E^u(y))$ is bounded by a constant times $d_{\mathfrak G}(\xi(x),\xi(y))$ on $X$,
Cor \ref{t.Osceledets-Stab} follows.
\qed

\subsection{Exponential Tails in Pesin Theory (Thm~\ref{t.Tail-Estimates-diffeos}, Cor~\ref{c.tail})}\label{s.Tail-Diffeos}
By Remark \ref{r.eps-tilde-arbit-small}, for any $\wt{\chi}$ small and all $\wt{\eps}$ small enough, we can choose the coding $(\Sigma,\pi)$ in Thm~\ref{t.SPR-coding} so that it satisfies the bornological property (Def.~\ref{d.bornological}(b))  with $\wt\chi, \wt\eps$.

Fix some vertex $a$ of the graph which defines the Markov shift $\Sigma$. By the bornological property, there is a $(\tilde{\chi},\tilde{\eps})$-Pesin block $\Lambda$ such that
$$
\Lambda\supset \pi(\{x\in\Sigma^\#: x_0=a\})\mod\nu,\text{ for all }\nu\in \Proberg(\chi,f|_X).
$$
By \eqref{e.chi-kappa-def}, the MME $\mu$ is in $\Proberg(\chi,f|_X)$. Hence the hitting time functions
\begin{align*}
&\tau_\Lambda(x):=\inf\{n\geq 1:f^n(x)\in \Lambda\}& \text{ (a function on $X$)}\\
&\tau_a(x):=\inf\{n\geq 1:x_n=a\}& \text{ (a function on $\Sigma$)}
\end{align*}
 satisfy the inequality $\tau_\Lambda\circ\pi\leq \tau_a$ $\hmu$-a.e. in $\Sigma$.
Consequently,
$$\mu[\tau_\Lambda>n]=\hmu[\tau_\Lambda\circ\pi>n]\leq \hmu[\tau_a>n].$$
The last expression is $O(\theta^n)$ as $n\to\infty$ for some $0<\theta<1$,  because $\Sigma$ is SPR, see Thm~\ref{c.Tail-CMS}.
We found a $(\tilde{\chi},\tilde{\eps})$-Pesin block $\Lambda$ whose hitting time has an exponential tail.
This is also the case for all sets $\Lambda'\supset \Lambda$.
Therefore, for all $\chi'\le\tilde{\chi}$ and $\eps'\ge\tilde{\eps}$,  there are $({\chi},{\eps})$-Pesin blocks $\Lambda'$ such that $\mu[\tau_{\Lambda'}>n]=O(\theta^n)$.

Since $\wt{\eps}$ is arbitrarily small (with $\wt{\chi}$ fixed), this argument shows that for every $\chi$ sufficiently small and for every $0<\eps<\chi$, there exists a $(\chi,\eps)$-Pesin block such that $\mu[\tau_\Lambda>n]=O(\theta^n)$. We have proved the first part of Thm~\ref{t.Tail-Estimates-diffeos}.
\smallskip

Next we consider the tail of the $(\tilde{\chi},\tilde{\eps})$-optimal Pesin bound $K_\ast(x)$ (Def.~\ref{d.Optimal-Pesin-Constant}).
$K_\ast$ is bounded on  $(\tilde{\chi},\tilde{\epsilon})$-Pesin blocks, and in particular, on $\Lambda$.

Let $C:=\sup_\Lambda K_\ast.$ By \eqref{e.temperability},
 $[K_\ast>Ce^{n\tilde{\eps}}]\subset [\tau_\Lambda>n]$.
Since $\mu[\tau_\Lambda>n]=O(\theta^n)$,   $\mu[\log K_\ast>n]=O(\gamma^n)$, with $\gamma:=\theta^{1/\wt{\eps}}$.

Note that $K'_\ast\leq K_\ast$ for all $(\chi,\eps)$-optimal Pesin bounds $K_\ast'$  such that $\tilde{\eps}\leq \eps<\chi\leq \tilde{\chi}$. Therefore $\mu[\log K'_\ast>n]=O(\gamma^n)$ for such $\chi,\eps$. Since $\wt{\eps}$ can be chosen arbitrarily small (with $\wt{\chi}$ fixed), this is the case for all $\chi$ sufficiently small and  $0<\eps<\chi$.
\hfill$\Box$

\section{Alternative Characterizations of the SPR Property}\label{s.converse-statements}
Some
of the consequences of the SPR property we discussed in \S \ref{s.statement-of-conseq} are actually equivalent to the SPR property. In particular we prove here Thm~\ref{t.tail}.
\medskip

Let $f$ be a $C^{1+}$ diffeomorphism of a closed manifold $M$,  $d=\dim(M)$ and $X$ be a Borel homoclinic class with $h_{\Bor}(f|_X)>0$.
Recall the definition of $\beta$-quasi-H\"older continuous functions and their norm $\|\cdot\|'_\beta$ in \S\ref{s.setup-for-consequences}.
$
$
Consider the following properties:\begin{enumerate}[(I)]
\item\label{i.SPR} {\bf SPR} (Def.~\ref{def-SPR-diffeo-local}). There is $\chi>0$ s.t.  for all $\eps>0$ there are a $(\chi,\eps)$-Pesin block $\Lambda$, $h_0<h_{\Bor}(f|_X)$, and $\tau>0$ s.t.: $\forall\nu\in\Proberg(f|_X)$, $h(f,\nu)>h_0\Rightarrow\nu(\Lambda)>\tau$.

\medskip
\item\label{i.Tail} {\bf MME with Exponential Tails.} There are $\mu\in\Proberg(f|_X)$ and $\chi,\delta>0$ s.t.:
\begin{enumerate}[(i)]
\item The measure $\mu$ is an MME for $f|_X$. \item For every $0<\eps<\chi$, there are a $(\chi,\eps)$-Pesin block $\Lambda$ and $0<\theta<1$ s.t.
    $
    \mu[\tau_\Lambda>n]=O(\theta^n)$ as $n\to\infty$, where $\tau_\Lambda(x):=\inf\{n\geq 0: f^n(x)\in \Lambda\}.
    $
\item
Every $\nu\in\Proberg(f|_X)$ with $h(f,\nu)>h_{\Bor}(f|_X)-\delta$ is $\chi$-hyperbolic. 
\end{enumerate}

\medskip
\item\label{i.SPR-coding} {\bf SPR Hyperbolic  Coding.}
There is an irreducible  SPR entropy-full hyperbolic coding $(\Sigma,\pi)$ with properties ($\Sigma$\ref{i.Sigma1})--($\Sigma$\ref{i.Sigma3}) of Thm~\ref{t.SPR-coding}.
    
\medskip
\noindent
\item\label{i.Robustness-Strong} {\bf Quantitative Robustness of the MME.} $f|_{X}$ has a unique MME $\mu$, and, for any $\beta\in(0,1)$, there exists $C_\beta$ such that: for every $f$-invariant measure $\nu$ on $X$
and every $\beta$-quasi-H\"ol\-der continuous function $\psi:X\to\R$
\begin{center} $\left|\int\psi d\mu-\int\psi d\nu\right|\leq C_\beta \|\psi\|_{\beta}'\sqrt{h_{\Bor}(f|_X)-h(f,\nu)}.$\end{center}

\item\label{i.Robustness-Weak} {\bf Qualitative Robustness of the MME.} $f|_{X}$ has a unique MME $\mu$,
 and for all  $\nu_n\in\Proberg(f|_X)$, if
    $h(f,{\nu_n})\to h_{\Bor}(f|_{X})$, then $\nu_n\wsc\mu$ on $M$, and $\Lambda^+(\nu_n)\to\Lambda^+(\mu)$.

\medskip
\item\label{i.Lambda-entropy} {\bf Entropy Hyperbolicity and Entropy Continuity} (\S\ref{s.thm-CE-SPR}).
 There is $\chi>0$ s.t. for all $\nu_n\in \Proberg(f|_X)$, if $h(f,\nu_n)\to h_{\Bor}(f|_X)$ and $\nu_n\wsc\nu$ on $M$, then:
(EH) $\exists i<d$ s.t. $\lambda^i(x)>\chi>-\chi>\lambda^{i+1}(x)$ $\nu$-a.e. \& (EC) $\Lambda^+(\nu_n)\to\Lambda^+(\nu)$.

\medskip
\item\label{i.Entropy-Tightness} {\bf Entropy Tightness} (Def.~\ref{d.entropy-tight-diffeo}). There are $0<{\eps}<{\chi}$ s.t. for all   $\tau\in (0,1)$, there are a $ ({\chi},{\eps})$-Pesin block $\Lambda$ and $h_\tau<h_{\Bor}(f|_X)$ satisfying: for all $\nu\in\Proberg(f|_X)$, $h(f,\nu)>h_\tau\Rightarrow\nu(\Lambda)>\tau$.
\end{enumerate}

\begin{theorem}\label{thm-characterizations}
Let $f$ be a $C^{1+}$ diffeomorphism of a closed manifold and $X$ a Borel homoclinic class with $h_{\Bor}(f|_X)>0$.
Then  properties \eqref{i.SPR}--\eqref{i.Entropy-Tightness} are all equivalent.
\end{theorem}
\begin{proof}
{\bf \eqref{i.SPR}$\Rightarrow$\eqref{i.Tail}:}  Thm~\ref{t.MME-exists} gives (i), Thm~\ref{t.Tail-Estimates-diffeos} gives (ii), and (iii) holds, since any ergodic measure which charges a $(\chi_\text{I},\eps)$-Pesin block is $\chi_\text{II}$-hyperbolic once $\chi_\text{II}<\chi_\text{I}$.

\medskip
\noindent
{\bf \eqref{i.Tail}$\Rightarrow$\eqref{i.SPR-coding}:}
Choose $\chi_\text{II}, \delta_\text{II}$ and $\mu$ as in \eqref{i.Tail} and let $h:=h_\Bor(f|_X)$.
The constructions of \S\ref{ss.construction-hyperbolic} and \S\ref{ss.construction-bornological} and in particular Thm~\ref{t.SPR-coding} give, for every sufficiently small $\chi\in (0,\chi_\text{II})$, a Markovian coding $(\Sigma,\pi)$ with  ($\Sigma$\ref{i.Sigma1})--($\Sigma$\ref{i.Sigma7}); it is irreducible
by ($\Sigma$\ref{i.Sigma1}). By \eqref{i.Tail}(iii) and Cor~\ref{c.coro-transitive-full-hyp-coding}, we find that $(\Sigma,\pi)$ is entropy-full and by ($\Sigma$\ref{i.Sigma5}) it is hyperbolic. It remains to show that $\Sigma$ is SPR.

 By Lemma \ref{l-entropy-equal}, $h=h_{\Bor}(f|_X)=h_{\Bor}(\Sigma)$.
 By Entropy-fullness and Lemma~\ref{l.Hyp-Mark-Coding-Prop}, $\mu=\pi_\ast\hmu$ where $\hmu$ is an ergodic $\sigma$-invariant measure on $\Sigma$, with the same entropy as $\mu$.  Since
 $
h(\sigma,\hmu)=h(f,\mu)=h_{\Bor}(f|_X)=h_{\Bor}(\Sigma)$, $\hmu$ is the MME of $\Sigma$.

By ($\Sigma$\ref{i.Sigma6}),  $(\Sigma,\pi)$ is $\chi$-bornological. Therefore,  there is some $\eps\in (0,\chi)$ so that
the pre-image of any $(\chi,\eps)$-Pesin by $\pi$ can be covered  by a union of  finitely many cylinders of $\Sigma$, and a set of zero $\hmu$-measure.
By \eqref{i.Tail}, there is a $(\chi_\text{II},\eps)$-Pesin block $\Lambda$ and $0<\theta<1$ such that $\mu[\tau_{\Lambda}>n]=O(\theta^n)$ as $n\to\infty$.  Since $\chi<\chi_\text{II}$, $\Lambda$ is a $(\chi,\eps)$-Pesin block, and there is a finite collection $W^\ast$ of symbols $w$ such that
\begin{equation}\label{pi-p}
\hmu\left(\pi^{-1}(\Lambda)\setminus\bigcup_{w\in W^\ast} [w]\right)=0.
\end{equation}

Let $\mathfs G$ denote the countable directed graph such that $\Sigma=\Sigma(\mathfs G)$. By a theorem of Gurevich \cite{Gurevich-Topological-Entropy}, there is a finite collection $W\supset W^\ast$ of vertices of $\mathfs G$ such that the restriction of $\mathfs G$ to $W$  is a connected graph $\mathfs G_{0}$ with $
h_0:= h_\Bor(\Sigma(\mathfs G_{0}))>h-|\ln\theta|.
$

The measure $\hmu$ is the MME of  $\Sigma$, and is therefore given by Thm~\ref{t.Parry-Gurevich-ASS}. So there are constants $c_{ab}>0$ so that for any non-empty cylinder $[a,\xi_1,\ldots,\xi_n,b]$,
$$
\hmu([a,\xi_1,\ldots,\xi_n,b])=c_{ab} e^{-nh}.
$$
Let $c:=\min\{c_{ab}:a,b\in W\}$, then for every $n\geq 0$, \begin{align*}
&\#\left\{(a,\xi_1,\ldots,\xi_n,b):\begin{array}{c}
a,b\in W\\
\xi_1,\ldots,\xi_n\not\in W\end{array}\ ,\ [a,\xi_1,\ldots,\xi_n,b]\neq \emptyset\right\}\\
&\hspace{1cm}\leq c^{-1} e^{nh}
\hmu\left\{\un{x}\in\Sigma^\#:\begin{array}{c}
x_0,x_{n+1}\in W\\
x_1,\ldots,x_n\not\in W\end{array}\right\}\\
&\hspace{1cm}\leq c^{-1} e^{nh}\hmu\left\{\un{x}\in\Sigma^\#:
x_1,\ldots,x_n\not\in W^\ast\right\}\ \ \ (\text{since } W^\ast\subset W)\\
&\hspace{1cm}\leq c^{-1} e^{nh}\hmu\{\un{x}\in\Sigma^\#: f^j(\pi(\un{x}))\not\in \Lambda\ (j=1,\ldots,n)\}
,\text{ by \eqref{pi-p}}\\
&\hspace{1cm}\leq c^{-1} e^{nh} \mu[\tau_\Lambda\circ f>n-1]=O(e^{nh}\theta^n)=O(e^{n(h_0-\kappa)}),
\end{align*}
where $\kappa:=h_0- (h-|\ln\theta|)>0$. By Thm~\ref{t.SPR-as-Entropy-Gap-Shifts}(6), $\Sigma$ is SPR.

\medskip
\noindent
{\bf \eqref{i.SPR-coding}$\Rightarrow$\eqref{i.Robustness-Strong}:}
By Lemma~\ref{l-entropy-equal} $h_\Bor(f|_X)=h_\Bor(\Sigma)$.
The existence and uniqueness of the MME of $f|_X$ is proved exactly as in \S\ref{s.Proof-existence-MME-diffeo}.
The other assertions of \eqref{i.Robustness-Strong} are trivial for measures with entropy bounded away from $h_{\Bor}(f|_X)$; therefore for proving them is sufficient to consider measures with entropy close to $h_{\Bor}(f|_X)$.

By entropy-fullness, these measures are carried by $\pi(\Sigma^\#)$, and can therefore be lifted without changing the entropy to $\Sigma$.
Now \eqref{i.Robustness-Strong} follows from the SPR property of $\Sigma$, as in the proof of Thm \ref{t.Kadyrov-Ineq-diffeos}
at \S\ref{s.Kadirov-Proof-diffeos}. \medskip

\noindent
{\bf \eqref{i.Robustness-Strong}$\Rightarrow$\eqref{i.Robustness-Weak}:}
Applying~\eqref{i.Robustness-Strong} to the geometric potentials $J^u$ in Example \ref{e.Geometric-Is-Quasi-Holder},
we get $\Lambda^+(\nu_n)\to\Lambda^+(\mu)$.
Since any continuous test function on $M$ can be approximated in the supremum norm by a H\"older continuous function on $M$, and every H\"older continuous function on $M$ is quasi-H\"older, \eqref{i.Robustness-Strong} also implies $\nu_n\wsc\mu$ on $M$.

\medskip
\noindent
{\bf \eqref{i.Robustness-Weak}$\Rightarrow$\eqref{i.Lambda-entropy}:}
By \eqref{i.Robustness-Weak}, $f|_X$ has a unique MME $\mu$. This measure is hyperbolic, because $X$ is a Borel homoclinic class. It is also ergodic (otherwise it would not be unique).
Suppose $\nu_n\in\Proberg(f|_X)$, $h(f,\nu_n)\to h_{\Bor}(f|_X)$, and  $\nu_n\wsc\nu$ on $M$. By \eqref{i.Robustness-Weak}, $\nu=\mu$, and (EC) follows. Since $\mu$ is ergodic and hyperbolic, (EH) follows.

\medskip
\noindent
{\bf \eqref{i.Lambda-entropy}$\Rightarrow$\eqref{i.Entropy-Tightness}:} This  is the content of  Thm~\ref{thm-CE-SPR} together with Remark~\ref{r.entropy-tightness}.

\medskip
\noindent
{\bf  \eqref{i.Entropy-Tightness}$\Rightarrow$\eqref{i.SPR}:} This the content of Prop.~\ref{prop-tight-implies-SPR}.
\end{proof}

}

\appendix

\part{Appendices}
\section{The Grassmannian Bundle}\label{app-grassmannian}
\subsection{The Grassmannian Bundle}
Suppose $V$ is an inner product space of dimension $n$, and $k\leq n$. The space of all $k$-dimensional linear subspaces of $V$ can be naturally identified with the homogeneous space
$O(n)/O(k)\times O(n-k)$. The smooth Riemannian manifold thus obtained  is called the  {\em $k$-th Grassmannian of $V$}, and is denoted by   $\mathbb G(k, V)$.

The {\em $k$-th Grassmannian bundle} of a closed smooth Riemannian manifold $M$  is the
smooth and compact bundle  $\mathfrak G(k,M)$ with base space $M$, and  fibres $\mathbb G(k,T_x M)$ $(x\in M)$.  Given $x\in M$ and $E\in \mathbb G(k,T_x M)$, there is a natural identification
$
T_{(x,E)}\, \mathfrak G(k,M)\cong T_x M\oplus T_E \, {\mathbb G}(k,T_x M).
$
The Riemannian metric on $\mathfrak G(k,M)$ is defined by declaring the two spaces on the right orthogonal, and endowing them with the inner products they inherit from the Riemannian structure on $M$ and ${\mathbb G}(k,T_x M)$.

\subsection{Lifts of Measures to the Grassmannian Bundle (Proof of Prop.~\ref{prop-erg-lifts})}\label{app-dilavg}
Suppose $\nu\in\Proberg(f)$  has unstable dimension $1\le i\le d$. Recall that its unstable lift $\wt{\nu}^+$ to $\mathfrak G(i,M)$ is the unique lift carried by $\{(x,E^+_x):x\in M\}$. Necessarily, $\wt{\nu}^+(\vf)=\int\log |\det(Df|_{E^+_x})|d\nu$. The invariance of $\nu$ and the chain rule give $\wt{\nu}^+(\vf)=(1/n)\int\log |\det(Df^n|_{E^+_x})|d\nu$. Taking the limit,  Oseledets theorem gives
$$\wt{\nu}^+(\vf)=\Lambda^+(\nu).$$

Next we show that any ergodic lift $\tnu\neq \tnu^+$ of $\nu$ to $\mathfrak G(i,M)$ satisfies $\wt{\nu}(\vf)<\Lambda^+(\nu)$.
For $\tnu$-a.e. $(x,F)\in\mathfrak G(i,M)$, let $\alpha_1>\dots>\alpha_r$ be the Lyapunov exponents of $Df$ in $F$.
(Since $\wt{\nu}$ is ergodic,  $\tnu$-a.e. $(x,F)$ has the same exponents.)
Using ergodicity and $\tnu\ne\tnu^+$, we have $F\ne E^\pci_x$ for $\tnu$-a.e. $(x,F)$. Therefore there exists a unit vector $u\in F\setminus E^\pci_x$ {such that} the limit of $\tfrac 1 n \log \|Df^{-n}. u\|$ is negative. This implies
$$\alpha_r<0<\lambda^i(\nu).$$
 Let $F= F^1\supset F^2\supset\dots\supset F^r\ne \{0\}$ be the flag associated to $\alpha_1>\dots>\alpha_r$. Pick inductively a basis $b_r$ of $F^r$, $b_{r-1}$ a collection of vectors such that $b_r\cup b_{r-1}$ is a basis of $F^{r-1}$, etc, so  that $b_1\cup\cdots\cup b_r=\{v^1_x,\ldots,v^i_x\}$ is a basis of $F$. Let
 $V_x^n$ denote the volume of the parallelepiped with sides $Df^n v_x^1,\ldots,Df^n v_x^i$. Then
\begin{align*}
\lim_{n\to\infty}&\tfrac{1}{n}\log|\det(Df^n|_F)|
=\lim_{n\to\infty}\tfrac{1}{n}\log \tfrac{V_x^n}{V_x^0}=\sum_{s=1}^r|b_s|\alpha_s\\
&\leq \lambda^1(\nu)+\lambda^2(\nu)+\dots+\lambda^{i-1}(\nu)+\alpha_r\leq \Lambda^+(\nu)-[\lambda^i(\nu)-\alpha_r].
\end{align*}
Integrating, we obtain by the bounded convergence theorem and the invariance of $\nu$ that
$
\displaystyle\wt{\nu}(\vf)=\int\log|\det(Df|_F)|d\wt{\nu}\leq \Lambda^+(\nu)-[\lambda^i(\nu)-\alpha_r]
<\Lambda^+(\nu).
$
\qed

\section{Spectral Gap for SPR Markov Shifts and Consequences}\label{appendix-Markov-Shifts}

The spectral gap of SPR Markov shifts manifests itself in the action of the ``transfer operator" on a space of regular ``one-sided functions", 
therefore we begin with a discussion of these objects.

\subsection{One-Sided Shifts, One-Sided Functions and One-Sided Measures}\label{s.one-sided}

Let $\mathfs G$ be a countable directed graph with set of vertices $\mathfs V$. Let
$$
\Sigma^+=\Sigma^+(\mathfs G):=\{(x_0,x_1,\cdots)\in \mathfs V^{\N\cup\{0\}}: x_i\to x_{i+1}\text{ for all }i\in\Z\}.
$$
We endow $\Sigma^+(\mathfs G)$ with the metric
\begin{equation}\label{e.natural-metric} d(\un{x},\un{y}):=
\exp[-\min\{i\geq 0:x_i\neq y_i\}] \text{ (or $0$ if  $\un{x}=\un{y}$)},
\end{equation}
and the action of the {\em left shift map}
$
\sigma:\Sigma^+(\mathfs G)\to \Sigma^+(\mathfs G),\ \ \sigma(\un{x}):=\un{y},\text{ where }y_i:=x_{i+1}.
$
This map is continuous, but not invertible.
The resulting topological dynamical system is called the {\em one-sided  Markov shift} associated to $\mathfs G$. We keep the name {\em Markov shift} for the invertible system $(\Sigma,\sigma):=(\Sigma(\mathfs G),\sigma)$, introduced in \S\ref{s.Markov-shifts} (using sometimes the phrase  \emph{two-sided} Markov shifts for emphasis).

All the results for (two-sided) Markov shifts stated in \S\ref{s.Markov-shifts} and \S\ref{s.SPR-shift}  extend to one-sided shifts, with two exceptions:
\begin{enumerate}[(1)]
\item Lemma \ref{l.local-compactness-crit}. Call a directed graph {\em forward locally finite} if every vertex has a finite {\em out}-going degree. The correct one-sided version of this lemma is as follows: 
\emph{Let $\mathfs G$ be a proper countable directed graph. $\Sigma^+(\mathfs G)$ is locally compact iff
$\mathfs G$ is forward locally finite. In this case the cylinders are compact.}
The proof is simple, and we omit it.

\medskip
\item Thm \ref{t.Parry-Gurevich-ASS}(1). In the case of $\Sigma^+(\mathfs G)$, the MME is isomorphic to the product of a system with Bernoulli {\em natural extension}, and a cyclic permutation of $p$ points.
\end{enumerate}

There is an obvious semi-conjugacy $\vartheta:(\Sigma,\sigma)\to(\Sigma^+,\sigma)$, given by
$$
\vartheta[(v_n)_{n\in\Z}]:=(v_n)_{n\geq 0}=:\un{v}^+.
$$

The {\em one-sided projection} of a $\sigma$-invariant probability measure $\mu$ on  $\Sigma$, is  the
 $\sigma$-invariant measure on $\Sigma^+$ defined by $$\mu^+:=\mu\circ \vartheta^{-1}.
 $$
Conversely, any $\sigma$-invariant measure $\mu^+$ on $\Sigma^+$ lifts to a unique $\sigma$-invariant measure $\mu$ on $\Sigma$ such that $\mu\circ\vartheta^{-1}=\mu^+$.  The lift is  determined by the following condition:
$$
\mu({_m}[v_0,\ldots,v_{n-1}])=\mu^+([v_0,\ldots,v_{n-1}])\text{ for all cylinders and all $m\in\Z$}.
$$
The lift and the projection  have the same entropy,  and $\mu$ is ergodic {(resp.\ strongly mixing)} iff $\mu^+$ is ergodic {(resp.\ strongly mixing)}.

Every function $\psi^+:\Sigma^+\to\R$ lifts to a function $\psi^+\circ\vartheta:\Sigma\to\R$.
Functions on $\Sigma$ of this  form are called {\em one-sided functions}.
A function on $\Sigma$ is generally not one-sided, however we have the following important result:

\begin{theorem}[Sinai Lemma]\label{t.Sinai-Lemma}
Let $\Sigma$ be a {(two-sided)} Markov shift. {For any $\beta>0$, there is $C_\beta>0$ s.t.}
every $\beta$-H\"older continuous $\psi:\Sigma\to \R$ can be written as
$
\psi={\widetilde\psi}+h-h\circ\sigma
$
where $\varphi,h$ are $\beta/2$-H\"older continuous, {$\widetilde \psi$} is one-sided,
{and $\|\widetilde \psi\|_{\beta/2}\leq C_\beta \|\psi\|_\beta$.}
\end{theorem}
\noindent
The proof for subshifts of finite type in \cite[Lemma 1.6]{Bowen-LNM} extends verbatim to the countable state case.
Note that since $\diam(\Sigma)<\infty$, $\phi,\psi$ and $h$ are bounded.

\begin{proposition}\label{p.conditional-measures}
{Let $\Sigma$ be a {(two-sided)} irreducible Markov shift with finite Gurevich entropy, let
$\mu$ be an equilibrium measure of a $\beta$-H\"older potential $\phi$, and let }
$\mathcal{A}_0^\infty$ be the $\sigma$-algebra generated by non-negative coordinates.

There is $C_\phi>0$ such that
for every $\beta$-H\"older continuous function
$\psi:\Sigma\to\R$, there exists a $\beta/2$-H\"older continuous function $\psi^+:\Sigma^+\to\R$ which satisfies
$
\|\psi^+\|_{\beta/2}\leq C_\phi\|\psi\|_\beta \; \text{ and }\; \E_{\mu}(\psi|\mathcal{A}_0^\infty)=\psi^+\circ\vartheta \quad \mu\text{-a.e.}$
\end{proposition}
The proof requires some tools and is postponed to the end of \S~\ref{s.transfer-section}.

\subsection{Transfer Operators}\label{s.transfer-section} Suppose $T:\Omega\to\Omega$ is a (possibly non-invertible) measurable map which preserves a probability measure $m$.  Fix $\vf\in L^1(m)$. $T$ maps the signed measure $dm_\vf:={\vf\, dm}$ to the measure $T_\ast(m_\vf):=m_\vf\circ T^{-1}$. Since $m_\vf\circ T^{-1}\ll m$, we can write
$$
T_\ast(\vf dm)=L(\vf)\, dm,\text{ where }L(\vf):=\frac{dm_\vf\circ T^{-1}}{dm}.
$$
The operator $L:L^1(m)\to L^1(m)$ defined this way is called the {\em transfer operator} of $m$.
The following is well-known, and straightforward.
\begin{lemma}\label{l.transfer-op-properties} $L\colon L^1(m)\to L^1(m)$ is a bounded linear operator, and:\begin{enumerate}[(1)]
\item $L(1)=1$;
\item $\forall\psi\in L^1(m)$, $\int L(\psi) d m=\int \psi d m$;
\item $\forall\psi\in L^1(m)$ and $\chi\in L^\infty$, $L((\chi\circ T)\cdot \psi)=\chi \cdot L(\psi)$.
\end{enumerate}
In particular $\int \chi\cdot (L\psi) dm= \int L((\chi\circ T) \cdot \psi) dm =\int(\chi\circ T) \cdot \psi dm$. Moreover, the last property characterizes the transfer operator.

\end{lemma}

Suppose $\Sigma^+$ is  a one-sided Markov shift $\Sigma^+$, and   $\varphi\colon \Sigma^+\to \RR$ is a function. {\em Ruelle's operator} is the formal sum (which may or may not converge)
$$
(L_\varphi \psi)(\un{x})=\sum_{\sigma(\un{y})=\un{x}}e^{\varphi(\un{y})}\psi(\un{y}).
$$
We recall the local H\"older continuity introduced in \S\ref{ss.metric}. We have the following.
\begin{proposition}\label{p.transfer-op-formula}
Let $\Sigma^+$ be an irreducible one-sided Markov shift with finite Gurevich entropy,
and $\mu$ an equilibrium measure for a $\beta$-H\"older continuous {function} $\phi\colon \Sigma^+\to \RR$.
Then there is a locally $\beta$-H\"older continuous function $u$ such that:
\begin{enumerate}
\item The {function}  $\varphi:=\phi+u-u\circ \sigma-P(\sigma,\mu,\phi)$ satisfies
$\sum_{\sigma(\un{y})=\un{x}}e^{\varphi(\un{y})}=1$. \item {$L_\varphi$ defines an operator $L^1(\mu)\to L^1(\mu)$ which coincides with
the transfer operator of $\mu$.}
\item
$L_\varphi$ preserves the space of $L^\infty$ functions having continuous representatives $\psi$, and satisfies
$\sup|L_\varphi \psi| \leq \sup|\psi|$. \end{enumerate}
\end{proposition}
\begin{proof}
Let us consider the operator $L_{\phi}$.
In  \cite{Buzzi-Sarig}, it is shown that the existence of an equilibrium measure for $\phi$ entails the positive recurrence of $\phi$, and implies the existence of $\lambda>0$, a positive continuous function $h$, and a (possible infinite) Borel measure $\nu$ which is finite and positive on cylinders, such that $L_{\phi} (h)=\lambda{h}$, ${(L_{\phi})_\ast} \nu=\lambda\nu$, and $\int h d\nu=1$. The equilibrium measure $\mu$ is then given by ${h\, d\nu}$ and $P(\sigma,\mu,\phi)=\log(\lambda)$.

Now let $g:=\frac{e^\phi h}{\lambda h\circ\sigma}$. Then  
$\sum_{\sigma(\un{y})=\un{x}} g(\un{y})=\lambda^{-1} h^{-1} L_{\phi} (h)= 1$.
{Setting} $u=\log h$ and $\varphi=\log(g)$,
the operator $\widehat L:=L_\varphi$ satisfies: $\widehat L(\psi)=\lambda^{-1} h^{-1} L_{{\phi}}(h\psi)$, hence
$
{\widehat L_\ast}\mu={\widehat L_\ast}(h \nu)=h \nu=\mu.
$
It follows that $\int \chi\cdot \widehat L(\psi)d\mu=\int \widehat L(\chi\circ\sigma \cdot \psi)d\mu=\int (\chi\circ\sigma)\cdot \psi d\mu$. Thus $\widehat L$ is the transfer operator of $\mu$.
The local H\"older regularity of $u$ follows from the regularity properties of $\log h$ stated in  \cite[Lemma 6]{Sarig-ETDS-99}.

Since $\sum_{\sigma(\un{y})=\un{x}} g(\un{y})=1$ and $\log g$ is locally H\"older continuous, $L_\varphi$ preserves the class of bounded continuous functions, and $\sup|L_{\varphi}\psi|\leq\sup|\psi|$.
\end{proof}

\begin{remark}\label{r.transferMME}
In the special case when $\mu$ is an MME and $\ell_u, p_u, p_{uv}$ are as in Thm \ref{t.Parry-Gurevich-ASS},  $g=e^\varphi$ is given by
$g(\un{y}):=\frac{p_{y_0}p_{y_0 y_1}}{p_{y_1}}=e^{-h_{\Bor}(\Sigma)}\ell_{y_1}^{-1}\ell_{y_0}$
and $u(\un{y})={\log}\ell_{y_0}$. 
\end{remark}

\begin{proof}[Proof of Proposition \ref{p.conditional-measures}]
By Thm \ref{t.Sinai-Lemma}, $\mu$ is an equilibrium measure of a {\em one-sided}  $\beta/2$-H\"older potential ${\widetilde \phi}$. 
{Let $g=e^{\varphi}$ where $\varphi$ is the function in Prop \ref{p.transfer-op-formula}.}
Define $g^{(n)}:=\prod_{i=0}^{n-1}g\circ \sigma^i$.
An inductive argument which starts with Prop~\ref{p.transfer-op-formula}(1) shows
\begin{equation}\label{e.consistency}
\sum_{\sigma^n(\un{y})=\un{x}} g^{(n)}(\un{y})=1.
\end{equation}
Given $\un{x}^+\in\Sigma^+$ and a cylinder ${_{-n}}[a_{-n},\ldots,a_{-1},x_0]$ in the {\em two sided} shift $\Sigma$, let 
$$
\mu_{\un{x}^+}\bigl({{_{-n}}[}a_{-n},\ldots,a_{-1},x_0]\bigr)= g^{(n)}(a_{-n},\ldots,a_{-1},x_0,x_1,\ldots).$$
By \eqref{e.consistency},  $\mu_{\un{x}^+}$ can be extended to a probability measure $\mu_{\un{x}^+}$ on $\{\un{y}:y_i=x_i\ (i\geq 0)\}$.

Given a $\beta$--H\"older continuous $\psi:\Sigma\to\R$, define $\psi^+$ on $\Sigma^+$ by
  $$
  \psi^+(\un{x}^+):=\int\psi d\mu_{\un{x}^+}.
  $$

\medskip
\noindent
{\em Claim 1.\/} ${\psi}^+\circ\vartheta$ is measurable with respect to the  $\sigma$-algebra $\mathcal A_0^\infty$.

\medskip
\noindent
{\em Proof.\/} If  $\psi=\mathds{1}_{[\un{a}]}$, then
$(\psi^+\circ\vartheta)(\un{x})={\mu}_{\un{x}^+}[\un{a}]$ is continuous and one-sided, whence $\mathcal A_0^\infty$-measurable. If $\psi$ is H\"older continuous, use a standard approximation argument.

\medskip
\noindent
{\em Claim 2.\/} $(\psi^+\circ\vartheta)(\un{x})=\E_\mu(\psi|\mathcal{A}_0^\infty)(\un{x})$ for $\mu$-a.e. $\un{x}$.

\medskip
\noindent
{\em Proof.\/}
We need to show
$\int \chi \cdot{\psi}^+\circ\vartheta d\mu=\int \chi\cdot{\psi} d\mu$ for all $\mathcal{A}_0^\infty$--measurable $\chi\in L^\infty$, or equivalently,
$\mu=\int_{\Sigma^+}\mu_{\un{x}^+}d\mu^+.$
For any cylinder $C={_{-n}}[a_{-n},\ldots,a_{-1},b_0,\ldots,b_m]$,
\begin{align*}
&\int_{\Sigma^+}\mu_{\un{x}^+}(C)d\mu^+(\un{x}^+)=\int_{[\un{b}]}g^{(n)}(a_{-n},\ldots,a_{-1},\un{x}^+)d\mu^+(\un{x}^+)\\
&=\int_{\Sigma^+}\mathds{1}_{[\un{b}]}(\un{x}^+)\sum_{\sigma^n(\un{y}^+)=\un{x}^+} g^{(n)}(\un{y}^+) \mathds{1}_{[\un{a}]}(\un{y}^+) d\mu^+(\un{x}^+)=\int \mathds{1}_{[\un{b}]} \cdot L^n (\mathds{1}_{[\un{a}]}) d\mu^+,
\end{align*}
where $L$ is the transfer operator of $\mu^+$, see Prop \ref{p.transfer-op-formula}. Thus by Lemma \ref{l.transfer-op-properties},
$$
\int_{\Sigma^+}\mu_{\un{x}^+}(C)d\mu^+(\un{x}^+)=\int \mathds{1}_{[\un{b}]}\circ\sigma^n \cdot \mathds{1}_{[\un{a}]} d\mu^+=\mu^+([\un{a},\un{b}])=\mu({_0}[\un{a},\un{b}])=\mu(C)
$$
(the last equality uses the shift-invariance of $\mu$).  
Thus
$\mu(C)=\int \mu_{\un{x}^+}(C)d\mu$ for all cylinders $C$, whence for  all Borel sets.

\medskip
\noindent
{\em Claim 3.\/}  $\psi^+\circ\vartheta$ is $\beta/2$-H\"older continuous and $\|\psi^+\|_{\beta/2}\leq C_\phi\|\psi\|_\beta$, where $C_\phi$ depends only on $\phi$.

\medskip
\noindent
{\em Proof.\/} Clearly, $\|\psi^+\|_\infty\leq \|\psi\|_\infty\leq \|\psi\|_\beta$.  To estimate the $\beta/2$-H\"older constant, we fix  $\un{x},\un{y}\in\Sigma$ such that $x_i=y_i$ for $|i|\leq n$.

If $n<2$, then $|\psi^+(\un{x})-\psi^+(\un{y})|\leq 2\|\psi\|_\infty e^{\beta} \cdot e^{-\beta n/2}$.
We thus assume $n\geq 2$.
Let $\Theta:\supp\mu_{\un{x}^+}\to \supp\mu_{\un{y}^+}$ be the map $
\Theta(\un{x})=(\ldots,x_{-2},x_{-1},x_0,y_1,y_2,\ldots).
$ Then $\mu_{\un{y}^+}\circ\Theta\sim \mu_{\un{x}^+}$ and 
$$
J(\un{z}):=\frac{d\mu_{\un{x}^+}}{d\mu_{\un{y}^+}\circ\Theta}=\prod_{k=0}^\infty \frac{g(z_{-k},\ldots,z_{-1},x_0,\ldots,x_n,x_{n+1},\ldots)}{g(z_{-k},\ldots,z_{-1},x_0,\ldots,x_n,y_{n+1},\ldots)}\  \ \ \ \ \mu_{\un{x}^+}\text{-a.e.}
$$

Since $n\geq 2$, we can use the regularity properties of $\log g$  in Prop \ref{p.transfer-op-formula} to see that $|J(z)-1|\leq Ce^{-\beta n/2}$, where $C$ depends only on  $\log g$, whence only on $\phi$. (The exponent is $\beta/2$ and not $\beta$, because {$g=e^{\varphi}$ where}  $\varphi$ is only $\beta/2$-H\"older.) 
Since $d\mu_{\un{x}^+}=J d\mu_{\un{y}^+}\circ\Theta$, we have 
\begin{align*}
&|\psi^+(\un{x}^+)-\psi^+(\un{y}^+)|=\left|\int \psi d\mu_{\un{x}^+}-\int \psi d\mu_{\un{y}^+}\right|=
\left|\int \psi J d\mu_{\un{y}^+}\circ\Theta-\int \psi d\mu_{\un{y}^+}\right|\\
&\leq \int |(\psi J)\circ\Theta^{-1}-\psi|d\mu_{\un{y}^+}\leq \int \biggl(|\psi\circ \Theta^{-1}-\psi|\cdot J\circ\Theta^{-1}+|\psi|\cdot|J\circ\Theta^{-1}-1| \biggr)d\mu_{\un{y}^+}\\
&\leq \|\psi\|_{\beta} e^{-\beta n}\cdot\|J\|_\infty+\|\psi\|_\infty \cdot C e^{-\beta n/2}\leq 
\|\psi\|_\beta e^{-\beta n/2}(2C+1),
\end{align*}
because $\Theta^{-1}$ preserves the $k$-coordinates with $k\leq n$, and  $|J|\leq C+1$.  

Thus $\|\psi^+\|_{\beta/2}\leq C_\phi \|\psi\|_\beta$ for some  $C_\phi$ which only depends on $\phi$.
\end{proof}

\subsection{SPR and the Spectral Gap}
We say that a one-sided shift $\Sigma^+$ is SPR for $\phi^+\colon \Sigma^+\to \RR$ iff the associated two-sided shift $\Sigma$ is SPR for the potential $\phi:=\phi^+\circ \vartheta$ (see Def \ref{d.SPR-for-Shifts}).
\begin{remark}\label{r.project}
Given some H\"older continuous $\phi$ on $\Sigma$, Sinai's lemma {(Thm~\ref{t.Sinai-Lemma})} constructs a function $\phi^+\colon \Sigma^+\to \RR$
such that $\phi^+\circ \vartheta$ is cohomologous to $\phi$. Hence if $\Sigma$ is SPR for $\phi$, then $\Sigma^+$ is SPR for $\phi^+$.
Their pressures are equal. Moreover if $\mu$ is the equilibrium measure on $\Sigma$, then the equilibrium measure on $\Sigma^+$ is $\mu^+:=\vartheta_*\mu$.
\end{remark}

\begin{theorem}[Cyr-Sarig \cite{Cyr-Sarig}]\label{t.SGP} Let {$\Sigma^+$} be an irreducible one-sided Markov shift
with finite Gurevich entropy and period $1$, let $\mu$ be the equilibrium measure of an SPR H\"older continuous potential $\phi$, 
and let $\varphi$ be the function given by Prop~\ref{p.transfer-op-formula}.
For any $\beta>0$ there exists a complex Banach space $(\mathcal L,\|\cdot\|_{\mathcal L})$ of continuous functions $\psi:\Sigma^+\to\mathbb C$ s.t.:
\begin{enumerate}[(a)]
\item\label{item0} For any $\psi\in \mathcal L$ and $\un{x}\in \Sigma^{+}$, the sum $(L_\varphi\psi)(\un{x})=\sum_{\sigma(\un{y})=\un{x}}e^{\varphi(\un{y})}\psi(\un{y})$ converges.
\item\label{item1} $L_\varphi(\mathcal L)\subset\mathcal L$ and $L_\varphi:\mathcal L\to \mathcal L$ is bounded.
\item\label{item2} \emph{Spectral gap:} $L_\varphi=P+N$, where $P$ and $N$ are bounded linear operators on $\mathcal L$,  $PN=NP=0$,
$P^2=P$, $\operatorname{rank}(P)=1$ and $N$ has spectral radius $\rho(N)<1$.
\item\label{item8} If $\chi$ is $\beta$-H\"older continuous and $\psi\in \mathcal L$, then $\chi\cdot \psi\in \mathcal L$ and $\|\chi\cdot \psi\|_{\mathcal L}\leq \|\chi\|_\beta \|\psi\|_{\mathcal L}$.
\item\label{item4} $\mathcal L$ contains the $\beta$-H\"older continuous functions $\psi:\Sigma^+\to\mathbb \RR$.
\item\label{item3} $\mathcal{L}\subset L^1(\mu)$, $L_\varphi$ coincides with the transfer operator of $\mu$ on $\mathcal L$, 
and the operator $P$ is given by $\displaystyle P\psi=\int \psi d\mu\cdot 1.$
\item\label{item6} $\psi\in\mathcal L\Rightarrow |\psi|\in\mathcal L$ and $\bigl\||f|\bigr\|_{\mathcal L}\leq \|f\|_{\mathcal L}$.
\item\label{item7} Convergence in $\mathcal L$ implies uniform convergence on cylinders.
\end{enumerate}
\end{theorem}

\begin{remark}
Thm \ref{t.SGP} as well as all other results in this section hold in the more general setup of  {\em weakly} H\"older continuous SPR potentials $\phi$ on irreducible Markov shifts, possibly with {\em infinite} Gurevich entropy  (and with or without  period one depending on the context),  provided we add the assumption that $\sup\phi<\infty$ and $\phi$ has finite pressure. When $\phi$ is H\"older continuous, it is bounded, and the finite pressure  condition is equivalent to the condition that $\Sigma$ has finite Gurevich entropy.
\end{remark}

\begin{proof}
We will apply the  theory of SPR potentials in \cite{Sarig-CMP-2001} and \cite{Cyr-Sarig}. This theory applies to every  SPR potential which is bounded from above, has finite Gurevich pressure, and is {\em weakly} H\"older continuous in the sense of \S\ref{ss.metric}. The potential $\varphi:=\log g:=\phi+\log h-\log h\circ\sigma-P(\sigma,\mu,\phi)$ from  Prop \ref{p.transfer-op-formula} belongs to this class because  $\log h$ is {\em locally} H\"older continuous, see \cite[Lemma 6]{Sarig-ETDS-99} and its proof, and because condition \eqref{e.SPR-potential2} is invariant under the addition of coboundaries and constants.

Consequently, Theorem 2.1 in \cite{Cyr-Sarig} applies to $\varphi$. Note that the pressure $P_\Bor(\Sigma^+,\varphi)$ vanishes.
We obtain a Banach space $(\mathcal L,\|\cdot\|_{\mathcal L})$ satisfying \eqref{item0}, \eqref{item1}, \eqref{item2},  \eqref{item6}, \eqref{item7}, and the following weak version of \eqref{item4}:
\begin{enumerate}[(e')]
\item\label{item5} $\mathcal L$ contains all the indicators of cylinder sets.
\end{enumerate}

Property \eqref{item8} is not mentioned explicitly in Theorem 2.1, but is noted on page 650 in \cite{Cyr-Sarig}, and is a direct consequence of the definition of $\|\cdot\|_{\mathcal L}$ on page 643.

The derivation of \eqref{item4} is more complicated.
By (\ref{item4}'), $\mathcal L$ contains some {(non-trivial)} bounded continuous functions, e.g. $\psi=\mathds{1}_{[v]}$. 
{As $\psi$ is bounded it belongs to $L^1(\mu)$ and by Prop~\ref{p.transfer-op-formula}, $L^n_\varphi(\psi)=L^n(\psi)$ for all $n\geq 0$.}
By~\eqref{item2}, $L_\varphi^n=P+N^n\to P$ in norm. By~\eqref{item7},
$ L_\varphi^n \psi\to P\psi$ pointwise,
Prop~\ref{p.transfer-op-formula}(3) implies that $\|L_\varphi^n\psi\|_\infty\leq \|\psi\|_\infty$, and by the dominated convergence theorem,
$ L_\varphi^n \psi\to P\psi$ in $L^1(\mu)$.
The measure $\mu$ is mixing, by Theorem ~\ref{t.Parry-Gurevich-ASS} (the period $p$ equals one, by  assumption). Hence for any $\chi\in L^\infty(\mu)$ one gets
$\int(\chi\circ T^n) \cdot \psi d\mu\to \int \chi d\mu \int \psi d\mu$.
By Item (3) in Lemma~\ref{l.transfer-op-properties} we also have
$\int(\chi\circ T^n) \cdot \psi d\mu=\int \chi\cdot (L^n\psi) d\mu
{=\int \chi\cdot (L_\varphi^n\psi) d\mu}
\to \int \chi\cdot P\psi d\mu$.
It follows that $P\psi=\int\psi d\mu$ $\mu$-a.e., whence by continuity, everywhere. Since $P(\mathcal L)\subset\mathcal L$, the constant functions are in $\mathcal L$. We now invoke \eqref{item8}, and obtain \eqref{item4}.

It remains to check \eqref{item3}. Let $\psi\in\mathcal L$.  By~\eqref{item6}, one can  assume $\psi\ge0$. Since $\operatorname{rank}(P)=1$ and $P$ is positive, the previous argument shows $P\psi=c\cdot 1$ for some $c\ge0$ and $L_\varphi^n\psi\to c$ pointwise. Since $L_\varphi$ is positive, $L_\varphi^n\psi\ge0$. By Fatou's lemma,
 $$
 c=\int P\psi\,d\mu = \int \lim_n L_\varphi^n\psi\,d\mu \le \lim_n \int L_\varphi^n\psi\,d\mu = \int \psi\, d\mu. 
 $$
To prove the converse inequality, note that the bounded continuous functions $\psi\wedge M$ ($M\ge1$) belong to $\mathcal L$ by~\eqref{item6}. 
Therefore $P(\psi\wedge M)= \int \psi\wedge M\, d\mu \leq\int \psi d\mu\leq  P\psi$, and  using monotone convergence,
one gets~\eqref{item3}:
\begin{align*}
c&=\int P \psi\, d\mu\geq \int P(\psi\wedge M)\,d\mu =
\int \psi\wedge M\,d\mu\xrightarrow[M\to\infty]{}\int\psi\,d\mu.
\end{align*}
{By Prop~\ref{p.transfer-op-formula}, we get $L_\varphi=L$ on $\mathcal L\subset L^1(\mu)$.}
\end{proof}

\begin{remark}
\eqref{item2} implies that the spectrum of $L_\varphi$ consists of a simple eigenvalue $1$, with eigenprojection $P$, and a compact subset of the open unit disk (indeed of $\{z:|z|\leq\rho(N)\}$, where $\rho(N)$ is the spectral radius). This is the ``spectral gap."
\end{remark}
{From now on we use the notation $L$ for both the transfer operator and the Ruelle operator $L_\varphi$, and we let $P$ be the operator $P\psi=
\int\psi d\mu\cdot 1$}.

\begin{corollary}\label{c.perturbation-theory}
Under the assumptions of Thm~\ref{t.SGP}, there are $\eps>0$ and $\rho_0\in (0,1)$ as follows.
For every $\beta$-H\"older continuous function $\psi$ {on $\Sigma^+$} with $\|\psi\|_\beta\leq 1$ and every $|z|<\eps$ there exist $\lambda_z\in\mathbb C$ and bounded linear operators $P_z,N_z$ on $\mathcal L$ s.t.:
\begin{enumerate}[(1)]
\item $L_z=\lambda_z P_z+N_z$, where $L_z\chi=L(e^{z\psi}\chi)$ for any $\chi\in\mathcal L$.
\item  $P_z^2=P_z$, $\mathrm{rank}(P_z)=1$, $P_z N_z=N_z P_z=0$, and $\rho(N_z)<\rho_0$.
\item $\lambda_0=1$, $P_0=P$, $N_0=N$.
\item $\lambda_z$ is holomorphic on $\{z:|z|<\eps\}$.
\item $P_z, N_z$ are holomorphic on $\{z:|z|<\eps\}$, i.e. they  can be expanded to a power series in $z$ which converges in the operator norm on this disk.
    \end{enumerate}
\end{corollary}
\begin{proof}
By the theory of   analytic perturbations of linear operators on Banach spaces \cite[Thm-IV.3.16]{Kato-Book}, there are $\eps_{pert}>0$, $\rho_0\in (0,1)$ such that for every bounded linear operator $T$ on $\mathcal L$ satisfying $\|T-L\|<\eps_{pert}$, there exist $\lambda(T)\in\mathbb C$ and bounded linear operators $P(T),N(T)$ such that $T=\lambda(T)P(T)+N(T)$, and:
\begin{enumerate}[$\bullet$]
\item  $P(T)^2=P(T)$, $\mathrm{rank}(P(T))=1$, $P(T) N(T)=N(T) P(T)=0$, $\rho(N(T))<\rho_0$;
\item $\lambda(0)=1$, $P(0)=P$, $N(0)=N$;
\item if $z\mapsto T_z$ is analytic, $\lambda(T_z), P(T_z), N(T_z)$ are analytic  on $\{z:\|T_z-L\|<\eps_{pert}\}$.
\end{enumerate}

In order to apply this to $T_z:=L_z$, we fix $\eps:=\min\{1,\eps_{pert}/(\|L\| e)\}$ and show that $\|L_z-L\|<\eps_{pert}$ for $|z|<\eps$ and any choice of {a H\"older continuous} $\psi$ with $\|\psi\|_\beta\leq 1$:
By Property~\eqref{item8} in Theorem~\ref{t.SGP}, the linear operator $M_\psi\colon \chi\mapsto \psi\cdot \chi$
is bounded on $\mathcal L$ and satisfies $\|M_\psi\|\leq \|\psi\|_\beta \leq 1$.
We have $L_z=\sum_{n=0}^\infty \frac{z^n}{n!}L M_\psi^n$, therefore $\|L_z-L\|\leq \|L\|\sum_{n=1}^\infty \frac{|z|^n}{n!}$. So  if $|z|<\eps$, then  $\|L_z-L\|\leq \eps_{pert}$.
\end{proof}

\subsection{Exponential Decay of Correlations}\label{ss.CMS.correlations}
Recall the H\"older norm \eqref{e.Holder-Norm}.

\begin{theorem}\label{Lemma-DOC}
Let $\mu$ be the {equilibrium measure} for a H\"older continuous potential $\phi$   on a one-sided or two-sided irreducible Markov shift with period $1$ and finite Gurevich entropy. For any $\beta>0$, there are  constants $C>0$ and $0<\theta<1$ s.t. for all $\beta$-H\"older continuous ${\varphi},{\psi}$,
$$
\left|\int {\vf}\cdot ( {\psi}\circ \sigma^{n} )d\mu-\int {\vf} d\mu\int {\psi} d\mu\right|<C\|\psi\|_{\beta}\|\vf\|_{\beta}\theta^{n}, \text{ for all }n\geq 1.
$$
\end{theorem}

In the one-sided case, the argument below allows to take $\psi\in L^\infty(\mu)$, and to replace  $\|\psi\|_\beta$ on the right hand side  by $\|\psi\|_{\infty}$.

\begin{proof} The special case when $\vf$ and $\psi$ are indicators of partition sets follows from Vere-Jones's work \cite{Vere-Jones-Geometric-Ergodicity}.
Here we explain how to do the general H\"older case.

\smallskip
\noindent
{\em The One-Sided Case\/}: This follows from \cite[Thm 1.1b]{Cyr-Sarig} up to a routine modifications. We give a direct proof for the convenience of the reader.

Suppose {$\psi\in L^\infty(\mu)$ and $\vf$ $\beta$-H\"older continuous} on $\Sigma^+$, and let $\mathcal L$ be the space given in Theorem \ref{t.SGP}.
Let $\ov{\vf}:=\vf-\int\vf d{\mu}$, $\ov{\psi}:=\psi-\int\psi d{\mu}$. {Note that $\ov{\vf}$ is a $\beta$-H\"older continuous function, therefore $\ov{\vf}\in\mathcal L$.}

Recall that $L=P+N$ where $P\psi=(\int\psi d{\mu})\cdot 1$, $PN=NP=0$, and $N$ has spectral radius strictly smaller than one. Then
\begin{align*}
& \left|\int {\vf} \cdot({\psi}\circ \sigma^{n}) d{\mu}-\int {\vf} d{\mu}\int {\psi} d{\mu}\right|=\left|\int
\ov{\vf}\cdot(\ov{\psi}\circ \sigma^n )d{\mu}\right|=\left|\int
L^n(\ov{\vf})\ov{\psi} d{\mu}\right|\\
&\leq \|\ov{\psi}\|_\infty \int |L^n\ov{\vf}|d{\mu}= \|\ov{\psi}\|_\infty\frac{\|P(|L^n\ov{\vf}|)\|_{\mathcal L}}{\|1\|_{\mathcal L}}\leq \frac{2\|P\|\|\psi\|_\infty}{\|1\|_{\mathcal L}} \bigl\||L^n\ov{\varphi}|\bigr\|_{\mathcal L}.
\end{align*}
By Thm \ref{t.SGP}\eqref{item2}, \eqref{item8} and \eqref{item6},
$\bigl\||L^n\ov{\varphi}|\bigr\|_{\mathcal L}\leq \bigl\|L^n\ov{\varphi}\bigr\|_{\mathcal L}= \|(P+N)^n\ov{\vf}\|_{\mathcal L}= \|P\ov{\vf}+N^n\ov{\vf}\|_{\mathcal L}= \|N^n\ov{\vf}\|_{\mathcal L}\leq \|\ov{\vf}\|_{\mathcal L}\|N^n\|\leq \|1\|_\mathcal L \|\ov{\varphi}\|_\beta\|N^n\|\leq 2\|1\|_\mathcal L\|\vf\|_\beta\|N^n\|$.

This tends to zero exponentially fast, because $\rho(N)<1$. Fix $C'>0$ and  $0<\theta<1$ such that $\|N^n\|\leq C'\theta^n$, and let $C:=4C'\|P\|$. Then
\begin{equation}\label{e.uniform-exp-dec}
\left|\int {\vf}\cdot( {\psi}\circ \sigma^{n} )d{\mu}-\int {\vf} d{\mu}\int {\psi} d{\mu}\right|\leq C\|\psi\|_{\infty}\|\vf\|_{\beta} \theta^n.
\end{equation}

\noindent
{\em The Two-Sided Case\/}: {By Sinai's lemma (Thm \ref{t.Sinai-Lemma}) one can replace the potential by a one-sided potential, without changing
the equilibrium measure $\mu$.}
Then $\mu^+:=\mu\circ\vartheta^{-1}$ is an equilibrium measure on $\Sigma^+$ by Remark~\ref{r.project}.
Let $\psi,\vf$ be two $\beta$-H\"older continuous functions on $\Sigma$.  

\smallskip
\noindent
{\sc Claim:} {\em  For every $n$, there is a  $\beta$-H\"older continuous function ${\psi}^{(n)}:\Sigma\to\R$ such that $\|\psi^{(n)}\|_\infty=\|\psi\|_\infty$, $\|\psi-\psi^{(n)}\|_\infty\leq \|\psi\|_{\beta}e^{-\lfloor n/2\rfloor \beta}$, and $\psi^{(n)}\circ\sigma^{\lfloor n/2\rfloor }$ is one-sided.}

\smallskip
\noindent
{\em Proof.\/} Take $\psi^{(n)}(\un{x}):=\sup\{\psi(\un{y}): y_i=x_i\text{ for all }i\geq -\lfloor n/2\rfloor\}$. \qed
\medskip

Let $\ov{\vf}:=\vf-\int\vf\, d\mu$, $\ov{\psi}:=\psi-\int \psi d\mu$. Recall that $\mathcal A_0^\infty$ denote the $\sigma$-algebra generated by the non-negative coordinates and let $\ov{\varphi}^+$ be as in Prop. \ref{p.conditional-measures}.
Let also $\ov{\psi}_n^+:\Sigma^+\to\R$ be the  unique function such that $\ov{\psi}^{(n)}\circ\sigma^{\lfloor n/2\rfloor}=\ov{\psi}_n^+\circ\vartheta$.
Then:
$$|\mathrm{Cov}_\mu(\vf,\psi\circ\sigma^n)|=|\mathrm{Cov}_\mu(\ov{\vf},\ov{\psi}\circ\sigma^n)|\leq |\mathrm{Cov}_\mu(\ov{\vf},\ov{\psi}^{(n)}\circ\sigma^n)|
+\|\ov{\vf}\|_\infty\|\ov{\psi}-\ov{\psi}^{(n)}\|_\infty.$$
Since $\ov{\psi}^{(n)}\circ \sigma^n$ is one-sided,
\begin{align*}
|\mathrm{Cov}_\mu(\ov{\vf},\ov{\psi}^{(n)}\circ\sigma^n)|
&=\left|\int \ov{\vf} \cdot(\ov{\psi}^{(n)}\circ\sigma^n)\, d\mu\right|
=
\left|\int \E_{\mu}(\ov{\vf}|\mathcal A_0^\infty) \cdot(\ov{\psi}^{(n)}\circ\sigma^n)\, d\mu\right|\\
&=|\mathrm{Cov}_{\mu}(\ov{\vf}^+\circ\vartheta,\ov{\psi}^{(n)}\circ\sigma^n)|
=|\mathrm{Cov}_{\mu^+}(\vf^+,\ov{\psi}_n^+\circ \sigma^{n-\lfloor n/2\rfloor})|\\
&\leq C\|\psi\|_\infty\|\ov{\varphi}^+\|_{\beta}\theta^{n/2}, \text{ by \eqref{e.uniform-exp-dec}.}
\end{align*}
Putting together these inequalities and recalling the choice of $\ov{\psi}^{(n)}$, we obtain:
\begin{align*}
|\mathrm{Cov}_\mu(\vf,\psi\circ\sigma^n)|
&\leq C\|\psi\|_\infty\|\ov{\varphi}\|_{\beta}\theta^{n/2}+\|\ov{\vf}\|_\infty \|\ov{\psi}\|_\beta e^{-\lfloor n/2\rfloor \beta}\\
&\leq 2C\|\psi\|_{\beta}\|\vf\|_{\beta}\theta^{n/2}+4e^{\beta}\|\vf\|_\beta\|\psi\|_\beta e^{-\beta n/2 },\text{ because $\|\cdot\|_{\infty}\leq \|\cdot\|_\beta$}.
\end{align*}
So
$|\mathrm{Cov}_\mu(\vf,\psi\circ\sigma^n)|\leq (2C+4e^{\beta})\|\psi\|_\beta\|\vf\|_\beta \rho^n$,
with
 $\rho:=\min\{\theta^{1/2},e^{-\beta/3}\}$.
\end{proof}

Next, suppose  $\Sigma$ is irreducible but with period $p>1$.
By Thm \ref{t.Parry-Gurevich-ASS},  $(\Sigma,\mathfs B,\mu,\sigma)$ is  measure theoretically isomorphic to the product of a Bernoulli scheme and a cyclic permutation of $p$ elements. So $\mu$ is $\sigma$-ergodic, but not $\sigma^p$ ergodic, and its $\sigma^p$-ergodic decomposition takes the form
$
 \mu=\frac{1}{p}\sum_{i=0}^{p-1}\mu'\circ \sigma^i,
$
where $\mu'$ is $\sigma^p$-mixing. The analogous statements hold for $\Sigma^+$.

\begin{theorem}\label{t.DOC-for-CMS}
Let $\mu$ be the {equilibrium measure} of an SPR H\"older continuous potential on a one-sided or two-sided irreducible  Markov shift with finite Gurevich entropy and period $p$.  For every $\beta>0$ there are  $C>0$, $0<\theta<1$ s.t. for all $\beta$-H\"older continuous ${\varphi},{\psi}$,
$$
\left|\int {\vf} \cdot ({\psi}\circ \sigma^{np}) d\mu'-\int {\vf} d\mu'\int {\psi} d\mu'\right|<C\|\vf\|_\beta\|\psi\|_{\beta}\theta^{n},\text{ for all  $n\geq 1$.}
$$
\end{theorem}
\noindent
The proof follows directly from the spectral decomposition (Lemma \ref{l.spectral-decomposition-CMS}).

\subsection{Asymptotic Variance}\label{ss.CMS.variance}
The results in this section were proved \cite{Cyr-Sarig}  in the  irreducible {\em  aperiodic} case (building on earlier work for subshifts  of finite type in \cite{Parry-Pollicott-Asterisque} and \cite{Guivarch-Hardy}). Here we explain how to treat the irreducible {\em periodic}  case. 
\begin{theorem}\label{t.asymp-var-CMS}
Let $\mu$ be the {equilibrium measure} of a H\"older continuous SPR potential $\phi$ on a one-sided or two-sided irreducible  Markov shift with finite Gurevich entropy and  period $p$. Let $\psi$ be a {$\beta$-}H\"older continuous function. Then the following statements hold.
\begin{enumerate}[(1)]
\item Let $\psi_n:=\sum_{j=0}^{n-1}\psi\circ\sigma^j$, then the limit $\sigma_\psi^2:=\lim\limits_{n\to\infty}\frac{1}{n}\mathrm{Var}_\mu(\psi_n)$ exists and
$$\sigma_\psi^2=\frac{1}{p}\left[\mathrm{Var}_\mu(\psi_p)+2\sum\limits_{n=1}^\infty
    \mathrm{Cov}_\mu(\psi_{p},\psi_{p}\circ\sigma^{np})\right] \quad \text{{\em (Green-Kubo Identity)}}.$$
    
\item {\em Linear Response Identity:}
    $$
    \sigma_\psi^2=\frac{d^2}{dt^2}\biggr|_{t=0} \left( {P_\Bor(\Sigma,\phi+t\psi)} \right).
    $$
\item {\em Asymptotic Laplace Transform:} If $\int\psi d\mu=0$, then for every $z\in\mathbb C$,
$$
\lim_{n\to\infty}\E_\mu(e^{z\psi_n/\sqrt{n}})=e^{\frac{1}{2}\sigma_\psi^2 z^2}.
$$

\item {\em Upper bound\/:} For some constant $M_\beta$ which only depends on $\phi$, $\beta$ and $\Sigma$,
\begin{equation}\label{e.sigma-bound-CMS}
\sigma_\psi\leq M_\beta\|\psi\|_\beta.
\end{equation}
\end{enumerate}
\end{theorem}
\noindent

\begin{proof}
It is sufficient to consider the case when $\int_\Sigma\psi d\mu=0$.
\medskip

\noindent
\emph{Part (1).}
Let $p$ be the period of $\Sigma$, and $\Sigma=\biguplus_{i=0}^{p-1}\sigma^i(\Sigma')$ be the spectral decomposition from Lemma \ref{l.spectral-decomposition-CMS}.
There is a natural identification of  $(\Sigma',\sigma^p)$ with an irreducible aperiodic Markov shift, and if $\psi$ is H\"older continuous with respect to  the natural metric of $\Sigma$, then $\psi_p$ is H\"older continuous with respect to the natural metric on $\Sigma'$.
Note that $\mu=\frac{1}{p}\sum_{i=0}^{p-1}\mu'\circ\sigma^{-i}$, where $\mu'$ is the measure $\mu'(E):=\mu(E\cap\Sigma')/\mu(\Sigma')$, and $\mu'$ is the
{equilibrium measure of $(\Sigma',\sigma^p)$ for the potential $\phi_p=:=\sum_{j=0}^{n-1}\phi\circ\sigma^j$.} Also,
$$
\int_{\Sigma'}\psi_p d\mu'\circ\sigma^{-i}=0\ \ (i=0,\ldots,p-1).
$$
To see this, let $\mu'':=\mu'\circ\sigma^{-i}$, and note that $\sum_{j=0}^{p-1}\int\psi\circ \sigma^j d\mu''=\int\psi d\mu=0$.

The covariance is bilinear, and $\mu'\circ \sigma^{-p}=\mu'$. Therefore,
\begin{align*}
&\mathrm{Var}_{\mu'}(\psi_{np})=\mathrm{Cov}_{\mu'}\left(\sum_{i=0}^{n-1}\psi_p\circ\sigma^{ip}, \sum_{j=0}^{n-1}\psi_p\circ\sigma^{jp}\right)=\sum_{i,j=0}^{n-1}\mathrm{Cov}_{\mu'}(\psi_p\circ \sigma^{ip},\psi_p\circ \sigma^{jp})\\
&=n\mathrm{Var}_{\mu'}(\psi_p)+2\sum_{0\leq i<j\leq n-1}\mathrm{Cov}_{\mu'}(\psi_p,\psi_p\circ\sigma^{(j-i)p})\ \ \text{(by shift invariance)}\\
&=n\mathrm{Var}_{\mu'}(\psi_p)+\sum_{k=1}^{n-1} (n-k)\mathrm{Cov}_{\mu'}(\psi_p,\psi_p\circ\sigma^{kp}).
\end{align*}
The series $\sum \mathrm{Cov}_{\mu'} (\psi_p,\psi_p\circ\sigma^{kp})$ converges absolutely,
by the exponential mixing of  $(\Sigma',\sigma^p)$ with respect to $\mu'$. Thus by the dominated convergence theorem,
\begin{equation}\label{e.var-mu-prime}
\frac{1}{n} \mathrm{Var}_{\mu'}(\psi_{np})\xrightarrow[n\to\infty]{}
\mathrm{Var}_{\mu'}(\psi_p)+\sum_{k=1}^{\infty} \mathrm{Cov}_{\mu'}(\psi_p,\psi_p\circ\sigma^{kp}).
\end{equation}

Similarly
$
\displaystyle\frac{1}{n}\mathrm{Var}_{\mu'\circ \sigma^{-i}}(\psi_{np})\xrightarrow[n\to\infty]{}
\mathrm{Var}_{\mu'\circ\sigma^{-i}}(\psi_p)+\sum_{k=1}^{\infty} \mathrm{Cov}_{\mu'\circ\sigma^{-i}}(\psi_p,\psi_p\circ\sigma^{kp}).
$
Summing over $i$, dividing by $p^2$, and recalling that $\E_\mu(\psi)=\E_{\mu'\circ\sigma^{-i}}(\psi_p)=0$, gives
$$
\lim_{n\to\infty}\frac{1}{np}\mathrm{Var}_{\mu}(\psi_{np})=\frac{1}{p}\left[
\mathrm{Var}_\mu(\psi_{p})+2\sum_{k=0}^{\infty}\mathrm{Cov}_{\mu}(\psi_p,\psi_p\circ\sigma^{jp})\right].
$$
Next, for each fixed $0\leq r\leq p-1$,
\begin{align*}
&\mathrm{Var}_\mu(\psi_{np+r})=\mathrm{Var}_\mu(\psi_{np})+2\mathrm{Cov}_{\mu}(\psi_r,\psi_{np})+\Var_{\mu}(\psi_r)\\
&=\mathrm{Var}_{\mu}(\psi_{np})+O(\sqrt{n})+O(1),
\end{align*}
 because by the Cauchy-Schwarz inequality and the boundedness of $\psi$,
\begin{align*}
&|\mathrm{Cov}_{\mu}(\psi_r,\psi_{np})|\leq \sqrt{\mathrm{Var}_\mu(\psi_r)\mathrm{Var}_{\mu}(\psi_{np})}\leq 2r\|\psi\|_\infty \sqrt{O(1)O(np)}=O(\sqrt{n}).
\end{align*}
So  $\displaystyle
\lim_{n\to\infty}\frac{1}{np+r}\mathrm{Var}_{\mu}(\psi_{np+r})
=\frac{1}{p}\left[
\mathrm{Var}_\mu(\psi_{p})+2\sum_{k=0}^{\infty}\mathrm{Cov}_{\mu}(\psi_p,\psi_p\circ\sigma^{jp})\right]$ for each $r$. Part (1) follows.
\medskip

\noindent
\emph{Part (2).} We first consider the  one-sided {\em aperiodic} case and assume $\|\psi\|_\beta=1$ without loss of generality.
Let 
 $$\phi_t:=\phi+t\psi.$$
By the variational principle for Gurevich pressure \cite{Sarig-ETDS-99},
$$
\sup\left\{h_\nu(\sigma)+\int {\phi_t} d\nu:\begin{array}{l} \mu\text{ is a shift invariant}\\ \text{probability measure}
    \end{array}\right\}={P_\Bor(\Sigma,{\phi_t})}.
$$
 Let $\lambda_z,P_z,N_z$ and $\eps$ be as in Cor \ref{c.perturbation-theory}.
By \cite[p. 665]{Cyr-Sarig} and references therein,  for all $t$ real and $|t|<\eps$
$$
{P_\Bor(\Sigma,{\phi_t})}=\log\lambda(t) \qquad \text{ setting }\lambda(t):=\lambda_t.
$$
In particular, $t\mapsto {P_\Bor(\Sigma,{\phi_t})}$ is real-analytic on a neighborhood of zero.

The first two derivatives of $\log\lambda(z)$ at $z=0$ can be found exactly as in the case of subshifts of finite type discussed in  \cite{Guivarch-Hardy}, \cite{Parry-Pollicott-Asterisque}. The calculation is reproduced in the countable state case in \cite[p. 662-664]{Cyr-Sarig},\footnote{But $\lambda_t$ there is what we call $\lambda(it)$ here.} and leads to
$$
\lambda'(0)=\int\psi d\mu^+=0\ \ ,\  \
\lambda''(0)=\int\left[{\psi}^2+2{\psi}\sum_{k=1}^\infty L^k ({\psi})\right]d\mu^+.
$$
Since $L$ is the transfer operator of $\mu^+$, this equals  the  right-hand-side of the Green-Kubo identity (because $p=1$). So
$
\lambda''(0)=\sigma_\psi^2.
$

Since $\lambda(0)=1$ and $\lambda'(0)=0$, $\frac{d^2}{dt^2}\big|_{t=0}(\log\lambda)''(0)=\lambda''(0)=\sigma_\psi^2$, and Part (2) follows in the one-sided topologically mixing case.

The one-sided irreducible {\em periodic} case readily follows, because of the following simple fact.  Let $\Sigma=\biguplus_{i=0}^{p-1}\sigma^{-i}\Sigma'$ be the spectral decomposition, then
$
{P_\Bor(\Sigma,{\phi_t})}=\frac{1}{p} {P_\Bor(\Sigma',\phi_p+t\psi_p)}.
$
So $\frac{d^2}{dt^2}|_{t=0}{P_\Bor(\Sigma,{\phi_t})}=
\frac{1}{p}\lim\limits_{n\to\infty}\frac{1}{n}\mathrm{Var}_{\mu'}(\psi_{np})
=\sigma_\psi^2$
by (2) in the aperiodic case, and by \eqref{e.var-mu-prime}. This proves (2) in the one-sided case.

To analyze the two-sided  case, we use Sinai's Lemma to write ${\phi=\wt{\phi}+u-u\circ \sigma}$ where ${\wt{\phi}}$ is one-sided and $u$ is H\"older continuous. In particular $u$ is {\em bounded} and continuous, and therefore $\int (u-u\circ\sigma) d\nu=0$ for all shift invariant $\nu$. {We introduce $\wt{\phi}^+$, the unique H\"older continuous function on $\Sigma^+$ such that $\wt{\phi}=\wt{\phi}^+\circ\vartheta$.}
{We write similarly $\psi=\wt{\psi}+v-v\circ \sigma$ and $\wt{\psi}=\wt{\psi}^+\circ\vartheta$.
Also $\wt{\psi}_t=\wt{\phi}+t\wt{\psi}$.} So:
$$
{P_\Bor(\Sigma,\phi_t)}={P_\Bor(\Sigma,\wt\psi_t)} \text{ and } \frac{d^2}{dt^2}\big|_{t=0} {P_\Bor(\Sigma,\phi_t)}=\frac{d^2}{dt^2}\big|_{t=0} {P_\Bor(\Sigma,\wt\psi^+_t)}=\sigma_{\wt{\psi}^+}^2.
$$

To see that $\sigma_{\wt{\psi}^+}^2=\sigma_{{\psi}}^2$, we note that $\sigma_{\wt{\psi}^+}^2=\sigma_{\wt{\psi}}^2$ and that
$
\sigma_{\wt{\psi}}^2=\sigma_{\psi}^2.
$
The first identity is clear, and the second can be checked using the identity
$$
\psi_n=\wt{\psi}_n+v-v\circ\sigma^n,
$$
the uniform boundedness of $u$, the Cauchy-Schwarz inequality, and the bound  $\|\wt{\psi}_n\|_2\sim \sigma_{\wt{\psi}}\sqrt{n}=O(\sqrt{n})$. Part (2) is now proved.
\medskip

\noindent
\emph{Part (3).} We first consider the one-sided aperiodic case.
\begin{align*}
&\E_{\mu}(e^{z\psi_n/\sqrt{n}})=\int e^{z\psi_n/\sqrt{n}}d\mu=\int e^{z\psi_n/\sqrt{n}}\cdot 1\circ\sigma^n\, d\mu=\int L^n(e^{z\psi_n/\sqrt{n}}) d\mu.
\end{align*}
Let $\lambda(z), P_z,N_z,L_z$ be as in Cor \ref{c.perturbation-theory}.
An inductive argument using Lemma \ref{l.transfer-op-properties}(3) gives $L^n(e^{z\psi_n/\sqrt{n}})=L_{z/\sqrt{n}}^n 1$. So,
\begin{align*}
&\E_{\mu}(e^{z\psi_n/\sqrt{n}})
=\int L_{z/\sqrt{n}}^n 1\, d\mu=P_0 L_{z/\sqrt{n}}^n 1,\text{ because }P_0\vf=P\vf=\bigg(\int \vf d\mu\bigg)\cdot 1\\
&=P_{0}\biggl(\lambda\big(\tfrac{z}{\sqrt{n}}\big) P_{z/\sqrt{n}}+N_{z/\sqrt{n}}\biggr)^n1=\lambda\big(\tfrac{z}{\sqrt{n}}\big)^nP_0 P_{z/\sqrt{n}} 1+P_0 N_{z/\sqrt{n}}^n 1.
\end{align*}
We first consider $\lambda(\frac{z}{\sqrt{n}})^nP_0 P_{z/\sqrt{n}} 1$.
Recall that $\lambda(0)=1$, $\lambda'(0)=\int\psi d\mu=0$, and $\lambda''(0)=\sigma_\psi^2$. Therefore, for every $z$,
$
\lambda(\frac{z}{\sqrt{n}})^n=\biggl(1+\frac{\sigma_\psi^2 z^2}{2n}+O(\frac{z^3}{n^{3/2}})\biggr)^n\xrightarrow[n\to\infty]{}e^{\frac{1}{2}\sigma_\psi^2 z^2}.
$
 Next, $P_0 P_{z/\sqrt{n}}\to P_0^2=P_0$ in norm because $w\mapsto P_w$ is analytic on a complex neighborhood of zero. Thus  $\lambda(\frac{z}{\sqrt{n}})^nP_0 P_{z/\sqrt{n}} 1$ tends to $e^{\frac{1}{2}\sigma_\psi^2 z^2}$.

\medskip
\noindent
We then consider $P_0 N^n_{z/\sqrt{n}} 1$.
Recall that $N_0$ has spectral radius less than $1$, therefore $\|N_0^k\|<\kappa<1$ for some $\kappa<1$ and  $k\in\N$. Since $w\mapsto N_w$ is analytic, $\|N_w^k\|<\kappa$ for all $|w|$ sufficiently small, whence $\|P_0 N^n_{z/\sqrt{n}} 1\|$ tends to zero.

\medskip
\noindent
In summary, $
\lambda(\frac{z}{\sqrt{n}})^nP_0 P_{z/\sqrt{n}} 1+P_0 N_{z/\sqrt{n}}^n 1\to e^{\frac{1}{2}z^2\sigma_\psi^2}$ in norm
as (constant) functions in $\mathcal L$. This proves Part (3) in the one-sided aperiodic case.

\medskip
The two-sided aperiodic case follows again from Sinai's Lemma. As before we write $\psi=\wt{\psi}+v-v\circ\sigma$. We have already seen that
$
\sigma_{\wt{\psi}}^2=\sigma_{{\psi}}^2,
$
so by the first part of the proof,
$
\E_\mu(e^{z\wt{\psi}_n/\sqrt{n}})\xrightarrow[n\to\infty]{}
e^{\frac{1}{2}\sigma_\psi^2 z^2}.
$
It remains to see (using the  uniform boundedness of $v$ and the first part of the proof) that
$\left|\E_\mu(e^{z{\psi}_n/\sqrt{n}})-\E_\mu(e^{z\wt{\psi}_n/\sqrt{n}})
\right|\leq  \E_\mu(e^{\mathrm{Re}(z)\wt{\psi}_n/\sqrt{n}})\left\|e^{z(u-u\circ\sigma^n)/\sqrt{n}}-1\right\|_\infty\to 0
$.
This gives Part (3) in the aperiodic case; the general irreducible case follows from the spectral decomposition.
\medskip

\noindent
\emph{Part (4).}  \eqref{e.sigma-bound-CMS} follows  from the Green-Kubo identity and Theorem \ref{t.DOC-for-CMS}.
\end{proof}

\begin{theorem}\label{t.zero-variance-CMS}
Let $\mu$ be the equilibrium measure of an SPR H\"older continuous potential $\phi$ on a  one-sided or two-sided irreducible SPR Markov shift with finite Gurevich entropy, and let $\psi$ be a H\"older continuous function. The following are equivalent.
\begin{enumerate}[(a)]
\item  $\sigma_\psi^2=0$.
\item $\psi-\int\psi d\mu=u-u\circ\sigma$ for some continuous (but possibly unbounded) real-valued function $u$.
\item If $\sigma^n(x)=x$, then $\psi_n(x)=n\int\psi d\mu$.
\item If $\nu$ is shift invariant, then $\int\psi d\nu=\int\psi d\mu$.
\end{enumerate}
\end{theorem}
\begin{proof}
First we prove the theorem in the {\em one-sided aperiodic} case, and then we extend the proof to the general irreducible one-sided or two-sided case.

 (a)$\Rightarrow$(b) follows from the Green-Kubo identity in its equivalent form
$$\sigma^2_\psi=\int (\psi^2+2\psi\sum_{k=1}^\infty L^k \psi)d\mu$$ as in \cite{Gouezel-2}. The proof is reproduced in a context close to ours in \cite[p. 664]{Cyr-Sarig}.

(b)$\Rightarrow$(d): Suppose $\nu$ is shift invariant. If $u\in L^1(\nu)$ then (d) is obvious. If not, then  $\liminf_{n\to\infty}|\frac{1}{n}\psi_n-\int\psi d\mu|=\liminf_{n\to\infty}\frac{|u-u\circ\sigma^{n}|}{n}=0$ $\nu$-a.e., by Poincar\'e's Recurrence Theorem. Since  $\lim\frac{1}{n}\psi_n$ exists $\nu$-a.e., it must be the case that  $\lim\frac{1}{n}\psi_n=\int\psi d\mu$ $\nu$--a.e. By the bounded convergence theorem,  $\int\psi d\nu=\int\psi d\mu$.

(d)$\Rightarrow$(c) is trivial, and (c)$\Rightarrow$(b) is the  Livsic Theorem for Markov shifts (the proof for subshifts of finite type in \cite{Bowen-LNM} works verbatim in the countable state case).

It remains to prove (b)$\Rightarrow$(a). It is sufficient to do this in the special case when $\int\psi d\mu=0$. In this case, for all $t\in\R$,
\begin{align*}
&\E_\mu(e^{it\psi_n/\sqrt{n}})=\E_\mu(e^{it(u-u\circ\sigma^n)/\sqrt{n}})=\E_\mu(e^{-itu\circ\sigma^n/\sqrt{n}})+O(|\E_\mu(|e^{itu/\sqrt{n}}-1|))
\\
&=\E_\mu(e^{-itu/\sqrt{n}})+O(\E_\mu(|e^{itu/\sqrt{n}}-1|)),\text{ because }\mu\circ\sigma^{-1}=\mu\\
&\xrightarrow[n\to\infty]{}1,\text{ by the bounded convergence theorem.}
\end{align*}
By Thm \ref{t.asymp-var-CMS}(3), 
$e^{-\frac{1}{2} t^2 \sigma_\psi^2}=1$ for all $t\in\R$. So $\sigma_\psi=0$.
This proves the theorem in the irreducible one-sided aperiodic case. 

\medskip
Next we consider the irreducible one-sided {\em periodic} case, with period $p>1$. It is sufficient to consider the case $\int \psi d\mu=0$.
Let $\Sigma=\biguplus_{i=0}^{p-1}\sigma^{-i}\Sigma'$ be the spectral decomposition.
We saw in the proof of Theorem \ref{t.asymp-var-CMS} that $\sigma_\psi^2$ is up to a constant the asymptotic variance of $\psi_p|_{\Sigma'}$ with respect to $\sigma^p$ and the $\sigma^p$-ergodic components  $\mu'$ of $\mu$. We also saw that $\int \psi_p d\mu'=0$. So the following are equivalent:
\begin{enumerate}[(a)]
\item[(a)] $\sigma_\psi=0$.
\item[(b')] $\psi_p=u-u\circ\sigma^p$ on $\Sigma'$, with $u:\Sigma'\to\R$ continuous.
\item[(c')] $\sigma^{np}(x)=x, x\in\Sigma'\Rightarrow\psi_{np}(x)=0$.
\item[(d')] If $\nu'$ is a shift invariant measure on $\Sigma'$, then $\int\psi_p d\nu'=\int\psi_p d\mu'$.
\end{enumerate}

We claim that (b') is equivalent to (b). The implication (b)$\Rightarrow$(b') is clear. To see (b')$\Rightarrow$(b), we  extend $u:\Sigma'\to\R$ to a function $\wh{u}:\Sigma\to\R$ by setting
$$
\wh{u}|_{\Sigma'}=u\text{ and }\wh{u}|_{\sigma^{-i}\Sigma'}=u\circ\sigma^i+\psi_i
\text{ for $i=1,\ldots,p-1$}.
$$
On $\sigma^{-i}\Sigma'$, we have $\psi_p=\psi_i+\psi_p\circ\sigma^i-\psi_i\circ\sigma^p=\wh{u}-\wh{u}\circ\sigma^p=\wh{u}_p-\wh{u}_p\circ\sigma$.
So
$$
\psi_p=\wh{u}_p-\wh{u}_p\circ\sigma\text{ on }\Sigma.$$
Next,
$
\psi-\psi\circ\sigma^k=\psi_k-\psi_k\circ\sigma
$ for all $k$. Summing over $k=0,\ldots,p-1$, we obtain
$
p\psi-\psi_p=\left(\sum_{k=0}^{p-1} \psi_k\right)-\left(\sum_{k=0}^{p-1} \psi_k\right)\circ\sigma
$
whence
 $$p\psi=\wh{u}_p-\wh{u}_p\circ\sigma+\left(\sum_{k=0}^{p-1} \psi_k\right)-\left(\sum_{k=0}^{p-1} \psi_k\right)\circ\sigma.$$ The right-hand-side is a continuous $\sigma$-coboundary, and (b) follows.

We claim that (c') is equivalent to (c). The implication (c)$\Rightarrow$(c') is clear. The implication (c')$\Rightarrow$(c) is because if $\sigma^k(x)=x$ then $k=np$ for some $n$, and $x=\sigma^j(y)$ for some $y\in\Sigma'$ and $j=0,\ldots,p-1$. So $\psi_k(x)=\psi_{np}(\sigma^j y)=\psi_{np}(y)=0$.

Finally, we claim that (d') is equivalent to (d). This is because $\nu'$ is a $\sigma^p$-invariant measure on $\Sigma'$ iff $\nu:=\frac{1}{p}\sum_{j=0}^{p-1}\nu'\circ\sigma^j$ is a $\sigma$-invariant measure on $\Sigma$ and $\nu'$ is the measure $\nu'(E)=\nu(E\cap\Sigma')/\nu(\Sigma')$, and in this case, $\int\psi_p d\nu'=p\int\psi d\nu$.

This completes the proof in the one-sided case. The two-sided case follows from Sinai's Lemma, by taking a one-sided $\psi^+$ cohomologous to $\psi$ by a bounded continuous transfer function, see the end of the proof of Theorem \ref{t.asymp-var-CMS}.
\end{proof}

\subsection{Large Deviations}\label{ss.CMS.deviations}
Let $\Sigma$ be an irreducible one-sided or a two-sided Markov shift with finite Gurevich entropy, and let {$\mu$ be the equilibrium state of an SPR
H\"older continuous potential $\phi$.} Let $\psi:\Sigma\to\R$ be some H\"older continuous function.
The  {\em log-moment generating function} and the {\em rate function} of $\psi$ are given, respectively, by
\begin{equation}\label{e.Lambda}
\Lambda_\psi(t):=\limsup_{n\to\infty}\tfrac{1}{n}\log \int e^{t\psi_n}d\mu\ ,\ I_\psi(s):=\sup_{t\in\R}\{st-{\L_\psi}(t)\}.
\end{equation}

\begin{lemma}\label{t.Free-Energy-CMS}
{For any $\beta>0$}
there are $\eps,{c}>0$,  as follows. Suppose $\psi$ is $\beta$-H\"older continuous,  $\int\psi d\mu=0$,  $\|\psi\|_\beta=1$, and $\sigma_\psi\neq 0$.  Then:
\begin{enumerate}[(1)]
\item \begin{enumerate}[(a)]
\item $\Lambda_\psi$ is finite, convex, and non-negative on $\R$. On $(-{\eps},{\eps})$,  $\Lambda_\psi$ is  {$C^\infty$}, strictly convex,
{and the limsup in \eqref{e.Lambda} is a limit.}
\item $I_\psi$ is  convex and  non-negative on $\R$. On $(-{c}{\sigma_\psi^4},{c}
    {\sigma_\psi^4})$,  $I_\psi$ is finite,  $C^\infty$, strictly convex and it satisfies $\frac{1}{2}\sigma_\psi^{-2}\leq I_\psi''\leq 2\sigma_\psi^{-2}$.
\end{enumerate}

\item
$ \Lambda_\psi(0)=0, \ \Lambda_\psi'(0)=0, \ \Lambda_\psi''(0)=\sigma^2_\psi$, \  $I_\psi(0)=0, \ I_\psi'(0)=0, \ I_\psi''(0)={\sigma^{-2}_\psi}.$
\item \begin{enumerate}[(a)]
\item If $|t|<\eps$, then
$
\Lambda_\psi(t)=
\sup\left\{P(\sigma,\nu,\phi+t\psi):\nu\in \mathbb P(\sigma)\right\}
{-P_\Bor(\Sigma,\phi)}
$.
\item If $|s|<{c}{\sigma_\psi^4}$,
    $
    I_\psi(s)={P_\Bor(\Sigma,\phi)-\sup}\left\{P(\sigma,\nu,\phi):\nu\in \mathbb P(\sigma), \int\psi d\nu=s\right\}.
    $
\end{enumerate}
\noindent
The suprema are attained by unique measures; these measures are ergodic.
\end{enumerate}
\end{lemma}
\begin{remark}\label{r.Free-Energy-CMS}
 If we do not assume that  $\|\psi\|_\beta=1$, the lemma  holds  with the numbers $\varepsilon,c$ replaced
by $\varepsilon/\|\psi\|_\beta$ and $c/\|\psi\|_\beta^3$.
\end{remark}
\begin{proof}
We prove the lemma in the one-sided irreducible and aperiodic case. The periodic case  and the two-sided case follow as in the previous sections from the spectral decomposition, Sinai's Lemma, and the uniform boundedness of $\psi$.

The finiteness of $\Lambda_\psi(t)$ holds because $\phi$ and $\psi$ are uniformly bounded, and $\Sigma$ has finite Gurevich entropy. The convexity of $\Lambda_{\psi}$ follows from H\"older's inequality. 
Strict convexity will be shown below.

Let $\lambda_z,P_z,N_z$ and $\eps_1>0$ given by Cor \ref{c.perturbation-theory}, and recall that $\eps_1$ is independent of $\psi$. Proceeding exactly as in the proof of Thm \ref{t.asymp-var-CMS}(3), one shows that for every $t$ {\em real} such that $|t|<\eps_1$, and for every $\beta$-H\"older continuous $\psi$ such that $\|\psi\|_\beta=1$,
$$
\Lambda_\psi(t)=\lim_{n\to\infty}\tfrac{1}{n}\log \E_{\mu}(e^{t\psi_n})=\log\lambda_t\text{ on }(-{\eps_1},{\eps_1}).
$$

By \cite[p. 665]{Cyr-Sarig}, $\log\lambda_t=P_\Bor(\Sigma,{\varphi}+t\psi)$, where {$\varphi$} is given by Lem~\ref{p.transfer-op-formula} and $P_\Bor$ here is the Gurevich pressure defined by~\eqref{e.gurevich-pressure}. We recall that $\varphi$ is bounded above, weakly H\"older continuous, and non-positive (see the proof of Thm \ref{t.SGP}).
Since {$\varphi$} is cohomologous to {$\phi-P_\Bor(\Sigma,\phi)$}, for all $|t|<\eps$, we get (3a):
$$
\Lambda_\psi(t)=\log\lambda_t={P_\Bor(\Sigma,\phi+t\psi)-P_\Bor(\Sigma,\phi)}.
$$
The function {$\mathfrak p_{\phi,\psi}(t):=P_\Bor(\Sigma,\phi+t\psi)$} is analyzed in detail in \cite[\S4]{Ruhr-Sarig}. Lemma 4.2 there  provides a positive constant $\eps_2$ (independent of $\psi$) so that $\mathfrak p_{\phi,\psi}(t)$ is real-analytic and strictly convex on $(-\eps_2,\eps_2)$. Let $\eps:=\min\{\eps_1,\eps_2\}$. Then $$\Lambda_\psi(t)=\mathfrak p_{\phi,\psi}(t)-P_\Bor(\Sigma,\phi)\text{ on } (-\eps,\eps).
$$
All the properties of $\Lambda_\psi(t)$  we want to prove now follow from \cite[Thm 3.2]{Ruhr-Sarig}.

\medskip
Let $\mathfrak q_{\phi,\psi}(t):=\sup\{h(\sigma,\nu)+\int\phi d\nu:\nu\in\mathbb P(\sigma),\ \int\psi d\nu=t\}$, with the understanding that $\sup\emptyset:=-\infty$.
The claim in \cite[p. 711]{Ruhr-Sarig} says that there is a positive constant ${c}$, which does not depend on $\psi$, as follows:
\begin{enumerate}[(1)]
\item The supremum defining $\mathfrak q_{\phi,\psi}(t)$ is attained by some unique measure for all $|t|<{c} {\sigma_\psi^4}$, and this measure is an equilibrium measure $\nu_t$ of some H\"older continuous potential. In particular $\nu_t$ is ergodic.
\item $\mathfrak q_{\phi,\psi}$ is minus the Legendre transform of $\mathfrak p_{\phi,\psi}$ on $(-{c}{\sigma_\psi^4},{c}{\sigma_\psi^4}) $. See \cite[Eq. (5.2)]{Ruhr-Sarig}. By definition, the Legendre transform of $\mathfrak p_{\phi,\psi}$ is $I_\psi{-P_\Bor(\Sigma,\phi)}$. Thus
    $$
    I_\psi(t)={P_\Bor(\Sigma,\phi)}-\mathfrak q_{\phi,\psi}(t)\text{ on $(-{c}{\sigma_\psi^4},{c}{\sigma_\psi^4}) $. }
    $$
\end{enumerate}
(3b) follows. The remaining properties of $I_\psi(t)$  are established in \cite[Lem 5.2]{Ruhr-Sarig}.
\end{proof}
\begin{theorem}[Large Deviations]\label{t.LDP-CMS}
{For any $\beta>0$, there is
$c>0$ as follows.} Suppose $\psi$ is $\beta$-H\"older continuous,  $\int\psi d\mu=0$, $\|\psi\|_\beta=1$, and $\sigma_\psi\neq 0$.  Then:
\begin{enumerate}[(1)]
\item For  closed sets $F\subset\R$,
$
\displaystyle\limsup_{n\to\infty}\tfrac{1}{n}\log\mu\{x: \tfrac{1}{n}\psi_n(x)\in F\}\leq -\inf_F I_\psi.$
\item For open sets $G\subset\R$, $
\displaystyle\liminf_{n\to\infty}\tfrac{1}{n}\log\mu\{x: \tfrac{1}{n}\psi_n(x)\in G\}\geq -\inf_{G\cap (-{c}\sigma_\psi^4,{c}\sigma_\psi^4)} I_\psi.$
\item
 $\displaystyle\lim_{n\to\infty}\tfrac{1}{n}\log\mu[\psi_n\geq na]=-I_\psi(a)$ for all {$a\in (0, c\sigma_\psi^4)$}.
 \item $I_\psi(a)=\frac{a^2}{{2}\sigma^2_\psi}(1+o(1))\text{ as $a\to 0^+$}.$
 \end{enumerate}
\end{theorem}

\begin{proof}
By the G\"artner-Ellis Theorem (see e.g. \cite[Thm 2.3.6]{Dembo-Zeitouni}) and Lemma \ref{t.Free-Energy-CMS}(1),
$
\displaystyle \limsup_{n\to\infty}\tfrac{1}{n}\log\mu\{\psi_n/n\in F\}\leq -\inf_F I_\psi
$
for all closed sets $F\subset\R$. In addition, for every open set $G\subset\R$,
$
\displaystyle\limsup_{n\to\infty}\tfrac{1}{n}\log\mu\{\psi_n/n\in G\}\geq -\inf_{G\cap \mathcal F} I_\psi,
$
 where $\mathcal F$ is the set of {\em exposed points} of $I_\psi$:
$$
\mathcal F:=\{y\in\R: \exists c\in\R\text{ s.t. }\forall x\neq y,\ c(y-x)\gneqq I_\psi(y)-I_\psi(x)\}.
$$
Since $I_\psi$ is differentiable and strictly convex on $(-{c}\sigma_\psi^4,{c}\sigma_\psi^4)$,
$$\mathcal F\supset (-{c}\sigma_\psi^4,{c}\sigma_\psi^4),$$
and Part (2) of the theorem follows as well.

To see Part (3), we take $0<a<{c}\sigma_\psi^4$, and apply Parts (1) and (2) with $F=[a,\infty)$, $G=(a,{c}\sigma_\psi^4)$,  noting that
$
\inf_{[a,\infty)}I_\psi=\inf_{(a,{c}\sigma_\psi^4)}I_\psi=I_\psi(a)
$,
because
$I_\psi$ is continuous and monotonic increasing on $(0,\infty)$ (since it is  a finite smooth convex function such that $I_\psi'(0)=0$).

Part (4) follows from smoothness of $I_\psi$ on a neighborhood of zero, Taylor's expansion,  and the identities $I_\psi(0)=0, I_\psi'(0)=0,I_\psi''(0)=\sigma_\psi^{-2}$.
\end{proof}

\subsection{Almost Sure Invariance Principle (ASIP) for Markov Shifts}\label{ss.CMS.asip} We refer the readers to the beginning of \S\ref{s.ASIP-Statement}, where they can find  a summary of the necessary definitions from probability theory.
\begin{theorem}\label{t.ASIP-Sigma}
Let $\mu$ be the equilibrium measure of an SPR H\"older continuous potential $\phi$ on a one-sided or two-sided irreducible SPR Markov shift with finite Gurevich entropy, and let $\psi$ be a H\"older continuous function such that $\int\psi d\mu=0$. Then
$(\Sigma,\mathcal{B}(\Sigma),\mu,\sigma,\psi)$ satisfies the ASIP with rate $o(n^{\frac{1}{4}+\eps})$ for all $\varepsilon>0$.
\end{theorem}
Lemma~\ref{l.Borelomania} allows to define the approximating Brownian motion on $\Sigma\times [0,1]$, see Remark~\ref{r-dynamical-ASIP}.

\begin{proof}
We begin with the aperiodic one-sided case.
{Let $L:\mathcal L\to\mathcal L$ be the transfer operator associated to $\mu$ and
$L_{it}$ be defined by
$
L_{it}\varphi:=L(e^{it\psi}\varphi)
$ as in Cor~\ref{c.perturbation-theory}.}

\medskip
\noindent
{\em Claim.\/}
The following conditions hold:
\begin{enumerate}[(1)]
\item $L=P+N$ where $P$ is a projection, $\dim\mathrm{Im}(P)=1$, $P N=N P=0$, and the spectral radius of $N$ satisfies $\rho(N)<1$;\item there is $C>0$ such that $\|L_{it}^n\|\leq C$ for all $n\in\N$ and all small (real) $t$;
\item $\<\mu,\vf\>:=\int\vf d\mu$ is a bounded linear functional on $\mathcal L$;
\item Nagaev's identity:  $\E_{\mu}(e^{i\sum_{\ell=0}^{n-1} t_\ell \psi\circ \sigma^\ell})=\<\mu, L_{t_{n-1}}\cdots L_{t_0}1\>$.
\end{enumerate}
In the terminology of \cite{Gouezel-ASIP}, (3) and (4) say that the characteristic function of the process $(\psi\circ\sigma^\ell)_{\ell\ge0}$ is coded by $(\mathcal L,(L_{it})_{t\in\RR},1,\mu)$.
\begin{proof} Item (1) is Theorem~\ref{t.SGP}\eqref{item2}. 

Let ${\lambda_z}, P_z, N_z$ be as in Cor \ref{c.perturbation-theory}.
Since $\rho(N_0)<|\lambda_0|=1$,  $\rho(N_z)<1$. By the spectral radius formula, there are positive numbers $k,\delta$ such that $\|N_0^{k}\|\leq e^{-\delta}$. Since $z\mapsto  \|N_z^{k}\|$ is continuous, there are $k,\delta>0$ so that $\|N^k_{z}\|<e^{{-\delta}}$ for all $z$ in a neighborhood of zero. Thus 
$$
\|N^{n}_{it}\|\xrightarrow[n\to\infty]{}0\text{ uniformly in $t$ on a neighborhood of zero}. 
$$
{For all $t$ small let $\lambda(t):=\lambda_{it}$. Then}
\begin{equation}\label{e.L-it-power-n}
\|L_{it}^n-{\lambda(t)}^n P_{it}\|=\|N_{it}^n\|\xrightarrow[n\to\infty]{}0 \text{ uniformly in $t$ near zero}.
\end{equation}
In $\mathcal L$, convergence in norm implies uniform convergence on partition sets. So \eqref{e.L-it-power-n}  implies that $|L_{it}^n 1-{\lambda(t)}^n P_{it} 1|\to 0$ uniformly on partition sets. It is easy to verify that  $\|L_{it}^n 1\|_\infty\leq \|L^n 1\|_\infty=1$. Since  $P_{it} 1\approx P_0 1=1\not\equiv 0$ for all  $|t|$ small,  $|{\lambda(t)}|\leq 1$. Returning to  \eqref{e.L-it-power-n}, we see that $\sup_n\|L_{it}^n\|<\infty$ for all $|t|$ small. This is (2).

Item (3) follows Thm \ref{t.SGP}\eqref{item3}:
$\mathcal L\subset L^1(\mu)$, $1\in \mathcal L$, and $P\vf=\int\vf d\mu \cdot 1$, so $|\<\mu,\vf\>|\leq (\|P\|/\|1\|_{\mathcal L})\|\vf\|_{\mathcal L}$.

To see Nagaev's identity we use Lemma \ref{l.transfer-op-properties}(3) to check by induction that
$$
L_{it_{n-1}}\cdots L_{it_0} 1= L^{n}\bigg(\exp\sum_{\ell=0}^{n-1}it_\ell\psi\circ\sigma^\ell\bigg).
$$
Integrating with respect to $\mu$ and using the identity $L^\ast\mu=\mu$, we obtain that
$
\<\mu,L_{it_{n-1}}\cdots L_{it_0} 1\>=\int  e^{i\sum_{\ell=0}^{n-1}t_\ell\psi\circ\sigma^\ell}d\mu=\E_{\mu}
(e^{i\sum_{\ell=0}^{n-1}t_\ell\psi_\ell\circ\sigma^\ell})
$.
\end{proof}

The claim we just proved verifies the conditions of Gou\"ezel's almost sure invariance principle \cite[Thm 2.1]{Gouezel-ASIP}: The {characteristic function of the} $L^\infty$-bounded process $(\psi\circ\sigma^\ell)_{\ell\ge0}$ is coded by $(\mathcal L,(L_{it})_{t\in\RR},1,\mu)$ where $L_{it}$ has spectral gap (1) and bounded iterates (2).
Therefore Gou\"ezel's theorem yields a standard probability space $(\Omega,\mathcal F,m)$ with two  families $(\wt{S}_n)_{n\geq 1}, (\wt{B}_t)_{t\geq 0}:\Omega\to\R$
such that:
\begin{enumerate}[(1)]
\item $(\wt{S}_n)_{n\geq 1}=(\psi_n)_{n\geq 1}$ in distribution,
\item $(\wt{B}_t)_{t\geq 0}$ is a standard Brownian motion,
\item $\forall\eps>0$, $m\{\omega\in \Omega: |\wt{S}_n(\omega)-{\sigma_\psi} \wt{B}_n(\omega)|=o(n^{\frac{1}{4}+\eps})$ as $n\to\infty\}=1$.
\end{enumerate}

\medskip
This proves the one-sided aperiodic case.
We now consider the one-sided case, with period $p>1$. Let $\Sigma=\biguplus_{i=0}^{p-1}\sigma^i(\Sigma')$ denote the spectral decomposition.

By the ASIP in the aperiodic case {and Lemma \ref{l.Borelomania}}, there is a probability measure $\nu'$ on $\Sigma'\times [0,1]$ which projects to $\mu'(\cdot):=\mu(\cdot\cap\Sigma')/\mu(\Sigma)$, and there is a standard Brownian motion $(B_t')_{t\geq 0}:\Sigma'\times [0,1]\to\R$ such that for all $\eps>0$
$$
|\psi_{np}-\sigma' B_n'|=o(n^{\frac{1}{4}+\eps})\text{ a.e., }
$$
where $\sigma'$ is the asymptotic variance of $\psi_p$ with respect to $\sigma^{p}:\Sigma'\to\Sigma'$.

Let us define {$\sigma=\sigma'/\sqrt{p}$ and} $\nu:=\frac{1}{p}\sum_{j=0}^{p-1}\nu'\circ (\sigma^{j}\times \mathrm{id})^{-1}$ on $\Sigma\times [0,1]$. We also extend $(B_t')_{t\geq 0}$ to a standard Brownian motion on {$\Sigma\times[0,1]$} by setting
$$
B_t(x,\xi):=\sqrt{p} B_{t/p}'(y,\xi)\text{ for the unique $y\in\Sigma'$ such that $x=\sigma^j(y)$ with $0\leq j<p$}.
$$

Given $0\leq j\leq p-1$ and  $(x,\xi)\in \sigma^j(\Sigma')\times [0,1]$, write $x=\sigma^j(y)$. Since $\psi\in L^\infty$,
{\begin{align*}
&|\psi_{m p}(x)-\sigma B_{m p}(x,\xi)|\leq |\psi_{m p}(y)-\sigma' B_{m}'(y,\xi)|+O(1)=o(n^{\frac{1}{4}+\eps}).
\end{align*}}
For a.e. path of Brownian motion,
$
|B_{n+j}-B_{n}|=o(n^{\frac{1}{4}})\text{ for all $n$ large enough},
$
(because   $\sum_{n\geq 1}\Prob(|B_{n+j}-B_{n}|>n^{1/5})<\infty$). It follows that for $\nu$-a.e. $(x,\xi)$,
\begin{align*}
&|\psi_{mp+j}(x)-\sigma B_{mp+j}(x,\xi)|\leq |\psi_{mp}(y)-\sigma' B_{m}'(y,\xi)|+O(1)+o(m^{\frac{1}{4}+\eps})=o(m^{\frac{1}{4}+\eps}).
\end{align*}
So $|\psi_{n}(x)-\sigma B_{n}(x,\xi)|=o(n^{\frac{1}{4}+\eps})$ $\nu$-a.s.

This completes the proof of the ASIP in the one-sided case.
The two-sided case follows immediately from Sinai's Lemma (Thm~\ref{t.Sinai-Lemma}).
\end{proof}

\noindent
The consequences of the ASIP stated in \S\ref{s.FCLT}--\ref{s.records} are explained in Appendix~\ref{appendix-ASIP}. 
\subsection{Effective Intrinsic Ergodicity}
\begin{theorem}[R\"uhr-Sarig~\cite{Ruhr-Sarig}]\label{t.effective-ergodicity-shift}
Let $\Sigma$ be a one-sided or two-sided irreducible SPR Markov shift with finite Gurevich entropy, and with an MME  $\mu$.
For any $\beta>0$,  there are $\eps^\ast,C^*,K>0$ such that the following holds for every $\beta$-H\"older continuous function $\psi:\Sigma\to\R$:
\begin{enumerate}[(a)]
\item  If
     $0<\eps\leq \eps^\ast$, then  for every $\nu\in\mathbb P(\sigma)$ s.t. $h(\sigma,{\nu})>h(\sigma,\mu)-
     C^*\eps^2(\sigma_{\psi}/\|\psi\|_{{\beta}})^6$,
    $$
    |\mu(\psi)-\nu(\psi)|\leq e^{\eps}
    \sqrt{2\sigma_{\psi}^2(h(\sigma,\mu)-h(\sigma,\nu))}.
    $$
\item There exist $\nu_n\in\Proberg(\sigma)$ such that $h(\sigma,\nu_n)\neq h(\sigma, \mu)$, $h(\sigma,\nu_n)\to h(\sigma, \mu)$, and
    $$
    \frac{|\mu(\psi)-\nu(\psi)|}
    {\sqrt{2\sigma_{\psi}^2(h(\sigma,\mu)-h(\sigma, \nu))}}\xrightarrow[n\to\infty]{}
    1.
    $$
\item For every $\nu\in\mathbb P(\sigma)$,
    $
|\mu(\psi)-\nu(\psi)|\leq K\|\psi\|_{{\beta}}\sqrt{h(\sigma,\mu)-h(\sigma,\nu)}.
    $
\end{enumerate}
\end{theorem}

\subsection{No Phase Transitions In High Temperature}

\begin{theorem}\label{t.no-phase-transition}
{Let $\Sigma$ be a one-sided or two-sided irreducible Markov shift with finite Gurevich entropy, which is SPR for a
H\"older continuous potential $\phi$. Then
there exists $\eps_0>0$ such that for any H\"older function $\psi:\Sigma\to\R$ with $\sup|\psi|<\eps_0$,
$\Sigma$ is SPR for $\phi+\psi$ and $t\mapsto P_{\Bor}(\Sigma,\phi+t\psi)$ is real-analytic on  $(-1,1)$.}
\end{theorem}
\begin{proof}
Looking at \eqref{e.SPR-potential2}, it is easy to see that the SPR property is stable under perturbation with small supremum norm of the potential.
This gives the first property. The second one
follows from \cite[Thm 1.1]{Cyr-Sarig}.
\end{proof}

\section{Consequences of the Almost Sure Invariance Principle}\label{appendix-ASIP}
Donsker and Strassen realized that many well-known properties of the Brownian motion can be generalized to other stochastic processes, using the approximation by Brownian motion given by the ASIP, which they invented. In this appendix, we recall how to obtain  Cor~\ref{c.FCLT-Diffeos}--\ref{c.Records} this way. See \cite{Billingsley} and references therein.

\indent Throughout this appendix, $f$ denotes a measurable map on a standard probability space $(X,\mathfs B,\mu)$ and $\psi\in L^1(\mu)$ is a function such that $\int\psi d\mu=0$.
\footnote{In the case which interests us in this paper, $\sup|\psi|<\infty$, and the condition $\int\psi d\mu=0$ follows from the ASIP and the strong law of large numbers for Brownian motion.}
 We will say that {\em $(X,f,\mu,\psi)$ satisfies the ASIP with parameter $\sigma$ and rate $o(n^\gamma)$}, if
$$
S_n(x):=\psi_n(x):=\psi(x)+\psi(f(x))+\cdots+\psi(f^{n-1}(x))
$$
satisfies the ASIP with the same parameter and rate (Def \ref{def-ASIP}). Specifically, by Remark \ref{r-dynamical-ASIP}, there is a Borel probability measure $\nu$ on $X\times [0,1]$ which projects to $\mu$, and there is a standard Brownian motion $B_t(x,\xi)$ on $(X\times [0,1],\nu)$ such that
\begin{equation}\label{e.C.1}
|S_n(x)-\sigma B_n(x,\xi)|=o(n^\gamma)\text{ as $n\to\infty$, $\nu$-a.e. in $X\times [0,1]$.}
\end{equation}
(We do not assume that  $f$ preserves the measure $\mu$.)

\subsection{Functional Central Limit Theorem (Cor~\ref{c.FCLT-Diffeos})}
Recall that  $\ov{B}$ denotes a random Brownian path (Example \ref{e.Wiener}), and $\ov{\psi}_n$ denotes the normalized linear interpolation of $(0,\psi(x),\psi_2(x),\ldots)$ (Example \ref{e.linear-interp}).

\begin{theorem}\label{thm-FCLT-from-ASIP}
Suppose  $(X,\mu,f,\psi)$ satisfies the ASIP with parameter $\sigma$ and rate $o(n^\gamma)$, $0<\gamma<\frac{1}{2}$. Then 
$
\overline{\psi}_n\xrightarrow[n\to+\infty]{}\sigma \overline{B}
$ in distribution.
\end{theorem}

Before giving a proof, we note the following.

\begin{lemma}\label{e.BM-fact}
Let $(B_t)_{t\geq 0}$ be a standard Brownian motion on $(\Omega,\mathfs F,m)$. Then
$$
\displaystyle\sup_{k=0,\ldots,n}\  \sup_{t\in [k,k+1]}\ \ \tfrac{1}{\sqrt n}|B_t(\omega)-B_k(\omega)|\xrightarrow[n\to\infty]{}0
\;\;\;\text{ $m$-a.e.}
$$
\end{lemma}
\begin{proof} Fix $\eps>0$ small and let  $p_k:=\displaystyle\Prob\big(|B_t-B_k|>\eps{\sqrt{k}}:\; \text{for some }t\in [k,k+1]\big)$. Since $(B_t-B_k)_{t\geq k}$ is a Brownian motion, $p_k=\Prob\big(|B_t|>\eps\sqrt{k}:\; \text{for some }t\in [0,1]\big)$.

By the reflection principle for the Brownian motion,  $
p_k \leq 4\Prob\big(B_1>\eps\sqrt{k}\big)$. Since $B_1$ is Gaussian with standard deviation one,
$$
p_k\leq \frac{4}{\sqrt{2\pi}}\int_{\eps\sqrt{k}}^\infty e^{-t^2/2}dt\leq \frac{4}{\sqrt{2\pi}}\int_{\eps\sqrt{k}}^\infty \frac{te^{-t^2/2}}{\eps\sqrt{k}}dt=O(e^{-\eps^2k/2}).
$$
So $\sum p_k<\infty$. By the Borel-Cantelli Lemma, for a.e. $\omega$, for all $k$ large enough, $\sup_{t\in [k,k+1]}|B_t(\omega)-B_k(\omega)|\leq \eps\sqrt{k}$.  The lemma follows, since $\eps$ is arbitrary.
\end{proof}

\begin{proof}
By construction, $\overline{\psi}_n(x)(k/n)=\psi_k(x)/\sqrt{n}$. By \eqref{e.C.1},
for $\nu$-a.e. $(x,\xi)\in X\times [0,1]$,
$
\displaystyle\sup_{k=0,1,\ldots,n}|\overline{\psi}_n(x)(\tfrac{k}{n})-\tfrac{\sigma}{\sqrt{n}}B_k(x,\xi)|=o(1)\text{ as }n\to+\infty.
$

Define $\wt{B}_n(x,\xi)(\cdot)\in \mathcal C([0,1])$ by
$ \wt{B}_n(x,\xi)(t):=\frac{1}{\sqrt{n}}B_{tn}(x,\xi)$.
Note that
$\wt{B}_n(x,\xi)(\cdot)$ is a standard Brownian motion for each $n$.
By Lemma \ref{e.BM-fact}
$$\sup_{t\in [0,1]}|\overline\psi_n(x)(t)-\sigma\wt{B}_n(x,\xi)(t)|
\xrightarrow[n\to\infty]{}0, \ \ \nu\text{-a.e.}
$$

Now fix some  bounded uniformly continuous function $g:\mathcal{C}([0,1])\to \R$.  Then
$
|g(\overline{\psi}_n(x))-g(\sigma \wt{B}_n(x,\xi))|\xrightarrow[n\to\infty]{}0\text{ $\nu$-a.e.}
$
By the bounded convergence theorem, and since $\nu$ projects to $\mu$,
$$
\lim\limits_{n\to+\infty}\E_\mu[g(\overline{\psi}_n)]=\lim\limits_{n\to \infty}\E_\nu[g(\sigma\wt{B}_n)].$$
$\wt{B}_n$ is a standard Brownian motion on $[0,1]$, so
$\E_\nu[g(\sigma\wt{B}_n)]=\E[g(\sigma\overline{B})]$ for all $n$.
In summary,
 $\E_\mu[g(\overline{\psi}_n)]\to \E_{\mu_W}[g(\sigma \overline{B})]$ for all bounded {\em uniformly} continuous $g:\mathcal C([0,1])\to\R$.  By \cite[Thm 2.1]{Billingsley}, this must also be the case for all bounded continuous functions $g:\mathcal C([0,1])\to\R$. So $\overline{\psi}_n\xrightarrow[n\to\infty]{}\sigma\overline{B}$ in distribution.
 \end{proof}

\subsection{Law of the Iterated Logarithm (Cor~\ref{c.LIL})}

{\begin{theorem}\label{thm-LIL-from-ASIP}
Suppose $(X,\mu,f,\psi)$ satisfies the ASIP with  $\sigma\neq 0$ and rate $o(n^\gamma)$, $0<\gamma<\frac{1}{2}$.
Then for any $c\in (0,1)$ the following holds for $Z_n=\psi_n$:
\begin{align}
&\limsup_{n\to\infty}\frac{Z_n}{\sigma\sqrt{2n\log\log n}}=1, \
\liminf_{n\to\infty}\frac{Z_n}{\sigma \sqrt{2n\log\log n}}=-1, \label{e.LIL1-Z}\\
&\limsup_{N\to\infty} \tfrac{1}{N}\#\left\{1\leq n\leq N: Z_n>c \sigma\sqrt{2n\log\log n}\right\}=1-e^{-4(c^{-2}-1)} \label{e.LIL2-Z}.
\end{align}
\end{theorem}}

\begin{proof}
Suppose $Z_n=X_1+\cdots+X_n$, where $X_i$ are independent random variables with centered Gaussian distribution and standard deviation $\sigma$. Then
 \eqref{e.LIL1-Z} holds a.s. by the Hartman-Wintner theorem, and \eqref{e.LIL2-Z} holds a.s., by Strassen's law of the iterated logarithm~\cite{Strassen}.

Suppose $Z_n=\sigma B_n$, where $B_t$ is a standard Brownian motion. Looking at the decomposition  $B_n=\sum_{k=1}^n (B_k-B_{k-1})$, we see that $Z_n$  is a sum of Gaussian random variables as above. Therefore  \eqref{e.LIL1-Z} and \eqref{e.LIL2-Z} hold with $Z_n=\sigma B_n$.

If \eqref{e.LIL1-Z} and \eqref{e.LIL2-Z} hold for $Z_n$, they also hold for $Z_n+o(n^{\gamma})$.
By \eqref{e.C.1},
$\psi_n(x)=\sigma B_n+o(n^{\gamma})$ a.s.,  where $B_n$ is a standard Brownian motion. Hence \eqref{e.LIL1-Z} and \eqref{e.LIL2-Z} hold for $\mu$-a.e. $x$, with $Z_n=\psi_n(x)$.
\end{proof}

\subsection{Arcsine Law (Cor~\ref{c.Arcsine})}
{\begin{theorem}\label{thm-Arcsine-from-ASIP}
Suppose  $(X,\mu,f,\psi)$ satisfies the ASIP with $\sigma\neq 0$ and rate $o(n^\gamma)$, $0<\gamma<\frac{1}{2}$.  Let $ d_n(x):=\frac{1}{n}\#\{1\leq k\leq n: \psi_k(x)>0\}$,  then
$$\lim_{n\to\infty}\mu\{x\in X: d_n(x)\leq s\}=\tfrac{2}{\pi}\arcsin(\sqrt{s}).$$
\end{theorem}}

The proof will use the following fact.
Let $\lambda$ denote Lebesgue's measure on $[0,1]$, and let $\mu_W$ denote Wiener's measure (see Example \ref{e.Wiener}).

\begin{lemma}\label{l.BM-zero-set}
In the notation of Example \ref{e.Wiener}, for $\mu_W$-a.e. $\omega$ and for every $\eps>0$ there is a ${\delta=}\delta(\omega,{\eps})>0$
such that $\lambda\{0<t<1:|B_t(\omega)|<\delta\}<\epsilon.$
\end{lemma}
\begin{proof}
It is well-known that for a.e. Brownian path,  $\lambda(\{t:B_t(\omega)=0\})=0$ (see~\cite[\S7.4.1]{durrett}). Therefore, for $\mu_W$-a.e. $\omega\in \mathcal C([0,1])$, the Lebesgue measure of
$\{t:|B_t(\omega)|<\delta\}$ tends to zero, as $\delta\to 0$.
\end{proof}

\begin{proof}[Proof of Thm \ref{thm-Arcsine-from-ASIP}]
L\'evy's
Arcsine law for Brownian motion~\cite{revuz-yor} says that for all $s$ in $[0,1]$,
$
\mu_W\bigl\{\omega: \lambda\{t\in [0,1]: B_t(\omega)>0\}\leq s  \bigr\}=\frac{2}{\pi}\arcsin\sqrt{s}.
$
By Lemma \ref{l.BM-zero-set},
\begin{equation}\label{e.soft-arcsine}
\begin{aligned}
&\lim_{\delta\to 0}\mu_W\bigl\{\omega: \lambda\{t\in [0,1]: B_t(\omega)>-\delta\}\leq s  \bigr\}=\tfrac{2}{\pi}\arcsin\sqrt{s}.
\end{aligned}
\end{equation}

\indent Without loss of generality, $\sigma=1$. By \eqref{e.C.1},
given any $\delta>0$, for $\nu$-a.e. $(x,\xi)$, for all $n$ large enough and $1\leq k\leq n$,
$|\psi_k(x)-B_k(x,\xi)|<\delta n^{\frac{1}{2}}$. For such $n$,
\begin{align*}
d_n(x)&:=\tfrac{1}{n}\#\{1\leq k\leq n:\psi_k(x)>0\}=\tfrac{1}{n}\#\{1\leq k\leq n:\tfrac{1}{\sqrt{n}}\psi_k(x)>0\}\\
&\leq \tfrac{1}{n}\#\{1\leq k\leq n:\tfrac{1}{\sqrt{n}}B_k(x,\xi)>-\delta\}\\
&\leq \tfrac{1}{n}\lambda\{0\leq t\leq n:\tfrac{1}{\sqrt{n}}B_t(x,\xi)>-2\delta\}+o(1)\quad \text{ by Lemma \ref{e.BM-fact}}\\
&\leq \lambda\{0\leq t\leq 1:\tfrac{1}{\sqrt{n}}B_{tn}(x,\xi)>-2\delta\} +o(1).\\
\text{Hence, }  \liminf_{n\to\infty}&\mu\{x\in X: d_n(x)\leq s\}=\liminf_{n\to\infty}\nu\{(x,\xi)\in X\times [0,1]: d_n(x)\leq s\}\\
&\geq \liminf_{n\to\infty}\nu\big\{(x,\xi): \lambda\{0\leq t\leq 1:\tfrac{1}{\sqrt{n}}B_{tn}(x,\xi)>-2\delta\}\leq s+o(1)\big\}\\
&=\mu_W\{\omega: \lambda\{0\leq t\leq 1: B_t(\omega)>-2\delta\}\leq s+o(1)\},
\end{align*}
because $\frac{1}{\sqrt{n}}B_{tn}(x,\xi)$ is a standard Brownian motion on $(X\times [0,1],\nu)$.

Invoking \eqref{e.soft-arcsine}, and passing to the limit $\delta\to 0$, we obtain
$$
\liminf_{n\to\infty}\mu\{x\in X: d_n(x)\leq s\}\geq \tfrac{2}{\pi}\arcsin\sqrt{s}.
$$
Similarly, one shows that $
\limsup\limits_{n\to\infty}\mu\{x\in X: d_n(x)\leq s\}\leq \frac{2}{\pi}\arcsin\sqrt{s}
$.
\end{proof}

\subsection{Law of Records (Cor~\ref{c.Records})}
\begin{theorem}\label{thm-Records-from-ASIP}
Suppose $(X,\mu,f,\psi)$ satisfies the ASIP
with $\sigma\neq 0$ and rate $o(n^\gamma)$, $0<\gamma<\frac{1}{2}$. Then for any $s>0$,
$$
\mu\left\{x\in X: \tfrac{1}{\sqrt{n}}\max_{1\leq k\leq n}\psi_k(x)\geq s\right\}\xrightarrow[n\to\infty]{}\sqrt{\tfrac{2}{\pi\sigma^2}}
\int_s^\infty e^{-t^2/2\sigma^2}dt.
$$
\end{theorem}
\begin{proof}
We use notations $\ov{B}$, $\ov{\psi}_n$ from Examples \ref{e.Wiener} and \ref{e.linear-interp}.
Let $g:\mathcal C([0,1])\to \R$ denote  the continuous function $g(\omega):=\max\{\omega(t):0\leq t\leq 1\}$. Let $h:\R\to\R$ be a bounded uniformly continuous function. Then $h\circ g: \mathcal C([0,1])\to \R$ is a bounded {uniformly} continuous function, and
by the functional CLT,
$$
\displaystyle\lim_{n\to\infty}\E_\mu[(h\circ g)(\ov{\psi}_n)]=\E_{\mu_W}[(h\circ g)(\sigma\ov{B})].
$$

The left hand side equals $\lim\limits_{n\to\infty}\E_\mu[h(\max\{\frac{\psi_k(x)}{\sqrt{n}}:0\leq k\leq n\})]$, with $\psi_0(x):=0$. Since 
$|\psi_1(x)/\sqrt{n}|\to 0$
and $h$ is bounded and uniformly continuous,
$$
\lim\limits_{n\to\infty}\E_\mu\left[h(\max\{\tfrac{\psi_k(x)}{\sqrt{n}}:1\leq k\leq n\})\right]=\E_{\mu_W}[(h\circ g)(\sigma\ov{B})]
.
$$
The portmanteau theorem \cite[Thm~2.1]{Billingsley} now tells us that
$$
\mu\{x\in X: \max\{\tfrac{\psi_k(x)}{\sqrt{n}}:k=1,\ldots,n\}\geq s\}\xrightarrow[n\to\infty]{} \mu_W\bigg[\max_{t\in [0,1]}\sigma\ov{B}_t(\omega)\geq s\bigg].
$$
We now recall that by  the reflection principle for Brownian motion,
$$
\mu_W\bigg[\max_{t\in [0,1]}\sigma\ov{B}_t(\omega)\geq s\bigg]=2\times \frac{1}{\sqrt{2\pi}}\int_{s/\sigma}^\infty e^{-\tau^2/2}d\tau.
$$
Changing variables $t=\sigma \tau$, we obtain the law of records.
\end{proof}

\small
\bibliographystyle{plain-like-initial} \bibliography{Omri's-bib-file}{}

\def\cprime{$'$} \def\cprime{$'$} \def\cprime{$'$} \def\cprime{$'$}
\begin{thebibliography}{100}

\bibitem{Abdenur-Crovisier}
F.~Abdenur and S.~Crovisier.
\newblock Transitivity and topological mixing for {$C^1$} diffeomorphisms.
\newblock In {\em Essays in mathematics and its applications}, pages 1--16.
  Springer, Heidelberg, 2012.

\bibitem{Abramov-Rokhlin}
L.~M. Abramov and V.~A. Rohlin.
\newblock Entropy of a skew product of mappings with invariant measure.
\newblock {\em Vestnik Leningrad. Univ.} \textbf{17} (1962), 5--13.

\bibitem{Adler-Shields-Smorodinsky}
R.~L. Adler, P.~Shields, and M.~Smorodinsky.
\newblock Irreducible {M}arkov shifts.
\newblock {\em Ann. Math. Statist.} \textbf{43} (1972), 1027--1029.

\bibitem{Adler-Weiss-Similarity-Toral-Automorphisms}
R.~L. Adler and B.~Weiss.
\newblock {\em Similarity of automorphisms of the torus}.
\newblock Memoirs of the American Mathematical Society \textbf{98}. AMS,
  Providence, R.I., 1970.

\bibitem{Alvez-Luzzatto-Pinheiro}
J.~F. Alves, S.~Luzzatto, and V.~Pinheiro.
\newblock Markov structures and decay of correlations for non-uniformly
  expanding dynamical systems.
\newblock {\em Ann. Inst. H. Poincar\'{e} C Anal. Non Lin\'{e}aire} \textbf{22}
  (2005), 817--839.

\bibitem{andersson-vasquez}
M.~Andersson and C.~H. V\'{a}squez.
\newblock Statistical stability of mostly expanding diffeomorphisms.
\newblock {\em Ann. Inst. H. Poincar\'{e} Anal. Non Lin\'{e}aire} \textbf{37}
  (2020), 1245--1270.

\bibitem{Anosov-Geodesic-Flows}
D.~V. Anosov.
\newblock {\em Geodesic flows on closed {R}iemann manifolds with negative
  curvature.}
\newblock Proceedings of the Steklov Institute of Mathematics \textbf{90}
  (1967). AMS, Providence, R.I., 1969.

\bibitem{Araujo-Lima-Poletti}
E.~Araujo, Y.~Lima, and M.~Poletti.
\newblock {\em Symbolic dynamics for nonuniformly hyperbolic maps with
  singularities in high dimension}.
\newblock Memoirs of the AMS \textbf{301}. AMS, 2024.

\bibitem{Castro}
A.~Armando~de Castro~J\'unior.
\newblock Fast mixing for attractors with a mostly contracting central
  direction.
\newblock {\em Ergodic Theory Dynam. Systems} \textbf{24} (2004), 17--44.

\bibitem{Barreira-Pesin-Non-Uniform-Hyperbolicity-Book}
L.~Barreira and Y.~Pesin.
\newblock {\em Nonuniform hyperbolicity}, volume 115 of {\em Encyclopedia of
  Mathematics and its Applications}.
\newblock Cambridge University Press, Cambridge, 2007.

\bibitem{Ben-Ovadia-Codings}
S.~Ben~Ovadia.
\newblock Symbolic dynamics for non-uniformly hyperbolic diffeomorphisms of
  compact smooth manifolds.
\newblock {\em J. Mod. Dyn.} \textbf{13} (2018), 43--113.

\bibitem{Ben-Ovadia-coded-set}
S.~Ben~Ovadia.
\newblock The set of points with {M}arkovian symbolic dynamics for
  non-uniformly hyperbolic diffeomorphisms.
\newblock {\em Ergodic Theory Dynam. Systems} \textbf{41} (2021), 3244--3269.

\bibitem{Berger-Henon}
P.~Berger.
\newblock Abundance of non-uniformly hyperbolic henon-like endomorphisms.
\newblock {\em Ast{\'e}risque} \textbf{410} (2019), 53--177.

\bibitem{Billingsley}
P.~Billingsley.
\newblock {\em Convergence of probability measures}.
\newblock Wiley Series in Probability and Statistics: Probability and
  Statistics. John Wiley \& Sons, Inc., New York, second edition, 1999.

\bibitem{bochi-viana}
J.~Bochi and M.~Viana.
\newblock The {L}yapunov exponents of generic volume-preserving and symplectic
  maps.
\newblock {\em Ann. of Math.} \textbf{161} (2005), 1423--1485.

\bibitem{bonatti-crovisier-connecting}
C.~Bonatti and S.~Crovisier.
\newblock R\'ecurrence et g\'en\'ericit\'e.
\newblock {\em Inventiones Mathematicae} \textbf{158} (2004), 33--104.

\bibitem{Bonatti-Diaz}
C.~Bonatti and L.~J. D\'{\i}az.
\newblock Persistent nonhyperbolic transitive diffeomorphisms.
\newblock {\em Ann. of Math.} \textbf{143} (1996), 357--396.

\bibitem{Bonatti-Viana}
C.~Bonatti and M.~Viana.
\newblock S{RB} measures for partially hyperbolic systems whose central
  direction is mostly contracting.
\newblock {\em Israel J. Math.} \textbf{115} (2000), 157--193.

\bibitem{Bowen-MP-Axiom-A}
R.~Bowen.
\newblock Markov partitions for {A}xiom {${\rm A}$} diffeomorphisms.
\newblock {\em Amer. J. Math.} \textbf{92} (1970), 725--747.

\bibitem{Bowen-LNM}
R.~Bowen.
\newblock {\em Equilibrium states and the ergodic theory of {A}nosov
  diffeomorphisms}.
\newblock Lecture Notes in Mathematics, Vol. 470. Springer-Verlag, Berlin,
  1975.

\bibitem{BoyleBuzziGomez2014}
M.~Boyle, J.~Buzzi, and R.~G\'{o}mez.
\newblock Borel isomorphism of {SPR} {M}arkov shifts.
\newblock {\em Colloq. Math.} \textbf{137} (2014), 127--136.

\bibitem{Burguet-Cr}
D.~Burguet.
\newblock Maximal measure and entropic continuity of {L}yapunov exponents for
  {${C}^r$} surface diffeomorphisms with large entropy.
\newblock {\em Ann. Henri Poincar\'{e}} \textbf{25} (2024), 1485--1510.

\bibitem{Buzzi-PQFT}
J.~Buzzi.
\newblock Puzzles of quasi-finite type, zeta functions and symbolic dynamics
  for multi-dimensional maps.
\newblock {\em Ann. Inst. Fourier} \textbf{60} (2010), 801--852.

\bibitem{Buzzi-SIM}
J.~Buzzi.
\newblock The intrinsic ergodicity of smooth interval maps.
\newblock {\em Israel J. Math.} \textbf{100} (1997), 125--161.

\bibitem{BuzziNoMax}
J.~Buzzi.
\newblock {$C^r$} surface diffeomorphisms with no maximal entropy measure.
\newblock {\em Ergodic Theory Dynam. Systems} \textbf{34} (2014), 1770--1793.

\bibitem{BCS-2}
J.~Buzzi, S.~Crovisier, and O.~Sarig.
\newblock Continuity of {L}yapunov exponents and entropy for surface
  diffeomorphisms.
\newblock {\em Inventiones Mathematicae} \textbf{230} (2022), 767--849.

\bibitem{BCS-1}
J.~Buzzi, S.~Crovisier, and O.~Sarig.
\newblock Measures of maximal entropy for surface diffeomorphisms.
\newblock {\em Annals of Math.} \textbf{195} (2022), 421--508.

\bibitem{Buzzi-Fisher-2013}
J.~Buzzi and T.~Fisher.
\newblock Entropic stability beyond partial hyperbolicity.
\newblock {\em J. Mod. Dyn.} \textbf{7} (2013), 527--552.

\bibitem{Buzzi-Fisher-Sambarino-Vasquez-2012}
J.~Buzzi, T.~Fisher, M.~Sambarino, and C.~Vasqu\`ez.
\newblock Maximal entropy measures for certain partially hyperbolic, derived
  from anosov systems.
\newblock {\em Ergodic Theory and Dynamical Systems} \textbf{32} (2012),
  63--79.

\bibitem{Buzzi-Fisher-Tahzibi}
J.~Buzzi, T.~Fisher, and A.~Tahzibi.
\newblock A dichotomy for measures of maximal entropy near time-one maps of
  transitive anosov flows.
\newblock {\em Ann. Sci. Ec. Norm. Sup{\'e}r.} \textbf{55} (2022), 969--1002.

\bibitem{Buzzi-Sarig}
J.~Buzzi and O.~Sarig.
\newblock Uniqueness of equilibrium measures for countable {M}arkov shifts and
  multidimensional piecewise expanding maps.
\newblock {\em Ergodic Theory Dynam. Systems} \textbf{23} (2003), 1383--1400.

\bibitem{Climenhaga-Tower}
V.~Climenhaga.
\newblock Specification and towers in shift spaces.
\newblock {\em Commun. Math. Phys.} \textbf{364} (2018), 441--504.

\bibitem{Climenhaga-Thompson}
V.~Climenhaga and D.~J. Thompson.
\newblock Intrinsic ergodicity beyond specification: {$\beta$}-shifts, {S}-gap
  shifts, and their factors.
\newblock {\em Isr. J. Math.} \textbf{192} (2012), 785--817.

\bibitem{Crovisier-Poletti}
S.~Crovisier and M.~Poletti.
\newblock Invariance principle and non-compact center foliations.
\newblock arXiv:2210.14989.

\bibitem{Cyr-Sarig}
V.~Cyr and O.~Sarig.
\newblock Spectral gap and transience for {R}uelle operators on countable
  {M}arkov shifts.
\newblock {\em Comm. Math. Phys.} \textbf{292} (2009), 637--666.

\bibitem{Daon}
Y.~Daon.
\newblock Bernoullicity of equilibrium measures on countable {M}arkov shifts.
\newblock {\em Discrete Contin. Dyn. Syst.} \textbf{33} (2013), 4003--4015.

\bibitem{Dembo-Zeitouni}
A.~Dembo and O.~Zeitouni.
\newblock {\em Large deviations techniques and applications}, volume~38 of {\em
  Stochastic Modelling and Applied Probability}.
\newblock Springer-Verlag, Berlin, 2010.

\bibitem{Denker-Philipp}
M.~Denker and W.~Philipp.
\newblock Approximation by {B}rownian motion for {G}ibbs measures and flows
  under a function.
\newblock {\em Ergodic Theory Dynam. Systems} \textbf{4} (1984), 541--552.

\bibitem{didier}
P.~Didier.
\newblock Stability of accessibility.
\newblock {\em Ergodic Theory Dynam. Systems} \textbf{23} (2003), 1717--1731.

\bibitem{Dolgopyat}
D.~Dolgopyat.
\newblock On dynamics of mostly contracting diffeomorphisms.
\newblock {\em Comm. Math. Phys.} \textbf{213} (2000), 181--201.

\bibitem{Dolgopyat-Wilkinson}
D.~Dolgopyat and A.~Wilkinson.
\newblock Stable accessibility is {$C^1$} dense.
\newblock {\em Ast\'{e}risque} \textbf{287} (2003), xvii, 33--60.
\newblock Geometric methods in dynamics. II.

\bibitem{durrett}
R.~Durrett.
\newblock {\em Probability -- theory and example}, volume~49 of {\em Cambridge
  Series in Statistical and Probabilistic Mathematics}.
\newblock Cambridge University Press, 2019.

\bibitem{ELPV}
M.~Einsiedler, E.~Lindenstrauss, P.~Michel, and A.~Venkatesh.
\newblock The distribution of closed geodesics on the modular surface, and
  {D}uke's theorem.
\newblock {\em Enseign. Math.} \textbf{58} (2012), 249--313.

\bibitem{Fisher-Potrie-Sambarino-2014}
T.~Fisher, R.~Potrie, and M.~Sambarino.
\newblock Dynamical coherence of partially hyperbolic diffeomorphisms of tori
  isotopic to anosov.
\newblock {\em Math. Z.} \textbf{278} (2014), 149--168.

\bibitem{Franks69}
J.~Franks.
\newblock Anosov diffeomorphisms on tori.
\newblock {\em Trans. Amer. Math. Soc.} \textbf{145} (1969), 117--124.

\bibitem{Ornstein-Friedman}
N.~A. Friedman and D.~S. Ornstein.
\newblock On isomorphism of weak {B}ernoulli transformations.
\newblock {\em Advances in Math.} \textbf{5} (1970), 365--394 (1970).

\bibitem{Gouezel-2}
S.~Gou\"{e}zel.
\newblock Regularity of coboundaries for nonuniformly expanding {M}arkov maps.
\newblock {\em Proc. Amer. Math. Soc.} \textbf{134} (2006), 391--401.

\bibitem{Gouezel-ASIP}
S.~Gou\"{e}zel.
\newblock Almost sure invariance principle for dynamical systems by spectral
  methods.
\newblock {\em Ann. Probab.} \textbf{38} (2010), 1639--1671.

\bibitem{GST}
S.~Gou\"ezel, B.~Schapira, and S.~Tapie.
\newblock Pressure at infinity and strong positive recurrence in negative
  curvature.
\newblock arXiv:2012.13226.

\bibitem{Guivarch-Hardy}
Y.~Guivarc'h and J.~Hardy.
\newblock Th\'eor\`emes limites pour une classe de cha\^\i nes de {M}arkov et
  applications aux diff\'eomorphismes d'{A}nosov.
\newblock {\em Ann. Inst. H. Poincar\'e Probab. Statist.} \textbf{24} (1988),
  73--98.

\bibitem{Gurevich-Topological-Entropy}
B.~M. Gurevic.
\newblock Topological entropy of a countable {M}arkov chain.
\newblock {\em Dokl. Akad. Nauk SSSR} \textbf{187} (1969), 715--718.

\bibitem{Gurevich-Measures-Of-Maximal-Entropy}
B.~M. Gurevic.
\newblock Shift entropy and {M}arkov measures in the space of paths of a
  countable graph.
\newblock {\em Dokl. Akad. Nauk SSSR} \textbf{192} (1970), 963--965.

\bibitem{Gurevich-Savchenko}
B.~M. Gurevich and S.~V. Savchenko.
\newblock Thermodynamic formalism for symbolic {M}arkov chains with a countable
  number of states.
\newblock {\em Uspekhi Mat. Nauk} \textbf{53} (1998), 3--106.

\bibitem{Gurevich-Zargaryan}
B.~M. Gurevich and A.~S. Zargaryan.
\newblock Conditions for the existence of a maximal measure for a countable
  symbolic {M}arkov chain.
\newblock {\em Vestnik Moskov. Univ. Ser. I Mat. Mekh.} \textbf{} (1988),
  14--18.

\bibitem{Gurevich-Stable}
B.~M. Gurevich.
\newblock Stably recurrent nonnegative matrices.
\newblock {\em Uspekhi Mat. Nauk} \textbf{51} (1996), 195--196.

\bibitem{Hofbauer-Tower}
F.~Hofbauer.
\newblock On intrinsic ergodicity of piecewise monotonic transformations with
  positive entropy. {II}.
\newblock {\em Israel J. Math.} \textbf{38} (1981), 107--115.

\bibitem{Iommi-Velozo-Cusp}
G.~Iommi, F.~Riquelme, and A.~Velozo.
\newblock Entropy in the cusp and phase transitions for geodesic flows.
\newblock {\em Israel J. Math.} \textbf{225} (2018), 609--659.

\bibitem{iommi-todd-velozo-2020}
G.~Iommi, M.~Todd, and A.~Velozo.
\newblock Upper semi-continuity of entropy in non-compact settings.
\newblock {\em Math. Res. Lett.} \textbf{27} (2020), 1055--1078.

\bibitem{Iommi-Todd-Velozo-Advances}
G.~Iommi, M.~Todd, and A.~Velozo.
\newblock Escape of entropy for countable {M}arkov shifts.
\newblock {\em Adv. Math.} \textbf{405} (2022), Paper No. 108507, 54.

\bibitem{Kadyrov-Effective-Uniqueness}
S.~Kadyrov.
\newblock Effective uniqueness of {P}arry measure and exceptional sets in
  ergodic theory.
\newblock {\em Monatsh. Math.} \textbf{178} (2015), 237--249.

\bibitem{Kato-Book}
T.~Kato.
\newblock {\em Perturbation theory for linear operators}.
\newblock Classics in Mathematics. Springer-Verlag, Berlin, 1995.

\bibitem{Katok79}
A.~Katok.
\newblock Bernoulli diffeomorphisms on surfaces.
\newblock {\em Ann. of Math.} \textbf{110} (1979), 529--547.

\bibitem{KatokIHES}
A.~Katok.
\newblock Lyapunov exponents, entropy and periodic orbits for diffeomorphisms.
\newblock {\em Inst. Hautes \'Etudes Sci. Publ. Math.} \textbf{51} (1980),
  137--173.

\bibitem{Katok-Hasselblatt-Book}
A.~Katok and B.~Hasselblatt.
\newblock {\em Introduction to the modern theory of dynamical systems},
  volume~54 of {\em Encyclopedia of Math and its Applications}.
\newblock Cambridge University Press, 1995.

\bibitem{Kechris}
A.~Kechris.
\newblock {\em Classical descriptive set theory}, volume 156 of {\em Graduate
  Texts in Mathematics}.
\newblock Springer-Verlag, New York, 1995.

\bibitem{Kifer}
Y.~Kifer.
\newblock Large deviations in dynamical systems and stochastic processes.
\newblock {\em Trans. Amer. Math. Soc.} \textbf{321} (1990), 505--524.

\bibitem{Kitchens-Book}
B.~P. Kitchens.
\newblock {\em Symbolic dynamics}.
\newblock Universitext. Springer-Verlag, Berlin, 1998.

\bibitem{Komlos-Major-Tusnady}
J.~Koml\'{o}s, P.~Major, and G.~Tusn\'{a}dy.
\newblock An approximation of partial sums of independent {${\rm RV}$}'s and
  the sample {${\rm DF}$}. {I}.
\newblock {\em Z. Wahrscheinlichkeitstheorie und Verw. Gebiete} \textbf{32}
  (1975), 111--131.

\bibitem{Lima-1D}
Y.~Lima.
\newblock Symbolic dynamics for one dimensional maps with nonuniform expansion.
\newblock {\em Ann. Inst. H. Poincar\'{e} C Anal. Non Lin\'{e}aire} \textbf{37}
  (2020), 727--755.

\bibitem{Martens-Liverani}
C.~Liverani and M.~Martens.
\newblock Convergence to equilibrium for intermittent symplectic maps.
\newblock {\em Comm. Math. Phys.} \textbf{260} (2005), 527--556.

\bibitem{Livsic}
A.~N. Liv\v{s}ic.
\newblock Cohomology of dynamical systems.
\newblock {\em Izv. Akad. Nauk SSSR Ser. Mat.} \textbf{36} (1972), 1296--1320.

\bibitem{Mane}
R.~Ma\~{n}\'{e}.
\newblock {\em Ergodic theory and differentiable dynamics}, volume~8 of {\em
  Ergebnisse der Mathematik und ihrer Grenzgebiete (3)}.
\newblock Springer-Verlag, Berlin, 1987.

\bibitem{Manning74}
A.~Manning.
\newblock There are no new {A}nosov diffeomorphisms on tori.
\newblock {\em Amer. J. Math.} \textbf{96} (1974), 422--429.

\bibitem{Martinchich}
S.~Martinchich.
\newblock Global stability of discretized anosov flows.
\newblock {\em J. Mod. Dyn.} \textbf{19} (2023), 561--623.

\bibitem{Melbourne-Nicol0}
I.~Melbourne and M.~Nicol.
\newblock Almost sure invariance principle for nonuniformly hyperbolic systems.
\newblock {\em Comm. Math. Phys.} \textbf{260} (2005), 131--146.

\bibitem{Melbourne-Nicol-ASIP}
I.~Melbourne and M.~Nicol.
\newblock A vector-valued almost sure invariance principle for hyperbolic
  dynamical systems.
\newblock {\em Ann. Probab.} \textbf{37} (2009), 478--505.

\bibitem{Misiurewicz}
M.~Misiurewicz.
\newblock Diffeomorphism without any measure with maximal entropy.
\newblock {\em Bull. Acad. Polon. Sci. S\'er. Sci. Math. Astronom. Phys.}
  \textbf{21} (1973), 903--910.

\bibitem{Mongez-Pacifico}
J.~C. Mongez and M.~J. Pacifico.
\newblock Finite measures of maximal entropy for an open set of partially
  hyperbolic diffeomorphisms.
\newblock {\em arXiv:2401.02776} \textbf{}.

\bibitem{Newhouse70}
S.~E. Newhouse.
\newblock On codimension one {A}nosov diffeomorphisms.
\newblock {\em Amer. J. Math.} \textbf{92} (1970), 761--770.

\bibitem{Newhouse-Homoclinic}
S.~E. Newhouse.
\newblock Hyperbolic limit sets.
\newblock {\em Trans. Amer. Math. Soc.} \textbf{167} (1972), 125--150.

\bibitem{Newhouse-Entropy}
S.~E. Newhouse.
\newblock Continuity properties of entropy.
\newblock {\em Ann. of Math.} \textbf{129} (1989), 215--235.

\bibitem{Parry-Intrinsic-MC}
W.~Parry.
\newblock Intrinsic {M}arkov chains.
\newblock {\em Trans. Amer. Math. Soc.} \textbf{112} (1964), 55--66.

\bibitem{Parry-Pollicott-Asterisque}
W.~Parry and M.~Pollicott.
\newblock Zeta functions and the periodic orbit structure of hyperbolic
  dynamics.
\newblock {\em Ast\'erisque} \textbf{187-188} (1990), 268.

\bibitem{Pesin-Izvestia-1976}
J.~B. Pesin.
\newblock Families of invariant manifolds that correspond to nonzero
  characteristic exponents.
\newblock {\em Izv. Akad. Nauk SSSR Ser. Mat.} \textbf{40} (1976), 1332--1379,
  1440.

\bibitem{Pollicott-Livsic}
M.~Pollicott.
\newblock Local {H}\"{o}lder regularity of densities and {L}ivsic theorems for
  non-uniformly hyperbolic diffeomorphisms.
\newblock {\em Discrete Contin. Dyn. Syst.} \textbf{13} (2005), 1247--1256.

\bibitem{Pratelli-2007}
A.~Pratelli.
\newblock On the equality between {M}onge's infimum and {K}antorovich's minimum
  in optimal mass transportation.
\newblock {\em Ann. Inst. H. Poincar\'e Probab. Statist.} \textbf{43} (2007),
  1--13.

\bibitem{revuz-yor}
D.~Revuz and M.~Yor.
\newblock {\em Continuous martingales and {B}rownian motion}, volume 293 of
  {\em Grundlehren der mathematischen Wissenschaften}.
\newblock Springer-Verlag, Berlin, 1999.

\bibitem{Young-Rey-Bellet}
L.~Rey-Bellet and L.-S. Young.
\newblock Large deviations in non-uniformly hyperbolic dynamical systems.
\newblock {\em Ergodic Theory Dynam. Systems} \textbf{28} (2008), 587--612.

\bibitem{Riquelme-Velozo}
F.~Riquelme and A.~Velozo.
\newblock Escape of mass and entropy for geodesic flows.
\newblock {\em Ergodic Theory Dynam. Systems} \textbf{39} (2019), 446--473.

\bibitem{RodriguezHertz-RodriguezHertz-Tahzibi-Ures}
F.~Rodriguez~Hertz, M.~A. Rodriguez~Hertz, A.~Tahzibi, and R.~Ures.
\newblock Maximizing measures for partially hyperbolic systems with compact
  center leaves.
\newblock {\em Ergodic Theory Dynam. Systems} \textbf{32} (2012), 825--839.

\bibitem{Rodriguez-Hertz-Squared-Tahzibi-Ures-CMP}
F.~Rodriguez~Hertz, M.~Rodriguez~Hertz, A.~Tahzibi, and R.~Ures.
\newblock Uniqueness of {SRB} measures for transitive diffeomorphisms on
  surfaces.
\newblock {\em Communications in Mathematical Physics} \textbf{306} (2011),
  35--49.

\bibitem{Ruelle-TDF-book}
D.~Ruelle.
\newblock {\em Thermodynamic formalism}, volume~5 of {\em Encyclopedia of
  Mathematics and its Applications}.
\newblock Addison-Wesley Publishing Co., Reading, Mass., 1978.

\bibitem{Ruette}
S.~Ruette.
\newblock On the {V}ere-{J}ones classification and existence of maximal
  measures for countable topological {M}arkov chains.
\newblock {\em Pacific J. Math.} \textbf{209} (2003), 366--380.

\bibitem{Ruhr-Sarig}
R.~R\"{u}hr and O.~Sarig.
\newblock Effective intrinsic ergodicity for countable state topological markov
  shifts.
\newblock {\em Israel J. Math.} \textbf{251} (2022), 679--735.

\bibitem{Sarig-ETDS-99}
O.~M. Sarig.
\newblock Thermodynamic formalism for countable {M}arkov shifts.
\newblock {\em Ergodic Theory Dynam. Systems} \textbf{19} (1999), 1565--1593.

\bibitem{Sarig-CMP-2001}
O.~M. Sarig.
\newblock Phase transitions for countable {M}arkov shifts.
\newblock {\em Comm. Math. Phys.} \textbf{217} (2001), 555--577.

\bibitem{Sarig-Bernoulli-JMD}
O.~M. Sarig.
\newblock Bernoulli equilibrium states for surface diffeomorphisms.
\newblock {\em J. Mod. Dyn.} \textbf{5} (2011), 593--608.

\bibitem{Sarig-JAMS}
O.~M. Sarig.
\newblock Symbolic dynamics for surface diffeomorphisms with positive entropy.
\newblock {\em J. Amer. Math. Soc.} \textbf{26} (2013), 341--426.

\bibitem{ST}
B.~Schapira and S.~Tapie.
\newblock Regularity of entropy, geodesic currents and entropy at infinity.
\newblock {\em Ann. Sci. E.N.S.} \textbf{54} (2021), 1--68.

\bibitem{Sinai-Construction-of-MP}
J.~G. Sina{\u \i}.
\newblock Construction of {M}arkov partitionings.
\newblock {\em Funkcional. Anal. i Prilo{\v z}en.} \textbf{2} (1968), 70--80.

\bibitem{Sinai-MP-U-diffeomorphisms}
J.~G. Sina{\u \i}.
\newblock Markov partitions and {U}-diffeomorphisms.
\newblock {\em Funkcional. Anal. i Prilo{\v z}en} \textbf{2} (1968), 64--89.

\bibitem{Sinai-Gibbs}
J.~G. Sina{\u \i}.
\newblock Gibbs measures in ergodic theory.
\newblock {\em Uspehi Mat. Nauk} \textbf{27} (1972), 21--64.

\bibitem{Smale}
S.~Smale.
\newblock Differentiable dynamical systems.
\newblock {\em Bull. Amer. Math. Soc.} \textbf{73} (1967), 747--817.

\bibitem{Strassen}
V.~Strassen.
\newblock An invariance principle for the law of the iterated logarithm.
\newblock {\em Z. Wahrscheinlichkeitstheorie und Verw. Gebiete} \textbf{3}
  (1964), 211--226.

\bibitem{Tahzibi-Yang}
A.~Tahizibi and J.~Yang.
\newblock Invariance principle and rigidity of high entropy measures.
\newblock {\em Trans. Amer. Math. Soc.} \textbf{371} (2019), 1231--1251.

\bibitem{Takahashi}
Y.~Takahashi.
\newblock Isomorphisms of {$\beta $}-automorphisms to {M}arkov automorphisms.
\newblock {\em Osaka Math. J.} \textbf{10} (1973), 175--184.

\bibitem{Ures-2012}
R.~Ur\`es.
\newblock Intrinsic ergodicity of partially hyperbolic diffeomorphisms with a
  hyperbolic linear part.
\newblock {\em Proc. Amer. Math. Soc.} \textbf{140} (2012), 1973--1985.

\bibitem{Vere-Jones-Geometric-Ergodicity}
D.~Vere-Jones.
\newblock Geometric ergodicity in denumerable {M}arkov chains.
\newblock {\em Quart. J. Math. Oxford Ser.} \textbf{13} (1962), 7--28.

\bibitem{Walters-Book}
P.~Walters.
\newblock {\em An introduction to ergodic theory}, volume~79 of {\em Graduate
  Texts in Mathematics}.
\newblock Springer-Verlag, New York, 1982.

\bibitem{JYang-u-entropy}
J.~Yang.
\newblock Entropy along expanding foliations.
\newblock arXiv:1601.05504.

\bibitem{Kakutani-Yosida}
K.~Yosida and S.~Kakutani.
\newblock Birkhoff's ergodic theorem and the maximal ergodic theorem.
\newblock {\em Proc. Imp. Acad. Tokyo} \textbf{15} (1939), 165--168.

\bibitem{Young-Towers-Annals}
L.-S. Young.
\newblock Statistical properties of dynamical systems with some hyperbolicity.
\newblock {\em Ann. of Math.} \textbf{147} (1998), 585--650.

\bibitem{young-tower-recurrence}
L.-S. Young.
\newblock Recurrence times and rates of mixing.
\newblock {\em Israel J. Math.} \textbf{110} (1999), 153--188.

\bibitem{Zang-P-C}
Y.~Zang.
\newblock Personal communication.

\end{thebibliography}
\bigskip

\hspace{-2.5cm}
\begin{tabular}{l l l l l}
\emph{J\'er\^ome Buzzi}
& &
\emph{Sylvain Crovisier}
& &
\emph{Omri Sarig}
\\

Laboratoire de Math\'ematiques d'Orsay
&& Laboratoire de Math\'ematiques d'Orsay
&& Faculty of Mathematics\\
CNRS - UMR 8628
&& CNRS - UMR 8628
&&  and Computer Science\\
Universit\'e Paris-Saclay
&&  Universit\'e Paris-Saclay
&& The Weizmann Institute of Science\\
Orsay 91405, France
&& Orsay 91405, France
&& Rehovot, 7610001,  Israel\\

\end{tabular}

\end{document}